\documentclass[openany]{memo-l}
\usepackage{amssymb}
\usepackage{pb-diagram,lamsarrow,pb-lams}
\usepackage{pst-all}

\theoremstyle{plain}
\newtheorem{thm}{Theorem}[section]
\newtheorem{cor}[thm]{Corollary}
\newtheorem{lem}[thm]{Lemma}
\newtheorem{prop}[thm]{Proposition}
\newtheorem{step}{Step}

\newtheorem{conj}[thm]{Conjecture}

\newtheorem{example}{Example}

\newtheorem*{thm_n}{Theorem}

\theoremstyle{definition}
\newtheorem*{defn}{Definition}

\theoremstyle{remark}

\newtheorem*{rem_n}{Remark}
\newtheorem*{conv_n}{Convention}

\numberwithin{section}{chapter}

\newcommand{\thmref}[1]{Theorem~\ref{#1}}

\newcommand{\lemref}[1]{Lemma~\ref{#1}}
\newcommand{\propref}[1]{Proposition~\ref{#1}}
\newcommand{\stepref}[1]{Step~\ref{#1}}

\newcommand{\corref}[1]{Corollary~\ref{#1}}
\newcommand{\exref}[1]{Example~\ref{#1} on page~\pageref{#1}}

\def\PP{{\mathbb P}}
\def\RR{{\mathbb R}}

\def\QQ{{\mathbb Q}}
\def\NN{{\mathbb N}}
\def\ZZ{{\mathbb Z}}

\def\FF{{\mathbb F}}

\def\AA{{\mathbb A}}
\def\BB{{\mathbb B}}

\def\LLL{{\mathcal L}}
\def\FFF{{\mathcal F}}
\def\EEE{{\mathcal E}}
\def\BBB{{\mathcal B}}

\def\UUU{{\mathcal U}}

\def\AAA{{\mathcal A}}
\def\SSS{{\mathcal S}}

\def\WWW{{\mathcal W}}
\def\XXX{{\mathcal X}}

\def\III{{\mathcal I}}

\def\CCC{{\mathcal C}}

\def\EEE{{\mathcal E}}
\def\GGG{{\mathcal G}}
\def\KKK{{\mathcal K}}

\def\MMM{{\mathcal M}}

\def\TTT{{\mathcal T}}

\def\kkkk{{\mathfrak K}}

\def\nqed{\renewcommand{\qed}{}}
\def\blackqed{\renewcommand{\qedsymbol}{$\blacksquare$}}

\newcommand{\booktitle}[1]{\textit{#1}}

\newcommand{\scal}[1]{\left\langle#1\right\rangle}
\newcommand{\setof}[1]{\left\{#1\right\}}

\def\lqs{\leqslant}
\def\gqs{\geqslant}

\def\nabla{\triangledown}

\def\im{\mathop{\rm im}\nolimits}
\def\ker{\mathop{\rm ker}\nolimits}
\def\coker{\mathop{\rm coker}\nolimits}
\def\card{\mathop{\rm card}\nolimits}
\def\ord{\mathop{\rm ord}\nolimits}
\def\inv{\mathop{\rm inv}\nolimits}
\def\Id{\mathop{\rm Id}\nolimits}
\def\id{\mathop{\rm id}\nolimits}
\def\pr{\mathop{\rm pr}\nolimits}
\def\Aut{\mathop{\rm Aut}\nolimits}
\def\dim{\mathop{\rm dim}\nolimits}

\def\rank{\mathop{\rm rank}\nolimits}
\def\Hom{\mathop{\rm Hom}\nolimits}
\def\Ext{\mathop{\rm Ext}\nolimits}

\def\Ph{\mathop{\rm Ph}\nolimits}
\def\Compon{\mathop{\rm Compon}\nolimits}

\def\const{{\rm const}}

\def\tim{\!\times\!}

\def\eps{\varepsilon}

\def\H{\Tilde H}

\def\+{\oplus}

\newenvironment{abc}[0]{

\begin{enumerate}
}{
\end{enumerate}

}

\newenvironment{enroman}[0]{

\begin{enumerate}
}{
\end{enumerate}

}
\newcounter{ovaj}
\newcommand{\navedi}[1]{\setcounter{ovaj}{#1}({\bf\roman{ovaj}})}

\def\colim{\qopname\relax n{colim}}

\def\ll{\lim{}^{\!\!1\,}}
\newcommand{\lm}[1]{\lim_{#1}{}^{\!\!1\,}}

\newgray{gray0}{.90}
\newgray{gray1}{.82}
\newgray{gray2}{.74}
\newgray{gray3}{.66}
\newgray{gray4}{.58}
\newgray{gray5}{.5}

\definecolor{brown}{rgb}{.65, .16, .16}
\definecolor{lightblue}{rgb}{.68, .85, .9}
\definecolor{palegreen}{rgb}{.6, .98, .6}
\definecolor{pink}{rgb}{1, .75, .8}
\definecolor{wheat}{rgb}{.96, .87, .7}
\definecolor{softred}{rgb}{1,.2,0}

\newcommand{\prevthmref}[1]{Theorem~\ref{#1} of \cite{smrekar}}

\newcommand{\prevlemref}[1]{Lemma~\ref{#1} of \cite{smrekar}}

\begin{document}

\frontmatter
%%%%%%%%%%%%%%%%%%%%%%%%%%%%%%%%%%%%%%%%%%%%%%%%%%%%%%%%%%%%%%%%%%%%%%%%

\title{CW Type of Inverse Limits and Function Spaces}
\author{Jaka Smrekar}
\address{Fakulteta za matematiko in fiziko, Univerza v Ljubljani, Jadranska ulica 19, SI-1111 Ljubljana, Slovenia}
\email{jaka.smrekar@fmf.uni-lj.si}

\thanks{Part of this work was developed while the author was a Marie Curie pre-doctoral fellow at the
	Centre de Recerca Matem\`{a}tica, Barcelona.}
\thanks{The author was supported in part by the European Community and the CRM, Barcelona under contract
	HPMT-CT-2000-00075, as well as by the Ministry for Education, Science and Sport of the Republic of Slovenia
	Research Program No. 101-509.}
%\date{\today}

%%%%%%%%%%%%%%%%%%%%%%%%%%%%%%%%%%%%%%%%%%%%%%%%%%%%%%%%%%%%%%%%%%%%%%%%%

\begin{abstract}
We continue the investigation of CW homotopy type of spaces of continuous functions
between two CW complexes begun by J.~Milnor in 1959 and P.~Kahn in 1984. Viewing function spaces as particular cases of
inverse limits we also study certain inverse systems of fibrations between CW homotopy type spaces.

If the limit space $Z_\infty$ of an inverse sequence $\setof{Z_i}$ of fibrations between CW type spaces has
CW type then a subsequence of $\setof{\Omega Z_i}$ splits into the product of a sequence of homotopy equivalences and
one of nullhomotopic maps.

If for some $N>0$, all spaces $Z_i$ have $\pi_k(Z_i)=0$ for $k>N$, then
the question of CW type of $Z_\infty$ depends solely on the induced functions $\pi_k(Z_j)\to\pi_k(Z_i)$.
This applies to $Z_i=Y^{L_i}$ where $\pi_k(Y)=0$ for $k>N$ and $\setof{L_i}$ is an
ascending sequence of finite complexes. Here $Z_\infty=Y^{\cup L_i}$, the space of continuous
functions $(\cup L_i)\to Y$ with the compact open topology.

In general, if the path component of $g\in Y^X$ has CW type then
$\Omega(Y^X,g)\to\Omega(Y^L,g\vert_L)$ is a homotopy equivalence for a countable subcomplex $L$ of $X$.
A suitable converse holds as well.

Function spaces of CW type lack phantom phenomena in a strong sense. This provides interesting
examples. One is a space of pointed maps that is weakly contractible but not contractible.
Next, for $X$ the localization of a finite complex at a set of primes $P$,
the question of CW type of $Y^X$ is related, and sometimes equivalent,
to that of eventual geometric H-space exponents of $Y$.

If $Y$ is a $P$-local and $X$ a simply connected complex then localization
$X\to X_{(P)}$ induces a genuine homotopy equivalence $Y^{X_{(P)}}\to Y^X$ regardless of whether $Y^X$ has
CW type or not.

For $Y=K(G,n)$ we give necessary and sufficient conditions for $Y^X$ to have CW type in terms of
homology of $X$. If $\oplus_n\pi_n(Y)$ is finitely generated and $X$ is $1$-connected
then we give necessary and `almost' sufficient conditions.

Some properties of CW complexes $X$ are equivalent to $Y^X$ having CW homotopy
type for a certain family of complexes $Y$. For example, $X$ is finitely dominated if and only if $\pi_1(X)$
is finitely presented and $Y^X$ has CW type for all complexes $Y$.
\end{abstract}

\maketitle

\tableofcontents

%%%%%%%%%%%%%%%%%%%%%%%%%%%%%%%%%%%%%%%%%%%%%%%%%%%%%%%%%%%%%%%%%%%%%%%%%%%%%%%%%
%										%
%				Introduction					%
%										%
%%%%%%%%%%%%%%%%%%%%%%%%%%%%%%%%%%%%%%%%%%%%%%%%%%%%%%%%%%%%%%%%%%%%%%%%%%%%%%%%%

\chapter*{Introduction}\label{intro}

\setcounter{thm}{0}

For topological spaces $X$ and $Y$ we denote by $Y^X$ the space of continuous functions $X\to Y$ with the compact
open topology.

One of the first results on CW homotopy type of function spaces is a result of Kuratowski \cite{kuratowski},
saying that if $Y$ is an absolute neighbourhood retract for the class of metric spaces and $K$ is a compact
metric space, then $Y^K$ is also an absolute neighbourhood retract for the class of metric spaces. (Later
it was shown that $Y$ has the homotopy type of a countable CW complex if and only if it has the homotopy
type of an absolute neighbourhood retract for the class of metric spaces; see Milnor \cite{milnor} for details.)

In his fundamental study of spaces having the homotopy type of a CW complex, Milnor \cite{milnor} strengthened
Kuratowski's result to show that $Y^X$ has CW homotopy type of a CW complex if $Y$ has and $X$ is a compact Hausdorff
space. Some effort has already been made to extend Milnor's theorem. P.J.Kahn \cite{kahn} proved that $Y^X$ has the
homotopy type of a CW complex when $Y$ has the homotopy type of a CW complex with $\pi_k(Y)$ trivial for $k\gqs n+1$
and $X$ has the homotopy type of a CW complex with finite $n$-skeleton.

We are interested in both sufficient and necessary conditions on CW complexes $X$ and $Y$ for the function space
$Y^X$ to have CW homotopy type. The restriction that $X$ be a CW complex (or have homotopy type of one) is motivated
by Kahn's result and is natural for homotopy theory.

Previous work of the author \cite{smrekar} on this problem shows that the conditions of Kahn are not necessary
and gives other examples of function spaces of CW homotopy type.

To develop sufficient conditions for the space $Y^X$ to have CW homotopy type, we view $Y^X$ as the limit of
inverse system $\setof{Y^L\,\vert\,L\in\LLL}$ where $\LLL$ is a suitable set of subcomplexes of $X$, and the bond for
$L\lqs L'$ is the restriction fibration $Y^{L'}\to Y^L$. Notably, choosing $\LLL$ to be the set $\KKK$ of finite
subcomplexes of $X$, the associated inverse system is one of spaces of CW homotopy type, by Milnor's theorem.

This inverse system approach makes it natural to consider the problem in a slightly more general setting.
Namely, if $Z$ is the limit space of an inverse system $\setof{Z_\lambda\,\vert\,\lambda}$ of fibrations
between CW homotopy type spaces, when has $Z$ the homotopy type of a CW complex?

Milnor has shown that CW homotopy type is intimately related to local contractibility notions.
One of them is equilocal convexity. Milnor introduced it in \cite{milnor} to show that if a space is paracompact
and equilocally convex, then it has the homotopy type of a CW complex. He used this to prove his fundamental result
concerning CW homotopy type of function spaces.

To investigate necessary conditions for a topological space to have the homotopy type of a CW complex it is natural
to consider a much weaker form of local contractibility, namely semilocal contractibility, since it is
homotopy invariant.

The lack of semilocal contractibility has been used a few times to demonstrate that spaces do not have the homotopy
type of a CW complex; see for example Skljarenko \cite{skljarenko} as well as Dydak and Geoghegan \cite{dyd-geog}.
In presence of additional properties, semilocal contractibility can be a necessary and sufficient condition for
a space to have CW homotopy type, see Allaud \cite{allaud} and Sch\"{o}n \cite{schon}.

In tune with the cited existing results on CW homotopy type of function spaces, we pursue the intuition that CW
complex $X$ must be `small' with respect to CW complex $Y$ for the space $Y^X$ to have CW type. This intuition is
justified in various ways in this paper.

{\it Acknowledgement.} This work was begun at the University of Ljubljana and completed at the Centre de Recerca
Matem\`{a}tica of Barcelona while on a one-year leave. The author wishes to express sincere appreciation to both
institutions for making this possible.

%%%%%%%%%%%%%%%%%%%%%%%%%%%%%%%%%%%%%%%%%%%%%%%%%%%%%%%%%%%%%%%%%%%%%%%%%%%%%%%%%%%%%%%%%%
%
%	Overview of results
%
%%%%%%%%%%%%%%%%%%%%%%%%%%%%%%%%%%%%%%%%%%%%%%%%%%%%%%%%%%%%%%%%%%%%%%%%%%%%%%%%%%%%%%%%%%

\section*{Overview of results}
In Chapter \ref{slc_cw_type} we treat CW homotopy type of certain inverse systems of fibrations.
Let $X$ be a CW complex. The set $\XXX$ of all subcomplexes of $X$ is a complete lattice,
and the subset $\KKK$ of finite subcomplexes is a sublattice with the property that every $L\in\XXX$
can be expressed as the supremum of a subset of $\KKK$.
 
In \S\ref{inverse_systems} and \S\ref{sufficient_necessary} we study CW homotopy type of limits of inverse systems
indexed over a complete lattice that admits an appropriate sublattice over which the system is one of CW homotopy
type spaces. We give necessary conditions for CW type which serve as our main obstruction throughout the paper,
and some sufficient conditions. The developed results are mostly results of technical nature to be applied to
function spaces in subsequent sections. 

However, the results obtained for sequences in \S\ref{results_on_sequences} are important independently of applications
to function spaces. Consider an inverse sequence $\dots\to W_3\to W_2\to W_1$ of fibrations between spaces of CW homotopy
type. We show the following. If the limit space $W_\infty$ also has CW homotopy type, then a subsequence of the sequence
of loop spaces $\setof{\Omega W_i}$ splits into the product of a sequence of nullhomotopic fibrations and a sequence of
homotopy equivalences. Conversely, if $\setof{P_i}$ is an inverse sequence of fibrations that are nullhomotopic, and
$\setof{Q_i}$ is an inverse sequence of fibrations that are homotopy equivalences, then the limit space of the sequence
$\setof{P_i\times Q_i}$ is homotopy equivalent to $Q_1$ (and has CW homotopy type if $Q_1$ has).

Thus the limit space of an inverse sequence of fibrations between spaces of CW type has itself CW type if and
\-- up to a looping \-- only if it is trivial in a certain sense; see \thmref{inverse_sequence} below.
In particular, this settles a gap in Theorem B of Dydak and Geoghegan \cite{dyd-geog}.

In \S\ref{phantom_components} we introduce the notion of `phantom path components' of the limit space $Z$ of
an inverse system of fibrations $\setof{Z_\lambda\,\vert\,\lambda}$ between spaces of CW type. In the particular
case when $Z=Y^X$ is the limit of $\setof{Y^L\,\vert\,L\text{ finite }\lqs X}$ this notion coincides with path
components of phantom maps $X\to Y$ with respect to the class of finite subcomplexes of $X$.

We show that if $Z$ has CW type then the systems of groups $\setof{\pi_k(Z_\lambda)\,\vert\,\lambda}$ satisfy
the Mittag-Leffler condition `uniformly with respect to $k$' (see \thmref{u_M-L_general}). In particular, $Z$
has no `phantom components'.

In \S\ref{sequences_of_postnikov_sections} we consider sequences of fibrations $\dots\to Y_3\to Y_2\to Y_1$ of spaces
of CW type with $\pi_k(Y_i)=0$ for $k>N$ independently of $i$. We show that, roughly, (see \thmref{postnikov_sequence})
the limit space $Y_\infty$ has CW homotopy type if and only if for each $k\gqs 1$ the sequence $\setof{\pi_k(Y_i)\,\vert\,i}$
satisfies the Mittag-Leffler condition and the canonical morphisms $\lim_j\pi_k(Y_j)\to\pi_k(Y_i)$ are injective for all
but finitely many $i$, and all $k\gqs 0$.

In Chapter \ref{function_spaces} we specialize to spaces of continuous functions between CW complexes.

In \S\ref{restriction} we define a family $\FFF(X,Y)$ of subspaces of $Y^X$ that contains `all function spaces arising
in practice' and is suitable for application of machinery developed in \S\ref{inverse_systems}. We give two simple-minded
examples of function spaces not of CW type.

In \S\ref{characterizations} we apply the machinery of \S\ref{inverse_systems} to spaces of functions between CW complexes
$X$ and $Y$.

We show that $Y^X$ has CW homotopy type if and only if it admits a numerable covering of open sets contractible
within $Y^X$ (see \thmref{fibre_SLC}). Such a covering exists in every paracompact semilocally contractible space.
For countable $X$, the space $Y^X$ is always paracompact, but it may lack semilocal contractibility.

Next, if $Y^X$ has CW homotopy type, then for every map $g\colon X\to Y$ and every countable subcomplex $L$ of $X$
there exists a bigger countable subcomplex $L'$ (depending on $g$ and $L$) such that the map
\[ \Omega(Y^X,g)\to\Omega(Y^{L'},g\vert_{L'}) \] is a homotopy equivalence (see \thmref{double_wicked}).
This is to say that if $Y^X$ has CW type, then $X$ is essentially {\it countable with respect to $Y$}.

Conversely, if for each countable subcomplex $L$ of $X$ there exists a bigger countable subcomplex $L'$ such that
$Y^X\to Y^{L'}$ is a weak equivalence onto image which has CW type, then $Y^X$ itself has CW type.

Both \thmref{fibre_SLC} and \thmref{double_wicked} are appropriately stated for arbitrary elements of $\FFF(X,Y)$.

A particular case of \thmref{wicked} is the following. If for each countable subcomplex $L$ of $X$ there exists a larger
countable subcomplex $L'$ such that $(Y,*)^{(L',*)}$ is contractible, then $(Y,*)^{(X,*)}$ is contractible.

In \S\ref{counting} we study restrictions on the set of path components of $Y^X$ of CW type, an issue which is
delicate because this set does not in general admit a natural group structure. We show that (\propref{homotopy_sets_1})
if $X$ has $\aleph_\xi$ cells and all homotopy groups of $Y$ have cardinality at most $\aleph_\eta$ then $Y^X$ has
at most $\aleph_{\max\setof{\xi,\eta}}$ path components.

If $X$ is countable and $Y$ has finite homotopy groups then (\propref{homotopy_sets_2}) $Y^X$ has actually finitely
many path components.

In \S\ref{phantom} we note that results of \S\ref{phantom_components} imply that if $Y^X$ has CW type then it
does not contain nontrivial phantom maps. In fact there are no nontrivial phantom maps $X\to\Omega^kY$ for any $k$,
in a `uniform way' with respect to $k$.

We give an example of a function space with essential phantom maps (which consequently does not have CW type)
and study its properties. Then we recall some examples from existing literature on phantom maps. In particular,
it follows from a famous theorem of Zabrodsky \cite{zabrodsky} that for simply connected CW complexes $X$ and $Y$
where $X$ has finitely many nontrivial homotopy groups and $Y$ is a finite complex the space $Y^X$ does not have
CW type if $(Y,*)^{(X,*)}$ is not weakly contractible.

In \S\ref{maps_into_postnikov_n_ads} we apply results of \S\ref{characterizations} and
\S\ref{sequences_of_postnikov_sections} to show that, roughly, for $Y$ a CW complex with finitely many nontrivial
homotopy groups the question of whether $Y^X$ has CW type only depends on the induced morphisms
$\pi_k(Y^M)\to\pi_k(Y^L)$ for $L\lqs M\lqs X$. (See \thmref{v}.) In particular, vanishing of phantom phenomena
together with an additional condition suffice for CW homotopy type in this case.

For $X$ a countable complex the statement becomes very simple.
Let $L_1\lqs L_2\lqs\dots$ be a filtration of finite subcomplexes for $X$ and assume that $Y$ has only finitely many
nontrivial homotopy groups. Then the path component $C$ of $g\in Y^X$ has CW homotopy type if and only if the
sequences $\setof{\pi_k(Y^{L_i},g\vert_{L_i})\,\vert\,i}$, for $k\gqs 1$, satisfy the Mittag-Leffler condition,
and the arrows $\lim_j\pi_k(Y^{L_j},g\vert_{L_j})\to\pi_k(Y^{L_i},g\vert_{L_i})$ are injective for all but finitely
many $i$, and all $k\gqs 0$ (for $k=0$ in the usual sense). In particular, if the homotopy groups of $Y$ are countable,
then $(Y,*)^{(X,*)}$ is contractible if and only if it is weakly contractible. The latter is `far from true' for
$Y$ with infinitely many nontrivial homotopy groups, see \exref{weakly_null}.

In \S\ref{general_constructions} we present a few general sufficient conditions for a function space to have CW
homotopy type; which we use in subsequent sections. In particular, in \propref{postnikov_adjunction} we recover
Theorem 1.1 of Kahn \cite{kahn}.

In Chapter \ref{localization_chapter} we investigate some effects of localization of both domain and target complexes.

In \S\ref{zabrodsky_ext} we prove an auxiliary result used later on; a strengthened `Zabrodsky lemma'.
The well known `Zabrodsky lemma' (see Miller \cite{miller}) says that if $E\to B$ is a fibration with fibre $F$
and for some CW complex $X$ the function space $(X,*)^{(F,*)}$ is weakly contractible, then
$X^B\to X^E$ is a weak homotopy equivalence. Upon assuming that $(X,*)^{(F,*)}$ is contractible, we observe that
then $X^B\to X^E$ is a genuine homotopy equivalence.

In \S\ref{local_target} we consider localizations of the target complex $Y$. We prove (see \thmref{localization})
that if $Y$ is a $P$-local CW complex with respect to a set of primes $P$ then the natural map
$(Y,*)^{(X_{(P)},*)}\to(Y,*)^{(X,*)}$ is a homotopy equivalence for any simply connected CW complex $X$.

In \S\ref{local_domain} we study the space of functions from $K_{(P)}$ where $K$ is the suspension of a finite
complex and $P$ is a set of primes. Applying  \S\ref{phantom_components}
we show that, given a CW complex $Y$ of finite type over $\ZZ_{(R)}$ where $P\cap R$ is nonempty, the space
$(Y,*)^{(K_{(\PP\setminus P)},*)}$ has CW type only if $(Y_{(P)},*)^{(K_{(\PP\setminus P)},*)}$ is contractible
and the H-space $(\Omega Y_{(P)},*)^{(K,*)}$ has an H-space exponent. We show that a partial converse holds in
that $(Y_{(P)},*)^{(K_{(\PP\setminus P)},*)}$ is contractible if $(Y_{(P)},*)^{(K,*)}$ has an H-space exponent.
(See \thmref{finite_localization}.) Applying this to spheres in place of $K$ (see \propref{f_l_spheres}) we relate
the problem to the Moore conjecture on homotopy (and H-space) exponents. We provide a restatement of the
(geometric part of the) conjecture in the language of CW homotopy type of function spaces.

Relying on results
from H-space exponent theory we give non-trivial examples of both function spaces that have CW homotopy type, and of such
that do not.

Relying on Serre's theorem on torsion in homotopy groups of spheres we show that
for any $m,n\gqs 2$, the space $Z=(S^n,*)^{(S^m_{(0)},*)}$ does not have CW homotopy type. In particular,
if $m>n$ then $Z$ is a weakly contractible space which is not contractible (indeed this holds for all $\Omega^kZ$,
$k\gqs 0$).

In Chapter \ref{eilenberg-maclane_target} we investigate spaces of maps into Eilenberg-MacLane spaces.

In \S\ref{thom_section} we note (see \thmref{thom_enhanced}) that the homotopy type of $K(G,n)^X$ (for abelian $G$)
depends only on homology groups of $X$, a result whose weak version was already established by Thom \cite{thom}.
 
According to \S\ref{thom_section} it suffices to study spaces of functions $M(A,m)\to K(G,n)$ where $M(A,m)$
denotes a Moore space of type $(A,m)$. In \S\ref{maps_moore_eilenberg} we give sharpened versions of results
of \S\ref{characterizations} for this particular case.

Applying almost everything of the preceding sections, we give in \S\ref{explicit_determinations} necessary and
sufficient conditions on abelian groups $A$ and $G$ for $K(G,n)^{M(A,m)}$ (see \thmref{moore_eilenberg}) respectively
$K(G,n)^{K(A,m)}$ (see \thmref{eilenberg_eilenberg}) to have CW homotopy type. The conditions fit well into the
intuition that `$A$ must be small with respect to $G$' in order for the mentioned function spaces to have CW type.

In \S\ref{aspherical_target} we complement \S\ref{thom_section} and \S\ref{explicit_determinations} to give
necessary and sufficient conditions for a function space $K(G,1)^X$ to have CW type, for arbitrary $G$.

The results of P.J.Kahn \cite{kahn} as well as those of the author \cite{smrekar}
exhibit function spaces $Y^X$ of CW homotopy type where the restriction
map $Y^X\to Y^L$, for a suitable finite subcomplex $L$ of $X$, is a homotopy
equivalence onto the union of some path components, reducing the problem to Milnor's theorem.
With this in mind one might conjecture that \-- at least for some CW decomposition of X \--
this is the only possible way for a function space $Y^X$ to have CW type.
This is not true, as we note in \exref{out_of_milnor} as a consequence of results of \S\ref{explicit_determinations}.

However, Milnor's result is best possible in a global sense, as we discuss in \S\ref{for_all_Y}
where we give necessary and sufficient conditions on $X$ for $Y^X$ to have CW homotopy type for all CW complexes
$Y$ belonging to certain classes. For example (see \thmref{milnor_forever}), a connected CW complex $X$
is finitely dominated if and only if $\pi_1(X)$ is finitely presented and $Y^X$ has CW type for {\it all} CW complexes $Y$.
For another example, $X$ has homological finite type if and only if $Y^X$ has CW type for all nilpotent CW complexes $Y$
of finite type.

In Chapter \ref{iff} we establish sufficient and `almost necessary' conditions for $Y^X$ to have the
homotopy type of a CW complex when $X$ is a simply connected CW complex, and $Y$ is
simply connected of finite type with only finitely many nontrivial homotopy groups. (See \thmref{monster}.)
Thereby we enhance Kahn's result in this case, and give a suitable partial converse.

In Appendix \ref{lifting_functions} we provide an explicit determination of a homotopy equivalence
arising from a fibration with fibre that contracts in the total space; a result that involves exhibiting
particular lifting functions for certain fibrations, and is (directly) applied once in the paper.
Furthermore we give the proof of a result (see \propref{mapping_space_covariant}) which also involves
lifting functions and is essentially applied once in the paper.

In Appendix \ref{quasi_groups} we introduce the notion of `quasitopological' group.
We use this essentially to circumvent difficulties arising from the fact that function spaces
in general are not compactly generated. The results of this appendix are of technical character
and are used on a single occasion to extend properties that hold for countable CW complexes
to the general case.

%%%%%%%%%%%%%%%%%%%%%%%%%%%%%%%%%%%%%%%%%%%%%%%%%%%%%%%%%%%%%%%%%%%%%%%%%%%%%%%%%%%%%%%%%%%%%%%%%
%												%
%		Conventions and basic tools							%
%												%
%%%%%%%%%%%%%%%%%%%%%%%%%%%%%%%%%%%%%%%%%%%%%%%%%%%%%%%%%%%%%%%%%%%%%%%%%%%%%%%%%%%%%%%%%%%%%%%%%

\section*{Conventions and basic tools}

Both domain and target spaces of our function spaces are assumed to be CW complexes (or have CW homotopy type).
Let $(X;A_1,\dots,A_n)$ be a CW $n$-ad and $(Y;B_1,B_2,\dots,B_n)$ of the homotopy type of a CW $n$-ad. We may
form the $n$-ad function space $(Y;B_1,B_2,\dots,B_n)^{(X;A_1,\dots,A_n)}$ of maps $f\colon X\to Y$ for which
$f(A_i)\subset B_i$ for all $i$. This $n$-ad function space is a subspace of $Y^X$ and belongs to the class
$\FFF(X,Y)$ discussed above and defined in \S\ref{restriction}. Some results, in particular those of more
topological nature in Chapter \ref{function_spaces}, are naturally
stated for $n$-ad function spaces. See also Milnor \cite{milnor}, Theorem 3.

However, in later chapters where we treat particular families of function spaces, and determine
explicit homological conditions, we consider spaces of free maps $Y^X$ and spaces of pointed maps $(Y,*)^{(X,*)}$,
for $X$ a CW complex and $Y$ a Hausdorff space of CW type. In this case the space $Y^X$ has
CW type if and only if the spaces $(Y,y_0)^{(X,*)}$ have, for $y_0$ ranging over different path components of $Y$,
see \corref{cor_stasheff_theorem} below. At the end of this section we note that for questions concerning $Y^X$ and
$(Y,y_0)^{(X,x_0)}$ it suffices to consider path-connected spaces and assume that $x_0$ is a $0$-cell of $X$ and
$y_0$ is a nondegenerate base point of $Y$.

When topological spaces are involved we use continuous function, function, and map synonymously. 

We use throughout the compact open topology.
While its compactly generated refinement (used for example by Kahn in \cite{kahn}) is more
appealing from a categorical point of view, it is not very useful when discussing properties
concerning open sets, such as semilocal contractibility. Fortunately when $X$ is a countable CW complex,
and $Y$ is any CW complex, the compact open topology on $Y^X$ and its refinement are homotopy equivalent,
see \prevlemref{stratifiable}. In general only our `positive' results can be applied to the compactly generated
refinement; if $Y^X$ has CW type then so has $\kkkk(Y^X)$ (see \cite{smrekar}, Remark 2.5).

We distinguish strictly between homotopy equivalent and {\it weakly} homotopy equivalent,
shortly just equivalent and weakly equivalent, respectively. Indeed, in this paper we are
concerned exactly with the difference between homotopy equivalences and weak homotopy equivalences,
since every topological space is {\it weakly} equivalent to a CW complex. In the course of
our investigation we prove some results which in their weak version are well known or obvious,
but not so in the genuine homotopy equivalence setting.

We use `type' for homotopy type.

The term fibration is used for a Hurewicz fibration, that is a not necessarily surjective map with
the homotopy lifting property with respect to all spaces.

We refer to Chapter 6 of Maunder's book \cite{maunder} for a concise treatment of basic facts regarding
the compact open topology that circumvents in a very nice way the problem of local compactness of the domain.
(Specific CW decompositions of our infinite domain complexes will almost never be locally compact.)

We use the following notation.
By $Y^X$ we denote the space of continuous functions $X\to Y$. Function spaces are endowed with the
compact open topology, and we use $G(K,V)$ for the standard subbasic open set,
$G(K,V)=\setof{f\in Y^X\,\vert\,f(K)\subset V}$, with $K$ compact and $V$ open.

The symbol $SX$ denotes the reduced suspension of the complex $X$ (with the standard CW structure).

We denote the $n$-skeleton of $X$ by $X^{(n)}$. When there is no danger of confusion, we denote by $Y_i$
the $i$-th stage of the Postnikov tower of $Y$.

The space of paths in a space $Z$ starting with a (sometimes tacitly understood) fixed base point is denoted by
$PZ$. The corresponding evaluation at endpoint is denoted by $\eps_1\colon PZ\to Z$.

We adopt the convention that all spaces considered are Hausdorff. For a brief explanation,
we want the domain space of a function space to be Hausdorff since every compact subspace
is normal. On the other hand, compact open topology is Hausdorff if the target space is.
Moreover, if two Hausdorff spaces are homotopy equivalent, then so also are their compactly
generated refinements. Occasionally we add the Hausdorff assumption to the statement of a proposition, for emphasis.

By $\ZZ_p$ we will denote the group of integers mod $p$, by $\ZZ_{p^\infty}$ the quasicyclic $p$-group
(divisible $p$-torsion group of rank $1$), and by $\Hat\ZZ_p$ the group of $p$-adic integers. For a short
exact sequence of groups $0\to B\to A\to Q\to 0$ we will say that $A$ is an extension of $B$ by $Q$
(in accord with Fuchs \cite{fuchs1}).

For a set of primes $P$ we will denote by $A_{(P)}$ the localization of $A$ at $P$ (for $A$ an abelian group
or a suitable CW complex).

We record here a handful of basic results on (CW) homotopy type of function spaces,
which will be extensively used in the paper.

\begin{lem}\label{very_first}
Let $X$ be a Hausdorff space.
\begin{enroman}
\item	If $Y$ is homeomorphic to $\prod_\lambda Y_\lambda$, then $Y^X$ is homeomorphic to
	$\prod_\lambda Y_\lambda^X$ and
	consequently $(Y,*)^{(X,*)}$ is homeomorphic to $\prod_\lambda(Y_\lambda,*_\lambda)^{(X,*)}$.
\item	If $X$ is homeomorphic to $\vee_{\lambda}X_\lambda$ then $(Y,*)^{(X,*)}$ is homeomorphic to
	the product $\prod_\lambda(Y,*)^{(X_\lambda,*_\lambda)}$. \qed
\end{enroman}
\end{lem}

\begin{prop}\label{homotopy_equivalence}
Let \[ \varphi\colon(A;A_1,\dots,A_n)\to(X;X_1,\dots,X_n)\text{ and }
\psi\colon(Y;Y_1,\dots,Y_n)\to(B;B_1,\dots,B_n) \] be maps of $n$-ads.
Define $\Phi(\varphi,\psi)\colon Y^X\to B^A$ by $f\mapsto\psi\circ f\circ\varphi$.
Then $\Phi(\varphi,\psi)$ is a map of $n$-ads
\begin{equation*} \tag{$*$} \Big(Y^X;(Y,Y_1)^{(X,X_1)},\dots,(Y,Y_n)^{(X,X_n)}\Big)\to
\Big(B^A;(B,B_1)^{(A,A_1)},\dots,(B,B_n)^{(A,A_n)}\Big). \end{equation*}
If there exist homotopies of $n$-ads $\varphi\simeq\varphi'$ and $\psi\simeq\psi'$, then
the maps $\Phi(\varphi,\psi)$ and $\Phi(\varphi',\psi')$ are homotopic as maps of $n$-ads ($*$).
Moreover, the homotopy sends the subspace $Y_i^X$ to the subspace $B_i^A$, for all $i$.
(The latter applies if we want to consider function spaces with constants as base points.)

In particular, if $\varphi$ admits a right (or double sided) homotopy inverse as a map of $n$-ads
and $\psi$ admits a left (or double sided) homotopy inverse as a map of $n$-ads, then $\Phi(\varphi,\psi)$
admits a left (respectively double sided) homotopy inverse as a map of $n$-ads, and induces a homotopy
domination (respectively equivalence)
\begin{equation*} (Y;Y_1,\dots,Y_n)^{(X;X_1,\dots,X_n)}\to(B;B_1,\dots,B_n)^{(A;A_1,\dots,A_n)}.
\end{equation*}
\end{prop}

\begin{proof}
See Maunder \cite{maunder}, Theorem 6.2.25, for $n=0$. The generalization is evident.
\end{proof}

The following is Theorem 3 of Milnor \cite{milnor}. In what follows we will simply call it `Milnor's theorem.'

\begin{thm}[Milnor's theorem]\label{milnor_theorem}
Let $(A;A_1,A_2,\dots,A_n)$ be a compact $n$-ad ($A$ is a compactum and the $A_i$ are closed in $A$)
and let $(Y;Y_1,Y_2,\dots,Y_n)$ have the homotopy type of a CW $n$-ad. Then
$\big(Y^X;(Y,Y_1)^{(X,X_1)},\dots,(Y,Y_n)^{(X,X_n)}\big)$ also has the homotopy type of a CW $n$-ad.
Consequently $(Y;Y_1,Y_2,\dots,Y_n)^{(A;A_1,A_2,\dots,A_n)}$ has the homotopy type of a CW complex. \qed
\end{thm}

A topological space is called {\it finitely dominated} if it is dominated by a finite CW complex
(see J.~Whitehead \cite{whitehead}), that is there exist a finite CW complex $K$ and continuous maps
$i\colon X\to K$ and $p\colon K\to X$ such that $p\circ i$ is homotopic to $\id_X$.

\begin{cor}\label{cor_milnor_theorem}
Let $X$ be finitely dominated. Then $Y^X$ has CW homotopy type for every space $Y$ of CW homotopy type.
\end{cor}
\begin{proof}
If $X$ is dominated by finite complex $K$, then $Y^X$ is dominated by $Y^K$ by \propref{homotopy_equivalence}.
The space $Y^K$ has CW homotopy type by \thmref{milnor_theorem}. Then $Y^X$ is dominated by a CW complex and
hence has CW homotopy type by Theorem 23 of Whitehead \cite{whitehead}.
\end{proof}

We note that using \propref{homotopy_equivalence} and \thmref{milnor_theorem} to full extent, together with Theorem 2 of
Milnor \cite{milnor} concerning homotopy type of $n$-ads, the above corollary generalizes trivially to function
spaces of $n$-ads.

\begin{lem}\label{wedge_splitting}
If $X$ is homotopy equivalent to $\vee_\lambda X_\lambda$ in the pointed sense,
then $(Y,*)^{(X,*)}$ is homotopy equivalent to $\prod_\lambda(Y,*)^{(X_\lambda,*_\lambda)}$.

Consequently if $(Y,*)^{(X,*)}$ has CW type, so has $(Y,*)^{(X_\lambda,*_\lambda)}$, for each $\lambda$.
If $\Lambda$ is finite, then the converse also holds.
\end{lem}
\begin{proof}
The first statement follows from \navedi{2} of \lemref{very_first} and \propref{homotopy_equivalence},
the second statement follows from Proposition 3 of Milnor \cite{milnor}.
\end{proof}

\begin{lem}\label{pushout-pullback}
Consider the following topological pushout diagram.
\begin{equation*}\tag{$*$}\begin{diagram}
\node{A} \arrow{e,t,V}{i} \arrow{s,l}{\varphi} \node{B} \arrow{s,r}{\Phi} \\
\node{C} \arrow{e,t,V}{j} \node{D=B\sqcup_AC}
\end{diagram}\end{equation*}
Given a space $Y$, we get an induced diagram
\begin{equation*}\tag{$**$}\begin{diagram}
\node{Y^D} \arrow{e,t}{\Phi^\#} \arrow{s,l}{j^\#} \node{Y^B} \arrow{s,r}{i^\#} \\
\node{Y^C} \arrow{e,t}{\varphi^\#} \node{Y^A}
\end{diagram}\end{equation*}
Then ($**$) is a pullback diagram if either
\begin{itemize}
	\item	$i$ is a closed embedding, and $\varphi$ is a proper closed map, or
	\item	$(B,A)$ is a CW pair, and $C$ admits a CW structure such that $\varphi\colon A\to C$ is a cellular map.
\end{itemize}
If $a_0\in A$, $b_0\in B$, $c_0\in C$, $d_0\in D$ are base points coherent with respect to ($*$),
and $y_0$ is a base point in $Y$, then ($**$) is a pullback also if $Y^A,Y^B,Y^C,Y^D$ are replaced by,
respectively, $(Y,y_0)^{(A,a_0)},(Y,y_0)^{(B,b_0)},(Y,y_0)^{(C,c_0)},(Y,y_0)^{(D,d_0)}$. \qed
\end{lem}

Taking the diagram $*\leftarrow L\hookrightarrow X$ for a CW pair $(X,L)$ we obtain the following
corollary (which is not true for general pairs $(X,L)$).

\begin{cor}\label{fibre_of_restriction_fibration}
Let $(X,L)$ be a CW pair and $(Y,y_0)$ a pointed space. Then the space
$(Y,y_0)^{(X,L)}$ is homeomorphic with $(Y,y_0)^{(X/L,*)}$. \qed
\end{cor}

\begin{cor}\label{principal_fibration}
Let $\varphi\colon A\to L$ be a cellular map of CW complexes (with respect to some decomposition of $L$)
and let $C_\varphi=CA\sqcup_\varphi L$ denote the (either reduced or unreduced) mapping cone of $\varphi$.
Choose $a_0\in A$ and set $x_0=\varphi(a_0)$. Then the following is a pullback diagram.
\begin{equation*}\tag{\dag}\begin{diagram}
\node{(Y,y_0)^{(C_\varphi,x_0)}} \arrow{e} \arrow{s} \node{(Y,y_0)^{(CA,a_0)}} \arrow{s} \\
\node{(Y,y_0)^{(L,x_0)}} \arrow{e} \node{(Y,y_0)^{(A,a_0)}}
\end{diagram}\end{equation*}
The vertical arrows are fibrations and $(Y,y_0)^{(CA,a_0)}$ is contractible. This exhibits
\[ (Y,y_0)^{(C_\varphi,x_0)}\to(Y,y_0)^{(L,x_0)} \]
as a principal fibration with all fibres either empty or homotopy equivalent
to the space $\Omega\big((Y,y_0)^{(A,a_0)},\const_{y_0}\big)\approx(Y,y_0)^{(SA,*)}$.

In addition, if $(X,L)$ is a CW pair such that $\varphi$ induces a homotopy equivalence
$C_\varphi\to X$ then \begin{equation*}\tag{\ddag} (Y,y_0)^{(C_\varphi,x_0)}\to(Y,y_0)^{(X,x_0)} \end{equation*}
(respectively $Y^{C_\varphi}\to Y^X$) is a fibre homotopy equivalence over $(Y,y_0)^{(L,x_0)}$
(respectively $Y^L$).
\end{cor}

\begin{proof}
The space $(Y,y_0)^{(CA,a_0)}$ is contractible by \propref{homotopy_equivalence},
hence (\dag) is the diagram of a principal fibration. (It is equivalent to the standard
one when we are pulling back the path fibration $P(Y,y_0)^{(A,a_0)}\to(Y,y_0)^{(A,a_0)}$
by the coglueing theorem of Brown and Heath \cite{brown-heath}.)

By \propref{homotopy_equivalence} the map (\ddag) is a homotopy equivalence hence it is a fibre homotopy
equivalence by \cite{brown-heath}, Corollary 3.7.
\end{proof}

The following is a result of Stasheff \cite{stasheff} (see also Sch\"{o}n \cite{schon}).
We will simply call it `Stasheff's theorem' in what follows.

\begin{thm}[Stasheff's theorem]\label{stasheff_theorem}
Let $p\colon E\to B$ be a fibration where $B$ has the homotopy type of a CW complex.

If $E$ has the homotopy type of a CW complex then so have all fibres of $p$.

Conversely if for every path component $C$ of $B$ and some $b_0\in C$ the fibre $p^{-1}(b_0)$ has
CW type, then the total space $E$ has CW homotopy type.\qed
\end{thm}

The quoted theorems of Milnor and of Stasheff, as well as the above cited coglueing theorem
of Brown and Heath (see \cite{brown-heath}), and Theorem 3.4.1 of Edwards and Hastings \cite{edwards-hastings}
together with its simplified proof by Geoghegan \cite{geoghegan} form the basis for our work.

\subsection*{Reduction to connected spaces with nondegenerate base points}
At the end of this section we show that to address the question of CW homotopy type of the function space
$Y^X$ or that of $(Y,y_0)^{(X,x_0)}$ it suffices to consider connected CW complexes $X$ with a single $0$-cell $x_0$,
and path-connected Hausdorff spaces $Y$ with nondegenerate base point $y_0$. Moreover it suffices to consider
spaces of pointed functions $(Y,y_0)^{(X,x_0)}$.

We begin with the following easy application of Stasheff's theorem.

\begin{cor}\label{cor_stasheff_theorem}
Let $X$ be a Hausdorff space with nondegenerate base point $x_0$ and let $Y$ be a space of CW homotopy type.
For each path component $D_\lambda$ of $Y$ choose a point $y_\lambda\in D_\lambda$. Then $Y^X$ has CW homotopy
type if and only if all spaces $(Y,y_\lambda)^{(X,x_0)}$ have CW homotopy type.

Fix $y_0\in Y$. The path component of $\const_{y_0}$ in $Y^X$ has CW type if and only if the path component of
$\const_{y_0}$ in $(Y,y_0)^{(X,x_0)}$ has CW type. (If $\pi_1(Y,y_0)$ is trivial then this holds for any map
$g\colon(X,x_0)\to(Y,y_0)$ in place of $\const_{y_0}$.)
\end{cor}

\begin{proof}
Since $x_0$ is nondegenerate, the evaluation map
\[ Y^X\to Y,\,\,\,f\mapsto f(x_0), \]
is a fibration. (See also \lemref{restriction_map} and the remark after it.)

Let $D$ be the path component of $y_0\in Y$.
Pick a pointed map $g\colon(X,x_0)\to(Y,y_0)$ and let $C$ be the path component of $g$ in $Y^X$.
Then $C\cap(Y,y_0)^{(X,x_0)}$ is a set of path components of $(Y,y_0)^{(X,x_0)}$. In case $C$ is the constant component,
$C\subset D^X$ and $C\cap(Y,y_0)^{(X,x_0)}=C\cap(D,y_0)^{(X,x_0)}$ is path-connected. (And if $\pi_1(Y,y_0)$ is trivial
this holds for any $g$.) The restricted map $C\to Y$ is still a fibration.

Now apply Stasheff's theorem.
\end{proof}

\begin{cor}\label{loop_space}
If $Z$ has CW type and $z_0$ is any point in $Z$ then $\Omega(Z,z_0)$ has CW type.
\end{cor}
\begin{proof}
The space $Z^{S^1}$ has CW type by Milnor's theorem, hence so has
the space $\Omega(Z,z_0)=(Z,z_0)^{(S^1,*)}$, by \corref{cor_stasheff_theorem}.
\end{proof}

Let $X$ be a CW complex and $Y$ a space of CW type. Write $X=\cup_{\lambda\in\Lambda}X_\lambda$ where
$\setof{X_\lambda\,\vert\,\lambda}$ is the set of path components of $X$. Pick $x_0\in X_{\lambda_0}\subset X$.
Then $Y^X$ is homeomorphic to $\prod_\lambda Y^{X_\lambda}$, and $(Y,y_0)^{(X,x_0)}$ is homeomorphic to
$(Y,y_0)^{(X_{\lambda_0},x_0)}\times\prod_{\lambda\neq\lambda_0}Y^{X_\lambda}$. In light of \exref{product}
it suffices to consider $X$ path-connected.

Let $X$ be a path-connected CW complex and let $x_0\in X$. Then any given CW decomposition may be refined
in such a way that $x_0$ becomes a $0$-cell. In particular, $x_0$ is a nondegenerate base point. 

Let $Y$ be a space of CW type and let $y_0\in Y$. Since $X$ is path-connected, evidently
$(Y,y_0)^{(X,x_0)}=(D,y_0)^{(X,x_0)}$ where $D$ is
the path component of $y_0$, and by \corref{cor_stasheff_theorem} the space $Y^X$ has CW type if and only if
$(Y,y_0)^{(X,x_0)}$ has for $y_0$ ranging over all different path components of $Y$. In this way
it suffices to consider $Y$ path-connected.

Let $X$ be a path-connected CW complex with a $0$-cell $x_0$ and let $T$ be a maximal tree with vertex $x_0$ in $X$.
Then $(X,x_0)\to(X/T,*)$ is a homotopy equivalence of pairs and in light of \propref{homotopy_equivalence} it suffices
to consider CW complexes $X$ with a single $0$-cell.

Now we take care of the base point in $Y$. Let $y_0$ be an arbitrary base point of a path-connected space $Y$ of CW type.
If $y_0$ is degenerate, attach a whisker at $y_0$ to obtain $Y'$ with nondegenerate base point $y_0'$. (That is $Y'$ is
the mapping cylinder of $\setof{y_0}\to Y$.) The retraction $Y'\to Y$ is a homotopy equivalence (sending $y_0'$ to $y_0$),
hence so is the induced function ${Y'}^X\to Y^X$, by \propref{homotopy_equivalence}. The diagram
\begin{equation*}\begin{diagram}
\node{{Y'}^X} \arrow{e} \arrow{s} \node{Y^X} \arrow{s} \\
\node{Y'} \arrow{e,t}{\text{retraction}} \node{Y}
\end{diagram}\end{equation*}
where the vertical arrows are evaluations at $x_0$, is a map of fibrations. Since the horizontal arrows
are homotopy equivalences, so is the induced map
\begin{equation*} \tag{\dag} (Y',y_0')^{(X,x_0)}\to(Y,y_0)^{(X,x_0)} \end{equation*}
on the fibres, by Brown and Heath \cite{brown-heath}, Corollary 1.5.
Hence we may assume that $Y$ has a nondegenerate base point.

%%%%%%%%%%%%%%%%%%%%%%%%%%%%%%%%%%%%%%%%%%%%%%%%%%%%%%%%%%%%%%%%%%%%%%%%%%%%%
%%%%%%%%%%%%%%%%%%%%%%%%%%%%%%%%%%%%%%%%%%%%%%%%%%%%%%%%%%%%%%%%%%%%%%%%%%%%%
%%%%%%%%%%%%%%%%%%%%%%%%%%%%%%%%%%%%%%%%%%%%%%%%%%%%%%%%%%%%%%%%%%%%%%%%%%%%%
%%%%%									%%%%%
%%%%%				MAIN MATTER				%%%%%
%%%%%									%%%%%
%%%%%%%%%%%%%%%%%%%%%%%%%%%%%%%%%%%%%%%%%%%%%%%%%%%%%%%%%%%%%%%%%%%%%%%%%%%%%
%%%%%%%%%%%%%%%%%%%%%%%%%%%%%%%%%%%%%%%%%%%%%%%%%%%%%%%%%%%%%%%%%%%%%%%%%%%%%
%%%%%%%%%%%%%%%%%%%%%%%%%%%%%%%%%%%%%%%%%%%%%%%%%%%%%%%%%%%%%%%%%%%%%%%%%%%%%

\mainmatter

%%%%%%%%%%%%%%%%%%%%%%%%%%%%%%%%%%%%%%%%%%%%%%%%%%%%%%%%%%%%%%%%%%%%%%%%%%%%%%%%%%%%%%%%%%%%%%%%%%%
%%%%%%%%%%%%%%%%%%%%%%%%%%%%%%%%%%%%%%%%%%%%%%%%%%%%%%%%%%%%%%%%%%%%%%%%%%%%%%%%%%%%%%%%%%%%%%%%%%%
%%%												%%%
%%%		CW homotopy type of certain inverse limits					%%%
%%%												%%%
%%%%%%%%%%%%%%%%%%%%%%%%%%%%%%%%%%%%%%%%%%%%%%%%%%%%%%%%%%%%%%%%%%%%%%%%%%%%%%%%%%%%%%%%%%%%%%%%%%%
%%%%%%%%%%%%%%%%%%%%%%%%%%%%%%%%%%%%%%%%%%%%%%%%%%%%%%%%%%%%%%%%%%%%%%%%%%%%%%%%%%%%%%%%%%%%%%%%%%%

\chapter{CW homotopy type of certain inverse limits}\label{slc_cw_type}

\setcounter{thm}{0}

\def\Z{{\bf Z}}
\def\P{{\bf P}}
\def\Q{{\bf Q}}
\def\p{{\bf p}}
\def\q{{\bf q}}
\def\W{{\bf W}}
\def\f{{\bf f}}
\def\g{{\bf g}}
\def\h{{\bf h}}
\def\k{{\bf k}}
\def\w{{\bf w}}
\def\r{{\bf r}}

A topological space $Z$ is {\it semilocally contractible} (see for example Dydak and Geoghegan \cite{dyd-geog} or
Fritsch and Piccinini \cite{fp})
if every point $z\in Z$ has a neighbourhood that deforms in $Z$ to a point. The deformation
is not required to be $z$-preserving, but we may clearly assume that the final point is $z$.

Semilocal contractibility is also referred to as {\it weak local contractibility},
see for example \cite{fadell2}, \cite{dugundji}.

\begin{defn}
A space $Z$ is a Dold space if it admits a numerable covering of open sets deformable in $Z$ to a point.
\end{defn}

Such spaces are called {\it locally contractible in the large} in Allaud \cite{allaud}.

Obviously a Dold space is semilocally contractible.

A space dominated by a semilocally contractible (respectively Dold) space is itself semilocally contractible
(respectively Dold); see Fadell \cite{fadell2} and Dold \cite{dold}. In particular, the two properties are
homotopy invariant, and we note (see \cite{dold})

%************************************************************************************************************

\begin{lem}
Spaces of CW homotopy type are Dold spaces.\qed
\end{lem}

%************************************************************************************************************

Path components of a semilocally contractible space are open, and therefore a space is semilocally contractible
(respectively Dold) if and only if its path components are open and semilocally contractible (respectively Dold).
Similarly a space has the homotopy type of a CW complex if and only if its path components are open and each has
the type of a CW complex.

%************************************************************************************************************

It is easy to see (see \cite{dold})
\begin{lem}\label{paracompact_vs_Dold}
A paracompact and semilocally contractible space is a Dold space. \qed
\end{lem}

%%%%%%%%%%%%%%%%%%%%%%%%%%%%%%%%%%%%%%%%%%%%%%%%%%%%%%%%%%%%%%%%%%%%%%%%%%%%%%%%%%%%%%%%%%%%%%%%%
%												%
%		General necessary conditions for CW type of certain inverse limits		%
%												%
%%%%%%%%%%%%%%%%%%%%%%%%%%%%%%%%%%%%%%%%%%%%%%%%%%%%%%%%%%%%%%%%%%%%%%%%%%%%%%%%%%%%%%%%%%%%%%%%%

\section{General necessary conditions for CW type of certain inverse limits}\label{inverse_systems}

Let $(\XXX,\lqs,\vee,\wedge)$ be a complete lattice. We use also $\sup$ in place of $\vee$, and $\inf$ in place of $\wedge$.

\begin{defn}
We say that an element $\kappa\in\XXX$ is {\it compact} (see for example Davey, Priestley \cite{davey-priestley})
if for an arbitrary subset $\Lambda\subset\XXX$ the relation $\kappa\lqs\sup\Lambda$ implies that already
$\kappa\lqs\sup F$ for some {\it finite} subset $F$ of $\Lambda$. Let $\KKK=\KKK(\XXX)$ denote the set of
compact elements of $\XXX$.

A subset $\III$ of $\XXX$ is an {\it ideal} if for each $\kappa,\kappa'\in\III$ and $\xi\in\XXX$
also $\kappa\vee\kappa'$ and $\kappa\wedge\xi$ belong to $\III$. (In particular, $\III$ is a sublattice.)

A subset $\GGG$ is {\it generating in} or {\it generates} $\XXX$ if for each $\xi\in\XXX$ there exists a subset
$\Xi$ of $\GGG$ with $\xi=\sup\Xi$.

If the set of compact elements generates the complete lattice $\XXX$ we call $\XXX$ an {\it algebraic} lattice.

Finally, we call $\XXX$ a {\it regular} lattice if it is a distributive algebraic lattice whose set of compact
elements is an ideal. (Note that the set of compact elements is always closed for finite supremums.)

Let $\XXX$ be a regular lattice and let $\xi\in\XXX$. We define the cardinality of $\xi$ with respect to
the set $\KKK$ of compact elements\footnote{In fact the definition would make sense with any generating subset.} as
\[ \card_{\KKK}\xi=\min\setof{\card K\,\vert\, K\subset\KKK\text{ such that }\xi=\sup K}. \]
As every set of cardinal numbers is well-ordered, the above definition makes sense. For $\xi\in\XXX$ the
cardinality $\card_\KKK\xi$ is finite if and only if it is equal to $1$ which is if and only if $\xi$ belongs to $\KKK$.
\end{defn}

\begin{rem_n}
Every distributive algebraic lattice $\XXX$ satisfies the strong (more precisely the `join-infinite') distributive law
\begin{equation*}	\tag{$\star$}
\mu\wedge\sup\setof{\lambda_i\,\vert\,i}=\sup\setof{\mu\wedge\lambda_i\,\vert\,i}
\end{equation*}
for all elements $\mu\in\XXX$ and subsets $\setof{\lambda_i\,\vert\,i}\subset\XXX$.
(See Davey, Priestley \cite{davey-priestley}, Exercise 4.22.)

Let us mention an alternative sufficient condition for regularity.
Let $\XXX$ be a complete lattice which satisfies the join-infinite distributive law ($\star$).
Further let $\FFF$ be a generating ideal in $\XXX$. Assume that for any element $\phi\in\FFF$
the set $\setof{\xi\in\XXX\,\vert\,\xi<\phi}$ is finite. Then $\XXX$ is a regular lattice with
$\KKK(\XXX)=\FFF$. An example is the lattice of all subcomplexes of a given CW complex.
\end{rem_n}

Let $\Z$ be an inverse system of topological spaces indexed by $\XXX$,
that is a cofunctor $\XXX\to\TTT op$. As customary we write $Z_\lambda$ for $\Z(\lambda)$ and employ an additional symbol
$p_{\lambda,\mu}$ for the arrow, or bonding morphism, $\Z(\mu)\to\Z(\lambda)$ whenever $\lambda<\mu$.

Let $\LLL$ be a directed subset of $\XXX$ and let $\mu=\sup\LLL$. The morphisms $Z_\mu\to Z_\lambda$, for
$\lambda\in\LLL$, induce a map \begin{equation*}\tag{$**$} Z_\mu\to\lim_{\lambda\in\LLL}Z_\lambda. \end{equation*}

\begin{defn}
The inverse system $\Z$ is restricted if the map ($**$) is a homeomorphism, for each directed subset $\LLL$ of $\XXX$.
\end{defn}

\begin{rem_n}
The only type of lattice that we are going to consider (apart from the lattice of natural numbers together with
$\infty$) is the lattice $\XXX$ of all subcomplexes of a given CW complex $X$. 
Given a topological space $Y$ we may form the induced inverse system of function
spaces $\setof{Y^L\,\vert\,L\in\XXX}$ together with restriction fibrations as bonding maps. This system is certainly
restricted in the sense defined above, and systems of this type will be considered below.

Therefore we restrict our attention to (restricted) inverse systems indexed by complete lattices. Some results of general
nature below, for instance \propref{fibre_SLC_general}, could be formulated in a more general setting. However, we
insist that all morphisms in the system be fibrations. This is a strong requirement which is difficult to check in
`practice.' In a complete lattice we can choose a maximal totally ordered subset which is well-ordered because of
completeness. For a restricted inverse system indexed by a well-ordered set all bonding maps are fibrations if and
only if all consecutive bonding maps are fibrations. (See \corref{consecutive_fibrations}.)

With `target lattice' being subcomplexes of a CW complex, it may appear artificial to work with abstract complete lattices.
We have two justifications. We exhibit some interesting results
for inverse sequences which do not involve function spaces. On the other hand, for function spaces with domain an uncountable
complex our results follow from easy transfinite induction arguments that rely mainly upon properties of the lattice of
subcomplexes of domain and not particularly on the function space structure. Stripped of the `unnecessary extra' structure,
proofs involving transfinite arguments appear more natural. Thus we have taken an intermediate approach.
\end{rem_n}

\begin{prop}\label{fibre_SLC_general}
Let $\Z$ be a restricted system of fibrations indexed by a regular lattice $\XXX$ with set of compact elements $\KKK$.
Set $\alpha=\sup\XXX=\max\XXX$. Assume that for each $\kappa\in\KKK$ the space $Z_\kappa$ has the homotopy type
of a CW complex. Assume also that $\alpha\notin\KKK$, for nontriviality.

Denote the limit space $Z_\alpha=\lim_{\lambda\in\XXX}Z_\lambda=\lim_{\kappa\in\KKK}Z_\kappa$ simply by $Z$.
Let $C$ be an arbitrary path component of $Z$, and let $C_\kappa$ denote the image of $C$ under $Z\to Z_\kappa$.

\begin{enroman}
\item	$C$ is semilocally contractible if and only if for each $\zeta\in C$ there exists $\kappa\in\KKK$ such that
	the fibre of $C\to C_\kappa$ over $\zeta\vert_\kappa$ contracts in $C$.
\item	$C$ is semilocally contractible and open if and only if for each $\zeta\in C$ there exists $\kappa\in\KKK$ such
	that the fibre of $Z\to Z_\kappa$ over $\zeta\vert_\kappa$ contracts in $Z$.
\item	If $C$ is semilocally contractible then the loop space $\Omega(Z,\zeta)$ has CW homotopy type for each $\zeta\in C$.
\item	$Z$ has the homotopy type of a CW complex if and only if it is a Dold space.
\end{enroman}
\end{prop}

\begin{proof}
Assume that $C$ is semilocally contractible and let $\zeta\in C$. There exists a neighbourhood, with no loss
of generality basic, of $\zeta\in C$ that contracts in $C$. Suppose this neighbourhood has the form
\begin{equation*} \tag{\dag} \scal{U_{\kappa_1},\dots,U_{\kappa_r}}\cap C \end{equation*} where
$\kappa_1,\dots,\kappa_r\in\KKK$. Since $\KKK$ is a sublattice, $\kappa=\sup\setof{\kappa_1,\dots,\kappa_r}$ belongs
to $\KKK$. Thus the neighbourhood (\dag) contains the neighbourhood
$\scal{U_{\kappa}}\cap C$ where \[ U_\kappa=\bigcap_{i=1}^rp_{\kappa_i,\kappa}^{-1}(U_{\kappa_i}). \]
The latter contains the fibre $F_\kappa$ of $C\to C_\kappa$ over $\zeta_\kappa=\zeta\vert_{\kappa}$.
Thus $F_\kappa$ contracts in the total space $C$ and therefore the loop space $\Omega(C_\kappa,\zeta_\kappa)$
has the same homotopy type as the product $F_\kappa\times\Omega(C,\zeta)$. (See Spanier \cite{spanier}, Corollary 2.8.15,
as well as \propref{fibre_contraction}.) In particular, $\Omega(Z,\zeta)=\Omega(C,\zeta)$
is dominated by $\Omega(C_\kappa,\zeta_\kappa)$ which has CW homotopy type by assumption and Milnor's theorem.
This shows \navedi{3} and the necessity part of \navedi{1}. To prove necessity of \navedi{2}, note that
if $C$ is open in $Z$, then (\dag) may be assumed equal to $\scal{U_{\kappa_1},\dots,U_{\kappa_r}}\cap Z$.

For the sufficiency part of \navedi{1}, let $\zeta\in C$ and assume that the fibre $F_\kappa$ over $\zeta\vert_\kappa$
of $C\to C_\kappa$ contracts in $C$. Since $C_\kappa$ has CW type, there exists a neighbourhood $U_\kappa$ of
$\zeta\vert_\kappa$ that contracts in $C_\kappa$ to $\zeta\vert_\kappa$. Let $U$ be the preimage of $U_\kappa$ in $C$.
By homotopy lifting property the deformation of $U_\kappa$ lifts to a homotopy that deforms $U$ into the fibre $F_\kappa$.
Since the latter deforms to a point, we may concatenate to obtain a deformation of $U$ into a point. This concludes
\navedi{1}.

Note that $C_\kappa$ is open in $Z_\kappa$, hence $U_\kappa$ is open in $Z_\kappa$, and the preimage $\Tilde U$ of
$U_\kappa$ in $Z$ is an open neighbourhood of $\zeta$. If the fibre of $Z\to Z_\kappa$ contracts in $Z$, then so does
$\Tilde U$, by the above. The image of the contracting homotopy is a path-connected subset of $Z$, hence $\Tilde U$ is
contained in $C$. This shows that $C$ is open, and proves the sufficiency part of \navedi{2}.

If $Z$ is semilocally contractible it follows by \navedi{3} that for each $\zeta\in Z$, the loop space
$\Omega(Z,\zeta)$ has CW homotopy type. If, in addition, $Z$ is a Dold space, it has CW homotopy type
by the `delooping' theorem of Allaud \cite{allaud}.
\end{proof}

\begin{example}\label{product}
Let $I$ be an infinite set and let $\setof{W_i\,\vert\,i\in I}$ be a family of nonempty spaces of CW type.
Let $\KKK$ denote the set of all finite subsets of $I$, and $\XXX$ the set of all subsets of $I$.
For $\lambda\in\XXX$ let $Z_\lambda$ denote the cartesian product of $W_i$ for $i\in\lambda$, and
for $\lambda\subset\mu$ let $Z_\mu\to Z_\lambda$ denote the projection. The limit space of the
system $\setof{Z_\lambda\,\vert\,\lambda\in\XXX}$ is the cartesian product $Z=\prod_iW_i$.

\propref{fibre_SLC_general} implies that $Z$ is semilocally contractible if and only if there exists $\kappa\in\KKK$
such that $Z'=\prod_{i\notin\kappa}W_i$ contracts in $Z$. By virtue of the projection $Z\to Z'$ this implies that $Z'$
is contractible and hence so are the spaces $W_i$ for $i\notin\KKK$. In this case $Z$ is homotopy equivalent to the
finite product $\prod_{i\in\kappa}W_i$ and has CW homotopy type.\qed
\end{example}

\begin{prop}\label{strong_obstruction_general}
Given hypotheses of \propref{fibre_SLC_general}, assume that $C$ has the homotopy type of a CW complex.
Then, given any $\kappa\in\KKK$ there exist $\lambda,\lambda'\in\KKK$ with $\kappa\lqs\lambda\lqs\lambda'$ so that the
following is true.
\begin{enroman}
\item For every $\mu\in\XXX$ such that $\mu\gqs\lambda$ the fibre $F_\mu$ of $C\to C_\mu$ over $\zeta_\mu$ contracts in $C$
	and the fibration $\Omega(C,\zeta)\to\Omega(C_\mu,\zeta_\mu)$ is a domination map. Moreover, there exists
	a commutative diagram
\begin{equation*}\begin{diagram}
\node{\Omega(Z_\mu,\zeta_\mu)} \arrow[3]{e,t}{\Omega(p_{\lambda,\mu},\zeta_\mu)} \node[3]{\Omega(Z_\lambda,\zeta_\lambda)} \\
\node{F_\mu\times\Omega(Z,\zeta)} \arrow{n,l}{\simeq} \arrow[3]{e,t}{\text{inclusion}} \node[3]{F_\lambda\times\Omega(Z,\zeta)}
	\arrow{n,r}{\simeq}
\end{diagram}\end{equation*}
	where the vertical arrows are homotopy equivalences of the form
	\[ F_\lambda\times\Omega(Z,\zeta)\to\Omega(Z_\lambda,\zeta_\lambda),\,\,\,
		(z,\gamma)\mapsto p_{\lambda,\alpha}{}_\#(\gamma)*\theta_\lambda(z)	\]
	where $\theta_\lambda\colon F_\lambda\to\Omega(Z_\lambda,\zeta_\lambda)$ is a map,
	and $*$ denotes concatenation.

\item For every $\mu\in\XXX$ with $\mu\gqs\lambda'$ the inclusion $F_\mu\to F_\lambda$ (and hence also $F_\mu\to F_\kappa$)
	is homotopic to a constant map. Consequently the inclusion $F_\mu\times\Omega(Z,\zeta)\hookrightarrow
	 F_\lambda\times\Omega(Z,\zeta)$ is homotopic to the projection onto $\setof{\zeta}\times\Omega(Z,\zeta)$,
	and the fibration $\Omega(Z_\mu,\zeta_\mu)\to\Omega(Z_\lambda,\zeta_\lambda)$ factors as
	\[ \Omega(Z_\mu,\zeta_\mu)\xrightarrow{j}\Omega(Z,\zeta)\xrightarrow{\Omega(p_{\lambda,\alpha},\zeta)}
		\Omega(Z_\lambda, \zeta_\lambda) \]
	where $j$ is a right homotopy inverse of the fibration $\Omega(Z,\zeta)\to\Omega(Z_\mu,\zeta_\mu)$.
\end{enroman}
\end{prop}

\begin{proof}
If $\mu\gqs\lambda$ then $F_\mu\subset F_\lambda$, so by \navedi{1} of \propref{fibre_SLC_general} and directedness of
$\KKK$ there exists $\lambda\gqs\kappa$ so that $F_\lambda$ contracts in $C$ and consequently so does $F_\mu$ for any
$\mu\gqs\lambda$.

View $\Omega(Z,\zeta)$ and $\Omega(Z_\lambda,\zeta_\lambda)$ as subspaces of $C^I$ and $C_\lambda^I$, respectively.
Let $k\colon F_\lambda\to C^I$ denote the adjoint of a contracting homotopy. By \propref{fibre_contraction} a homotopy
equivalence $F_\lambda\times\Omega(Z,\zeta)\to\Omega(Z_\lambda,\zeta_\lambda)$ is given by
\begin{equation*}
\tag{\ddag} (z,\gamma)\mapsto p_{\lambda,\alpha}\circ\big[\gamma*k^{-1}(z)\big]=
		[p_{\lambda,\alpha}\circ\gamma]*[p_{\lambda,\alpha}\circ k^{-1}(z)].
\end{equation*}
For fixed $z$, the map (\ddag) amounts to the induced fibration
$\Omega(Z,\zeta)\to\Omega(Z_\lambda,\zeta_\lambda)$ followed by a multiplication homotopy equivalence;
which concludes the proof of \navedi{1}.

Since both $C$ and $C_\lambda$ have CW type, so has $F_\lambda$ by Stasheff's theorem. Hence also $F_\lambda$ is semilocally
contractible and there exists a basic neighbourhood of $\zeta$ in $F_\lambda$ that contracts in $F_\lambda$. Using
directedness of $\KKK$ as in the proof of \propref{fibre_SLC_general} we may assume that the neighbourhood has the form
$\scal{U_{\lambda'}}\cap F_\lambda$ for some $\lambda'\in\KKK$, $\lambda'\gqs\lambda$, and open $U_{\lambda'}\subset Z_{\lambda'}$.
But this neighbourhood contains the set $\setof{z\,\vert\,z_{\lambda'}=\zeta_{\lambda'}}\cap F_{\lambda}$ which equals
$F_{\lambda'}$. The commutativity of the diagram of \navedi{1} follows from (\ddag) upon having chosen $k\vert_{F_\mu}$
for the `contraction' $F_\mu\to C^I$.
\end{proof}

%%%%%%%%%%%%%%%%%%%%%%%%%%%%%%%%%%%%%%%%%%%%%%%%%%%%%%%%%%%%%%%%%%%%%%%%%%%%%%%%%%%%%%%%%%%%%%%%%
%												%
%			Results on sequences							%
%												%
%%%%%%%%%%%%%%%%%%%%%%%%%%%%%%%%%%%%%%%%%%%%%%%%%%%%%%%%%%%%%%%%%%%%%%%%%%%%%%%%%%%%%%%%%%%%%%%%%

\section{Results on sequences}\label{results_on_sequences}

In this section we study inverse sequences of fibrations between spaces of CW type, and show that
the inverse limit space has CW type only if the sequence of loop spaces splits into the product of
a sequence of nullhomotopic maps and one of homotopy equivalences. Moreover, a converse holds as well
(see \thmref{inverse_sequence} below). Thus the limit space has CW type if and only if the sequence
is trivial in some sense, in the spirit of Theorem B of Dydak and Geoghegan \cite{dyd-geog} (see also
\cite{corr-dyd-geog}).

We first recall a well known result due to Edwards and Hastings (see Geoghegan \cite{geoghegan} for details).

\begin{prop}\label{E-H}
Assume given a commutative diagram of topological spaces and continuous functions
\begin{equation*}\begin{diagram}
\node{\dots} \arrow{e} \node{W_3} \arrow{e,t}{p_3} \arrow{s,r}{f_3} \node{W_2} \arrow{e,t}{p_2} \arrow{s,r}{f_2}
	\node{W_1} \arrow{s,r}{f_1} \\
\node{\dots} \arrow{e} \node{Z_3} \arrow{e,t}{r_3} \node{Z_2} \arrow{e,t}{r_2} \node{Z_1}
\end{diagram}\end{equation*}
where the maps $p_i$ as well as $r_i$ are fibrations and the maps $f_i$ are homotopy equivalences.
Then the induced inverse limit map $f_\infty\colon W_\infty\to Z_\infty$ is also a homotopy equivalence. \qed
\end{prop}

\begin{cor}\label{cor_E-H}
Suppose $\dots\to Z_3\to Z_2\to Z_1$ is an inverse sequence of fibrations which are homotopy equivalences.
Then any canonical projection from the inverse limit $Z_\infty\to Z_i$ is a homotopy equivalence.
\qed
\end{cor}

Apply homotopy lifting property to inductively construct commutative squares in order to infer

\begin{cor}\label{E-H-homotopy}
In order to conclude that $W_\infty$ and $Z_\infty$ are homotopy equivalent
it is enough to assume commutativity only up to homotopy in \propref{E-H}. \qed
\end{cor}

The following proposition generalizes Lemma 2.8 of Kahn \cite{kahn}, and is crucial for understanding
the problem of CW homotopy type of inverse limits.

\begin{prop}\label{contractible_limit}
Let $\dots\to Z_3\xrightarrow{r_3}Z_2\xrightarrow{r_2}Z_1$ be an inverse sequence of fibrations
and let $Z_\infty$ denote the inverse limit space. If each fibration $r_{i+1}\colon Z_{i+1}\to Z_i$
is homotopic to a constant map, then the limit space $Z_\infty$ is contractible.
\end{prop}

\begin{proof}
Define $W_1=Z_1$ and let $f_1\colon W_1\to Z_1$ be the identity. Assume having constructed a suitable homotopy
equivalence $f_i\colon W_i\to Z_i$. Let $r_{i+1}\colon Z_{i+1}\to Z_i$ be homotopic to the constant $\zeta_i\in Z_i$.
Let $w_i$ be a point in $W_i$ that is mapped by $f_i$ into the path component of $\zeta_i$. Let $PW_i$ denote the
space of paths in $W_i$ that start in $w_i$ and set $W_{i+1}=Z_{i+1}\times PW_i$. Let $p_{i+1}\colon W_{i+1}\to W_i$ be the
composite $Z_{i+1}\times PW_i\xrightarrow{\pr}PW_i\xrightarrow{\eps}W_i$ where $\pr$ is the obvious projection
and $\eps$ denotes the evaluation at end point. Let $f_{i+1}\colon W_{i+1}\to Z_{i+1}$ be the projection
$Z_{i+1}\times PW_i\to Z_{i+1}$. Clearly $f_{i+1}$ is a homotopy equivalence. The composite $r_{i+1}\circ f_{i+1}$
is homotopic to $\const_{\zeta_i}$ while the composite $f_{i}\circ p_{i+1}$ is homotopic to $\const_{f_i(w_i)}$. The
two constant maps are homotopic by choice of $w_i$. We proceed inductively, and appeal to \corref{E-H-homotopy} to
conclude that $Z_\infty$ is homotopy equivalent to $W_\infty=\lim W_i$. But the latter is homeomorphic to the
limit space of the sequence $\dots\to PW_{i+1}\to PW_i\to\dots\to PW_1$ where the map $PW_{i+1}\to PW_i$ is
a fibration. By \corref{cor_E-H} the limit space is contractible.
\end{proof}

\begin{defn}
Let $\Z$ be an inverse system indexed by $\XXX$ and let $\zeta=\setof{\zeta_\lambda\,\vert\,\lambda\in\XXX}$ be a coherent
choice of base points, that is $\zeta\in\lim\Z$. We denote by $\Omega(\Z,\zeta)$ the inverse system of spaces
$\Omega(Z_\lambda,\zeta_\lambda)$ with the induced bonds
\[ \Omega(Z_{\mu},\zeta_{\mu})\xrightarrow{\Omega(Z_\mu\to Z_\lambda)}\Omega(Z_{\lambda},\zeta_\lambda). \]

Let $\P$ and $\Q$ be inverse systems indexed by the same set $\XXX$.
We define the product system $\P\times\Q$ as the inverse system of products with product bonding maps.
\end{defn}

\begin{lem}\label{continuity_of_loop_and_product}
Let $\Z,\P,\Q$ be restricted inverse systems of fibrations indexed by complete lattice $\XXX$, and let $\zeta\in\lim\Z$.
Then $\Omega(\Z,\zeta)$ and $\P\times\Q$ are restricted inverse systems of fibrations. \qed
\end{lem}

\begin{defn}
We say that the inverse sequence $\W$ splits as a product of inverse sequences $\P$ and $\Q$ if
there exists a level preserving morphism (up to homotopy) of sequences $\f\colon\P\times\Q\to\W$ 
which is a homotopy equivalence on each level.
\end{defn}

(Care should be taken to define splitting up to homotopy of a general inverse system. We will not need this.)

\begin{thm}\label{inverse_sequence}
\begin{enroman}
\item	Let $(\W,\w)$ be an inverse sequence of fibrations that splits into the product of $(\P,\p)$ and $(\Q,\q)$ where
	$\P$ is a sequence of nullhomotopic maps and $\Q$ is a sequence of homotopy equivalences.
	Then the limit space $W_\infty$ is homotopy equivalent to $Q_1$.

	In particular, if the space $Q_1$ has CW type, so has $W_\infty$.

\item	Let $(\Z,\r)$ be an inverse sequence of fibrations between spaces of CW homotopy type.
	If the limit space $Z_\infty$ also has CW type then, for any $\zeta\in Z_\infty$, a subsequence
	of the induced inverse sequence $\Omega(\Z,\zeta)$ splits into
	the product of a sequence of nullhomotopic maps and the identity
	sequence $\setof{\dots=\Omega(Z_\infty,\zeta_\infty)=\Omega(Z_\infty,\zeta_\infty)}$.
\end{enroman}
\end{thm}

\begin{proof}
We change $(\P,\p)$ and $(\Q,\q)$ inductively to obtain sequences of fibrations $(\P',\p')$ and $(\Q',\q')$
together with (strict) level preserving morphisms $\P\to\P'$ and $\Q\to \Q'$ consisting of
homotopy equivalences. Thus there exists a level preserving morphism (up to homotopy)
$\P'\times\Q'\to\W$ consisting of homotopy equivalences. The limit space of $\P'\times\Q'$ is
$P'_\infty\times Q'_\infty$ where $P'_\infty=\lim\P'$ is contractible by \propref{contractible_limit}, and
$Q'_\infty=\lim\Q'$ is homotopy equivalent to $Q_1'=Q_1$ by \corref{cor_E-H}. By \corref{E-H-homotopy} the space
$W_\infty=\lim\W$ is homotopy equivalent to $P'_\infty\times Q'_\infty$ which in turn is homotopy equivalent to
$Q'_\infty$. This proves \navedi{1}.

Statement \navedi{2} follows immediately from \propref{strong_obstruction_general}.
\end{proof}

\begin{cor}\label{contractibility}
Given the assumptions of \navedi{2} of \thmref{inverse_sequence}, if $Z_\infty$ has CW type and
$\Omega(Z_\infty,\zeta_\infty)$ is contractible, then a subsequence of $\Omega(\Z,\zeta)$ is a sequence
of nullhomotopic maps.\qed
\end{cor}

Note that while \propref{contractible_limit} yields a number of counterexamples to Theorem B of Dydak and Geoghegan
\cite{dyd-geog}, \thmref{inverse_sequence} says these represent essentially the only kind of counterexample. In a sense
therefore the cited theorem is `not far' from being correct.

%%%%%%%%%%%%%%%%%%%%%%%%%%%%%%%%%%%%%%%%%%%%%%%%%%%%%%%%%%%%%%%%%%%%%%%%%%%%%%%%%%%%%%%%%%%%%%%%%
%												%
%			Some sufficient and some necessary conditions				%
%												%
%%%%%%%%%%%%%%%%%%%%%%%%%%%%%%%%%%%%%%%%%%%%%%%%%%%%%%%%%%%%%%%%%%%%%%%%%%%%%%%%%%%%%%%%%%%%%%%%%

\section{Some sufficient and some necessary conditions}\label{sufficient_necessary}

The purpose of this section is proving \thmref{wicked_general} and \thmref{double_wicked_general},
to be applied in later sections. We begin with some results on inverse systems indexed by
well-ordered sets; then we show that given a regular lattice,
suitable well-ordered subsets can be chosen, for better grip on the inverse limit space.

\begin{defn}
For any ordinal number $\alpha$ we will denote by $W(\alpha)$ the ordinal number segment consisting
of the set of ordinal numbers $\lambda$ such that $\lambda<\alpha$. Recall that every well-ordered
set $W$ can be indexed uniquely in an order preserving manner by an ordinal number segment $W(\alpha)$
for $\alpha=\ord W$.

For convenience denote
$\Bar W=W(\alpha)\cup\setof{\alpha}=W(\alpha+1)=\setof{\lambda\,\vert\,\lambda\lqs\alpha}$.

Let $\alpha$ be a limit ordinal and let $(\Z,\p)$ be an inverse system of fibrations indexed by $W(\alpha)$.
Note that $\Bar W(\alpha)$ is a complete lattice. We let $Z_\alpha$ denote the limit
of $\Z$, and may view $\Z$ as an inverse system indexed by a complete lattice. Our notion of restrictedness
of the system $(\Z,\p)$ now coincides with terminology used by J.~Cohen \cite{joel_cohen_1}.

For emphasis, the system $(\Z,\p)$ is restricted if for each limit ordinal $\mu\in W(\alpha)$ the canonical map
\[	\vartheta_\mu\colon Z_\mu\to\Bar Z_\mu=\lim\Z\vert_{W(\mu)}						\]
is a homeomorphism. (In fact for our applications it would suffice for $\vartheta_\mu$ to be a fibration
and a homotopy equivalence.)

The projections $p_{\lambda,\lambda+1}$ are called {\it consecutive} bonding maps (or bonds).
\end{defn}

\begin{rem_n}
If $(\Z,\p)$ is indexed by $W(\alpha)$ where $\alpha$ has a predecessor $\alpha-1$ then $Z_{\alpha-1}$,
together with $p_{\lambda,\alpha-1}$, is the limit of the system. \qed
\end{rem_n}

\begin{prop}\label{transfinite_fibration}
Let $\alpha$ be a limit ordinal and $(\Z,\p)$ a restricted inverse system of fibrations, indexed by $W(\alpha)$.
Set $Z_\alpha=\lim\Z$. The canonical projection $p_{0,\alpha}\colon Z_\alpha\to Z_0$ is a fibration.
\end{prop}

\begin{proof}
Let $Y$ be a topological space. Suppose we are given continuous functions $h\colon Y\times I\to Z_0$ and
$f\colon Y\to Z_\alpha$ such that $p_{0,\alpha}\circ f=h\vert_{Y\times 0}$.

Let $T$ be the set of pairs $(\lambda,h_\lambda)$ where $\lambda<\alpha$ and $h_\lambda\colon Z\times I\to Z_\lambda$
makes the following diagram commutative.
\begin{equation*}\begin{diagram}
\node{Y\times 0} \arrow{s} \arrow{e,t}{f} \node{Z_\alpha} \arrow{e} \node{Z_\lambda} \arrow{s} \\
\node{Y\times I} \arrow{nee,t}{h_\lambda} \arrow[2]{e,t}{h} \node[2]{Z_0}
\end{diagram}\end{equation*}
Using transfinite construction we will define $\Phi\colon W(\alpha)\to T$ such that
$\Phi(\lambda)=(\lambda,h_\lambda)$, and $p_{\lambda,\lambda'}\circ h_{\lambda'}=h_{\lambda}$
for $\lambda<\lambda'$.

Assume $\Phi$ defined for $\lambda<\mu$. If $\mu$ has a predecessor, then we may lift $h_{\mu-1}$
along $p_{\mu-1,\mu}$ to a homotopy $h_\mu$. (To fulfill set theoretical requirements we may
use a previously chosen well ordering on the set $Z_\mu^{Y\times I}$ and take a minimal $h_\mu$.)

If $\mu$ is limiting, the coherent system $\setof{h_\lambda\vert\,\lambda<\mu}$ defines a unique limit
map $h_\mu\colon Y\times I\to\lim_{\lambda<\mu}Z_\lambda$. Restrictedness guarantees that
$\lim_{\lambda<\mu}Z_\lambda=Z_\mu$, so in fact this yields the successor $h_\mu\colon Y\times I\to Z_\mu$. 

Thus there exists a function $\Phi$ as required, and the coherent system $\setof{h_\lambda\,\vert\,\lambda<\alpha}$
yields the asserted lifting $h_\alpha\colon Y\times I\to Z_\alpha$ of $h$.
\end{proof}

Using \propref{transfinite_fibration} and transfinite induction it is straightforward to prove

\begin{cor}\label{consecutive_fibrations}
Let $(\Z,\p)$ be a restricted inverse system indexed by $W(\alpha)$ where the consecutive bonds are fibrations.
Then all bonds are fibrations. \qed
\end{cor}

%%  See Thesis for proof.

The following theorem is the transfinite analogue of \propref{E-H}.
The proof depends upon restrictedness and the induction step employed by Geoghegan in \cite{geoghegan}.

\begin{thm}\label{transfinite_geoghegan}
Let $\alpha$ be a limit ordinal and let $(\W,\p)$ and $(\Z,\q)$ be restricted inverse systems of fibrations indexed
by $W(\alpha)$. If $\f\colon\W\to\Z$ is a level-preserving system consisting of homotopy equivalences, then the limit
map $f_\alpha\colon W_\alpha\to Z_\alpha$ is also a homotopy equivalence.
\end{thm}

\begin{proof}
Let $T$ be the set of triples $(\lambda,g_\lambda,h_\lambda)$ where $\lambda<\alpha$,
the map $g_\lambda\colon Z_\lambda\to W_\lambda$ is a homotopy inverse of $f_\lambda$, and
$h_\lambda$ is a homotopy between $g_\lambda\circ f_\lambda$ and $\id_{W_\lambda}$.

We use transfinite construction to define $\Phi\colon W(\alpha)\to T$, where
$\Phi(\lambda)=(\lambda,g_\lambda,h_\lambda)$, such that for $\lambda<\lambda'$ the following coherence conditions hold.
\[ p_{\lambda,\lambda'}\circ g_{\lambda'}=g_{\lambda'}\circ q_{\lambda,\lambda'} \text{ and }
	p_{\lambda,\lambda'}\circ h_{\lambda'}=h_\lambda\circ(p_{\lambda,\lambda'}\times\id_I). \]
Assume $\Phi$ defined for $\lambda<\mu$. If $\mu$ is limiting, then the coherent systems
$\setof{g_\lambda\,\vert\,\lambda<\mu}$ and $\setof{h_\lambda\,\vert\,\lambda<\mu}$ trivially
define $g_\mu$ and $h_\mu$, relying upon restrictedness. Universality guarantees the required properties.

If $\mu$ has a predecessor $\mu-1$, the maps $g_{\mu-1}$ and $h_{\mu-1}$ are already given.
Now we obtain $g_\mu$ and $h_\mu$ using the inductive step from Geoghegan \cite{geoghegan}.
More precisely, consider the diagram
\begin{equation*}\begin{diagram}
\node{W_{\mu}} \arrow{e,t}{\id} \arrow{s} \node{W_{\mu}} \arrow{s} \\
\node{M(f_{\mu})} \arrow{ne,..} \arrow{e} \node{W_{\mu-1}}
\end{diagram}\end{equation*}
where $M(f_{\mu})$ is the mapping cylinder of $f_{\mu}$, and the arrow $M(f_{\mu})\to W_{\mu-1}$
is induced by the composite $h_{\mu-1}\circ(p_{\mu-1,\mu}\times\id)$ on $W_{\mu}\times I$ and by the composite
$g_{\mu-1}\circ q_{\mu-1,\mu}$ on $Z_{\mu}$.

The lifting $M(f_{\mu})\to W_{\mu}$ (minimal one in a previously established well ordering on $W_{\mu}^{M(f_{\mu})}$)
induces a homotopy $h_{\mu}\colon W_{\mu}\times I\to W_{\mu}$ as well as a map $g_{\mu}\colon Z_{\mu}\to W_{\mu}$
with the requisite properties.

Thus there exists a function $\Phi$ as required. Coherent systems $\setof{g_\lambda}$, $\setof{h_\lambda}$
for $\lambda<\alpha$ yield a limit map $g_\alpha\colon Z_\alpha\to W_\alpha$ and a limit map
$h_\alpha\colon W_\alpha\times I\to W_\alpha$. The latter is a homotopy between $g_\alpha\circ f_\alpha$
and $\id_{W_\alpha}$, by universality.

Repeating the above for $\g\colon\Z\to\W$ we get a homotopy inverse $f_\alpha'$ of $g_\alpha$.
Therefore $g_\alpha$ has both a left and a right homotopy inverse, hence is a homotopy equivalence.
\end{proof}

Using \thmref{transfinite_geoghegan}, restrictedness, and transfinite induction it is easy to prove

\begin{cor}\label{cor_transfinite_geoghegan_1}
Let $\f\colon\W\to\Z$ be a level-preserving morphism of restricted inverse systems of fibrations as above.
Let $f_0\colon W_0\to Z_0$ be a homotopy equivalence. If upon assuming that $f_\lambda$ is a homotopy
equivalence it follows that also $f_{\lambda+1}$ is a homotopy equivalence for each $\lambda<\alpha$,
then the limit map $f_\alpha\colon W_\alpha\to Z_\alpha$ is a homotopy equivalence. \qed
\end{cor}

%%  See Thesis for proof.

Immediately we infer

\begin{cor}\label{cor_transfinite_geoghegan_2}
Let $(\W,\p)$ be a restricted inverse system of fibrations indexed by $W(\alpha)$ for some limit ordinal $\alpha$.
If each consecutive bonding map $p_{\lambda,\lambda+1}$ is a homotopy equivalence then the limit projection
$W_\alpha\to W_0$ is a homotopy equivalence. \qed
\end{cor}

%%  See Thesis for proof.

For an abelian group $G$, a well-ordered ascending chain of subgroups
\[ N_0\lqs N_1\lqs\dots\lqs N_\lambda\lqs\dots\,\,\,\,(\lambda<\alpha)	\]
is called a {\it smooth chain} if for each limit ordinal $\mu$ the subgroup
$N_\mu$ equals the union (`supremum') of all $N_\lambda$ with $\lambda<\mu$.
(See for example Fuchs \cite{fuchs2} or Loth \cite{loth}.)

Thus we may think of the assignment $\lambda\to N_\lambda$ as `limit preserving,'
and could call it `continuous.' In accord with the group-theoretic terminology and to avoid confusion
we agree to the following

\begin{defn}\label{smooth_function}
Let $f\colon W(\alpha)\to\XXX$ be a monotone function from the segment $W(\alpha)$ to a complete lattice $\XXX$.
The function $f$ is {\it smooth} if for each limit ordinal $\mu$ in $W(\alpha)$, the following equality holds
\[ 		f(\mu)=\sup\setof{f(\lambda)\,\vert\,\lambda<\mu}.		\]
\end{defn}

If a smooth chain $\setof{N_\lambda\,\vert\,\lambda<\alpha}$ in abelian group $G$ has all successive quotients
$N_{\lambda+1}/N_{\lambda}$ cyclic groups (of infinite or prime order), and $N_0=0$,
$\cup_{\lambda<\alpha}N_\lambda=G$, then $\setof{N_\lambda}$ is called a {\it composition series} for $G$.
We view the notion of {\it good filtration}, defined in \lemref{good_filtration_general} below,
as analogous to that of composition series for abelian groups.

The proof of the following lemma is straightforward and will be omitted.

\begin{lem}\label{ordinal_quotient}
Let $W$ be a well-ordered set, and let $f\colon W\to\XXX$ be an order preserving smooth function into a complete lattice.
Then the preimages of $f$ are of the form $\setof{\mu\,\vert\,\lambda_1\lqs\mu<\lambda_2}$ where $\lambda_2$ has a predecessor.
The quotient set $\WWW$ of the preimages is well-ordered, and the natural function $q\colon W\to\WWW$ is order preserving.
Moreover $\ord\WWW\lqs\ord W$. An element $\Lambda\in\WWW$ is limiting if and only if $\min\Lambda$ is limiting in $W$.
Furthermore the ordinal number of $W$ is limiting if and only if the ordinal number of $\WWW$ is.

The induced function $F\colon\WWW\to\XXX$ is an order preserving smooth injection and $\sup f(W)=\sup F(\WWW)$. \qed
\end{lem}

\begin{lem}\label{good_filtration_general}
Let $\XXX$ be a regular lattice with set of compact elements $\KKK$. Set $\xi=\sup\XXX$.
Let $\alpha$ denote the initial ordinal of cardinality $\card_{\KKK}(\xi)$.

There exists an order preserving smooth injection $U\colon W(\alpha)\to\XXX$ whose image $\UUU$
is well-ordered, and exhaustive in the sense that $\sup\UUU=\xi$. Moreover, for each subset $\Lambda\subset\UUU$ either
$\sup\Lambda=\xi$ or $\sup\Lambda\in\UUU$ (in $\XXX$). In addition, the following two properties hold.
\begin{itemize}
\item	For each $\mu\in W(\alpha)$, the image $U(\mu+1)$ of $\mu+1$ equals
	$U(\mu)\vee\kappa=\sup\setof{U(\mu),\kappa}$ for some $\kappa=\kappa(\mu)\in\KKK$, and
\item	for each $\mu<\alpha+1$, the image $U(\mu)$ is bounded as $\card_{\KKK}U(\mu)\lqs\card W(\mu)$.
\end{itemize}
We name $\UUU$ a good filtration for $\XXX$ with respect to $\KKK$.
\end{lem}

\begin{proof}
Let $E$ be a subset of $\KKK$ with $\sup E=\xi$ and $\card E=\card_{\KKK}\xi$. Abusing notation we assume that $E$ is
an initial ordinal corresponding to the cardinal of $E$. To avoid ambiguity we denote the (new) relation in $E$ by $\prec$.
We define $u\colon E\to\XXX$ by setting
\[ u(f)=\sup\setof{e\,\vert\,e\in E,\,\,e\prec f}. \]
Obviously $u$ is monotone. Since it is initial, $E$ is limiting, and hence $f\lqs u(f+1)$ (in $\XXX$) for each $f\in E$.
In particular, $\sup\setof{u(f)\,\vert\,f\in E}=\sup E=\xi$. Moreover, note that if $g$ is limiting in $E$, then
\begin{equation*} u(g)=\sup\setof{u(f)\,\vert\,f\prec g}. \end{equation*}
This says that $u$ is smooth.

Let $\EEE$ denote the quotient set of preimages, and let $U\colon\EEE\to\XXX$ denote the induced function.
Furthermore let $j\colon\EEE\to E$ denote the function defined by $j(\lambda)=\max_E\lambda$.
By \lemref{ordinal_quotient} the function $U$ is an order preserving smooth injection, and $j$
is a well-defined order preserving injection. In particular, $\ord\EEE\lqs\ord E$.

Note that $U=u\circ j$, and an easy transfinite argument shows that \[ U(\mu)=\sup\setof{j(\lambda)\,\vert\,\lambda<\mu}. \]
This yields $\card_\KKK U(\mu)\lqs\card W(\mu)$.
Moreover, $\xi=\sup E=\sup\setof{j(\lambda)\,\vert\,\lambda\in\EEE}$ from which it follows that $\card\EEE=\card E$.
Since $E$ is initial it follows that $\ord\EEE=\ord E$, as claimed.
\end{proof}

%%  We note the evident
%%  \begin{lem}
%%  Let $\alpha$ be a limit ordinal and let $(\W,\p)$ be an inverse system indexed by $W(\alpha)$ with
%%  $\W\vert_{W(\mu)}$ restricted for all $\mu\in W(\alpha)$. Then $(\W,\p)$ is restricted. \qed
%%  \end{lem}

%%  \begin{proof}
%%  Let $\eta$ be a limit ordinal smaller than $\alpha$. But then $\eta+1<\alpha$, the system $\W\vert_{W(\eta+1)}$
%%  is restricted and we are done.
%%  \end{proof}

\begin{lem}\label{restricted_trick}
Let $\alpha$ be a limit ordinal and let $\Z$ be a restricted inverse system of fibrations
indexed by $W(\alpha)$. Let $Z$ denote the limit space and let $C$ be the union of some path components of $Z$.
For each $\lambda<\alpha$ denote by $C_\lambda$ the image of $C$ under $Z\to Z_\lambda$. If for each
$\lambda$ the induced fibration $C_{\lambda+1}\to C_{\lambda}$ is a homotopy equivalence, then the
system $\setof{C_\lambda\,\vert\,\lambda\in W(\alpha)}$ is restricted and hence a system of homotopy
equivalences. In particular, its limit space is $C$ and is homotopy equivalent to $C_0$.
\end{lem}

\begin{proof}
Let $S$ be the set of those $\mu\in W(\alpha)$ for which the system $\setof{C_\lambda\,\vert\,\lambda\lqs\mu}$
is restricted (and hence a system of homotopy equivalences). The set $S$ is nonempty, and it is evidently an
initial segment of $W(\alpha)$. Suppose $S=W(\eta)$ for some $\eta\in W(\alpha)$.

Immediately we infer that the system $\setof{C_\lambda\,\vert\,\lambda<\eta}$ is restricted.

If $\eta$ has a predecessor $\eta-1$, then $\setof{C_\lambda\,\vert\,\lambda\lqs\eta}$ is automatically a restricted
system.

Suppose that $\eta$ is limiting. Denote $\Bar C=\lim_{\lambda<\eta}C_\lambda$.
Trivially $\Bar C$ is a subspace of $\lim_{\lambda<\eta}Z_\lambda$ which is equal (homeomorphic) to
the space $Z_{\eta}$, by restrictedness. By definition of the spaces $C_\lambda$, the space $C_{\eta}$
is a subset of $\Bar C$. Moreover $C_{\eta}\to C_0$ is surjective. By \corref{cor_transfinite_geoghegan_2}
the map $\Bar C\to C_0$ is a homotopy equivalence. Therefore $\Bar C$ cannot meet path components of $Z_{\eta}$
different from those that meet $C_{\eta}$. But $C_{\eta}$ is the union of some path components,
therefore in fact $\Bar C$ equals $C_{\eta}$, and $C_{\eta}\to C_0$ is a homotopy equivalence. Thus the system
$\setof{C_\lambda\,\vert\,\lambda<\eta}$ is restricted.

The contradiction yields $S=W(\alpha)$, and it follows that $\setof{C_\lambda\,\vert\,\lambda<\alpha}$ is
a restricted system.
\end{proof}

We come to the two main results of the section.

\begin{thm}\label{wicked_general}
Let $\Z$ be a restricted inverse system of fibrations indexed by a complete lattice $\XXX$ with a set of finite
elements $\KKK$. Assume that $\XXX$ is uncountable with respect to $\KKK$. Let $Z$ denote the inverse limit space
of $\Z$ and let $C$ be a set of path components of $Z$. For $L\in\XXX$ let $C_L$ denote the image of $C$ under
$Z\to Z_L$. Let $\Gamma$ be a topological space which maps to the system $\setof{C_L\,\vert\,L}$.

Let $\LLL$ be the set of those $L\in\XXX$ for which $C_L$ has CW type and the map $\Gamma\to C_L$ is
a weak equivalence.

Let $\XXX(\eta)$ (respectively $\LLL(\eta)$) denote the subset of $\XXX$ (respectively $\LLL$) of 
elements of $\KKK$-cardinality at most $\aleph_\eta$.

If $\LLL(0)$ is cofinal in $\XXX(0)$, then $\LLL(\eta)$ is cofinal in $\XXX(\eta)$, for any ordinal $\eta$.
In particular, $C$ has CW homotopy type, and $\Gamma\to C$ is a weak equivalence.
\end{thm}

\begin{rem_n}
The role of $\Gamma$ mapping into $\setof{C_L\,\vert\,L}$ is the implication
\begin{equation*}\tag{$*$}
L',L''\in\LLL,\,\,\,L'\lqs L''\,\implies\,C_{L''}\to C_{L'}\text{ is a homotopy equivalence},
\end{equation*}
by Whitehead's theorem.
\end{rem_n}

We will make use of the following observation.
\begin{align*}\tag{$**$}
&\text{Let $\eta$ be a limit ordinal and let $f\colon W(\eta)\to\LLL$ be an order preserving}\\
&\text{smooth function. Then $L=\sup\setof{f(\lambda)\,\vert\,\lambda<\eta}\in\LLL$. }
\end{align*}

\begin{proof}[Proof of ($**$)]
By \lemref{ordinal_quotient} the function $f$ induces an order preserving smooth injection $F\colon\WWW\to\XXX$ for which
$\sup F=\sup f$. Here $\WWW$ denotes the well-ordered quotient set. Since $\Z$ is a restricted inverse system, so is
$\setof{Z_{F(\lambda)}\,\vert\,\lambda\in\WWW}$. The limit space of the latter is $Z_{\sup_\lambda F(\lambda)}$:

Indeed, if $\mu$ is a limit ordinal in $\WWW$, then smoothness implies that
$F(\mu)=\sup_{\XXX}\setof{F(\lambda)\,\vert\,\lambda<\mu}$. The set $F(W(\mu))$ is a directed (in fact well-ordered)
subset of $\XXX$, and since $\Z$ is a restricted system, it follows that
\[ Z_{\sup_{\XXX}F(W(\mu))}=\lim\big(\Z\vert_{F(W(\mu))}\big). \] This is to say that
$\Z\vert_{F(\WWW)}$ is a restricted system. Again using restrictedness of the system
$\Z$ we infer the equality $\lim\big(\Z\vert_{F(\WWW)}\big)=Z_{\sup_\XXX F(\WWW)}$.

Let $L=\sup f=\sup F$. For $\lambda\in\WWW$ the set $C_{F(\lambda)}$ is the image of
$C_L$ under projection $Z_L\to Z_{F(\lambda)}$. The system $\setof{C_{F(\lambda)}\,\vert\,\lambda\in\WWW}$
is one of homotopy equivalences by ($*$), therefore it is restricted by \lemref{restricted_trick}, and
$C_L$ is its limit space. In particular, $C_L\to C_{F(0)}$ is a homotopy equivalence, and $L$ belongs to $\LLL$.
\end{proof}

\begin{proof}[Proof of the theorem]
Let $S$ be the set of those $\lambda\in W(\xi)$ for which $\LLL(\lambda)$ is cofinal in $\XXX(\lambda)$.
By hypothesis $0\in S$.

We will prove that $S$ has the inductive property. To this end suppose $W(\eta)\subset S$,
and choose $L$ whose $\KKK$-cardinality is exactly $\aleph_\eta$. \lemref{good_filtration_general}
guarantees a good filtration $\setof{L_\lambda\,\vert\,\lambda\in W(\alpha)}$ for $L$, selected from the regular lattice
$\setof{M\in\XXX\,\vert\,M\lqs L}$ with the set of compact elements $\KKK\cap L$. In particular, $\alpha$
is an initial ordinal of cardinality $\aleph_\eta$, and for each element $\lambda\in W(\alpha)$, the
$\KKK$-cardinality of $L_\lambda$ is at most $\card W(\lambda)$.

Let $\XXX'$ be the subset of those elements of $\XXX$ whose $\KKK$-cardinality is strictly smaller than
$\aleph_\eta$. Set $\LLL'=\LLL\cap\XXX'$, and choose a well-ordering of $\LLL'$.

Using transfinite construction we define an order preserving smooth function $f\colon W(\alpha)\to\LLL'$
for which $f(\lambda)\gqs L_\lambda$, $\forall\lambda$.

Assume $f$ already given for $\lambda<\mu$ where $\mu<\alpha$. Note that
$\card W(\mu)<\aleph_\eta$, and that the following implication holds.
\begin{equation*}\tag{\dag}
(\forall \lambda<\mu\colon\,\card_\KKK M_\lambda<\aleph_\eta)\implies\card_\KKK\big(\sup\setof{M_\lambda\,\vert\,
	\lambda<\mu}\big)<\aleph_\eta.
\end{equation*}

Suppose that $\mu$ has a predecessor $\mu-1$. Then the $\KKK$-cardinality of the element
$f(\mu-1)\vee L_{\mu}$ is strictly smaller than $\aleph_\eta$. Since by assumption $W(\eta)$ is contained in $S$,
there exists an element of $\LLL'$, which is bigger than $f(\mu-1)\vee L_{\mu}$ with respect to the lattice
ordering. Let $f(\mu)$ be the least of those in the well-ordering of $\LLL'$.

If $\mu$ is limiting we define $f(\mu)=\sup\setof{f(\lambda)\,\vert\,\lambda<\mu}$.
Equality $L_\mu=\sup_{\lambda<\mu}L_\lambda$ evidently implies $f(\mu)\gqs L_\mu$. Moreover,
$f(\mu)$ belongs to $\LLL'$ by ($**$) and (\dag).

Thus there exists a function $f\colon W(\alpha)\to\LLL'$ with the requisite properties. Now set
$L'=\sup\setof{f(\lambda)\,\vert\,\lambda<\alpha}$. The $\KKK$-cardinality of $L'$ is at most
$\card W(\alpha)=\aleph_\eta$. Again from the equality $L=\sup_{\lambda<\alpha}L_\lambda$ it follows that $L'\gqs L$.
By ($**$), the element $L'$ belongs to the set $\LLL$, whence we infer that $L'\in\LLL(\eta)$, and $\eta\in S$.

Transfinite induction yields $S=W(\xi)$, and applying the definition of $S$ to the element $X$ we conclude the proof.
\end{proof}

\begin{lem}\label{allaud_refined}
Let $p\colon E\to B$ be a fibration where $E$ is a Dold space.
Suppose $p(e_0)=b_0$ and let $F$ be the fibre of $p$ over $b_0$.
Then the induced map $\Omega p\colon\Omega(E,e_0)\to\Omega(B,b_0)$ is a homotopy equivalence if and only if
$F$ is contractible.
\end{lem}
\begin{proof}
The proof is contained in the proof of the `delooping theorem' of Allaud \cite{allaud}.
\end{proof}

\begin{defn}
Let $\XXX$ be a regular lattice with set of compact elements $\KKK$.
A property $P=P(\alpha,\beta)$ of pairs $(\alpha,\beta)\in\KKK\times\KKK$ is
an {\it order preserving property} if for any $\alpha',\alpha,\beta,\beta'\in\KKK$ the implication
\[ \alpha'\lqs\alpha,\,\,\,\beta\lqs\beta',\,\,P(\alpha,\beta)\implies P(\alpha',\beta'). \]
(For example $\lqs$ is an order preserving property.)

If, in addition, for each $\alpha\in\KKK$ there exists $\beta\in\KKK$ with $\beta\gqs\alpha$ and $P(\alpha,\beta)$
then $P$ is a {\it directed} order preserving property.

We say that an ascending sequence \[ \lambda_1\lqs\lambda_2\lqs\dots \] of elements of $\KKK$
is a $P$-sequence if $P(\lambda_i,\lambda_{i+1})$ holds for all $i\gqs 1$.
\end{defn}

The proof of the following lemma is trivial.
\begin{lem}\label{P-sequence}
Let $\XXX$ be a regular lattice with set of compact elements $\KKK$.
Suppose $P$ is a directed order preserving property of ordered pairs.
Then for any ascending sequence \[ \kappa_1\lqs\kappa_2\lqs\dots \] of elements of $\KKK$
there exists a $P$-sequence \[ \lambda_1\lqs\lambda_2\lqs\dots \]
of elements of $\KKK$ which dominates $\setof{\kappa_i}$ in the sense that $\lambda_i\gqs\kappa_i$.\qed
\end{lem}

\begin{thm}\label{double_wicked_general}
Let $\Z$ be a restricted inverse system of fibrations indexed by a regular lattice $\XXX$ with set of compact
elements $\KKK$, and let $Z$ denote the limit space of $\Z$. Let $\zeta\in Z$.
\begin{enroman}
\item	If the path component of $\zeta$ has CW homotopy type then for every $\KKK$-countable $\kappa_\infty$ there exists
	a bigger $\KKK$-countable $\lambda_\infty$ such that
	\begin{equation*}
		\tag{$\star$}	\Omega(Z,\zeta)\to\Omega(Z_{\lambda_\infty},\zeta\vert_{\lambda_\infty})
	\end{equation*}
	is a homotopy equivalence. Consequently the fibre of $C\to C_{\lambda_\infty}$ is contractible.

	In addition, if $P$ is a directed order preserving property of ordered pairs in $\KKK$,
	then $\lambda_\infty$ may be chosen to be the supremum of a $P$-sequence.
\item	Conversely, if for each $\KKK$-countable $\kappa_\infty$ there exists a bigger $\KKK$-countable element
	$\lambda_\infty$ such that $\Omega(Z_{\lambda_\infty},\zeta\vert_{\lambda_\infty})$ has CW type and the map
	($\star$) is a weak homotopy equivalence, then $\Omega(Z,\zeta)$ has CW homotopy type.
\end{enroman}
\end{thm}

\begin{proof}
Assume the path component $C$ containing $\zeta$ has CW homotopy type. As usual, let $C_\mu$ denote the image of
$C$ under $Z\to Z_\mu$. 

Assume $\kappa_\infty=\sup\setof{\kappa_1,\kappa_2,\dots}$ where $\kappa_i\in\KKK$. By \propref{strong_obstruction_general}
there exists $\kappa_0\in\KKK$ such that for each $\mu\in\XXX$ with $\mu\gqs\kappa_0$ the fibre $F_\mu$ of $C\to C_\mu$
over $\zeta_\mu$ contracts in $C$.

For $\alpha,\beta\in\KKK$ we define $Q(\alpha,\beta)$ if $\alpha\lqs\beta$, $P(\alpha,\beta)$, and the inclusion
$F_\beta\to F_\alpha$ is nullhomotopic. Clearly $Q$ is an order preserving property. By assumption on $P$ and
by \propref{strong_obstruction_general} the property $Q$ is also directed.

We apply \lemref{P-sequence} to obtain a $Q$-sequence $\setof{\lambda_i}$ dominating the sequence
$\setof{\kappa_i\vee\kappa_0}$. In particular $P(\lambda_i,\lambda_{i+1})$ for all $i\gqs 1$ and $\lambda_i\gqs\kappa_0$
for all $i$. Let $\lambda_\infty=\sup\setof{\lambda_i}$. By domination $\lambda_\infty\gqs\kappa_\infty$.

To continue we write $i$ instead of $\lambda_i$ for $1\lqs i\lqs\infty$ to ease the notation.
\propref{strong_obstruction_general} further guarantees commutative diagrams
\begin{equation*}\begin{diagram}
\node{\Omega(Z_i,\zeta_i)} \arrow[3]{e,t}{p_{i-1,i}{}_\#} \node[3]{\Omega(Z_{i-1},\zeta_{i-1})} \\
\node{F_i\times\Omega(Z,\zeta)} \arrow{n,l}{\simeq} \arrow[3]{e,t}{\text{inclusion}} \node[3]{F_{i-1}\times\Omega(Z,\zeta)}
	\arrow{n,r}{\simeq}
\end{diagram}\end{equation*}
where the vertical homotopy equivalences are given by
\[ F_i\times\Omega(Z,\zeta)\to\Omega(Z_i,\zeta_i),\,\,\,
	(z,\gamma)\mapsto p_{i,\alpha}{}_\#(\gamma)*\theta_i(z), \]
and the maps $\theta_i$ satisfy $p_{i-1,i}{}_\#\circ\theta_i=\theta_{i-1}\circ\text{inclusion}$.

Set $\Phi_1=F_1$ and let $\dots\to\Phi_3\to\Phi_2\to\Phi_1$ be the inverse sequence obtained by changing
$\dots\to F_2\to F_1$ inductively into a sequence of fibrations. Thus we obtain (strictly) commutative diagrams
\begin{equation*}\begin{diagram}
\node{F_i} \arrow{s,lr}{f_i}{\simeq} \arrow{e} \node{F_{i-1}} \arrow{s,lr}{\simeq}{f_{i-1}} \\
\node{\Phi_i} \arrow{e,t}{q_{i-1,i}} \node{\Phi_{i-1}}
\end{diagram}\end{equation*}
where the $f_i\colon F_i\to\Phi_i$ are homotopy equivalences. Pick homotopy inverses $g_i$ for $f_i$. Since the
maps $p_{i-1,i}{}_\#$ are fibrations, we may homotope the composites $\theta_i\circ g_i$ inductively to maps
$\varphi_i\colon\Phi_i\to\Omega(Z_i,\zeta_i)$ for which
\[	p_{i-1,i}{}_\#\circ\varphi_i=\varphi_{i-1}\circ q_{i-1,i}. \]
By defining $\psi_i\colon\Phi_i\times\Omega(Z,\zeta)\to\Omega(Z_i,\zeta_i)$ as
$\psi_i(x,\gamma)=p_{i,\alpha}{}_\#(\gamma)*\varphi_i(x)$, we 
get a (strict) map between towers of fibrations
\begin{equation*}\begin{diagram}
\node{\Omega(Z_i,\zeta_i)} \arrow[3]{e,t}{p_{i-1,i}{}_\#} \node[3]{\Omega(Z_{i-1},\zeta_{i-1})} \\
\node{\Phi_i\times\Omega(Z,\zeta)}
	\arrow{n,lr}{\psi_i}{\simeq}
	\arrow[3]{e} \node[3]{\Phi_{i-1}\times\Omega(Z,\zeta)}
	\arrow{n,lr}{\simeq}{\psi_{i-1}}
\end{diagram}\end{equation*}
consisting of homotopy equivalences. By \propref{E-H} the limit map $\psi_\infty$ is also a homotopy equivalence.

The limit space of $\setof{\Phi_i\times\Omega(Z,\zeta)}$ is $\Phi_\infty\times\Omega(Z,\zeta)$ where
$\Phi_\infty$ is the limit space of $\setof{\Phi_i}$. If we let
$\varphi_\infty\colon\Phi_\infty\to\Omega(Z_\infty,\zeta_\infty)$ denote the limit map of the sequence
$\setof{\varphi_i}$, then by universality,
\begin{equation*}\tag{$\star\star$} \psi_\infty(x,\gamma)=p_{\infty,\alpha}{}_\#(\gamma)*\varphi_\infty(x). \end{equation*}
Let $x\in\Phi_\infty$. By \propref{contractible_limit} the space $\Phi_\infty$ is contractible, hence the inclusion
\[ \iota\colon\Omega(Z,\zeta)\equiv\setof{x}\times\Omega(Z,\zeta)\hookrightarrow\Phi_\infty\times\Omega(Z,\zeta) \]
is a homotopy equivalence. Therefore so is the map $\psi_\infty\circ\iota$. But by ($\star\star$), the map $\psi_\infty$
for a fixed $x$ is the restriction $\Omega(Z,\zeta)\to\Omega(Z_\infty,\zeta_\infty)$ followed by a multiplication
homotopy equivalence. Hence the map ($\star$) is a homotopy equivalence, and an application of \lemref{allaud_refined}
concludes the proof of \navedi{1}.

Statement \navedi{2} follows immediately from \thmref{wicked_general}, by setting $\Gamma=\Omega(Z,\zeta)$, and
using \lemref{continuity_of_loop_and_product}.
\end{proof}

%%%%%%%%%%%%%%%%%%%%%%%%%%%%%%%%%%%%%%%%%%%%%%%%%%%%%%%%%%%%%%%%%%%%%%%%%
%									%
%	Phantom components and the uniform Mittag-Leffler property	%
%									%
%%%%%%%%%%%%%%%%%%%%%%%%%%%%%%%%%%%%%%%%%%%%%%%%%%%%%%%%%%%%%%%%%%%%%%%%%

\section{Phantom components and the uniform Mittag-Leffler property}\label{phantom_components}

We introduce the notion of phantom path components of the limit space of an inverse system of fibrations
between CW type spaces. If $X$ is a CW complex then $Y^X$ is the limit of $\setof{Y^K}$ with $K$ ranging
over the set $\KKK$ of finite subcomplexes of $X$. Phantom path components of $Y^X$ then coincide with
path components of phantom maps $X\to Y$ with respect to $\KKK$ (see also Section \ref{phantom} below).

\begin{defn}
Let $\Lambda$ be a directed set and let $\setof{G_\lambda}$ be an inverse system of groups indexed by $\Lambda$.
For $\lambda_1\lqs\lambda_2$ we denote the corresponding bonding morphism
$p_{\lambda_1\lambda_2}\colon G_{\lambda_2}\to G_{\lambda_1}$. We say that $\setof{G_\lambda}$ satisfies
the Mittag-Leffler condition if for each $\lambda\in\Lambda$ there exists $\mu\in\Lambda$ with $\mu\gqs\lambda$
so that for each $\nu\gqs\mu$ the images of $p_{\lambda\nu}$ and $p_{\lambda\mu}$ coincide.
\end{defn}

Recall the following well known result (see for example Marde\v{s}i\'{c} and Segal \cite{mardesic-segal} II,
Theorem 7.1.1, and J.~Cohen \cite{joel_cohen_1} for generalizations)

\begin{prop}\label{homotopy_groups_of_inverse_limits}
Let $(Z,\setof{P^{i}})$ be the limit of inverse sequence of fibrations
$\dots\to Z_3\xrightarrow{p_3}Z_2\xrightarrow{p_2}Z_1$. For each number $k\gqs 0$ there exists a natural
exact sequence (of pointed sets)
\[ *\to\ll\pi_{k+1}(Z_i,*_i)\xrightarrow{\phi}\pi_k(Z,*)\xrightarrow{\lim\pi_k(P^i)}\lim\pi_k(Z_i,*_i)\to*. \]
In particular, $\phi$ is an injection. \qed
\end{prop}

We refer to \cite{mardesic-segal}, II, Theorems 6.2.10 and 6.2.11 for a proof of

\begin{prop}\label{lim1_mittag-leffler}
Let $\setof{G_j,p_j}$ be an inverse sequence of groups. Then $\ll G_j$ is trivial
if the sequence satisfies the Mittag-Leffler condition. If the groups $G_j$ are
countable, the converse also holds. \qed
\end{prop}

\begin{defn}
Let $\Z$ be a restricted system of fibrations indexed by a regular lattice $\XXX$ with set of compact elements $\KKK$.
Let $Z$ denote the inverse limit of $\Z$. 

Path components $C$, $D$ of $Z$ form {\it a phantom pair with respect to $\KKK$} if
the images of $C$ and $D$ under $Z\to Z_\kappa$ coincide for each $\kappa\in\KKK$.
They form a nontrivial phantom pair if $C\neq D$.

Clearly the relation `phantom pair' is an equivalence relation, and we denote by
$\Ph(C)$ the equivalence class of $C$, i.e. the set of all path components $D$
such that $C$, $D$ form a phantom pair.
\end{defn}

\begin{lem}\label{phantom_countable}
Assume that $\XXX$ is $\KKK$-countable, and let $\sup\XXX$ equal $\sup\setof{\kappa_i}$ where
$\kappa_1\lqs\kappa_2\lqs\dots$ is a countable ascending chain in $\KKK$. Clearly $\setof{\kappa_i}$
is cofinal in $\KKK$ and by \propref{homotopy_groups_of_inverse_limits} we may identify
\[	\Ph(C)=\ll\pi_{1}(Z_{\kappa_i},\zeta_i) \]
where $\zeta$ belongs to $C$ and $\zeta_i=\zeta\vert_{\kappa_i}$.\qed
\end{lem}

\begin{thm}\label{u_M-L_general}
Given hypotheses of \propref{fibre_SLC_general}, assume that $C$ has the homotopy type of a CW complex.
Then, given any $\kappa\in\KKK$ there exists $\lambda'\in\KKK$ such that for all $k\gqs 1$
and any $\mu\in\XXX$, for which $\mu\gqs\lambda'$, the image of
\[ \pi_k(C_\mu,\zeta_\mu)\to\pi_k(C_\kappa,\zeta_\kappa) \] equals that of
\begin{equation*}\tag{$*$} \pi_k(C,\zeta)\to\pi_k(C_\kappa,\zeta_\kappa). \end{equation*}
Moreover, for all large enough $\kappa$, the morphism ($*$) is injective for $k\gqs 1$.

In particular, the $\KKK$-indexed induced inverse system of groups 
\begin{equation*}\tag{$**$} \setof{\pi_k(C_\lambda,\zeta_\lambda)} \end{equation*} satisfies
the Mittag-Leffler condition uniformly with respect to $k\gqs 1$.

If $C$ is open then for all large enough $\kappa\in\KKK$ the preimage of $C_\kappa$ under 
$Z\to Z_\kappa$ equals $C$. Consequently $\Ph(C)=\setof{C}$.

If $\XXX$ is $\KKK$-countable then $\Ph(C)=\setof{C}$ regardless of whether $C$ is open or not.
\end{thm}

\begin{proof}
The first statement follows from \navedi{2} of \propref{strong_obstruction_general} by applying homotopy groups to the
diagram of \navedi{1} of the same proposition.

We may first choose $\kappa'\in\KKK$ such that the fibre of $C\to C_{\kappa'}$ over $\zeta_{\kappa'}$ contracts in $C$.
Then if $\kappa\gqs\kappa'$ the morphism ($*$) is an injection.

If $\XXX$ is $\KKK$-countable then by \lemref{phantom_countable} we may identify $\Ph(C)$
with $\ll$ of a subsequence of ($**$) which is trivial by \propref{lim1_mittag-leffler}.

If $C$ is open then by \propref{fibre_SLC_general} we may assume that the fibre of $Z\to Z_{\kappa'}$
over $\zeta_{\kappa'}$ contracts in the total space. By applying $\pi_0$ this shows that the preimage of
$C_{\kappa'}$ under $Z\to Z_{\kappa'}$ is exactly $C$.
\end{proof}

%%%%%%%%%%%%%%%%%%%%%%%%%%%%%%%%%%%%%%%%%%%%%%%%%%%%%%%%%%%%%%%%%%%%%%%%%
%									%
%		Sequences of Postnikov sections				%
%									%
%%%%%%%%%%%%%%%%%%%%%%%%%%%%%%%%%%%%%%%%%%%%%%%%%%%%%%%%%%%%%%%%%%%%%%%%%

\section{Sequences of Postnikov sections}\label{sequences_of_postnikov_sections}

Let $\dots\to Y_i\xrightarrow{p_i}Y_{i-1}\to\dots\to Y_1$ be an inverse sequence of fibrations between CW type spaces.
We do not know whether in general the question of CW type of the limit space $Y_\infty$ can be reduced to studying
the morphisms $\pi_k(Y_i)\to\pi_k(Y_{i-1})$ induced on homotopy groups. However, this is the case when all the $Y_i$
have trivial homotopy groups above a certain level (independent of $i$).

While \propref{whitehead_tower_sequence}, the key ingredient of \thmref{postnikov_sequence}, is more general
in nature, we need the vanishing of homotopy groups to `reach an end' of an inductive if and only if argument.

\begin{defn}
Let $\dots\to G_3\to G_2\to G_1$ be an inverse sequence of groups with inverse limit $G_\infty$.
We say that $\setof{G_i}$ is {\it injectively Mittag-Leffler} if it satisfies the Mittag-Leffler
condition and the canonical projections $G_\infty\to G_i$ are injective for all but finitely many $i$.
\end{defn}

Let $\kkkk$ denote the `$k$-ification functor', that is the functor which to
every topological space assigns the space with the same underlying set whose
topology is the compactly generated refinement of the original one. (See Steenrod \cite{steenrod}.)

\begin{thm}\label{postnikov_sequence}
Let $\dots\to Y_i\xrightarrow{p_i}Y_{i-1}\to\dots\to Y_1$ be an inverse sequence of fibrations with a coherent set
of nondegenerate base points $\setof{\eta_i}$. Let $Y_\infty$ be the limit space and let $C$ be the path component of
$\eta_\infty=\setof{\eta_i}$ in $Y_\infty$. Further let $C_i$ be the image of $C$ under $Y_\infty\to Y_i$.
Assume that all spaces $Y_i$ have CW homotopy type and that there exists
a number $N$ so that $\pi_k(Y_i,\eta_i)=0$ for $k\gqs N+1$ and all $i$.

\begin{enroman}
\item	If $C$ has the homotopy type of a CW complex then
	for each $k\gqs 1$ the sequence $\setof{\pi_k(Y_i,\eta_i)\,\vert\,i}$ is injectively Mittag-Leffler.
	
	If, in addition, $C$ is open, then the preimage of $C_i$ under $Y_\infty\to Y_i$ equals $C$ for
	all but finitely many $i$.
\item	Conversely, if for each $k\gqs 1$ the sequence $\setof{\pi_k(Y_i,\eta_i)\,\vert\,i}$ is injectively
	Mittag-Leffler, then $\kkkk(C)$ has CW homotopy type. If, in addition, $\pi_k(Y_i,\eta_i)$ is countable
	for $k\gqs 2$ and all $i$ then $C$ has CW type.

	If, in addition, the preimage of $C_i$ under $Y_\infty\to Y_i$ equals $C$ then $C$ is open in $Y_\infty$.
\end{enroman}
\end{thm}

\begin{cor}
Let $\dots\to Y_i\xrightarrow{p_i}Y_{i-1}\to\dots\to Y_1$ be an inverse sequence of fibrations with a coherent set
of nondegenerate base points $\setof{\eta_i}$. Assume that all spaces $Y_i$ have homotopy types of countable CW complexes
and that there exists a number $N$ so that $\pi_k(Y_i,\eta_i)=0$ for $k\gqs N+1$ and all $i$. 

Then the inverse limit of $\setof{Y_i}$ is contractible if and only if it is weakly contractible.
\end{cor}

\begin{proof}[Proof of Corollary]
Weak contractibility implies the vanishing of $\ll$ terms by \propref{homotopy_groups_of_inverse_limits}.
By \propref{lim1_mittag-leffler} the sequences $\setof{\pi_k(Y_i,*)\,\vert\,i}$ satisfy the Mittag-Leffler property.
\end{proof}

\begin{example}\label{solenoid}
Consider the inverse sequence of circles 
\[ \dots\to S^1\xrightarrow{p_2}S^1\xrightarrow{p_1}S^1. \]
where $p_i$ denotes the $n_i$-sheeted cover $\zeta\mapsto\zeta^{n_i}$ for a number $n_i$,
and we understand $S^1$ as the set of complex numbers of absolute value $1$.

The induced sequence on $\pi_1$ is $\dots\to\ZZ\xrightarrow{n_2}\ZZ\xrightarrow{n_1}\ZZ$. This has the Mittag-Leffler
property if and only if either 
\begin{itemize}
\item	$n_i=\pm 1$ for all but finitely many $i$ in which case the limit is $S^1$, or
\item	$n_i=0$ for infinitely many $i$ in which case the limit is a point.
\end{itemize}
In particular, the $p$-adic solenoid $T_p$ for a prime $p$ obtained by setting $n_i=p$ for all $i$,
has $\Hat\ZZ_p/\ZZ$ path components all of which form a single phantom class, and each is weakly contractible
but not contractible.\qed
\end{example}

The rest of this section is devoted to the proof of \thmref{postnikov_sequence}.

\begin{prop}\label{E-M_sequence}
Let $k$ be a natural number and let
\[ \dots\to K_i\xrightarrow{p_i}K_{i-1}\to\dots\to K_1 \]
be an inverse sequence of fibrations where the $K_i$ are Eilenberg-MacLane spaces
with the single nonvanishing homotopy group in dimension $k$.

Let $K_\infty$ denote the inverse limit of the sequence, together with natural projections $P^i\colon K_\infty\to K_i$.
Pick a base point $\zeta=\setof{\zeta_i}\in K_\infty$ and let $G_i=\pi_k(K_i,\zeta_i)$. Denote the induced morphism
$G_i\to G_{i-1}$ simply by $p_i$ and let $G_\infty$ denote the inverse limit of $\setof{G_i,p_i}$. 

Assume that the sequence $\setof{G_i,p_i}$ is injectively Mittag-Leffler,
and, in case $k=1$, that the point $\zeta_i$ is nondegenerate in $K_i$, for each $i$.

Then the limit space $K_\infty$ has the homotopy type of $K(G_\infty,k)$. In particular, it has the homotopy type
of a CW complex.
\end{prop}

\begin{proof}
Note that since $\pi_{k+1}(K_i,\zeta_i)=0$, $\pi_k(K_\infty,\zeta)\cong G_\infty$, by
\propref{homotopy_groups_of_inverse_limits}. By the same proposition, the space $K_\infty$
has a single nonvanishing homotopy group, since the sequence $\setof{G_i,p_i}$ has the Mittag-Leffler property.

More precisely, the Mittag-Leffler property implies that the images of $G_j\to G_i$ stabilize for large enough $j$.
By replacing the sequence $\setof{G_i}$ with an appropriate subsequence we may assume that for each $i$ the image
$S_{i-1}$ of $G_i\to G_{i-1}$ equals that of $G_j\to G_{i-1}$ for all $j\gqs i$ (including $j=\infty$).

In addition, we may assume that for each $i$, the morphism $G_\infty\to G_i$ is injective.
This implies that the morphisms $G_\infty\to S_i$ are bijective and consequently so are $S_i\to S_{i-1}$,
for all $i$.

Now we treat cases $k\gqs 2$ and $k=1$ separately.

Assume first that $k\gqs 2$. The composite $G_\infty\to G_2\to G_1$ is then an isomorphism onto $S_1$.
Hence the injection $G_\infty\to G_2$ has a left inverse, and since the groups are abelian, $G_\infty\to G_2$ splits. By
possibly neglecting the first term we may therefore assume that all injections $G_\infty\to G_i$ are split. Let $K_0$
be a CW complex of type $K(G_\infty,k)$ and let $p_1\colon K_1\to K_0$ be a map inducing a `projector' $G_1\to G_\infty$
on the homotopy group. Split $p_1$ canonically as $K_1\xrightarrow{f_1}L_1\xrightarrow{q_1}K_0$ where $f_1$ is a homotopy
equivalence and $q_1$ is a fibration. Given $f_i\colon K_i\to L_i$ split $f_ip_{i+1}\colon K_{i+1}\to L_i$ as
$K_{i+1}\xrightarrow{f_{i+1}}L_{i+1}\xrightarrow{q_{i+1}}L_i$ with $f_{i+1}$ a homotopy equivalence and $q_{i+1}$
a fibration. Let $f_\infty\colon K_\infty\to L_\infty$ denote the induced homotopy equivalence on the limit spaces.
We note that the projection $L_\infty\to L_1$ is a fibration and consequently so is the composite map $L_\infty\to K_0$.
Clearly the fibre of $L_\infty\to K_0$ over $b_0\in K_0$ is the inverse limit of the fibres $F_i$ of $L_i\to K_0$ over $b_0$.
But the induced maps $F_i\to F_{i-1}$ are nullhomotopic, and hence the limit $F_\infty$ is contractible.
Therefore $L_\infty\simeq K_0$, as claimed.

Now assume that $k=1$.
For each $i$ choose a CW complex $C_i$ having the homotopy type of $K_i$. Next choose a base point $\gamma_i\in C_i$.
Since the base points $\gamma_i\in C_i$ and $\zeta_i\in K_i$ are nondegenerate, $(K_i,\zeta_i)$ and $(C_i,\gamma_i)$
are equivalent as pairs; that is there exist pointed maps $g_i\colon K_i\to C_i$, $f_i\colon C_i\to K_i$ such that
$g_i\circ f_i$ and $f_i\circ g_i$ are homotopic to the respective identity maps via base point preserving homotopies.

Set $\Bar S_i=g_i{}_\#(S_i)$ and let $\Bar C_i\to C_i$ be a covering associated to the subgroup
$\Bar S_i\lqs\pi_1(C_i,\gamma_i)$. Pulling back this covering over $g_i\colon K_i\to C_i$ yields a fibration
$q_i\colon W_i\to K_i$ with discrete fibres, which maps $\pi_1(W_i,w_i)$ isomorphically onto $S_i$.

By construction the composite $C_{i+1}\xrightarrow{f_{i+1}}K_{i+1}\to K_i$ lifts to a map into $W_i$
(sending $\gamma_{i+1}$ to $w_i$). Precomposition with $g_{i+1}$ defines a map $K_{i+1}\to W_i$ which
lifts $K_{i+1}\to K_i$ up to homotopy. Since $W_i\to K_i$ is a fibration, a strict lifting $r_{i+1}$ can be
obtained. Since $\zeta_{i+1}$ is nondegenerate in $K_{i+1}$, we may arrange that $r_{i+1}(\zeta_{i+1})=w_i$.

Note that since the fibres of $q_i$ are discrete, the map $r_{i+1}$ is a fibration.
The limit space is thus homeomorphic to the limit space of
\[ \dots\to W_3\to W_2\to W_1 \]
which is a sequence of fibrations that are homotopy equivalences.
\end{proof}

We recall a definition due to J.~Cohen \cite{joel_cohen_2} in a slightly modified version
that is suitable for our purposes.

\begin{defn}
Let $\Phi\colon G\times E\to E$ denote a free left action of a topological group $G$ on a space $E$.
The action $\Phi$ is {\it open} if for every space $Z$ and every pair of maps
$f_1,f_2\colon Z\to E$ such that $f_1(z)$ and $f_2(z)$ lie in the same orbit for each $z$ there
exists a map $f\colon Z\to G$ so that \[ \Phi(f(z),f_1(z))=f_2(z),\,\,\,\forall z. \]
A fibration $p\colon E\to B$ is called a {\it principal open fibration}
if $E$ is a free left $G$-space with an open action such that $p(e_1)=p(e_2)$ if and only
if $e_1$ and $e_2$ lie in the same orbit. (This is to say that $p$ induces a continuous injection
$E/G\to B$ for which in general we do not require to be onto or open.)

For example, if $G$ is a closed subgroup of a topological group $H$, then the action of
$G$ on $H$ by multiplication on the left is an open action.
\end{defn}

The proofs of the following two lemmas are straightforward.
\begin{lem}\label{lem_g_action_1}
\begin{enroman}
	\item	Let $p\colon E\to B$ be a morphism of topological groups and let $G:=\ker p$.
		Let $G$ act on the left on $E$. If $p$ is a fibration then it is a principal
		open fibration.
	\item	Let $p\colon E\to B$ be a principal open fibration and let $f\colon X\to B$ be
		a map. Then the pullback fibration $f^*E\to X$ is a principal open fibration.\qed
\end{enroman}
\end{lem}

\begin{lem}\label{lem_g_action_2}
Let the groups $G$ and $G'$ act freely on, respectively, $E$ and $E'$. Let $p\colon E\to B$ and $p'\colon E'\to B'$
be fibrations coherent with the respective actions. Suppose given a commutative diagram
\begin{equation*}\begin{diagram}
\node{E'} \arrow{s,l}{p'} \arrow{e,t}{\eta} \node{E} \arrow{s,r}{p} \\
\node{B'} \arrow{e,t}{\beta} \node{B}
\end{diagram}\end{equation*}
Assume that $p$ is a principal open fibration, and that $\eta$ is an equivariant map in the sense that
$\eta(g\cdot x)=\omega(g)\cdot\eta(x)$ for some continuous morphism $\omega\colon G'\to G$.
If $\omega$ and $\beta$ are fibrations, so is $\eta$. \qed
\end{lem}

The following proposition is due to J.~Cohen; see \cite{joel_cohen_2}, Theorem 1.1.
\begin{prop}[J.~Cohen]\label{prop_j_cohen}
Let $\setof{p_i\colon E_i\to B_i}$ be a level preserving morphism of inverse sequences
where for each $i$, the space $E_i$ is a free left $G_i$-space and $p_i\colon E_i\to B_i$ is a principal
open fibration. Assume that the maps $E_i\to E_{i-1}$ are equivariant given by homomorphisms
$\omega_i\colon G_i\to G_{i-1}$.

Let $p_\infty\colon E_\infty\to B_\infty$ denote the induced inverse limit map, and set
$G_\infty=\lim G_i$.

If the $\omega_i$ are fibrations then $p_\infty$ is a principal open fibration. \qed
\end{prop}

\begin{prop}\label{whitehead_tower_sequence}
Let $\dots\to Z_i\xrightarrow{p_i}Z_{i-1}\to\dots\to Z_1$ be an inverse sequence of fibrations with a coherent set
of base points $\setof{\zeta_i}$. Assume that the $Z_i$ are $(k-1)$-connected spaces of CW type where $k\gqs 1$ and denote
$G_i=\pi_k(Z_i,\zeta_i)$. Further let $G_\infty$ denote the limit group of the induced sequence
$\dots\to G_i\to G_{i-1}\to\dots\to G_1$.

If $k=1$, assume also that $\zeta_i$ is a nondegenerate base point of $Z_i$ for each $i$.

There exists a morphism of towers
\begin{equation*}\begin{diagram}
\node{\dots} \arrow{e} \node{\Bar Z_3} \arrow{e} \arrow{s} \node{\Bar Z_2} \arrow{e} \arrow{s} \node{\Bar Z_1} \arrow{s} \\
\node{\dots} \arrow{e} \node{Z_3} \arrow{e} \node{Z_2} \arrow{e} \node{Z_1} \\
\end{diagram}\end{equation*}
where all arrows are fibrations and the arrows $\Bar Z_i\to Z_i$ are $k$-connected covers. The limit map
$\Bar Z_\infty\to Z_\infty$ fits into a pullback square
\begin{equation*}\begin{diagram}
\node{\Bar Z_\infty} \arrow{s} \arrow{e} \node{P_\infty} \arrow{s} \\
\node{Z_\infty} \arrow{e} \node{K_\infty}
\end{diagram}\end{equation*}
where $P_\infty$ is contractible and the vertical arrows are always fibrations for the class of compactly generated spaces.

If the sequence $\setof{G_i}$ satisfies the Mittag-Leffler condition and the canonical projections
$G_\infty\to G_i$ are injective, then the space $K_\infty$  has CW homotopy type. Hence $\kkkk(Z_\infty)$
has CW homotopy type if and only if $\kkkk(\Bar Z_\infty)$ has.

If $k=1$ or the groups $G_i$ are countable, then the arrows $\Bar Z_\infty\to Z_\infty$ and
$P_\infty\to K_\infty$ are fibrations and the space $Z_\infty$ has CW homotopy type if and only if $\Bar Z_\infty$ has.
\end{prop}

\begin{proof}
Assume that $k\gqs 2$. First inductively construct an inverse sequence
\[ \dots\to K(G_i,k)\xrightarrow{\alpha_i} K(G_{i-1},k)\to\dots\to K(G_1,k) \]
of fibration homomorphisms between topological abelian groups which on the nontrivial
homotopy group agree with the induced maps $\pi_k(Z_i,\zeta_i)\to\pi_k(Z_{i-1},\zeta_{i-1})$.
This can be done as follows. Using infinite symmetric products (see Dold and Thom \cite{dold_thom}) or geometric realization of
simplicial groups (see Milnor \cite{milnor3}) it is possible to construct an inverse sequence $\dots\to A_3\to A_2\to A_1$ with
$A_i\to A_{i-1}$ a morphism of topological abelian groups for all $i$. Since the space of free
paths on a topological group is canonically equipped with a compatible topological group structure,
the factorization of $A_2\to A_1$ as a composite of homotopy equivalence $A_2\to A_2'$ followed by fibration
$A_2'\to A_1$ is one of morphisms of topological abelian groups. Hence $A_3\to(A_2\to A_2')$ is a morphism of
topological groups and we may proceed inductively.

In case the groups $G_i$ are not countable, the above construction yields an inverse sequence of
morphism-fibrations of quasitopological groups (see Appendix \ref{quasi_groups}).

By applying the functorial contractible path-space `covering' construct a commutative ladder
\begin{equation*}\begin{diagram}
\node{\dots} \arrow{e} \node{P_3} \arrow{e} \arrow{s} \node{P_2} \arrow{e} \arrow{s} \node{P_1} \arrow{s} \\
\node{\dots} \arrow{e} \node{K(G_3,k)} \arrow{e} \node{K(G_2,k)} \arrow{e} \node{K(G_1,k)}
\end{diagram}\end{equation*}
where all arrows are both fibrations and morphisms of topological abelian groups.

Pick maps $Z_i\to K(G_i,k)$ corresponding to an isomorphism on the homotopy group $\pi_k$.
Using homotopy lifting property to change $Z_i\to K(G_i,k)$ for homotopic maps inductively construct a commutative ladder
\begin{equation*}\begin{diagram}
\node{\dots} \arrow{e} \node{Z_3} \arrow{e} \arrow{s} \node{Z_2} \arrow{e} \arrow{s} \node{Z_1} \arrow{s} \\
\node{\dots} \arrow{e} \node{K(G_3,k)} \arrow{e} \node{K(G_2,k)} \arrow{e} \node{K(G_1,k)}
\end{diagram}\end{equation*}

%%  \begin{equation*}\begin{diagram}
%%  \divide\dgARROWLENGTH by2
%%  \node[3]{P_i} \arrow[4]{e} \arrow{s,-} \node[4]{P_{i-1}} \arrow[3]{s} \\
%%  \node{\Bar Z_i} \arrow{nee} \arrow[4]{e} \arrow[3]{s} \node[2]{} \arrow[2]{s} \node[2]{\Bar Z_{i-1}} \arrow{nee} \arrow[3]{s} \\ \\
%%  \node[3]{K(\pi_i,k)} \arrow[2]{e,-} \node[2]{} \arrow[2]{e} \node[2]{K(\pi_{i-1},k)} \\
%%  \node{Z_i} \arrow{nee} \arrow[4]{e} \node[4]{Z_{i-1}} \arrow{nee}
%%  \end{diagram}\end{equation*}

Finally form an inverse sequence of pullback squares
\begin{equation*}\begin{diagram}
\divide\dgARROWLENGTH by4
\node[3]{P_\infty} \arrow[4]{e} \arrow{s,-} \node[4]{\dots} \arrow[4]{e} \arrow{s,-} \node[4]{P_i} \arrow[4]{e} \arrow{s,-} \node[4]{P_{i-1}} \arrow[3]{s} \\
\node{\Bar Z_\infty} \arrow{nee} \arrow[4]{e} \arrow[3]{s} \node[2]{}\arrow[2]{s} \node[2]{\dots} \arrow{nee} \arrow[4]{e} \arrow[3]{s} \node[2]{}\arrow[2]{s} \node[2]{\Bar Z_i} \arrow{nee} \arrow[4]{e} \arrow[3]{s} \node[2]{}\arrow[2]{s} \node[2]{\Bar Z_{i-1}} \arrow{nee} \arrow[3]{s} \\ \\
\node[3]{K_\infty} \arrow[2]{e,-} \node[2]{} \arrow[2]{e} \node[2]{\dots} \arrow[2]{e,-} \node[2]{} \arrow[2]{e} \node[2]{K(G_i,k)} \arrow[2]{e,-} \node[2]{} \arrow[2]{e} \node[2]{K(G_{i-1},k)} \\
\node{Z_\infty} \arrow{nee} \arrow[4]{e} \node[4]{\dots} \arrow{nee} \arrow[4]{e} \node[4]{Z_i} \arrow{nee} \arrow[4]{e} \node[4]{Z_{i-1}} \arrow{nee}
\end{diagram}\end{equation*}

Since $P_i\to K(G_i,k)$ is a fibration morphism, it is a principal open fibration. Hence so is the pullback map
$\Bar Z_i\to Z_i$ by \lemref{lem_g_action_1}. Further note that the restriction of $P_i\to P_{i-1}$ to the fibre
is trivially a fibration morphism. Hence the induced equivariant map $\Bar Z_i\to\Bar Z_{i-1}$ is also a fibration
by \lemref{lem_g_action_2}. The morphism $P_\infty\to K_\infty$ is the limit of principal open fibrations, and
is itself such by \propref{prop_j_cohen}. Moreover, the square
\begin{equation*}\begin{diagram}
\node{\Bar Z_\infty} \arrow{s} \arrow{e} \node{P_\infty} \arrow{s} \\
\node{Z_\infty} \arrow{e} \node{K_\infty}
\end{diagram}\end{equation*}
is a pullback by construction. Hence $\Bar Z_\infty\to Z_\infty$ is a principal open fibration by \lemref{lem_g_action_1}.

In case the groups $G_i$ are not countable, the actions are only compactly open (see \propref{compactly_open}), and
the limit maps under consideration have the homotopy lifting property with respect to compactly generated spaces
(see \corref{quasi_j_cohen}).

If the inverse sequence of groups $\setof{G_i}$ satisfies the hypotheses of the proposition,
$K_\infty$ has the homotopy type of a $K(G_\infty,k)$ by \propref{E-M_sequence}. In particular,
it has CW homotopy type.

Hence $Z_\infty$ has CW homotopy type if and only if $\Bar Z_\infty$ has. 

If the groups $G_i$ are not countable then after $\kkkk$-ification, the square remains a pullback,
and the vertical arrows are fibrations in the category of compactly generated Hausdorff spaces.
Stasheff's theorem remains valid in that category (see for example P.~Kahn \cite{kahn}), hence
in this case $\kkkk(Z_\infty)$ has CW type if and only if $\kkkk(\Bar Z_\infty)$ has.

In case $k=1$, first inductively construct a based inverse sequence of fibrations
\[ \dots\to(K_i,\kappa_i)\xrightarrow{\alpha_i}(K_{i-1},\kappa_{i-1})\to\dots\to(K_1,\kappa_1) \]
where $\kappa_i$ is nondegenerate in $K_i$, the pair $(K_i,\kappa_i)$ has the type of a $(K(G_i,1),*)$,
and $\alpha_i{}_\#\colon\pi_1(K_i,\kappa_i)\to\pi_1(K_{i-1},\kappa_{i-1})$ realizes the morphism $G_i\to G_{i-1}$.

For each $i$, let $q_i\colon P_i\to K_i$ denote the fibration with discrete fibre, obtained by pulling back
a universal covering $E_i\to K(G_i,1)$ over the homotopy equivalence $K_i\to K(G_i,1)$. Since the $P_i$ are
contractible, any map $P_i\to P_{i-1}$ is a lifting of $K_i\to K_{i-1}$ up to homotopy, and since $P_i\to K_i$
is a fibration, a strict lifting may be obtained. The composite $P_i\to P_{i-1}\to K_{i-1}$ is a fibration
since it equals $P_i\to K_i\to K_{i-1}$ and since $P_{i-1}\to K_{i-1}$ has discrete fibres, also $P_i\to P_{i-1}$
is a fibration.

We use well-pointedness of the $Z_i$ to inductively construct a level-preserving map of inverse sequences
$\setof{Z_i}$ to $\setof{K_i}$. Form $\Bar Z_i$ as the pullback of $Z_i\to K_i\leftarrow P_i$ as before.
Since $P_i\to K_i$ and $\Bar Z_i\to Z_i$ are fibrations with discrete fibres, the limit maps $P_\infty\to K_\infty$
and $\Bar Z_\infty\to Z_\infty$ are always fibrations.
\end{proof}

%%  \begin{cor}\label{postnikov_sequence}
%%  Let $\dots\to Z_i\xrightarrow{p_i}Z_{i-1}\to\dots\to Z_1$ be an inverse sequence of fibrations with a coherent set
%%  of nondegenerate base points $\setof{\zeta_i}$. Assume that all spaces $Z_i$ are path-connected spaces of CW homotopy
%%  type and that there exists a number $N$ so that $\pi_k(Z_i,\zeta_i)=0$ for $k\gqs N+1$ and all $i$.
%%  
%%  Assume that for each $k\gqs 1$ the sequence $\setof{\pi_k(Z_i,\zeta_i)\,\vert\,i}$ is injectively Mittag-Leffler.
%%  Then the limit space $Z_\infty$ is a path-connected space such that $\kkkk(Z_\infty)$ has the homotopy type of a CW complex.
%%  \qed
%%  \end{cor}

\begin{proof}[Proof of \thmref{postnikov_sequence}] 
Statement \navedi{1} is clear by \thmref{u_M-L_general}.

As for \navedi{2}, let $C_\infty$ be the limit space of $\setof{C_i}$. Since $\setof{\pi_1(C_i,\eta_i)=\pi_1(Y_i,\eta_i)}$
satisfies the Mittag-Leffler property, $C_\infty$ is path-connected by \propref{homotopy_groups_of_inverse_limits}.
Evidently $C\subset C_\infty$ and thus $C=C_\infty$ since $C$ is a path component. By inductive application of
\propref{whitehead_tower_sequence} it follows that $\kkkk(C)$ has CW type if and only if $\kkkk(Z)$ has where
$Z$ is the limit space of $N$-connected covers of $C_i$. By \propref{contractible_limit} the space $Z$ and consequently
$\kkkk(Z)$ are contractible. Hence $\kkkk(C)$ has CW type.
\end{proof}

%%%%%%%%%%%%%%%%%%%%%%%%%%%%%%%%%%%%%%%%%%%%%%%%%%%%%%%%%%%%%%%%%%%%%%%%%%%%%%%%%%%%%%%%%%%%%%%%%%%
%%%%%%%%%%%%%%%%%%%%%%%%%%%%%%%%%%%%%%%%%%%%%%%%%%%%%%%%%%%%%%%%%%%%%%%%%%%%%%%%%%%%%%%%%%%%%%%%%%%
%%%												%%%
%%%				CW homotopy type of function spaces				%%%
%%%												%%%
%%%%%%%%%%%%%%%%%%%%%%%%%%%%%%%%%%%%%%%%%%%%%%%%%%%%%%%%%%%%%%%%%%%%%%%%%%%%%%%%%%%%%%%%%%%%%%%%%%%
%%%%%%%%%%%%%%%%%%%%%%%%%%%%%%%%%%%%%%%%%%%%%%%%%%%%%%%%%%%%%%%%%%%%%%%%%%%%%%%%%%%%%%%%%%%%%%%%%%%

\chapter{CW homotopy type of function spaces} \label{function_spaces}
\setcounter{thm}{0}

In this chapter we discuss various properties of function spaces of CW homotopy type.
We give some necessary conditions, and some sufficient conditions.

First we show that certain spaces of continuous functions between CW complexes fit into the framework of Chapter
\ref{slc_cw_type}. In particular, $n$-ad function spaces $(Y;Y_1,Y_2,\dots,Y_n)^{(X;A_1,A_2,\dots,A_n)}$ are suitable
for study.

Then we give our main `topological' characterizations of CW type function spaces.
For CW complexes $X$ and $Y$, the function space $Y^X$ has CW type if and only if $Y^X$ admits a numerable covering
of open sets contractible within $Y^X$. In particular, if $X$ is countable then $Y^X$ has CW type if and only if it
is semilocally contractible. Hence function spaces between CW complexes carry so much structure that local contractibility
properties are sufficient.

Next, $Y^X$ has CW type if and only if $X$ is essentially countable with respect to $Y$. More precisely,
if $Y^X$ has CW type then for every $g\colon X\to Y$ and every countable subcomplex $L$ of $Y$ there exists
a bigger countable subcomplex $L'$ such that the fibre over $g\vert_{L'}$ of the fibration $Y^X\to Y^{L'}$ is
contractible and consequently $\Omega(Y^X,g)\to\Omega(Y^{L'},g\vert_{L'})$ is a homotopy equivalence.

Conversely, if for each countable subcomplex $L$ there exists a bigger countable subcomplex $L'$ such that
$Y^X\to Y^{L'}$ is a weak equivalence onto image of CW type, then $Y^X$ has CW type.

The results of Section \ref{inverse_systems} (see also Section \ref{phantom_components}) impose severe
restrictions on homotopy groups $\pi_k(Y^X,g)$, where $k\gqs 1$, for $Y^X$ of CW type. Below we show that
there are also restrictions on the set of path components. In particular, path components must be small
in number, and there can be no phantom path components.

For $Y$ a CW complex with finitely many nontrivial homotopy groups, the question of CW type of $Y^X$
depends only on the induced morphisms $\pi_k(Y^M)\to\pi_k(Y^L)$ for $L\lqs M\lqs X$, and $k\gqs 0$,
and we give necessary and sufficient conditions. 

At the end of this chapter we present three types of `constructions' of CW type function spaces,
which we apply in later chapters.

%%%%%%%%%%%%%%%%%%%%%%%%%%%%%%%%%%%%%%%%%%%%%%%%%%%%%%%%%%%%%%%%%%%%%%%%%%%%%%%%%
%										%
%		Function spaces as inverse limits				%
%										%
%%%%%%%%%%%%%%%%%%%%%%%%%%%%%%%%%%%%%%%%%%%%%%%%%%%%%%%%%%%%%%%%%%%%%%%%%%%%%%%%%

\section{Function spaces as inverse limits}\label{restriction}

Let $(X;A_1,\dots,A_n)$ be a CW $n$-ad and let $(Y;B_1,\dots,B_n)$ have the homotopy type of a CW $n$-ad.
Our general aim is to study the homotopy type of $n$-ad function spaces $(Y;B_1,\dots,B_n)^{(X;A_1,\dots,A_n)}$.

Let $\SSS_X$ be the set of all subcomplexes of $X$, and let $\AAA_X$ be the direct sum $\oplus_{i=1}^\infty\SSS_X$
whose elements are sequences $\AA=(A_1,A_2,A_3,\dots)$ of subcomplexes of $X$ with $A_j=\emptyset$ for all
but finitely many $j$. Let $\BBB_Y$ denote the subset of $\oplus_{i=1}^\infty 2^Y$ consisting of those
sequences $\BB=(B_1,B_2,\dots)$ of subsets of $Y$ where $(Y;B_1,\dots,B_j)$ has the homotopy type of a CW $j$-ad
for each $j$.

Let $K$ be a compact CW complex and let $\SSS_K$ be the set of subcomplexes of $K$. Then $X\times K$ is a CW complex and 
the cartesian product defines a function $\SSS_X\times\SSS_K\to\SSS_{X\times K}$.

We define obvious operations $\AAA_X\times\SSS_X\xrightarrow{\cap}\AAA_X$,
$\AAA_X\times\SSS_K\xrightarrow{\times}\AAA_{X\times K}$, and
$\AAA_X\times\AAA_X\xrightarrow{\cup}\AAA_X$.

Let $T$ be a subcomplex of $X$ and let $\gamma\colon T\to Y$ be a map. 
For $\AA\in\AAA_X$ and $\BB\in\BBB_Y$ let $(Y;\BB)^{(X;\AA)}$ denote the `$n$-ad' function space,
and $(Y;\BB)^{(X;\AA)}[T,\gamma]$ its subspace
\[ \setof{f\colon X\to Y\,\vert\,f(A_i)\subset B_i\,\,\forall i,\,\,\,f\vert_T=\gamma}. \]

\begin{lem}\label{restriction_map}
Let $L$ be a subcomplex of $X$. The restriction map
\[ (Y;\BB)^{(X;\AA)}[T,\gamma]\to(Y;\BB)^{(L;\AA\cap L)}[T\cap L,\gamma\vert_{T\cap L}] \]
is a fibration.
\end{lem}

\begin{proof}
We may assume $T\lqs L$, since
$(Y;\BB)^{(L;\AA\cap L)}[T\cap L,\gamma\vert_{T\cap L}]$ is homeomorphic to
$(Y;\BB)^{(L';\AA\cap L')}[T,\gamma]$ where $L'=L\cup T$. Then $(Y;\BB)^{(X;\AA)}[T,\gamma]$
is the preimage of $(Y;\BB)^{(L;\AA\cap L)}[T,\gamma]$
under the restriction map $(Y;\BB)^{(X;\AA)}\to(Y;\BB)^{(L;\AA\cap L)}$.
Thus it suffices to show that the latter is a fibration.

For a compactly generated space $M$ and a compactum $C$ the product $M\times C$ is compactly generated
and $(Y^M)^C$ is homeomorphic to $Y^{M\times C}$. In particular, this holds for a CW complex $M$. Thus
$[(Y;\BB)^{(X;\AA)})]^I$ is homeomorphic to $(Y;\BB)^{(X\times I,\AA\times I)}$ and
the pullback of $(Y;\BB)^{(X;\AA)}\to(Y;\BB)^{(L;\AA\cap L)}\leftarrow[(Y;\BB)^{(L;\AA\cap L)}]^I$
is homeomorphic to $(Y;\BB)^{(X\times 0\cup L\times I;\AA\times 0\cup(\AA\cap L)\times I)}$.

Let $\rho\colon X\times I\to X\times 0\cup L\times I$ be a retraction such that
$\rho(A_i\times I)\subset A_i\times 0\cup (A_i\cap L)\times I$ for all $i$.
Composition with $\rho$ defines a lifting function
\[ \lambda\colon (Y;\BB)^{(X\times 0\cup L\times I;\AA\times 0\cup(\AA\cap L)\times I)}\to(Y;\BB)^{(X\times I,\AA\times I)} \]
for $(Y;\BB)^{(X;\AA)}\to(Y;\BB)^{(L;\AA\cap L)}$.
\end{proof}

\begin{rem_n}
Note that \lemref{restriction_map} remains valid if we assume that the entries $A_i$ of $\AA$ are closed subsets
of a space $X$ such that all possible intersections $\cap_i A_i$ are cofibered subsets of $X$.
\end{rem_n}

If $(Y;\BB)^{(X;\AA)}[T,\gamma]$ is non-empty then it contains a map $\Gamma\colon X\to Y$.
Therefore it makes sense to define
\[ \FFF(X,Y)=\setof{(Y;\BB)^{(X;\AA)}[T,\Gamma\vert_T]\,\vert\,\AA\in\AAA_X,\,\,\BB\in\BBB_Y,\,\,T\in\SSS_X,\,\,
	\Gamma\in(Y;\BB)^{(X;\AA)}}. \]
For $Z=(Y;\BB)^{(X;\AA)}[T,\Gamma\vert_T]\in\FFF(X,Y)$ define
\[ Z(L)=(Y;\BB)^{(L;\AA\cap L)}[T\cap L,\Gamma\vert_{T\cap L}]. \]

\begin{lem}\label{second_lemma}
Let $Z\in\FFF(X,Y)$.
\begin{enroman}
\item	For arbitrary subcomplexes $L\lqs L'\lqs X$, the restriction map $Z(L')\to Z(L)$ is a fibration.
	The fibre of $Z\to Z(L)$ is an element of $\FFF(X,Y)$.
\item	Let $M\lqs X$ and let $\LLL$ be an exhaustive directed system of proper subcomplexes of $M$.
	Then $Z(M)$, together with restriction maps $Z(M)\to Z(L)$, is the limit of the inverse system
	$\setof{Z(L)\,\vert\,L\in\LLL}$.
\item	For $K$ a finite subcomplex of $X$, the space $Z(K)$ has CW homotopy type.
\end{enroman}
\end{lem}
\begin{proof}
Statement \navedi{1} is contained in \lemref{restriction_map}.

To prove \navedi{2}, note that directedness of $\LLL$ guarantees that every compact subset of $M$ is already contained
in a member of $\LLL$. Thus \navedi{2} is an easy consequence of the fact that CW complexes have the weak topology with
respect to the family of closed cells.

Let $Z=(Y,\BB)^{(X,\AA)}[T,\Gamma\vert_T]$. Then $Z(K)=(Y,\BB)^{(K,\AA\cap K)}[T\cap K,\Gamma\vert_{T\cap K}]$
is the fibre of $(Y,\BB)^{(K,\AA\cap K)}\to (Y,\BB)^{(T\cap K,\AA\cap T\cap K)}$ over $\Gamma\vert_{T\cap K}$.
The total space and the base space have CW homotopy types by Theorem 3 of \cite{milnor} hence so has the fibre
$Z(K)$ by Stasheff's theorem.
\end{proof}

We have established

\begin{prop}\label{function_space}
Let $Z\in\FFF(X,Y)$, and let $\XXX$ denote the set of all subcomplexes of $X$. Then the functor $\XXX\to\TTT op$,
\[	L\mapsto Z(L),\,\,\,(L\lqs M)\mapsto(Z(M)\xrightarrow{\text{restriction}}Z(L)) \]
is a restricted inverse system of fibrations indexed by the regular lattice $\XXX$. The sublattice $\KKK$ of $\XXX$
consisting of finite subcomplexes is the set of compact elements and $Z(K)$ has CW homotopy type for each $K\in\KKK$.\qed
\end{prop}

One of the nice features of $\FFF(X,Y)$ is that it is `closed under forming loop spaces.' More precisely,
let $g\in Z=(Y;\BBB)^{(X;\AAA)}[T,\Gamma\vert_T]$. Then $\Omega(Z,g)$ is homeomorphic to
\[ (Y;\BBB)^{(X\times I;\AAA\times I)}[T\times I\cup X\times\partial I,\,\Gamma\vert_T\circ\pr_T\sqcup g\circ\pr_X]. \]
In this way we may view $\Omega(Z,g)\in\FFF(X\times I,Y)$, and the restriction $\Omega(Z,g)\to\Omega(Z(L),g\vert_L)$
can be identified with $\Omega(Z,g)\to\Omega(Z,g)(L\times I)$.

Also note that in this way $\Omega(Z,g)$ may naturally be viewed as the limit of a restricted inverse system of
fibrations indexed by $\XXX$.

\begin{conv_n}
If $X$ is a CW complex then under cardinality of $X$ we understand the cardinality of the set of cells $X$.
\end{conv_n}

\lemref{good_filtration_general} implies

\begin{lem}\label{good_filtration}
Let $X$ be an infinite CW complex, and let $\XXX$ be the set of all subcomplexes of $X$. For an initial ordinal $\alpha$
of cardinality $\card{X}$ there exists an order-preserving injection $W(\alpha)\to\XXX$, $\lambda\mapsto X_\lambda$,
whose image is well-ordered and exhaustive in the sense that $\cup_{\lambda<\alpha}X_\lambda=X$. For each
$\mu\in W(\alpha)$ the quotient complex $X_{\mu+1}/X_\mu$ is finite and, if $\mu$ is limiting,
$X_\mu=\cup_{\lambda<\mu}X_\lambda$. Moreover, for each $\mu<\alpha+1$, the subcomplex $X_\mu$ has $\card W(\mu)$ cells.\qed
\end{lem}

\begin{rem_n}
If necessary, we may replace each `successive pair' $(X_{\mu+1},X_\mu)$ by a finite refinement
$X_\mu=X_0<X_1<\dots<X_r=X_{\mu+1}$ where for each $i$, the pair $(X_i,X_{i-1})$ is the adjunction of a single
cell. \qed
\end{rem_n}

\begin{defn}
Let $X$ be a CW complex and let $X_\lambda$, $\lambda\in W(\alpha)$, be a filtration of subcomplexes of $X$.
We say that $\setof{X_\lambda\,\vert\,\lambda}$ is a good filtration if it is exhaustive in the sense that
$\cup_\lambda X_\lambda=X$, the assignment $\lambda\to X_\lambda$ is order-preserving, and for each limiting
$\mu<\alpha$ the subcomplex $X_\mu$ is given by $X_\mu=\cup_{\lambda<\mu}X_\lambda$.	\qed
\end{defn}

\lemref{good_filtration} asserts that every complex $X$ admits a good filtration indexed by
$W(\alpha)$ where $\alpha$ is a limit ordinal such that consecutive steps are adjunctions of finitely
many cells (or a single one, if desired) and for each $\mu<\alpha+1$, the subcomplex $X_\mu$ has exactly
$\card W(\mu)$ cells.

The `goodness' of a good filtration is apparent from

\begin{lem}\label{good_yields_restricted}
Let $\alpha$ be an ordinal number and let $\setof{X_\lambda\,\vert\,\lambda<\alpha}$ be a good filtration for $X$.
Let $Y$ be any CW complex. We may form an inverse system indexed by $W(\alpha)$ consisting
of spaces $Y^{X_\lambda}$ where for each $\lambda<\mu$ the bonding map $Y^{X_\mu}\to Y^{X_\lambda}$ is the restriction
fibration. Since $\cup_\lambda X_\lambda=X$, the limit of this system is $Y^X$, together with restriction projections.

Moreover, if $\mu$ is a limit ordinal, $\mu<\alpha$, then \[ Y^{X_\mu}=\lim_{\lambda<\mu}Y^{X_\lambda}, \]
which is to say that $\setof{Y^{X_\lambda}\,\vert\,\lambda\in W(\alpha)}$ is a restricted inverse system.\qed
\end{lem}

We close this section with two simple-minded examples of function spaces that are not of CW type.

\begin{lem}[Splittings of domain]\label{splittings}
Assume $X=\vee_{\lambda\in\Lambda}X_\lambda$. Let $C$ (respectively $C_\lambda$) denote
the path component of the constant map in $(Y,*)^{(X,*)}$ (respectively $(Y,*)^{(X_\lambda,*)}$).
If $C$ has CW type, then $C_\lambda$ is contractible for all but finitely many $\lambda$.
\end{lem}

\begin{proof}
Since $(Y,*)^{(X,*)}$ (respectively $C$) is homeomorphic to $\prod_\lambda(Y,*)^{(X_\lambda,*)}$
(respectively $\prod_\lambda C_\lambda$), this follows immediately from \exref{product}.
\end{proof}

%************************************************************************************************************

\begin{example}[Infinitely many ``free'' cells]
Assume that $Y^T$ has CW type and let $X$ be obtained from $T$ by attaching an infinite family of cells
$\setof{e_\lambda^{m_\lambda}\,\vert\,\lambda}$ along an attaching map $\sqcup_\lambda S^{m_\lambda-1}\to T$.
This is to say that we attach infinitely many free cells. If for each $\lambda$ there exists a number
$n_\lambda\gqs m_\lambda$ with $\pi_{n_\lambda}(Y)$ nontrivial, then $Y^X$ cannot have CW type.

If $Y^X$ has CW type then so has the space $(Y,*)^{(X/T,\setof{T})}$, by Stasheff's theorem.
But $X/T=\bigvee_{\lambda}S^{m_\lambda}$, and $\pi_k\big((Y,*)^{(S^{m_\lambda},*)},*\big)\cong\pi_{k+m_\lambda}(Y,*)$,
by adjunction. The assumptions on $Y$ contradict \lemref{splittings}.

For an explicit example we may take $X=M(\FF,m)$ and $Y=K(\ZZ_2,n)$ where $m\lqs n$. Here
$M(\FF,n)$ denotes a Moore space of type $(\FF,n)$, and $\FF$ is the free abelian group on
countably many generators.

The case $m=n$ generalizes Milnor's example $X=\NN$, $Y=\setof{0,1}$
(see \cite{milnor}) if we understand $\NN$ as a $M(\FF,0)$ and $\setof{0,1}$ as a $K(\ZZ_2,0)$.
It was pointed out by Milnor that $(\setof{0,1},*)^{(\NN,*)}$ is homeomorphic to the Cantor set.
We note that $(K(\ZZ_2,n),*)^{(M(\FF,n),*)}$ is homeomorphic to $(\Omega^nK(\ZZ_2,n),*)^{(\NN,*)}$
which in turn is homotopy equivalent to the Cantor set. \qed
\end{example}

%************************************************************************************************************

\begin{example}[A stable splitting]
Let $Y$ be a CW complex with infinitely many nontrivial homotopy groups. (By a theorem of Serre \cite{serre},
see also Neisendorfer \cite{neisendorfer}, $Y$ may be a finite CW complex which is simply connected but not
contractible.) Then the path component $C$ of the constant map in $(Y,*)^{(\Omega S^{p+1},*)}$ does not have CW type
for any $p\gqs 1$.

Since $\Omega S^{p+1}$ is homotopy equivalent to the James construction $J(S^p)$ (see James \cite{james}, and also Puppe
\cite{puppe}), we may take $X=J(S^p)$. If $C$ has CW type, so has the loop space $\Omega(C,*)\approx(Y,*)^{(SX,*)}$. It is
well known that $SX=SJ(S^p)$ splits as $SX\simeq\vee_{j=1}^\infty S^{pj+1}$. By adjunction the group
$\pi_k\big((Y,*)^{(S^{pj+1},*)},*\big)$ is isomorphic to $\pi_{k+pj+1}(Y,*)$. The latter is nontrivial for infinitely many
$k$, contradicting \lemref{splittings}. \qed
\end{example}

%%%%%%%%%%%%%%%%%%%%%%%%%%%%%%%%%%%%%%%%%%%%%%%%%%%%%%%%%%%%%%%%%%%%%%%%%%%%%%%%%
%										%
%		Properties of function spaces of CW type			%
%										%
%%%%%%%%%%%%%%%%%%%%%%%%%%%%%%%%%%%%%%%%%%%%%%%%%%%%%%%%%%%%%%%%%%%%%%%%%%%%%%%%%

\section{Properties of function spaces of CW type}\label{characterizations}

%************************************************************************************************************
%*
%*	Proposition \label{fibre_SLC}
%*
%************************************************************************************************************

\propref{fibre_SLC_general} says

\begin{thm}\label{fibre_SLC}\label{CW_SLC}
Let $X$ and $Y$ be CW complexes and $Z\in\FFF(X,Y)$. Let $C$ be the union of some path components of $Z$,
and for each subcomplex $L$ of $X$ let $C_L$ denote the image of $C$ under $Z\to Z(L)$. 
\begin{enroman}
\item	$C$ is semilocally contractible if and only if for each map $g\in C$ there exists a finite subcomplex $K$ of $X$
	such that the fibre $F_K$ over $g\vert_K$ of $C\to C_K$ contracts in the total space.
\item	$C$ is semilocally contractible and open if and only if for each map $g\in C$ there exists a finite subcomplex
	$K$ such that the fibre of $Z\to Z(K)$ over $g\vert_K$ contracts in the total space.
\item	If $C$ is semilocally contractible then for each map $g\in C$ the loop space $\Omega(Z;g)$ has CW homotopy type.
\item	$Z$ has CW homotopy type if and only if it is a Dold space. \qed
\end{enroman}
\end{thm}

\begin{lem}
If $Z\in\FFF(X,Y)$ where $X$ is a countable complex, then $Z$ is hereditarily paracompact (in fact stratifiable).
\end{lem}
\begin{proof}
By Cauty \cite{cauty}, the function space $Y^X$ is stratifiable, and hence hereditarily paracompact.
\end{proof}

%************************************************************************************************************

\begin{cor} Let $Z\in\FFF(X,Y)$ where $X$ is countable, and let $C$ be a set of path components of $Z$.
Then $C$ has CW homotopy type if and only if it is semilocally contractible.

In particular, if $C$ is globally well-pointed, it has CW homotopy type.
Compare \prevthmref{countable_hammer}. \qed
\end{cor}

\begin{lem}\label{homology_P}
\begin{enroman}
\item
Let $X$ be a connected CW complex and $K$ a finite subcomplex of $X$. Then there exists a finite subcomplex $L$
of $X$ such that the inclusion induced morphism $\H_*(L)\to\H_*(X)$ is injective on the image of $\H_*(K)\to\H_*(L)$.

\item
Let $L_0\lqs L_1\lqs L_2\lqs\dots$ be an ascending sequence of finite subcomplexes of $X$ with union $L_\infty$.
If for each $i\gqs 1$ the morphism $\H_*(L_i)\to\H_*(X)$ is injective on the image of $\H_*(L_{i-1})\to\H_*(L_i)$
then the morphism \[ \H_*(L_\infty)\to\H_*(X) \] is injective.
\end{enroman}
\end{lem}

\begin{proof}
For \navedi{1} see the proof of Lemma 4.1.3 of Gray and McGibbon \cite{gray-mcgibbon}, and \navedi{2} is evident.
\end{proof}

%************************************************************************************************************
%*
%*
%*
%************************************************************************************************************

\thmref{double_wicked_general} says

\begin{thm}\label{double_wicked}
Let $Z\in\FFF(X,Y)$, let $C$ be a set of path components of $Z$, and let $g\in C$ be a map.
Let $C_L$ denote the image of $C$ under $Z\to Z(L)$.

\begin{enroman}
\item
If $C$ has CW homotopy type then for every countable complex $L$ of $X$
there exists a bigger countable complex $L'$ such that $\Omega(C_{L'},g\vert_{L'})$ has CW type and
\begin{equation*}\tag{$*$} \Omega(C;g)\to\Omega(C_{L'},g\vert_{L'}) \end{equation*}
is a homotopy equivalence. Consequently the fibre of $C\to C_{L'}$ over $g\vert_{L'}$ is contractible.

In addition, if $P$ is a directed order preserving property defined on pairs $(K,K')$ of finite subcomplexes
with $K\lqs K'$ then $L'$ may be assumed to be the union of a $P$-sequence.

In particular, we may assume $H_*(L';\ZZ)\to H_*(X;\ZZ)$ to be injective.

\item
Conversely, if for each countable complex $L$ there exists a bigger countable complex $L'$ such that
$\Omega(C_{L'},g\vert_{L'})$ has CW homotopy type and ($*$) is a weak homotopy equivalence then
$\Omega(C,g)$ has CW homotopy type.
\end{enroman}
\end{thm}
\begin{proof}
We need only discuss the statement about $H_*(L';\ZZ)\to H_*(X;\ZZ)$ of \navedi{1}.

Let $Q(K,K')$ if $\H_*(K')\to\H_*(X)$ is injective on the image of $\H_*(K)\to\H_*(K')$. Then $Q$
is a directed order preserving property by \navedi{1} of \lemref{homology_P}. Now apply \navedi{2}
of \lemref{homology_P} together with \thmref{double_wicked_general}.
\end{proof}

\thmref{wicked_general} translates into

\begin{thm}\label{wicked}
Let $X$ be a CW complex with uncountably many cells, and let $Y$ be an arbitrary CW complex.
Let $Z\in\FFF(X,Y)$ and let $C$ be a set of path components of $Z$. For each subcomplex $L$ of $X$
denote by $C_L$ the image of $C$ under $Z\to Z(L)$.

Let $\Gamma$ be a topological space and assume that there exists a map of $\Gamma$ into the inverse system
$\setof{C_L\,\vert\,L}$.

Let $\LLL$ be the set of those subcomplexes $L$ of $X$ for which $C_L$ has CW homotopy type and
$\Gamma\to C_L$ is a weak equivalence. 

Assume that for each countable subcomplex $L$ there exists a countable subcomplex $L'\in\LLL$ containing $L$.

Then for any ordinal $\eta$ and every subcomplex $L$ of $X$ with $\aleph_\eta$ cells there exists
a larger subcomplex $L'\in\LLL$ with $\aleph_\eta$ cells.

In particular $C$ has CW homotopy type, and $\Gamma\to C$ is a weak equivalence. \qed
\end{thm}

\begin{cor}\label{cor_wicked}
Let $X$ be a CW complex with uncountably many cells such that for each countable subcomplex $L$
there exists a countable subcomplex $L'$ containing $L$ so that $(Y,*)^{(L',*)}$ is contractible.
Then $(Y,*)^{(X,*)}$ is contractible.
\end{cor}
\begin{proof}
Note that $(Y,*)^{(L',*)}$ is contractible if and only if the section $Y\to Y^{L'}$ is a homotopy
equivalence. By \thmref{wicked} the space $Y^X$ has CW type and $Y\to Y^X$ is a (weak and hence genuine) homotopy
equivalence.
\end{proof}

\begin{lem}\label{obstruction}
Let $X$ and $Y$ be connected CW complexes. Assume that the path component $C$ of $g\in(Y,*)^{(X,*)}$ has
CW homotopy type. Then there exists a finite subcomplex $K$ of $X$ such that for any $L\gqs K$ and every $k\gqs 2$
the sequence of groups and homomorphisms 
\begin{equation*}
\tag{$*$} 0\to\pi_k((Y,*)^{(X,*)},g)\to\pi_k((Y,*)^{(L,*)},g\vert_L)\to\pi_{k-1}(F_L,g)\to 0
\end{equation*}
is exact. Here $F_L$ denotes the fibre of $(Y,*)^{(X,*)}\to(Y,*)^{(L,*)}$ over $g\vert_L$.
In addition, the morphism 
\begin{equation*}\tag{$**$}\pi_1((Y,*)^{(X,*)},g)\to\pi_1((Y,*)^{(L,*)},g\vert_L) \end{equation*}
is injective.
If $C$ is open in $(Y,*)^{(X,*)}$ then for $k=1$, ($*$) is an exact sequence of pointed sets,
and so is $*\to[X,Y]_*\to[L,Y]_*$ at the base point $g$.

If $g\equiv *$ then ($*$) transforms as $0\to[S^k\wedge X,Y]_*\to [S^k\wedge L,Y]_* \to [S^{k-1}\wedge(X/L),Y]_*\to 0$,
and similarly also ($**$).
\end{lem}

\begin{proof}
Follows from \navedi{1} and \navedi{2} of \propref{fibre_SLC_general} by applying the long homotopy exact
sequence of fibration together with the observation that $\pi_k(F_L,g)\cong\pi_k(F_L\cap C,g)$ for $k\gqs 1$.
For $g$ constant apply adjunction $\pi_k\big((Y,y_0)^{(L,x_0)},\const_{y_0}\big)\cong[S^k\wedge L,Y]_*$.
\end{proof}

%************************************************************************************************************

\begin{rem_n}
If the function $[X,Y]_*\to[L,Y]_*$ is naturally equivalent to a group morphism, the conclusion of
\lemref{obstruction} in case $C$ is the path component of the constant map (and is open) implies it is an injection.
\end{rem_n}

%************************************************************************************************************

\begin{cor}\label{cohomology_obstruction}
Let $Y$ be an Eilenberg-MacLane space $K(G,n)$, and let $X$ be an arbitrary CW complex.
If the path component $C$ of the constant map in $(Y,*)^{(X,*)}$ has CW type then there exists a finite subcomplex
$K$ of $X$ such that for any $L\gqs K$
and every $k$ with $1\lqs k\lqs n-1$ the inclusion induced morphism
\begin{equation*}	\tag{$*$}	\H^k(X;G)\to\H^k(L;G)	\end{equation*}
is injective. If, in addition, $C$ is open, then ($*$) is also injective for $k=n$. \qed
\end{cor}

%************************************************************************************************************

%%%%%%%%%%%%%%%%%%%%%%%%%%%%%%%%%%%%%%%%%%%%%%%%%%%%%%%%%%%%%%%%%%%%%%%%%%%%%%%%%
%										%
%				Counting path components			%
%										%
%%%%%%%%%%%%%%%%%%%%%%%%%%%%%%%%%%%%%%%%%%%%%%%%%%%%%%%%%%%%%%%%%%%%%%%%%%%%%%%%%

\section{Counting path components}\label{counting}

Here we focus on the restrictions that we may impose on path components of a function
space of CW type, an issue which is delicate since the set of path components in general
does not carry a group structure. We begin with the case when it does.

\begin{prop}\label{group_components}
Let $X$ and $Y$ be CW complexes such that the space $Y^X$ has CW type. If either $X$ is an $H$-cogroup or
$Y$ is an $H$-group, then for all large enough finite $K$, the functions
\[ \pi_0((Y,*)^{(X,*)})\to\pi_0((Y,*)^{(K,*)})\text{ and }\pi_0(Y^X)\to\pi_0(Y^K) \]
are injective.
\end{prop}

\begin{proof}
We prove the pointed version first. If $Y$ is an $H$-group then the contravariant functor $C(\_,Y)_*$,
$(L,*)\mapsto(Y,*)^{(L,*)}$ maps into the category of $H$-groups and $H$-morphisms, hence this case is
an immediate consequence of \thmref{u_M-L_general}.

Assume that $X$ is an $H$-cogroup and denote comultiplication $\nu\colon X\to X\vee X$ and coinverse
$\iota\colon X\to X$. The path component $C$ of the constant map in $(Y,*)^{(X,*)}$ is open and has CW
type hence by \thmref{u_M-L_general} there exists a finite subcomplex $K$ of $X$ such that the fibre of
\[ (Y,*)^{(X,*)}\to(Y,*)^{(L,*)} \] contracts in $C$ for every $L\gqs K$. In particular, if for some
$f\colon(X,*)\to(Y,*)$ and $L\gqs K$ the restriction $f\vert_L$ is nullhomotopic, then also $f$ is
nullhomotopic.

The image $\nu(K)$ is a compact subset of $X\vee X$, hence there exists a finite subcomplex $L$ of $X$
such that $L\gqs K$ and $\nu(K)\subset L\vee L$. 

Let $M$ be a finite subcomplex that contains $L$. Let $f,g\colon(X,*)\to(Y,*)$ be such that $f\vert_M$
and $g\vert_M$ are homotopic. Then so also are $f\vert_L,g\vert_L$. Pick a (pointed) homotopy $h\colon L\times I\to Y$,
and define $H\colon K\times I\to Y$ as follows
\begin{equation*}\tag{$*$}
K\times I\xrightarrow{\nu}(L\vee L)\times I=L\times I\sqcup_{*\times I}L\times I
	\xrightarrow{h\sqcup g\circ\iota\circ\pr_L}Y\vee Y\xrightarrow{\nabla}Y.
\end{equation*}
Then $H_0$ equals the composite $K\to X\xrightarrow{\nu}X\vee X\xrightarrow{f\vee g\circ\iota}Y\vee Y\xrightarrow{\nabla}Y$,
and $H_1$ equals the composite $K\to X\xrightarrow{\nu}X\vee X\xrightarrow{g\vee g\circ\iota}Y\vee Y\xrightarrow{\nabla}Y$.
This is to say that $H_0$ is the restriction of $f-g$ to $K$ and $H_1$ is the restriction of $g-g$ to $K$.
Since $g-g$ is nullhomotopic, so is $H_1$. Thus $(f-g)\vert_K$ is nullhomotopic. But then $f-g$ is also nullhomotopic
which is to say that $f$ and $g$ are homotopic, using the $H$-cogroup structure. Thus
$\pi_0((Y,*)^{(X,*)})\to\pi_0((Y,*)^{(M,*)})$ is injective for all $M\gqs L$.

The homotopy ladder with exact rows which is induced by the morphism of evaluation fibrations
$(Y^X\to Y)\to(Y^M\to Y)$ implies that $\pi_0(Y^X)\to\pi_0(Y^M)$ is injective. (See \navedi{1} of \lemref{easy}.)
\end{proof}

\begin{defn}
A space $Y$ has finite homotopy groups (respectively of cardinality at most $\aleph_\eta$) if it has finitely many
(respectively at most $\aleph_\eta$) path components and the homotopy groups of each path component are finite
(respectively at most $\aleph_\eta$).

An $n$-ad $(Y;\BB)$ has finite homotopy groups if $Y$ and all possible intersections $\cap_iB_i$ have finite
homotopy groups.

Similarly, an $n$-ad $(Y;\BB)$ has homotopy groups of cardinality at most $\aleph_\eta$ if $Y$
and all possible intersections $\cap_iB_i$ have homotopy groups of cardinality at most $\aleph_\eta$.

An $n$-ad $(Y;\BB)$ is said to be an $r$-Postnikov section if $Y$ and all possible intersections $\cap_iB_i$ are
$r$-Postnikov sections.
\end{defn}

\begin{prop}\label{homotopy_groups_of_function_space}
\begin{enroman}
\item	Let $(L;\AA)$ be an $n$-ad with $L$ a finite complex.
	Let $(Y;\BB)$ have finite homotopy groups (respectively of cardinality $\aleph_\eta$). Then
	$(Y;\BB)^{(L;\AA)}$ has finite homotopy groups (respectively of cardinality $\aleph_\eta$).
\item	Let $(Y;\BB)$ be an $r$-Postnikov $n$-ad. Let $(L';\AA)$ be an $n$-ad with $L'$ arbitrary.
	Suppose $(L',L)$ is an adjunction of one cell of dimension $d$. If $d\gqs r+1$ then the restriction fibration
	\[ (Y;\BB)^{(L';\AA)}\to(Y;\BB)^{(L;\AA\cap L)} \] is a homotopy equivalence onto its image, and if $d\gqs r+2$,
	then it is also surjective.
\end{enroman}
\end{prop}

\begin{lem}\label{easy}
\begin{enroman}
\item	Let $p\colon Y\to X$ be the principal fibration obtained by pulling back $f\colon X\to B$ over
	the evaluation $PB\to B$. Assume $b_0$ is the base point of $B$. Then there exists a left action of
	$\pi_1(B,b_0)$ on $\pi_0(Y)$ such that $p_\#\colon\pi_0(Y)\to\pi_0(X)$ collapses precisely the orbits
	of this action. Moreover for $g\in\pi_1(B,b_0)$ and $[y]\in\pi_0(Y)$ we have $g[y]=[y]$ if and only if
	there exists $\Tilde g\in\pi_1(X,p(y))$ such that $f_\#(\Tilde g)=g$.
\item	Let $E\to B$ be a fibration. If $E$ and $B$ have finite (resp. $\aleph_\eta$) homotopy groups, then so
	has each fibre. Conversely, if $B$ and all fibres of $E\to B$ have finite (resp. $\aleph_\eta$) homotopy
	groups, then so has $E$.\qed
\end{enroman}
\end{lem}

\begin{proof}[Proof of \propref{homotopy_groups_of_function_space}]
Denote for convenience $Y=B_0$, $L=A_0$.
Suppose $L=\setof{x_1,\dots,x_r}$, a finite discrete set. Let $K_j$ denote the set of indices $i$ for which
$x_j\in A_i$. Then $(Y;\BB)^{(L;\AA)}=\prod_{j=1}^r\big(\cap_{i\in K_j}B_i\big)$ which has finite ($\aleph_\eta$)
homotopy groups.

Thus we can make an induction on the number of cells of dimension $\gqs 1$. In other words, it suffices to prove
the following. If $(L';\AA)$ is an $n$-ad where $L'$ is obtained from a subcomplex $L$ by attaching a cell
$e$, and $(Y;\BB)^{(L;\AA\cap L)}$ has finite ($\aleph_\eta$) homotopy groups, then so has $(Y;\BB)^{(L';\AA)}$.

Let $x_0$ be an appropriate point (not necessarily a $0$-cell) in $L'$. Let $J$ denote the set of $j\gqs 0$ for
which $x_0\in A_j$.
We can consider the fibration (see the remark after \lemref{restriction_map})
\begin{equation*}\tag{\dag}
(Y;\BB)^{(L',\AA)}\xrightarrow{\eps}(Y;\BB)^{(\setof{x_0},\AA\cap\setof{x_0})}=\cap_{j\in J}B_j=B''.
\end{equation*}
By \lemref{easy} it suffices to show that for any $y_0\in B''$ the fibre of $\eps$ over $y_0$ has finite
($\aleph_\eta$) homotopy groups. The fibre of (\dag) over $y_0$ is the space $(Y,y_0\oplus\BB)^{(L',x_0\oplus\AA)}$.

Suppose $x_0\in L$.
By inductive assumption and \lemref{easy} also the space $(Y,y_0\oplus\BB)^{(L,x_0\oplus\AA\cap L)}$ has finite
($\aleph_\eta$) homotopy groups.

Let $K$ denote the set of indices $i\gqs 0$ for which $e\subset A_i$, and let $A'=\cap_{i\in K}A_i$.
Let $\varphi\colon S^{d-1}\to A'\cap L\lqs L$
be the attaching map, and $\Phi\colon B^d\to A'\lqs L'$ the characteristic map for $e$. Pick a base point
$\zeta_0\in S^{d-1}$ and set $x_0=\varphi(\zeta_0)$. Note that since $\Bar e\subset A'$, the set $K$ is
a subset of $J$. Thus $y_0\in\cap_{j\in J}B_j\subset\cap_{i\in K}B_i=B'$.

Consider the following diagram.
\begin{equation*}\tag{$\bullet$}\begin{diagram}
\node{(Y;y_0\oplus\BB)^{(L',x_0\oplus\AA)}} \arrow{e,t}{\text{restriction}} \arrow{s}
	\node{(B',y_0)^{(A',x_0)}}
	\arrow{e,t}{\Phi^\#} \arrow{s} \node{(B',y_0)^{(B^d,\zeta_0)}} \arrow{s} \\
\node{(Y;y_0\oplus\BB)^{(L,x_0\oplus\AA\cap L)}} \arrow{e,t}{\text{restriction}}
	\node{(B',y_0)^{(A'\cap L,x_0)}} \arrow{e,t}{\varphi^\#}
	\node{(B',y_0)^{(S^{d-1},\zeta_0)}}
\end{diagram}\end{equation*}
The left-hand square is a pullback square by construction while the right-hand square is a pullback induced
by the pushout diagram of the adjunction of $e$. By `horizontal composition', the large rectangle is also a pullback.
This exhibits
\begin{equation*}\tag{\ddag} (Y;y_0\oplus\BB)^{(L',x_0\oplus\AA)}\to(Y;y_0\oplus\BB)^{(L,x_0\oplus\AA\cap L)} \end{equation*}
as a principal fibration with fibre $(B',y_0)^{(S^{d},\zeta_0)}$. By inductive hypothesis and another application of
\lemref{easy} we conclude the proof of \navedi{1}.

Consider the following morphism of fibrations over $B''$.
\begin{equation*}\begin{diagram}
\node{(Y;\BB)^{(L';\AA)}} \arrow[2]{e} \arrow{se,b}{\eps_{x_0}} \node[2]{(Y;\BB)^{(L;\AA\cap L)}} \arrow{sw,b}{\eps_{x_0}} \\
\node[2]{B''}
\end{diagram}\end{equation*}
By Theorem 6.3 of Dold \cite{dold}, the map $(Y;\BB)^{(L';\AA)}\to(Y;\BB)^{(L;\AA\cap L)}$ is a homotopy equivalence onto
image if and only if for every $y_0\in B''$ the induced map of fibres over $y_0$ is a homotopy equivalence onto image.
The induced map of fibres is (\ddag).

Since $d\gqs r+1$ and $B'$ is an $r$-Postnikov section, the space $(B',y_0)^{(S^{d-1},\zeta_0)}$ is homotopy discrete.
In particular, the path component of the constant map is contractible, and the pullback rectangle of diagram ($\bullet$)
implies that (\ddag) is a homotopy equivalence onto image. If $d\gqs r+2$ then $(B',y_0)^{(S^{d-1},\zeta_0)}$ is contractible
hence (\ddag) is surjective. This concludes the proof of \navedi{2}.
\end{proof}

The following is a very rough estimate which we do by simply counting.

\begin{prop}\label{homotopy_sets_1}
Let $(Y;\BB)^{(X;\AA)}$ have CW homotopy type. Assume that the homotopy groups of $Y$ are of cardinality at most
$\aleph_\eta$, and that $X$ has $\aleph_\xi$ cells. Then $(Y;\BB)^{(X;\AA)}$ has at most
$\aleph_\xi\times\aleph_\eta=\aleph_{\max\setof{\xi,\eta}}$ path components.

In particular, if $Y$ is a countable complex, then $(Y;\BB)^{(X;\AA)}$ has at most $\aleph_\xi$ path components.
\end{prop}
\begin{proof}
Let $\KKK$ be the set of finite subcomplexes of $X$. Note that the cardinality of $\KKK$ is also $\aleph_\xi$.
By \thmref{u_M-L_general} there exists for each $[g]\in[(X;\AA),(Y;\BB)]$
an element $K=K(g)\in\KKK$ such that the preimage of $[g\vert_K]$ under $[(X;\AA),(Y;\BB)]\to[(K;\AA\cap K),(Y;\BB)]$
is exactly $[g]$. By making a choice over all path components we define a function $[(X;\AA),(Y;\BB)]\to\KKK$. Evidently
for each $K\in\KKK$ the function
\[ \Phi^{-1}(K)\to[(K;\AA\cap K),(Y;\BB)] \]
is an injection. By \navedi{1} of \propref{homotopy_groups_of_function_space} the cardinality of $\Phi^{-1}(K)$ is at most
$\aleph_\eta$. Since $[(X;\AA),(Y;\BB)]=\coprod_{K\in\KKK}\Phi^{-1}(K)$, the assertion follows.
\end{proof}

\begin{defn}
For a topological space $Z$ let $\Compon(Z)$ denote the quotient space of $Z$ obtained by collapsing path components.
\end{defn}

\begin{lem}\label{component_set}
Let $\dots\to Z_3\xrightarrow{p_3} Z_2\xrightarrow{p_2}Z_1$ be an inverse sequence of spaces of CW homotopy type,
and let $\zeta=\setof{\zeta_i}$ be an element of the limit space $Z_\infty$. If $\ll\pi_1(Z_i,\zeta_i)$ vanishes
then $\Compon Z_\infty$ is homeomorphic to $\lim\Compon Z_i$.
\end{lem}
\begin{proof}
For each $j$ let $\Bar p_j\colon\Compon(Z_j)\to\Compon(Z_{j-1})$ denote the map induced by $p_j\colon Z_j\to Z_{j-1}$.
By assumption on $\ll$ and \propref{homotopy_groups_of_inverse_limits} the natural function
$\Compon(Z)\to\lim_j\Compon(Z_j)$ is a continuous bijection. We claim that it is a homeomorphism.
Since the spaces $Z_j$ have CW type, the spaces $\Compon(Z_j)$ are discrete and
by a trivial induction we may construct left inverses $\vartheta_j\colon\Compon(Z_j)\to Z_j$ for the
quotient maps $q_j\colon Z_j\to\Compon(Z_j)$ so that $p_j\circ\vartheta_j=\vartheta_{j-1}\circ\Bar p_j$
for all $j$. The maps $\vartheta_j$ then furnish a continuous map $\lim(\Compon Z_j)\to Z$ which we compose
with the quotient $Z\to\Compon(Z)$ to obtain the desired inverse $\lim(\Compon Z_j)\to\Compon(Z)$.
\end{proof}

\begin{prop}\label{homotopy_sets_2}
Let $(X;\AA)$ be an $n$-ad where $X$ is countable. Let $(Y;\BB)$ be an $n$-ad with finite homotopy groups.
If $(Y;\BB)^{(X;\AA)}$ has CW homotopy type, it has finitely many path components.
\end{prop}
\begin{proof}
Take any filtration $L_1\lqs L_2\lqs\dots$ of finite subcomplexes for $X$ (i.e. a countable cofinal subset of $\KKK$)
and consider the associated inverse sequence $Z_j=(Y;\BB)^{(L_j;\AA\cap L_j)}$. Denote also $Z=(Y;\BB)^{(X;\AA)}$.
By \propref{u_M-L_general} and \lemref{component_set} $\Compon Z$ is homeomorphic to $\lim\Compon Z_j$. Since  the $Z_j$ have CW
homotopy types, the spaces $\Compon Z_j$ are discrete, and by \navedi{1} of \propref{homotopy_groups_of_function_space}
they are finite. Hence the limit space is compact. Since it is discrete by assumption, it must be finite.
\end{proof}

%%%%%%%%%%%%%%%%%%%%%%%%%%%%%%%%%%%%%%%%%%%%%%%%%%%%%%%%%%%%%%%%%%%%%%%%%%%%%%%%%
%										%
%				Phantom maps					%
%										%
%%%%%%%%%%%%%%%%%%%%%%%%%%%%%%%%%%%%%%%%%%%%%%%%%%%%%%%%%%%%%%%%%%%%%%%%%%%%%%%%%

\section{Phantom maps}\label{phantom}

If a path component of a function space has CW type and is {\it open} then it does
not form a nontrivial phantom pair by \thmref{u_M-L_general}. Indeed, \thmref{u_M-L_general}
says much more but we will apply it in full strength in later chapters. 
Here we list a few classical examples of function spaces with phantom maps
which are consequently examples of function spaces not of CW type.
We make a brief introduction to phantom maps so as to identify them
with our previously defined concept of phantom path components from section \ref{phantom_components}.
The reader is invited to consult McGibbon's survey article \cite{mcgibbon}
on phantom maps for a comprehensive treatment.

A (pointed) map $f\colon X\to Y$ between CW complexes is called a phantom map if the restriction $f\vert_K\colon K\to Y$
to any finite subcomplex $K$ of $X$ is nullhomotopic. This may be restated in the following homotopy invariant manner as
(cf. Roitberg \cite{roitberg})

\begin{defn}
A pointed map $f\colon X\to Y$ between connected CW complexes is a {\it phantom map} with respect to the class
$\KKK$ of finite connected CW complexes if for any $W\in\KKK$
and any map $j\colon W\to X$ the composition $W\xrightarrow{j} X\xrightarrow{f} Y$ is nullhomotopic.
\end{defn}

This definition agrees with that studied by, for example, Meier \cite{meier} and Zabrodsky \cite{zabrodsky}.
The other standard definition (see McGibbon \cite{mcgibbon}) takes instead of $\KKK$ the class of all finite
dimensional spaces.

For investigations involving the compact open topology, it is natural to consider phantom maps with respect to $\KKK$.
It is convenient to generalize the concept to functions between $n$-ads as follows.

\begin{defn}
Let $X$ be a CW complex and let $Y$ have CW type. Further let $\AA\in\SSS_X$ and $\BB\in\SSS_Y$.
We say that maps $f,g\colon(X;\AA)\to(Y;\BB)$ form a phantom pair if for every finite subcomplex
$K$ of $X$ the restrictions $f\vert_K,\,g\vert_K$ are homotopic as maps $(K;\AA\cap K)\to(Y;\BB)$.
This is to say that the restrictions $f\vert_K,\,g\vert_K$ belong to the same path component of the
function space $(Y;\BB)^{(K;\AA\cap K)}$.
\end{defn}

The definition depends only on the path components of the maps involved, whence 

\begin{lem}
If two maps $f,g$ form a phantom pair then their path components form a phantom
pair of path components as defined in section \ref{phantom_components},
with respect to the inverse system $\setof{(Y;\BB)^{(L;\AA\cap L)}\,\vert\,L}$
indexed by the regular lattice $\XXX$ of all subcomplexes of $X$ with set of compact elements
$\KKK$ the subset of finite subcomplexes. \qed
\end{lem}

As before we denote by $\Ph[g]$ the set of phantom components that form a phantom pair with the path component $[g]$
of $g\colon(X;\AA)\to(Y;\BB)$. Denoting the natural function
\[ \Psi\colon\pi_0\big((Y;\BB)^{(X;\AA)}\big)\to\lim_K\pi_0\big((Y;\BB)^{(K;\AA\cap K)}\big) \]
it is clear that $\Psi^{-1}\big(\Psi[g]\big)=\Ph[g]$.

\begin{prop}
Let $X$ be a countable CW complex. Let $(X;\AA)$ be a CW $n$-ad and let $(Y;\BB)$ have the type of a CW $n$-ad.
Let $L_1\lqs L_2\lqs L_3\lqs\dots$ be a filtration of finite subcomplexes for $X$. Let
$g\colon(X;\AA)\to(Y;\BB)$ be a map of $n$-ads. Then the following natural identification can be made
\[ \Ph[g]=\lm{j}\pi_1\big((Y;\BB)^{(L_j;\AA\cap L_j)},g\vert_{L_j}\big). \]
\end{prop}
\begin{proof}
Follows directly from \propref{homotopy_groups_of_inverse_limits}.
\end{proof}

\begin{example}\label{quasi_solenoid}
Let $m<n$. We consider the space of maps $M(\ZZ[\frac{1}{p}],m)\to K(\ZZ,n)$.
\begin{enroman}
\item	The space $K(\ZZ,n)^{M(\ZZ[\frac{1}{p}],m)}$ has the weak homotopy type of a product   \\   %% linebreak
$K(\Hat{\ZZ}_p/\ZZ,n-m-1)\times K(\ZZ,n)$, but
does not have the type of a CW complex.

In particular, for $m=n-1$ the space in question has uncountably many homotopy equivalent path components
that have the weak homotopy type of $K(\ZZ,n)$ but do not have CW homotopy type.

\item The space $(K(\ZZ,n),*)^{\big(M(\ZZ[\frac{1}{p}],m),*\big)}$ of pointed maps has the weak homotopy
type of a $K(\Hat{\ZZ}_p/\ZZ,n-m-1)$ but does not have CW homotopy type.

In particular, for $m=n-1$, the space in question has uncountably many homotopy equivalent path components
that are weakly contractible but not contractible.
\end{enroman}

If we take for $M(\ZZ[\frac{1}{p}],m)$ the usual telescoping construction
(see for example Hatcher \cite{hatcher}, \S 3.F) and denote the finite stages
of the telescope by $L_k$, say, then it can be easily verified that the sequence
$\setof{\pi_{n-m}(K(\ZZ,n)^{L_k})\,\vert\,k}$ fails to be Mittag-Leffler.

In case $m=n-1$ this implies the existence of phantom components. In particular,
the space of pointed maps has weakly contractible path components none of which has the
homotopy type of a CW complex. The path components are those of an H-group and hence
homotopy equivalent. See \propref{f_l_spheres} for a detailed proof of a more general result.

Here we focus on a more geometric representation of case $m=n-1$.
First consider the short exact sequence $0\to\ZZ\to\ZZ[\frac{1}{p}]\to\ZZ_{p^\infty}\to 0$.
This we can realize as a cofibration sequence 
\begin{equation*}
	\tag{$*$} S^{n-1}\xrightarrow{\varphi} M(\ZZ[\tfrac{1}{p}],n-1)\xrightarrow{\lqs}M(\ZZ_{p^\infty},n-1).
\end{equation*}
Suppose we construct $M(\ZZ[\frac{1}{p}],n-1)$ as the reduced mapping telescope of the sequence
\begin{equation*} \tag{$**$} S^{n-1}\xrightarrow{p}S^{n-1}\xrightarrow{p}\dots.	\end{equation*}
The first stage of the telescope $L_1$ may be identified with the sphere $S^{n-1}$ and the map
$\varphi\colon S^{n-1}\to M(\ZZ[\frac{1}{p}],n-1)$ is then also a degree $p$ map $S^{n-1}\to L_1$
composed by the inclusion into the telescope.

Set $Y=K(\ZZ,n)$. The associated fibrations $(Y,*)^{(L_i,*)}\to(Y,*)^{(L_{i-1},*)}$ may be identified
as $\Omega^{n-1}K(\ZZ,n)\xrightarrow{p}\Omega^{n-1}K(\ZZ,n)$ where $p$ denotes the H-group $p$-th power map.
An application of \propref{E-H} yields a homotopy equivalence \[ (K(\ZZ,n),*)^{(M(\ZZ[\frac{1}{p}],n-1),*)}\simeq T_p \]
where $T_p$ denotes the $p$-adic solenoid of \exref{solenoid}. Moreover the induced map 
\[ \varphi^\#\colon(Y,*)^{(M(\ZZ[\frac{1}{p}],n-1),*)}\to(Y,*)^{(S^{n-1},*)} \]
may be identified with the canonical projection $T_p\to S^1$.

By \lemref{principal_fibration} the restriction fibration
\[ (Y,*)^{(M(\ZZ_{p^\infty},n-1),*)}\to(Y,*)^{(M(\ZZ[\frac{1}{p}],n-1),*)} \] is the principal fibration
obtained by taking the homotopy fibre of $\varphi^\#$. Since $T_p\to S^1$ is a fibration,
$(Y,*)^{(M(\ZZ_{p^\infty},n-1),*)}$ is homotopy equivalent to the fibre of $T_p\to S^1$,
the group $\Hat\ZZ_p$ of $p$-adic integers with its profinite topology.

Thus while $(Y,*)^{(M(\ZZ[\frac{1}{p}],n-1),*)}\simeq T_p$ has uncountably many weakly contractible path components none
of which is contractible, $(Y,*)^{(M(\ZZ_{p^\infty},n-1),*)}\simeq\Hat\ZZ_p$ has every finite set of path components
(homotopy) discrete, contains countable sets of path components which are not homotopy discrete (hence not of CW type),
and countable sets of path components which are homotopy discrete (hence of CW type). \qed
\end{example}

%% \begin{rem_n}
%% Let $X$ be a countable CW complex and $Y$ an arbitrary connected CW complex. Assume that $(Y,*)^{(X,*)}$ has CW type.
%% Using either Milnor's result that iterated loop spaces of CW type spaces also have CW type and then applying
%% \thmref{u_M-L_general}
%% or applying \thmref{u_M-L_general} in conjunction with the $\ll$ representation of phantom sets, we note that there are no
%% phantom maps $X\to\Omega^kY$ (or, equivalently, $S^kX\to Y$) for any $k$. In other words, the function space
%% $(Y,*)^{(X,*)}$ is `stably free' of phantom maps. \qed
%% \end{rem_n}

We apply \thmref{u_M-L_general} to give an alternative proof of a theorem of Meier \cite{meier}.

\begin{cor}[Theorem 2. (a) of Meier \cite{meier}]\label{c_meier}
Let $X$ be a nilpotent CW complex of finite type. If there exists an integer $n$ such that $H_i(X)=0$ for $i\gqs n+1$
then the natural function $[X,Y]_*\to\lim[K,Y]_*$ is bijective. Here $K$ ranges over the finite subcomplexes
of $X$.
\end{cor}
\begin{proof}
The hypotheses on $X$ imply that $X$ is dominated by a finite complex $L$ (see Mislin \cite{mislin2}).
Hence $Y^X$ and $(Y,*)^{(X,*)}$ have CW type by \corref{cor_milnor_theorem}. By \thmref{u_M-L_general}
the space $(Y,*)^{(X,*)}$ does not have phantom path components, and the result follows.
\end{proof}

Function spaces with essential phantom maps exist in abundance. In particular,
we recall the following Zabrodsky's generalization of Miller's theorem on the Sullivan conjecture.

\begin{thm_n}[Zabrodsky \cite{zabrodsky}] \label{zabrodsky}
Let $X$ and $Y$ be nilpotent connected CW complexes of finite type with finite fundamental groups.
If $X$ is a Postnikov space, or an iterated suspension of such a space, and if $Y$ has the homotopy type
of a finite complex, or an iterated loop space of such a space, then {\it all} maps $X\to Y$ are phantom maps. \qed
\end{thm_n}

The following is a trivial consequence.

\begin{cor}[of Zabrodsky's theorem and \thmref{u_M-L_general}]\label{ruled_out}
Let $X$ and $Y$ be nilpotent connected CW complexes of finite type with finite fundamental groups.
Assume that $X$ is a Postnikov space, or an iterated suspension of such a space, and that $Y$ has the
homotopy type of a finite complex, or an iterated loop space of such a space.

Then if the space $(Y,*)^{(X,*)}$ is not contractible, it does not have the homotopy type of a CW complex.
Consequently the space $Y^X$ has CW type if and only if the evaluation $Y^X\to Y$ is a homotopy equivalence.
\end{cor}

\begin{proof}
Let $(Y,*)^{(X,*)}$ have the homotopy type of a CW complex that is not contractible,
and let $k$ denote the least nonnegative integer for which
$\pi_k((Y,*)^{(X,*)},*)$ is non-trivial. Since iterated loop spaces of CW type spaces have CW type,
the space $Z=\Omega^k\big((Y,*)^{(X,*)},*\big)\approx (\Omega^kY,*)^{(X,*)}$ has CW type. By Zabrodsky's theorem,
every map in $Z$ is a phantom map, and since $Z$ is not path-connected, it contains an essential phantom map,
a contradiction.
\end{proof}

\begin{example}
The space of (pointed) maps $K(\ZZ,2n)\to S^{2n+1}$ has $\Ext(\QQ,\ZZ)$ phantom components
and does not have CW homotopy type. \qed
\end{example}

If $X$ is a Postnikov section and $Y$ is a simply connected finite complex, then Zabrodsky's
theorem conveniently implies that the function space $(Y,*)^{(X,*)}$ cannot have CW homotopy type \---
if it has a nontrivial homotopy group.

If $(Y,*)^{(X,*)}$ is weakly contractible, then we cannot say whether it has CW homotopy type or not.
In particular, this is the case with `Miller function spaces' $(Y,*)^{(K(G,1),*)}$
for a locally finite group $G$ and a finite dimensional complex $Y$. We conjecture that these spaces
do not have CW homotopy type, having in mind the following observation. Let $L_1\lqs L_2\lqs\dots$ be
a filtration for $X=K(G,1)$. If $(Y,*)^{(X,*)}$ is contractible, then \corref{contractibility} yields
a subsequence $\setof{L_{\varphi(i)}\,\vert\,i}$ such that for each $i$ the restriction fibration
$(\Omega Y,*)^{(L_{\varphi(i+1)},*)}\to(\Omega Y,*)^{(L_{\varphi(i)},*)}$ is nullhomotopic. In particular,
the induced morphisms on homotopy groups are all trivial. By adjunction, this means that the morphisms
\[ [S^k\wedge L_{\varphi(i+1)},Y]_*\to[S^k\wedge L_{\varphi(i)},Y]_* \] are trivial for all $k\gqs 1$.

However, refining Miller's proof of the Sullivan conjecture for the case $X=\RR P^\infty$,
Lesh \cite{lesh} exhibits a function $j(k,i)$,
{\it going to infinity with $k$}, such that $[S^k\RR P^{j(k,i)},Y]_*\to[S^k\RR P^i,Y]_*$ is trivial.
The nature of dependence of $j$ on $k$ gives reason to believe that this cannot be done uniformly
with respect to $k$.

%%%%%%%%%%%%%%%%%%%%%%%%%%%%%%%%%%%%%%%%%%%%%%%%%%%%%%%%%%%%%%%%%%%%%%%%%%%%%%%%%
%										%
%		Spaces of maps into Postnikov $n$-ads				%
%										%
%%%%%%%%%%%%%%%%%%%%%%%%%%%%%%%%%%%%%%%%%%%%%%%%%%%%%%%%%%%%%%%%%%%%%%%%%%%%%%%%%

\section{Spaces of maps into Postnikov $n$-ads}\label{maps_into_postnikov_n_ads}

\begin{thm}\label{v}
Let $(Y,\BB)$ be an $r$-Postnikov $n$-ad and let $(X,\AA)$ be a CW $n$-ad. Let $g\colon(X;\AA)\to(Y;\BB)$ be a map.

Let for a subcomplex $L$ of $X$ the symbol $F_L$ denote the fibre of the restriction fibration
$(Y;\BB)^{(X;\AA)}\to(Y;\BB)^{(L;\AA\cap L)}$ over $g\vert_L$.

The path component $Z$ of $(Y;\BB)^{(X;\AA)}$ containing $g$ is open in $(Y;\BB)^{(X;\AA)}$ and has CW homotopy type
if and only if for each finite subcomplex $L$ there exists a larger finite subcomplex $L'$ such that the inclusion
$F_{L'}\hookrightarrow F_L$ induces trivial morphisms on all homotopy groups.
\end{thm}

\def\tot{(Y;\BB)^{(X;\AA)}}
\newcommand{\rest}[1]{(Y;\BB)^{(#1;\AA\cap #1)}}

\newcommand{\ovaj}[1]{\pi_k\big(\rest{#1},g\vert_{#1}\big)}

\begin{rem_n}
Let $\KKK$ denote the set of finite subcomplexes of $X$, and let $\III$ denote the subset of $\KKK$ consisting of those
finite subcomplexes $L$ for which the restriction induced morphism
\[ \pi_k\big((Y;\BB)^{(X;\AA)},g\big)\to\pi_k\big((Y;\BB)^{(L;\AA\cap L)},g\vert_L\big) \] is injective
for all $k\gqs 0$ (for $k=0$ in the usual sense).

The condition in the theorem is then equivalent to the following two.
\begin{enroman}
\item	$\III$ is cofinal in $\KKK$, and
\item	for each $L$ in $\KKK$ there exists a larger $L'\in\KKK$ such that
	for any subcomplex $M$ of $X$ containing $L'$ and any $k\gqs 1$, the
	image of \[ \ovaj{M}\to\ovaj{L} \] coincides with that of \[ \ovaj{L'}\to\ovaj{L}. \]
\end{enroman}
Note that if $\III$ is cofinal in $\KKK$, this implies that $Z$ is open.
\end{rem_n}

\begin{proof}
The necessity part is clear by \thmref{u_M-L_general}; we have to prove sufficiency.

Let $L$ be a countable subcomplex of $X$ and let $K_1\lqs K_2\lqs\dots$ be a filtration for $L$ of finite subcomplexes.
There exists a sequence of finite subcomplexes of $X$
\[ L_1\lqs L_2\lqs L_3\lqs\dots \]
such that $L_1\in\III$ (hence $L_i\in\III$ for all $i$), $L_i\gqs K_i$, and for each $i$, the image of
$\ovaj{L_i}\to\ovaj{L_{i-1}}$ equals that of
$\ovaj{M}\to\ovaj{L_{i-1}}$ for any $M$ containing $L_{i-1}$
(including infinite $M$). In particular $\pi_*(F_{L_i},g)\to\pi_*(F_{L_{i-1}},g)$ is trivial.

Let $L_\infty=\cup_iL_i$. Note that the fibre of $\tot\to\rest{L_\infty}$ over $g\vert_{L_\infty}$ is weakly contractible.
In particular, it is path-connected and is equal to the fibre of $Z\to Z_\infty$ where $Z_\infty$ is the image of $Z$.
Let $Z_i$ denote the path component of $g\vert_{L_i}$ in $\rest{L_i}$.  By \thmref{postnikov_sequence} the space
$\kkkk(Z_\infty)$ has the homotopy type of a CW complex. But $L_\infty$ is a countable complex, hence the natural map
$\kkkk(\rest{L_\infty})\to\rest{L_\infty}$ is a homotopy equivalence (see \prevlemref{stratifiable}).
Since compactly generated refinement preserves path components and inclusions, also the natural map
$\kkkk(Z_\infty)\to Z_\infty$ is a homotopy equivalence. In particular, $Z_\infty$ has CW homotopy type.
By \thmref{wicked}, $Z$ has CW homotopy type.
\end{proof}

For countable domain complexes $X$ the statement of \thmref{v} can be somewhat simplified.

\begin{cor}\label{cor_v_1}
Let $(X;\AA)$ be a CW $n$-ad with $X$ countable and let $(Y;\BB)$ be an arbitrary Postnikov CW $n$-ad.
Let $L_1\lqs L_2\lqs\dots$ be a filtration of finite subcomplexes for $X$. Further let $C$ be the
path component of some $g\in\tot$.

Then $C$ is open and has CW type if and only if
\begin{enroman}
\item	for each $k\gqs 1$ the sequence $\setof{\ovaj{L_i}\,\big\vert\,i}$ satisfies the Mittag-Leffler condition, and
\item	there exists $j\gqs 0$ such that for each $k\gqs 0$ (for $k=0$ in the usual sense) the `morphism'
	$\pi_k\big(\tot,g\big)\to\ovaj{L_j}$ is injective.
\end{enroman}
In particular, $C$ is contractible if and only if for each $k\gqs 1$ and each $i\gqs 1$ there exists $s(i)$
such that $\ovaj{L_{s(i)}}\to\ovaj{L_i}$ is trivial, and, in addition, exists $j$ such that
$[(X;\AA),(Y;\BB)]\to[(L_j;\AA\cap L_j),(Y;\BB)]$ is `injective' at the base point $g$. \qed
\end{cor}

\begin{cor}\label{cor_v_2}
If $X$ and $Y$ are countable CW complexes and $Y$ has only finitely many nontrivial homotopy groups then
$(Y,*)^{(X,*)}$ is contractible if and only if it is weakly contractible.
\end{cor}

\begin{proof}
Let $L_1\lqs L_2\lqs\dots$ be a filtration of finite subcomplexes for $X$. By \navedi{1} of
\lemref{homotopy_groups_of_function_space} the spaces $(Y,*)^{(L_i,*)}$ have countable homotopy
groups. By \propref{homotopy_groups_of_inverse_limits} the groups $\ll\pi_{k}\big((Y,*)^{(L_i,*)},*\big)$
vanish for all $k\gqs 1$, hence by \propref{lim1_mittag-leffler} the sequences
$\setof{\pi_{k}\big((Y,*)^{(L_i,*)},*\big)\,\big\vert\,i}$ satisfy the Mittag-Leffler condition.
\end{proof}

\corref{cor_v_2} is in sharp contrast with the case of general $Y$, see \exref{weakly_null}.

%%%%%%%%%%%%%%%%%%%%%%%%%%%%%%%%%%%%%%%%%%%%%%%%%%%%%%%%%%%%%%%%%%%%%%%%%%%%%%%%%
%										%
%		General constructions of function spaces of CW type		%
%										%
%%%%%%%%%%%%%%%%%%%%%%%%%%%%%%%%%%%%%%%%%%%%%%%%%%%%%%%%%%%%%%%%%%%%%%%%%%%%%%%%%

\section{General constructions of function spaces of CW type} \label{general_constructions}

%%  \begin{lem}
%%  If in the pullback diagram of topological spaces
%%  \begin{equation*} \begin{diagram}
%%  \node{Y} \arrow{s} \arrow{e} \node{E} \arrow{s,r}{p} \\
%%  \node{X} \arrow{e} \node{B}
%%  \end{diagram} \end{equation*}
%%  the arrow $p$ is a fibration and $E$ is contractible then
%%  $Y$ is homotopy equivalent to the homotopy fibre of $X\to B$.
%%  \end{lem}

\begin{prop}\label{postnikov_adjunction}
Let $(Y;\BB)$ be an $r$-Postnikov $n$-ad, and let $(X;\AA)$ be a CW $n$-ad. Let $T$ be a subcomplex of $X$.
If the pair $(X,T)$ has relative dimension $\gqs r+1$ (i.e. every cell in $X-T$ has dimension at least $r+1$),
then the fibration
\begin{equation*} \tag{$\bullet$} (Y;\BB)^{(X;\AA)}\to(Y;\BB)^{(T;\AA\cap T)} \end{equation*} is a homotopy
equivalence onto image. In particular, this is true for $T=X^{(r)}$.

If $(X,T)$ has relative dimension $\gqs r+2$, then the fibration ($\bullet$) is a homotopy equivalence.
In particular, $T$ may be $X^{(r+1)}$.
\end{prop}

\begin{proof}
Let $\setof{L_\lambda\,\vert\,\lambda\in W(\alpha)}$ be a good filtration for the pair $(X,T)$ with each consecutive
step the adjunction of a single cell. Set $Z_\lambda=(Y;\BB)^{(L_\lambda;\AA\cap L_\lambda)}$. Then
$\setof{Z_\lambda\,\vert\,\lambda}$ is a restricted inverse systems of fibrations indexed by $W(\alpha)$. The
limit space of this system is $(Y;\BB)^{(X;\AA)}=Z$. For each $\lambda$ denote by $C_\lambda$ the image of $Z$
under $Z\to Z_\lambda$. By \navedi{2} of \propref{homotopy_groups_of_function_space} for each $\lambda$ the fibration
$Z_{\lambda+1}\to Z_\lambda$ is a homotopy equivalence onto image. Hence $C_{\lambda+1}\to C_{\lambda}$ is a homotopy
equivalence. By \navedi{2} of \lemref{restricted_trick} the system $\setof{C_\lambda\,\vert\,\lambda}$ is a system
of homotopy equivalences with limit $Z$. In particular $Z\to C_0$ is a homotopy equivalence, and $C_0$ is the
image of the map ($\bullet$).

If $(X,T)$ is an adjunction of cells of dimension $\gqs r+2$, then each $Z_{\lambda+1}\to Z_\lambda$ is a homotopy
equivalence, and the assertion follows.
\end{proof}

\begin{cor}
Given the hypothesis of \propref{postnikov_adjunction}, if $(Y;\BB)^{(T;\AA\cap T)}$ has CW type then so has
$(Y;\BB)^{(X;\AA)}$. In particular, if the $r$-skeleton $X^{(r)}$ of $X$ is finite, we may take $T=X^{(r)}$. \qed
\end{cor}

\begin{prop}\label{homology_decomposition}
Let $X$ be a simply connected, and $Y$ an arbitrary CW complex. Assume that for each $m$ the function space
$(Y,*)^{(M(H_m(X),m),*)}$ has CW type. Then $(Y,*)^{(X,*)}$ has CW type in the following three cases.
\begin{enroman}
	\item	The complex $X$ is finite dimensional.
	\item	The complex $Y$ has finitely many nontrivial homotopy groups.
	\item	The spaces $(Y,*)^{(M(H_m(X),m),*)}$ are contractible for all $m$.
\end{enroman}
\end{prop}

\begin{proof}
By possibly replacing $X$ with a homotopy equivalent complex we may assume a homology decomposition for $X$,
i.e. a filtration of subcomplexes $X_2\lqs X_3\lqs X_4\lqs\dots$ where $X_2$ is of type $M(H_2(X),2)$ and
for each $i\gqs 3$ the subcomplex $X_i$ is the mapping cone of a (based) cellular map $M(H_i(X),i-1)\to X_{i-1}$.
(For a particular choice of homology decomposition see Chapter \ref{eilenberg-maclane_target}.)

Note that \navedi{1} implies \navedi{2} since by \propref{postnikov_adjunction} the
fibrations $(Y,*)^{(X_i,*)}\to(Y,*)^{(X_{i-1},*)}$ are homotopy equivalences for all large enough $i$,
and we may apply \corref{cor_E-H}.

We prove by induction that the space $(Y,*)^{(X_i,*)}$ has CW type for each $i$. The basis $i=2$ holds trivially.
Assume that $(Y,y_0)^{(X_{i-1},x_0)}$ has CW type.

By \corref{principal_fibration} the fibration $R\colon (Y,y_0)^{(X_i,x_0)}\to(Y,y_0)^{(X_{i-1},x_0)}$
is a principal fibration with fibres either empty or homotopy equivalent to $(Y,y_0)^{(M(H_iX,i),*)}$.
The latter has CW type by assumption, and therefore $(Y,*)^{(X_i,*)}$ has CW type by Stasheff's theorem.

If $X$ is finite-dimensional then $X=X_N$ for some $N$ and \navedi{1} follows. In case
the spaces $(Y,*)^{(M(H_i(X),i),*)}$ are contractible for all $i$, the above induction shows
that also the spaces $(Y,*)^{(X_i,*)}$ are contractible for all $i$, hence \navedi{3} follows from
\propref{contractible_limit}.
\end{proof}

\propref{homology_decomposition} has the dual

\begin{prop}\label{finite_postnikov_decomposition}
Let $Y_N\to Y_{N-1}\to\dots\to Y_2\to Y_1$ be a sequence of fibrations
where $Y_1=K(G_1,n_1)$ with $G_1$ abelian and for each $i\gqs 2$ the fibration
$Y_i\to Y_{i-1}$ is a principal fibration obtained by taking the homotopy fibre
of a pointed map $(Y_{i-1},y_{i-1})\to(K(G_i,n_i+1),*)$ of well-pointed spaces
(where $G_i$ is abelian). Let $(X,x_0)$ be an arbitrary pointed CW complex. 

If for each $i$ the space $K(G_i,n_i)^X$ has the homotopy type of a CW complex then so has the space $Y_N^X$.
\end{prop}

\begin{proof}
By \lemref{pullback_covariant} and \propref{mapping_space_covariant} the map
$(Y_i,y_i)^{(X,x_0)}\to(Y_{i-1},y_{i-1})^{(X,x_0)}$ is a principal fibration with fibres
either empty or homotopy equivalent to $(K(G_i,n_i),*)^{(X,x_0)}$ which has CW homotopy type by hypothesis.
Hence we may proceed inductively using Stasheff's theorem, the basis of the induction holding trivially.
\end{proof}

\begin{cor}
If $Y$ is a simply connected CW complex with $\pi_k(Y)=0$ for all $k\gqs n+1$ and
$X$ is a CW complex such that the space $K(\pi_iY,i)^X$ has CW type for each $i$,
then also the space $Y^X$ has CW type. \qed
\end{cor}

%%%%%%%%%%%%%%%%%%%%%%%%%%%%%%%%%%%%%%%%%%%%%%%%%%%%%%%%%%%%%%%%%%%%%%%%%%%%%%%%%%%%%%%%%%%%%%%%%%%
%%%%%%%%%%%%%%%%%%%%%%%%%%%%%%%%%%%%%%%%%%%%%%%%%%%%%%%%%%%%%%%%%%%%%%%%%%%%%%%%%%%%%%%%%%%%%%%%%%%
%%%												%%%
%%%					Localization						%%%
%%%												%%%
%%%%%%%%%%%%%%%%%%%%%%%%%%%%%%%%%%%%%%%%%%%%%%%%%%%%%%%%%%%%%%%%%%%%%%%%%%%%%%%%%%%%%%%%%%%%%%%%%%%
%%%%%%%%%%%%%%%%%%%%%%%%%%%%%%%%%%%%%%%%%%%%%%%%%%%%%%%%%%%%%%%%%%%%%%%%%%%%%%%%%%%%%%%%%%%%%%%%%%%

\chapter{Localization} \label{localization_chapter}

\setcounter{thm}{0}

%%%%%%%%%%%%%%%%%%%%%%%%%%%%%%%%%%%%%%%%%%%%%%%%%%%%%%%%%%%%%%%%%%%%%%%%%%%%%%%%%
%										%
%		A genuine Zabrodsky lemma					%
%										%
%%%%%%%%%%%%%%%%%%%%%%%%%%%%%%%%%%%%%%%%%%%%%%%%%%%%%%%%%%%%%%%%%%%%%%%%%%%%%%%%%

\section{A genuine Zabrodsky lemma}\label{zabrodsky_ext}

In this section our aim is to prove

\begin{prop}\label{z_prop}
Let $E\to B$ be a fibration with fibre $F$ where $B$ has the homotopy type of a connected CW complex.
Assume that $E$ as well as $F$ have the homotopy type of compactly generated Hausdorff spaces,
and let $X$ be a Hausdorff space. If the section $X\to X^F$ is a homotopy equivalence, then so is $X^B\to X^E$.
\end{prop}

\begin{rem_n}
The condition on $E$ and $F$ is fulfilled for example if $E$ has CW homotopy type.

Let $y_0$ be a base point of $F$. Since the evaluation map $X^F\to X$, $f\mapsto f(y_0)$, is a left inverse for the
section $X\to X^F$, one of the maps is a homotopy equivalence if and only if the other is. If $y_0$ is a nondegenerate
base point in $F$, then the evaluation at $y_0$ is a fibration. In this case by Theorem 6.3 of Dold \cite{dold} for any
path-connected Dold space $X$ (in particular any connected space of CW type)
%% Let $y_0\in F$. Since evaluation $X^F\to X$, $f\mapsto f(y_0)$, has a section
%% $X\to X^F$, one of the maps is a homotopy equivalence if and only if the other is. If $y_0$ is nondegenerate
%% in $F$, evaluation at $y_0$ is a fibration. Then by Theorem 6.3 of Dold \cite{dold} for any
%% path-connected Dold space $X$ (in particular of CW type)
\begin{enroman}
\item $(X,*)^{(F,*)}$ is contractible if and only if the evaluation $X^F\to X$ is a homotopy equivalence, and also
\item $X^B\to X^E$ is an equivalence if and only if $(X,*)^{(B,*)}\to(X,*)^{(E,*)}$ is.
\end{enroman}
In `most' cases we may assume that $F$ has a nondegenerate base point. Notably,
if $(Y,y_0)\xrightarrow{g}(B,b_0)$ is a map of well-pointed spaces, then the base point $(y_0,\const_{b_0})$ is
nondegenerate in the homotopy fibre of $g$ (formed as a subspace of $Y\times B^I$).
\end{rem_n}

The proposition analogous to \ref{z_prop} that assumes {\it weak} contractibility, and infers {\it weak} homotopy
equivalence is known as `the Zabrodsky lemma' (see Miller \cite{miller}, \S 9 or McGibbon \cite{mcgibbon}, Lemma 5.5). 

Our proof of \propref{z_prop} is essentially that of Lemma 1.5 of Zabrodsky \cite{zabrodsky2} modulo some
care about genuine homotopy type; whence our name.

\begin{proof}[Proof of \propref{z_prop}]
There exists a homotopy equivalence $B'\xrightarrow{f}B$ where $B'$ is a simplicial complex. Consider the diagram
\begin{equation*}\begin{diagram}
\node{\kkkk(E')} \arrow{se,b}{\Bar p} \arrow{e} \node{E'} \arrow{e,t}{\Bar f} \arrow{s,l}{\Bar p} \node{E} \arrow{s,r}{p} \\
\node[2]{B'} \arrow{e,t}{f} \node{B}
\end{diagram}\end{equation*}
The square above is a pullback, and by Corollary 1.4 of Brown, Heath \cite{brown-heath} the map
$\Bar f\colon E'\to E$ is an equivalence.
The map $\kkkk(E')\to E'$ is the natural map from the compactly generated refinement,
and $\Bar p\colon\kkkk(E')\to B'$ is a fibration for the class of compactly generated Hausdorff spaces.
Since $E'$ has the homotopy type of one, $\kkkk(E')\to E'$ is an equivalence.
For each $b_0\in B'$ the induced map on fibres over $b_0$ in $\kkkk(E')$ and $E'$ is exactly the natural
map $\kkkk(F)\to F$, where $F$ is the fibre in $E'$.

Using the mapping space functor $X^{(-)}$ and applying \propref{homotopy_equivalence} it suffices to prove
\lemref{z_lemma} below.
\end{proof}

\begin{lem}\label{z_lemma}
Let $B$ be a connected simplicial complex, let $E$ be a compactly generated Hausdorff space, and let $X$ be
a Hausdorff space. Let $p\colon E\to B$ be a fibration for the class of compactly generated Hausdorff, and let
$F$ be a fibre. If the map $X\to X^F$ is a homotopy equivalence, so is the map $X^B\to X^E$.
\end{lem}

\begin{rem_n} For compactly generated spaces $E$ and $B$ let $p\colon E\to B$ be a fibration for the class
of compactly generated spaces. Note that the following hold.
\begin{enroman}
\item	If $A\hookrightarrow B$ is a closed cofibration, so is $E\vert_A\hookrightarrow E$.
\item	If $B$ is a compact contractible space then $E$ is fibre homotopy equivalent to the trivial space $B\times F$
	where $F$ is the fibre.
\item	Fibres over points in the same path component are homotopy equivalent.
\end{enroman}
The standard proofs for Hurewicz fibrations carry over because, given our assumptions, in each particular case
the homotopy lifting property is applied to compactly generated spaces. (Notably for \navedi{1} see Str\o m \cite{strom2},
Theorem 12.)\end{rem_n}

\begin{proof}
Let $\BBB$ be the set of subcomplexes of $B$. By \lemref{good_filtration}
there exists a limit ordinal $\alpha$ and an order preserving injection $U\colon W(\alpha)\to\BBB$ such
that for a limiting $\mu<\alpha$, $U(\mu)=\cup_{\lambda<\mu}U(\lambda)$. Moreover $\cup_{\lambda<\alpha}U(\lambda)=B$,
and we may assume that $U(\lambda+1)=U(\lambda)\cup\sigma_\lambda$ where $\sigma_\lambda$ is a simplex, for each $\lambda$.
Clearly $U(0)$ is a $0$-simplex. We write $B_\lambda=U(\lambda)$, and set $E_\lambda=p^{-1}(B_\lambda)$. Let
$\lambda<\lambda'$. Since $B_\lambda$ is cofibered in $B_{\lambda'}$, also $E_\lambda$ is cofibered in $E_{\lambda'}$.
(See the above remark.) Since the system $\setof{B_\lambda\,\vert\,\lambda}$ dominates compact subsets of $B$,
also $\setof{E_\lambda\,\vert\,\lambda}$ dominates compact subsets of $E$. Since $E$ and $B$ are compactly generated,
the systems $\setof{X^{B_\lambda}\,\vert\,\lambda}$ and $\setof{X^{E_\lambda}\,\vert\,\lambda}$ are restricted
inverse systems of fibrations with limit spaces, respectively, $X^E$ and $X^B$, and the map $p^*\colon X^B\to X^E$
is the limit of maps $p_\lambda^*\colon X^{B_\lambda}\to X^{E_\lambda}$.

In light of \corref{cor_transfinite_geoghegan_1} it is therefore enough to prove the following. Suppose $L$ and $L'$
are subcomplexes of $B$ such that $L'=L\cup\sigma$ where $\sigma$ is a simplex. If $X^L\to X^{E\vert_L}$ is a homotopy
equivalence, then so is $X^{L'}\to X^{E\vert_{L'}}$.

Note $E\vert_{L'}=E\vert_L\cup E\vert_\sigma$. Since $\sigma$ is contractible, there exists a fibre homotopy
equivalence $\sigma\times F\to E\vert_\sigma$ (over $\sigma$). In particular this equivalence restricts to a (fibre)
homotopy equivalence $\partial\sigma\times F\to E\vert_{\partial\sigma}$.

Thus there exists the following commutative diagram.

\begin{equation*}\begin{diagram}
\divide\dgARROWLENGTH by2
\node[3]{X^{E\vert_{L'}}} \arrow[4]{e} \arrow{s,-} \node[4]{X^{E\vert_\sigma}} \arrow[3]{s} \arrow[4]{e,t}{\simeq} \node[4]{X^{\sigma\tim F}} \arrow[3]{s} \\
\node{X^{L'}} \arrow{nee} \arrow[4]{e} \arrow[3]{s} \node[2]{} \arrow[2]{s} \node[2]{X^{\sigma}} \arrow{nee} \arrow[3]{s} \\ \\
\node[3]{X^{E\vert_L}} \arrow[2]{e,-} \node[2]{} \arrow[2]{e} \node[2]{X^{E\vert_{\partial\sigma}}} \arrow[4]{e,t}{\simeq} \node[4]{X^{\partial\sigma\tim F}} \\
\node{X^{L}} \arrow{nee} \arrow[4]{e} \node[4]{X^{\partial\sigma}} \arrow{nee}
\end{diagram}\end{equation*}

Since $\sigma\times F\to E\vert_\sigma$ is a map over $\sigma$, the composite
$X^\sigma\to X^{E\vert_\sigma}\to X^{\sigma\times F}$ is induced by projection $\sigma\times F\to\sigma$.
Similarly for the composite $X^{\partial\sigma}\to X^{E\vert_{\partial\sigma}}\to X^{\partial\sigma\times F}$.

Since the projection induced maps $X^{\sigma}\to X^{\sigma\times F}$, respectively
$X^{\partial\sigma}\to X^{\partial\sigma\times F}$, are homotopy equivalences, so are the maps
$X^\sigma\to X^{E\vert_\sigma}$, respectively
$X^{\partial\sigma}\to X^{E\vert_{\partial\sigma}}$. By assumption, $X^L\to X^{E\vert_L}$ is a homotopy equivalence.
The cube in the diagram is a morphism of pullback diagrams with vertical arrows fibrations, so by coglueing homotopy
equivalences \cite{brown-heath} also the map (of pullback spaces) $X^{L'}\to X^{E\vert_{L'}}$ is a homotopy equivalence.

(We note that expressing the sphere $S^n$ as the union of two disks meeting in a common boundary it follows by
a similar finite inductive argument that $X^{S^n}\to X^{S^n\times F}$ is a homotopy equivalence for all $n$.)
\end{proof}

We immediately infer

\begin{cor}\label{cor_zabrodsky}
Assume $Y$ is a connected CW complex, and $X$ is a connected CW complex.
If $(Y,*)^{(\Omega X,*)}$ is contractible, so is $(Y,*)^{(X,*)}$. \qed
\end{cor}

%%%%%%%%%%%%%%%%%%%%%%%%%%%%%%%%%%%%%%%%%%%%%%%%%%%%%%%%%%%%%%%%%%%%%%%%%
%									%
%			Local target complex				%
%									%
%%%%%%%%%%%%%%%%%%%%%%%%%%%%%%%%%%%%%%%%%%%%%%%%%%%%%%%%%%%%%%%%%%%%%%%%%

\section{Local target complex}
\label{local_target}

In the theory of homotopy localization with respect to a map (see Dror Farjoun \cite{dror})
function spaces play a central role. Given a map of CW complexes $f\colon A\to B$, a CW complex $Y$
is called $f$-local if the induced map $Y^B\to Y^A$ is a weak homotopy equivalence. In particular,
in case $f$ is the map $W\to *$ then $f$-local CW complexes $Y$ are called $W$-null.

There exist the notions of strongly $f$-local and strongly $W$-null obtained by replacing
weak homotopy equivalence above for genuine homotopy equivalence. It would be of some importance
in homotopy theory to be able to characterize the difference between $f$-local and strongly $f$-local
(see V.~Halperin \cite{halperin}). However, this seems to be a difficult task in general.

In Section \ref{local_domain} below (see \thmref{finite_localization} and \exref{weakly_null}) we exhibit
pairs $Y,W$ such that $Y$ is weakly $W$-null but not strongly $W$-null.

For any map $f\colon A\to B$ there exists a `localization functor' $L_f$, that is a homotopy idempotent functor
equipped with a natural transformation $\Id\to L_f$ so that for every space $X$ of CW type the map
$X\to L_fX$ is a homotopically universal map to $f$-local spaces.

If $Y$ is an $f$-local space then the natural map $X\to L_fX$ induces a weak homotopy equivalence (see \cite{dror})
\begin{equation*}\tag{$\star$}	Y^{L_fX}\to Y^X. \end{equation*}
So a problem related to the one above would be to investigate when the map ($\star$) is a genuine equivalence.
We show below that this is the case for localization with respect to a set of primes which
is a particular case of localization with respect to a map (see \cite{dror}).

Our basic reference for localization with respect to a set of primes is Hilton, Mislin, Roitberg \cite{h-m-r}.

If $P$ is a set of primes then \[ \ZZ_{(P)}=\setof{\tfrac{a}{b}\,\vert\,\text{ prime divisors of $b$ belong
to the complement of }P}\lqs\QQ. \] For an abelian group $A$ we will denote $A_{(P)}=A\otimes\ZZ_{(P)}$ the
localization of $A$ at the set $P$ (or just $A_{(p)}$ in case of a single prime $P=\setof{p}$), and for
a (nilpotent) CW complex $X$ we will denote by $X_{(P)}$ the localization of $X$ at the set $P$.

\begin{thm}\label{localization}
Let $P$ be a set of primes and let $X$ be a simply connected CW complex.
Let $Y$ be a  $P$-local complex. Then the map $(Y,*)^{(X_{(P)},*)}\to(Y,*)^{(X,*)}$
induced by localization $X\to X_{(P)}$ is a homotopy equivalence.
\end{thm}

We list first a few corollaries and an example, and later prove the theorem.

\begin{cor}\label{the_square}
Let $\PP=P_1\cup P_2$ be a partition of the primes. Let $X$ and $Y$ be simply connected CW complexes.
There exists a pullback square
\begin{equation*}\begin{diagram}
\node{(Y,*)^{(X,*)}} \arrow{e} \arrow{s} \node{(Y_{(P_2)},*)^{(X_{(P_2)},*)}} \arrow{s} \\
\node{(Y_{(P_1)},*)^{(X_{(P_1)},*)}} \arrow{e} \node{(Y_{(0)},*)^{(X_{(0)},*)}}
\end{diagram}\end{equation*}
\end{cor}

\begin{proof}
We construct the square by first constructing a rationalization $Y_{(0)}$ (which we may assume to be a CW complex),
followed by localizations $Y_{(P_1)}$, $Y_{(P_2)}$ so that the maps $Y_{(P_i)}\to Y_{(0)}$ are fibrations. The
topological pullback $Y'$ of $Y_{(P_1)}\to Y_{(0)}\leftarrow Y_{(P_2)}$ is then homotopy equivalent to $Y$,
and $(Y',*)^{(X,*)}$ is the pullback of $(Y_{(P_1)},*)^{(X,*)}\to(Y_{(0)},*)^{(X,*)}\leftarrow(Y_{(P_2)},*)^{(X,*)}$
(see \lemref{pullback_covariant}). The resulting diagram is one of fibrations by \propref{mapping_space_covariant},
and an application of \thmref{localization} to its nodes proves the assertion.
\end{proof}

\begin{cor}\label{razcep} Let $X$ be a simply connected CW complex with torsion homology $\H_*$.
\begin{enroman}
\item For any partition of the primes $\PP=P_1\cup P_2$, the space
of pointed maps $(Y,*)^{(X,*)}$ is homotopy equivalent to the
product \[ (Y_{(P_1)},*)^{(X_{(P_1)},*)}\times (Y_{(P_2)},*)^{(X_{(P_2)},*)}. \]
In particular if $Y^X$ has CW homotopy type, so has the space $[Y_{(p)}]^X\cong [Y_{(p)}]^{X_{(p)}}$
for any prime $p$.
\item If in fact the homology $\H_*(X)$ is $P$-torsion for some $P\subset\PP$, the restriction map
	$(Y,*)^{(X,*)}\to(Y_{(P)},*)^{(X,*)}$ is a homotopy equivalence.
\end{enroman}
\end{cor}
\begin{proof}
If $\H_*(X)$ is torsion, the rationalization $X_{(0)}$ is contractible. In particular, if $\H_*(X)$ is $P$-torsion
then already $X_{(\PP\setminus P)}$ is contractible. Our assertions now follow from \corref{the_square},
\propref{homotopy_equivalence}, and the coglueing theorem (Brown, Heath \cite{brown-heath}.)
\end{proof}

It is natural to ask for some kind of converse to \navedi{1} above. The following example shows it is not possible.

\begin{example}
Set $X=\vee_{p\in\PP}M(\ZZ_p,p)$ and let $Y$ be a CW complex with $\pi_{p^2}(Y)=\ZZ_p$
for $p\in\PP$ and $\pi_k(Y)=0$ for $k\notin\PP$. 
The localization $Y_{(p)}$ at each prime $p$ is an Eilenberg-MacLane
space $K(\ZZ_p,p^2)$, hence $Y_{(p)}^X$ has the homotopy type of a CW complex
by \propref{postnikov_adjunction}.

However, since $(Y,*)^{(M(\ZZ_p,p),*)}$ is not contractible for any $p$,
$Y^X$ does not have the homotopy type of a CW complex, by \lemref{splittings}.\qed
\end{example}

The rest of this section is devoted to the proof of \thmref{localization}.

\begin{prop}\label{disjoint_local}
Let $\PP=P\cup P'$ be a partition of the primes with $P'\neq\emptyset$ and
let $X$ be a simply connected CW complex with $\H_*(X)$ torsion $P'$-primary.
Further let $Y$ be a $P$-local complex. Then $(Y,*)^{(X,*)}$ is contractible.
\end{prop}

\begin{lem}\label{lem_disjoint_local}
\propref{disjoint_local} holds for countable CW complexes $X$.
\end{lem}

\begin{proof}
In light of \propref{homology_decomposition} it is enough to prove the special case
$X=M(A,m)$ where $A$ is a $P'$-torsion group and $m\gqs 2$. 
Since $A$ is a countable torsion group, its $q$-primary components are totally $q$-projective (see \cite{fuchs2}),
and consequently $A$ is the union of a possibly infinite ascending sequence of finite subgroups
$A_0\lqs A_1\lqs\dots$ where $A_0\cong\ZZ_{q_0}$ and $A_i/A_{i-1}\cong\ZZ_{q_i}$ where $q_i\in P'$ for each $i$.
Then $X$ may be obtained as the union of an ascending sequence of subcomplexes $M(A_0,m)\lqs M(A_1,m)\lqs\dots$.
Consequently $(Y,*)^{(X,*)}$ is the limit of $(Y,*)^{(M(A_i,*),*)}$. Since the fibre of
$(Y,*)^{(M(A_i,*),*)}\to(Y,*)^{(M(A_{i-1},*),*)}$ over the constant map is $(Y,*)^{(M(\ZZ_{q_i},m),*)}$ it
suffices for a proof by induction to prove that $(Y,*)^{(M(\ZZ_q,m),*)}$ is contractible for each $q\in P'$.
Since $Y$ is $P$-local, $(Y,*)^{(M(\ZZ_q,m),*)}$ is weakly contractible. But $M(\ZZ_q,m)$ is a finite complex,
and hence $(Y,*)^{(M(\ZZ_q,m),*)}$ has CW type by Milnor's theorem.
\end{proof}

\begin{proof}[Proof of \propref{disjoint_local}]
Let $K_\infty$ be a countable subcomplex of $X$. By \lemref{homology_P} and \lemref{P-sequence}
(see also the proof of \thmref{double_wicked}) there exists a countable subcomplex $L_\infty$ of $X$ containing
$K_\infty$ such that the inclusion induced morphism \[ H_*(L_\infty)\to H_*(X) \] is injective. In particular,
$\H_*(L_\infty)$ is torsion $P'$-primary. By \lemref{lem_disjoint_local} the space $(Y,*)^{(L_\infty,*)}$ is contractible.
By \corref{cor_wicked}, also $(Y,*)^{(X,*)}$ is contractible.
\end{proof}

\begin{proof}[Proof of \thmref{localization}]
Let $F$ denote the homotopy fibre of the localization map $X\to X_{(P)}$. Then $F$ is a connected space of CW type
with homotopy groups $\pi_*(F)$ abelian torsion $P'$-primary. Moreover $\pi=\pi_1(F)$ is divisible.
Let $f\colon F\to K(\pi,1)$ denote a map inducing the identity on the fundamental group, and let $p\colon\Tilde F\to F$
be the homotopy fibre of $f$. Then $\Tilde F$ has the homotopy type of a simply connected CW complex with torsion
$P'$-primary homotopy groups. Since these form a Serre class, also $\H_*(\Tilde F)$ is $P'$-primary torsion. Then
by \propref{disjoint_local}, the space $(Y,*)^{(\Tilde F,*)}$ is contractible. (Note that $F$ has nondegenerate base points.)
By \propref{z_prop}, the map $(Y,*)^{(K(\pi,1),*)}\to(Y,*)^{(F,*)}$ is a homotopy equivalence. We claim that
$(Y,*)^{(K(\pi,1),*)}$ is contractible.

Since $\pi$ is divisible $P'$-torsion, it is isomorphic to the direct sum of some copies of the 
quasicyclic groups $\ZZ_{q^\infty}$, for various $q\in P'$. This is to say $\pi$ is the direct
sum of totally projective $q$-groups (for various $q$), and as such admits a composition series
(see the discussion preceding \lemref{good_filtration_general} on page \pageref{smooth_function}, and Fuchs \cite{fuchs2})
\[ N_0\lqs N_1\lqs\dots\lqs N_\lambda\lqs\dots\,\,\,\,(\lambda<\alpha) \]
where the successive quotients $N_{\lambda+1}/N_\lambda$ are cyclic of prime power $q_\lambda$ with $q_\lambda\in P'$.

Using transfinite construction we may represent $K(\pi,1)$ as the colimit of a system of CW complexes
$\setof{K(N_\lambda,1)\,\vert\,\lambda<\alpha}$, where $K(N_\lambda,1)\hookrightarrow K(N_{\lambda+1},1)$
is a cofibration for each $\lambda$, and for a limit ordinal $\mu<\alpha$ the space $K(N_\mu,1)$
is the colimit of $\setof{K(N_\lambda,1)\,\vert\,\lambda<\mu}$.

By \corref{cor_transfinite_geoghegan_1} it suffices to show that contractibility of
$(Y,*)^{(K(N_\lambda,1),*)}$ implies contractibility of $(Y,*)^{(K(N_{\lambda+1},1),*)}$.
The extension $0\to N_\lambda\to N_{\lambda+1}\to\ZZ_q\to 0$ induces a fibration
\[ K(N_{\lambda},1)\to K(N_{\lambda+1},1)\to K(\ZZ_{q},1). \]
If $(Y,*)^{(K(N_\lambda,1),*)}$ is contractible, the map $(Y,*)^{(K(\ZZ_q,1),*)}\to(Y,*)^{(K(N_{\lambda+1},1),*)}$
is a homotopy equivalence by \propref{z_prop}. Since $q$ belongs to $P'$, the space $(Y,*)^{(K(\ZZ_q,1),*)}$ is
contractible by Example 1 of \cite{smrekar}.
\end{proof}

%%%%%%%%%%%%%%%%%%%%%%%%%%%%%%%%%%%%%%%%%%%%%%%%%%%%%%%%%%%%%%%%%%%%%%%%%%%%%%%%%
%										%
%		Local domain complex and					%
%			relation to the Moore conjecture			%
%										%
%%%%%%%%%%%%%%%%%%%%%%%%%%%%%%%%%%%%%%%%%%%%%%%%%%%%%%%%%%%%%%%%%%%%%%%%%%%%%%%%%

\section{Local domain complex and relation to the Moore conjecture}
\label{local_domain}

Let $K$ be the suspension of a finite complex and $P$ a set of primes. The $H$-cogroup structure allows us to form
a sequence \begin{equation*}\tag{$\star$} K\xrightarrow{q_1}K\xrightarrow{q_2}\dots \end{equation*} where $q_i$ denotes
$H$-cogroup multiplication. If $\setof{q_i}$ is a sequence of primes in $\PP\setminus P$ with each member occurring
infinitely many times, then the telescope of the sequence ($\star$) is the $P$-localization of $K$. Then for a space $Y$
the function space $(Y,*)^{(K_{(P)},*)}$ can be viewed as the inverse limit of
\[ \dots\to(Y,*)^{(K,*)}\xrightarrow{q_2}(Y,*)^{(K,*)}\xrightarrow{q_1}(Y,*)^{(K,*)} \]
where now $q_i$ denotes the induced $H$-group multiplication by $q_i$. This construction
is particulary suitable for application of \thmref{u_M-L_general} and we do so in this section.

For the particular case when $K$ is a sphere, we obtain the following result.

\begin{prop}\label{f_l_spheres}
Let $\PP=P\cup P'$ be a partition of the primes and let $Y$ be a simply connected finite complex. Further let $R$ be
a set of primes with $R\cap P\neq\emptyset$ and $m$ a positive integer.
\begin{enroman}
\item	If the space $(Y_{(R)},*)^{(S^m_{(P')},*)}$ has the homotopy type of a CW complex then for $k\gqs 1$ the
	homotopy groups $\pi_{m+k}(Y_{(P\cap R)})$ are finite abelian groups bounded by a common bound. Moreover,
	the loop space $\Omega^{m+1}Y_{(P\cap R)}$ has an H-space exponent, the space $(Y_{(P\cap R)},*)^{(S^m_{(0)},*)}$
	is contractible and $(Y_{(R)},*)^{(S^m_{(P')},*)}\simeq\Omega^mY'$ where $Y'$ denotes the
	homotopy fibre of the localization $Y_{(R)}\to Y_{(P\cap R)}$. In particular, $Y$ is an elliptic complex.
\item	If, in addition, the homology group $H_k(Y)$ is infinite for some positive $k$, then $P\cap R$ must be finite.
\item	Conversely, if for some $l$ the space $\Omega^{l}Y_{(P\cap R)}$ admits an H-space exponent, then
	$(Y_{(R)},*)^{(S^l_{(P')},*)}$ has CW homotopy type, and is homotopy equivalent to $\Omega^lY'$.
\end{enroman}
\end{prop}

Recall that a simple space $Z$ is said to have a homotopy exponent at prime $p$ if there exists $k$
such that $p^k$ annihilates the $p$-primary component of $\pi_l(Z)$ for all $l$.

If $Z$ is a homotopy associative H-space then $Z$ is said to have an H-space (or {\it geometric}) exponent
if there exists a number $b$ such that the map $b\colon Z\to Z$, sending $z$ to $z^b$, is nullhomotopic. Note
that if $Z$ has H-space exponent $b$, then also $\Omega Z$ has H-space exponent $b$.

We refer the reader to the survey article of Neisendorfer and Selick \cite{neisen-selick} for an introduction
to (geometric) exponents in homotopy theory.

A simply connected finite complex $Y$ is {\it elliptic} if $\pi_k(Y)$ is torsion for all but finitely
many $k$. Otherwise it is {\it hyperbolic}. See Felix, Halperin, Thomas \cite{fht}.

We will prove \propref{f_l_spheres} as a corollary to the more general \thmref{finite_localization} below.

\begin{example}\label{weakly_null}
The function space $(S^n,*)^{(M(\QQ,m),*)}$ does not have CW type for any $m$.
In particular, for $m>n$ the space is weakly contractible but not contractible.
Thus for $m>n$ the sphere $S^n$ is weakly $M(\QQ,m)$-null but not strongly so.

Since $\Omega\big((S^n,*)^{(M(\QQ,m),*)},*\big)\approx(S^n,*)^{(M(\QQ,m+1),*)}$
the situation does not `improve' after looping.\qed
\end{example}

Until the end of this section let $X_m$ denote the localization of the $m$-sphere away from $p$
where $p$ is a prime understood from the context.

By exploiting results from the theory of homotopy exponents and H-space exponents
(see \cite{neisen-selick}) we give the following rather nontrivial

\begin{example}\label{spheres}
For all large enough $m$, the function space $(S^n,*)^{(X_m,*)}$ has the homotopy type of a CW complex
(and is homotopy equivalent to $\Omega^mX_n$).

This follows from the fact that $\Omega^lS^n_{(p)}$ has an H-space exponent for all large enough $l$
(the case $n$ even is due to James \cite{james}, the case $n,p$ odd is due to  Cohen, Moore,
Neisendorfer \cite{cmn}, and the case $n$ odd, $p=2$ is attributed to an unpublished result of Moore).

Let $W$ denote the localization at $p$ of the $n$-connected cover of $S^n$. It was shown by Neisendorfer and Selick
\cite{neisen-selick} that $\Omega^{n-3}W$ does not have an H-space exponent. However, $\Omega^{n-2}W$ does, see
Neisendorfer \cite{neisendorfer2}.

In particular, the space $(W,*)^{(X_{n-4},*)}$ does not have the type of a CW complex, and the space
$(W,*)^{(X_{n-2},*)}$ does.\qed
\end{example}

As a consequence we infer

\begin{prop}
Let $Y$ be a rationally elliptic simply connected finite complex.
For large enough $m$ and almost all primes $p$ the function space $(Y,*)^{(X_m,*)}$
has the homotopy type of a CW complex.
\end{prop}

\begin{proof}
By McGibbon and Wilkerson \cite{mcgibbon_wilkerson}
the loop space $\Omega Y$ is $p$-equivalent to a finite product of spheres and loop spaces of spheres,
for almost all primes $p$. Our assertion now follows from \propref{f_l_spheres} in conjunction with \exref{spheres}.
\end{proof}

We recall the following conjecture due to John C. Moore (see Selick \cite{selick}).
\begin{conj}[Moore]\label{moore_conjecture}
Let $Y$ be the localization at $p$ of a simply connected finite complex. Then the following are equivalent.
\begin{enumerate}
\item	There exists a number $l$ such that $\Omega^lY$ has an H-space exponent.
\item	The complex $Y$ has a homotopy exponent.
\item	The rational homotopy groups $\pi_*(Y)\otimes\QQ$ are totally finite-dimensional.
\end{enumerate}
\end{conj}

This conjecture can now be rephrased in the following manner.
\begin{conj}
Let $Y$ be the localization at $p$ of a simply connected finite complex. Then the following are equivalent.
\begin{itemize}
\item[($1^*$)]	There exists a number $m$ such that the function space $(Y,*)^{(X_m,*)}$ has the homotopy type of
		a CW complex.
\item[($2$)]	The complex $Y$ has a homotopy exponent.
\item[($3^*$)]	There exists a number $m$ such that the function space $(Y,*)^{(X_m,*)}$ is weakly contractible.
\end{itemize}
\end{conj}

We turn to the proof of \propref{f_l_spheres}.

Consider the space of maps $K_{(P)}\to Y$ where $K$ is a simply connected finite complex,
$P$ is a set of primes (possibly empty), and $Y$ is any simply connected complex. Let $P'=\PP\setminus P$.

\corref{the_square} yields the pullback diagram
\begin{equation*}\begin{diagram}
\node{(Y,*)^{(K_{(P)},*)}} \arrow{e} \arrow{s} \node{(Y_{(P)},*)^{(K_{(P)},*)}} \arrow{s} \\
\node{(Y_{(P')},*)^{(K_{(0)},*)}} \arrow{e} \node{(Y_{(0)},*)^{(K_{(0)},*)}}
\end{diagram}\end{equation*}
Note that reapplying \thmref{localization} we deduce $(Y_{(P)},*)^{(K_{(P)},*)}\simeq(Y_{(P)},*)^{(K,*)}$
and $(Y_{(0)},*)^{(K_{(0)},*)}\simeq(Y_{(0)},*)^{(K,*)}$ whence by Milnor's theorem we infer

\begin{prop}\label{rational_reduction}
\begin{enroman}
\item	The space $(Y,*)^{(K_{(P)},*)}$ has CW type if and only if the image
	of $(Y,*)^{(K_{(P)},*)}\to (Y_{(P')},*)^{(K_{(P)},*)}\simeq(Y_{(P')},*)^{(K_{(0)},*)}$ has.

	In particular, if $Y$ is $R$-local where $R$ is a set of primes contained in $P$, then $(Y,*)^{(K_{(P)},*)}$
	has CW homotopy type.

\item	If $Y$ is a loop space, then $(Y,*)^{(K_{(P)},*)}$ has CW type if and only if \linebreak
	$(Y_{(P')},*)^{(K_{(P)},*)}\simeq(Y_{(P')},*)^{(K_{(0)},*)}$ has.
\end{enroman}
\end{prop}

\begin{proof}
Only \navedi{2} needs be discussed.
Suppose $Y\simeq\Omega(Z,*)$. Let $F$ be the homotopy fibre of $Z\to Z_{(P')}$,
and $F'$ the homotopy fibre of $F\to Z$. Then $F'$ is homotopy equivalent to $\Omega(Z_{(P')},*)\simeq Y_{(P')}$,
and $F'\to F$ is a principal fibration with fibre homotopy equivalent to $\Omega(Z,*)\simeq Y$.
Then \[ (F',*)^{(K_{(P)},*)}\to(F,*)^{(K_{(P)},*)} \] is a (principal) fibration with fibre
(over the constant map) homotopy equivalent to $(Y,*)^{(K_{(P)},*)}$.
Note that $F$ is $P$-local. Thus by \thmref{localization}, $(F,*)^{(K_{(P)},*)}$ is homotopy equivalent
to $(F,*)^{(K,*)}$ which has CW type. Thus if $(Y,*)^{(K_{(P)},*)}$ has CW type,
so has $(F',*)^{(K_{(P)},*)}\simeq(Y_{(P')},*)^{(K_{(P)},*)}$, by Stasheff's theorem.
\end{proof}

\begin{cor}\label{cor_rational_reduction}
If $(Y_{(P')},*)^{(K_{(0)},*)}$ happens to be contractible, then the space
$(Y,*)^{(K_{(P)},*)}$ is homotopy equivalent to $(Y_{\tau P},*)^{(K,*)}$ where $Y_{\tau P}$ denotes
the homotopy fibre of the localization $Y\to Y_{(P')}$. \qed
\end{cor}

\begin{lem}\label{pc}
Let $P$ and $R$ be sets of primes with $P\cap R\neq\emptyset$, and
let $Y$ be a simply connected complex of finite type over $\ZZ_{(R)}$.
Further let $A$ be a finite complex, and denote $\Gamma=[SA,Y]_*$. Let $p_1,p_2,\dots$ be a sequence of
primes in $P$ with each prime in $P$ occurring infinitely many times. The limit group of the inverse sequence
$\dots\to\Gamma\xrightarrow{\xi\mapsto\xi^{p_2}}\Gamma\xrightarrow{\xi\mapsto\xi^{p_1}}\Gamma$ is trivial.
\end{lem}
\begin{proof}
Since $A$ is finite, $[SA,Y]_*\cong[SA,Y_i]_*$ for some Postnikov section $Y_i$ of $Y$.
Using left-exactness of the functor $\lim$ and the fact that no element of a finitely generated
$\ZZ_{(R)}$-module is infinitely divisible by a prime belonging to $P\cap R\subset R$,
it follows easily by induction that $[SA,Y_i]_*$ is trivial for all $i$.
\end{proof}

Before stating and proving \thmref{finite_localization} we recall some properties of abelian groups,
which we will need to a larger extent in Section \ref{explicit_determinations}. We refer
to Fuchs \cite{fuchs1},\cite{fuchs2} for all results on abelian groups.

\label{divisible_reduced_bounded} 

Let $P$ be a set of primes and $A$ an abelian group. Then $A$ is called $P$-divisible if it is divisible
by each prime belonging to $P$, i.e. if for each $y\in A$ and $p\in P$ there exists $x\in A$ with $px=y$.
If $A$ is divisible by all primes, it is called a {\it divisible} group. Divisible abelian groups are
exactly injective abelian groups.

It is well known that every divisible group is isomorphic with a direct sum of some copies of
the rationals $\QQ$ and some copies of the quasicyclic groups $\ZZ_{p^\infty}$ for various primes $p$,
say $A\cong(\oplus_{\Lambda_0}\QQ)\oplus\oplus_{p\in\PP}(\oplus_{\Lambda_p}\ZZ_{p^\infty})$. 
The cardinalities of $\Lambda_0$ and $\Lambda_p$ for $p\in\PP$ form a complete system of invariants
of $A$ (see Fuchs \cite{fuchs1}). Moreover in this case the rank of $A$ equals $\Lambda_0+\sum_{p\in\PP}\Lambda_p$.

An abelian group is called {\it reduced} if it contains no nontrivial
divisible subgroups. For an abelian group $A$ we may define $d(A)$ to be the unique
maximal divisible subgroup of $A$. Then $A/d(A)$ is a reduced group and since $d(A)$ is
injective, $A$ is isomorphic to $d(A)\oplus A/d(A)$. Although $A/d(A)$ is unique only
up to isomorphism, we may still define the {\it `reduced part'} of $A$ as a representative
of the isomorphism class of $A/d(A)$, and denote it by $r(A)$. This makes sense when we are only interested in the
isomorphism class.

We say that an abelian group $G$ is bounded (or that it admits an exponent) if there exists a number $b$
such that $b\cdot x=0$ for all $x\in G$. Such a number $b$ is called an exponent of $G$.
(See also \propref{abelian_groups}.)

\begin{thm}\label{finite_localization}\label{domain_divisible}
Let $P$ be a set of primes.
Let $Y$ be a simply connected CW complex and let $K$ be the suspension of a finite complex.
Denote $\Gamma_k=\pi_k\big((Y,*)^{(K,*)},*\big)$.
\begin{abc}
\item	If the function space $(Y,*)^{(K_{(\PP\setminus P)},*)}$ has the homotopy type of a CW complex, then there
	exists an integer $b\in(P)$ so that for all $k\gqs 1$ the group $D_k=b\cdot\Gamma_k$ is $P$-divisible and
	contains no $P$-torsion summands. The short exact sequence $0\to D_k\to\Gamma_k\to\Gamma_k\otimes\ZZ_b\to 0$
	induces the split short exact sequence
	\[ 	0\to D_k{}_{(P)}\to\Gamma_k{}_{(P)}\to\Gamma_k\otimes\ZZ_b\to 0		\]
	where $D_k{}_{(P)}$ is a rational group.

	Moreover, the divisible part of $\Gamma_0$ does not contain $P$-torsion,
	and therefore the divisible part of $\Gamma_0{}_{(P)}$ is rational.
\item	If, in addition to the assumption of {\bf a.}, $Y$ is of finite type over $\ZZ_{(R)}$ where $P\cap R$
	is nonempty then for all $k\gqs 1$ the localization $\Gamma_k{}_{(P)}$ is a finite abelian $b$-bounded group.
	Consequently $(Y_{(P)},*)^{(K_{(\PP\setminus P)},*)}$ is contractible and the abelian H-space
	$(\Omega Y_{(P)},*)^{(K,*)}$ admits an H-space exponent.
\item	Conversely, if the H-space $(Y_{(P)},*)^{(K,*)}$ admits an H-space exponent,
	then $(Y_{(P)},*)^{(K_{(\PP\setminus P)},*)}$ is contractible.
\end{abc}
\end{thm}

\begin{proof}
We construct localization $X=K_{(\PP\setminus P)}$ as the reduced infinite mapping telescope of the sequence
\[ K\xrightarrow{p_1}K\xrightarrow{p_2}K\xrightarrow{p_3}\dots \]
where $p_1,p_2,\dots$ is a sequence of primes in $P$ with each prime in $P$ occurring infinitely many times.

More precisely, we assume $X$ filtered by a based sequence $L_0\lqs L_1\lqs L_2\lqs\dots$ of finite
subcomplexes each of which is homotopy equivalent to $K$, and for all $i$, the inclusion $L_{i-1}\lqs L_{i}$
corresponds to H-cogroup multiplication by $p_i$.
Then $(Y,*)^{(X,*)}$ is the limit of the associated inverse sequence $(Y,*)^{(L_i,*)}$ where
the bonding map $(Y,*)^{(L_i,*)}\to(Y,*)^{(L_{i-1},*)}$ is equivalent to H-group multiplication by $p_i$
on $(Y,*)^{(K,*)}$.

In particular, the morphisms $\pi_k((Y,*)^{(L_{i},*)},*)\to\pi_k((Y,*)^{(L_{i-1},*)},*)$ correspond
to multiplication by $p_i$ on $\Gamma_k$ (the $p_i$-th power map on $\Gamma_0$).
Since $(Y,*)^{(X,*)}$ has the homotopy type of a CW complex, it follows by
\thmref{u_M-L_general} that there exists a number $i$ such that for all $j\gqs i$ and for all $k\gqs 1$,
the image of the morphism $\pi_k((Y,*)^{(L_j,*)},*)\to\pi_k((Y,*)^{(L_0,*)},*)$ equals that of
$\pi_k((Y,*)^{(L_{i},*)},*)\to\pi_k((Y,*)^{(L_0,*)},*)$. Set $b=p_1\dots p_i$. We have shown
\begin{equation*} \tag{$*$} \forall\,j>i\,\,\,\forall\,k\gqs 1:\,\,\,p_{i+1}\dots p_j(b\Gamma_k)=b\Gamma_k. \end{equation*}
Denote $D_k=b\Gamma_{k}$. Property ($*$) shows that $D_k$ is $P$-divisible, hence $D_k{}_{(P)}$ is divisible
and consequently \begin{equation*} \tag{$**$} 0\to D_k{}_{(P)}\to\Gamma_k{}_{(P)}\to\Gamma_k\otimes\ZZ_{b}\to 0 \end{equation*}
is a split exact sequence. In particular, $r(\Gamma_k{}_{(P)})\cong\Gamma_k\otimes\ZZ_{b}$.

For $k=0$ let $D_0$ be the divisible part of $\Gamma_0$. We may also record this in a short (split) exact sequence
$0\to D_0\to\Gamma_0\to R_0\to 0$.

The sequence of fundamental groups $\pi_1((Y,*)^{(L_i,*)},*)$ satisfies the Mittag-Leffler property, and
by \propref{homotopy_groups_of_inverse_limits} the set of path components $[X,Y]_*$ is the limit of 
\begin{equation*} \tag{$\star$} \dots\to[L_i,Y]_*\to[L_{i-1},Y]_*\to\dots. \end{equation*}
Hence $[X,Y]_*$ has a group structure with which the restrictions $[X,Y]_*\to[L_i,Y]_*$ are homomorphisms.
Thus \thmref{u_M-L_general} implies that for some $N>0$, the restriction induced morphisms
\begin{equation*} \tag{\dag} \pi_k((Y,*)^{(X,*)},*)\to\pi_k((Y,*)^{(L_i,*)},*) \end{equation*}
are injective for all $k$ and all $i\gqs N$.

Let $k\gqs 0$. Let for any group $G$ the symbol $\lim_{p_i}G$ denote the limit group of the sequence
$\dots\to G\xrightarrow{p_2}G\xrightarrow{p_1}G$. Functoriality of $\lim$ guarantees commutativity of
\begin{equation*}\begin{diagram}
\node{\lim_{p_i}D_k} \arrow{e} \arrow{s} \node{\lim_{p_i}\Gamma_k} \arrow{s} \\
\node{D_k} \arrow{e} \node{\Gamma_k}
\end{diagram}\end{equation*}
Here the vertical arrows denote canonical projections. Left-exactness of $\lim$ implies that the morphism
$\lim_{p_i}D_k\to\lim_{p_i}\Gamma_k$ is injective (actually bijective for $k\gqs 1$). If $D_k$ contains a
$\ZZ_{p^\infty}$ summand for some prime $p\in P$, then the kernel of any canonical projection
$\lim_{p_i}D_k\to D_k$ is uncountable. In particular, the kernel is nontrivial. Hence the same holds for
$\lim_{p_i}\Gamma_k\to\Gamma_k$, contradicting injectivity of (\dag).

Thus for all $k\gqs 0$, $D_k$ does not contain $\ZZ_{p^\infty}$-summands for $p\in P$.
For $k\gqs 1$ it follows that $D_k{}_{(P)}$ is torsion-free, and since it is divisible, it is a rational group.
For a reduced group $R$, the divisible part $d(R_{(P)})$ of localization $R_{(P)}$ is always rational, hence
since $D_0=d(\Gamma_0)$ does not contain $P$-torsion summands,
$d(\Gamma_0{}_{(P)})\cong D_0{}_{(P)}\oplus d(R_0{}_{(P)})$ is rational.
This proves {\bf a.}

If $Y$ is of finite type over $\ZZ_{(R)}$ then the groups $\Gamma_k$ are finitely generated as modules over $\ZZ_{(R)}$,
and $\Gamma_k{}_{(P)}$ finitely generated as modules over $\ZZ_{(P\cap R)}$. Since $P\cap R$ is nonempty the rational group
$D_k{}_{(P)}$ must be trivial and $\Gamma_k{}_{(P)}$ a finite abelian group bounded by $b$.
We form a sequence \[ i=\varphi(0)<\varphi(1)<\varphi(2)<\dots \]
such that for each $j$, the product $p_{\varphi(j-1)+1}p_{\varphi(j-1)+2}\dots p_{\varphi(j)}$ is a multiple of $b$.
Then the restriction $(Y_{(P)},*)^{(L_{\varphi(j)},*)}\to(Y_{(P)},*)^{(L_{\varphi(j-1)},*)}$ induces trivial morphisms
on homotopy groups $\Gamma_k{}_{(P)}$ for $k\gqs 1$ and consequently all corresponding homotopy groups of the limit space
$(Y_{(P)},*)^{(X,*)}$ are trivial. We have to show that $(Y_{(P)},*)^{(X,*)}$ is also path-connected.

As discussed above, $[X,Y_{(P)}]_*$ is the limit of ($\star$), i.e. of the sequence
\[ \dots\to[K,Y_{(P)}]_*\xrightarrow{\xi\mapsto\xi^{p_2}}[K,Y_{(P)}]_*\xrightarrow{\xi\mapsto\xi^{p_1}}[K,Y_{(P)}]_*. \]
The limit vanishes by \lemref{pc}. Hence $(Y_{(P)},*)^{(X,*)}$ is contractible. \corref{contractibility}
yields the rest of {\bf b.} while \thmref{inverse_sequence} gives {\bf c.}
\end{proof}

\begin{proof}[Proof of \propref{f_l_spheres}]
Statements \navedi{1} and \navedi{3} follow from \thmref{finite_localization} in conjunction with
\propref{rational_reduction} and \corref{cor_rational_reduction} by taking $K=S^m$.

In order to prove \navedi{2} we note that by
the generalized Serre's theorem (see McGibbon and Neisendorfer \cite{mcgibbon_neisen})
the group $\pi_k(Y)$ contains a subgroup of order $p$ for infinitely many $k$, and each prime $p$.
\end{proof}

%%%%%%%%%%%%%%%%%%%%%%%%%%%%%%%%%%%%%%%%%%%%%%%%%%%%%%%%%%%%%%%%%%%%%%%%%%%%%%%%%%%%%%%%%%%%%%%%%%%
%%%%%%%%%%%%%%%%%%%%%%%%%%%%%%%%%%%%%%%%%%%%%%%%%%%%%%%%%%%%%%%%%%%%%%%%%%%%%%%%%%%%%%%%%%%%%%%%%%%
%%%												%%%
%%%				Maps to Eilenberg-MacLane spaces				%%%
%%%												%%%
%%%%%%%%%%%%%%%%%%%%%%%%%%%%%%%%%%%%%%%%%%%%%%%%%%%%%%%%%%%%%%%%%%%%%%%%%%%%%%%%%%%%%%%%%%%%%%%%%%%
%%%%%%%%%%%%%%%%%%%%%%%%%%%%%%%%%%%%%%%%%%%%%%%%%%%%%%%%%%%%%%%%%%%%%%%%%%%%%%%%%%%%%%%%%%%%%%%%%%%

\chapter{Maps to Eilenberg-MacLane spaces} \label{eilenberg-maclane_target}

\setcounter{thm}{0}

%%%%%%%%%%%%%%%%%%%%%%%%%%%%%%%%%%%%%%%%%%%%%%%%%%%%%%%%%%%%%%%%%%%%%%%%%
%									%
%		Maps to Eilenberg-MacLane spaces			%
%									%
%%%%%%%%%%%%%%%%%%%%%%%%%%%%%%%%%%%%%%%%%%%%%%%%%%%%%%%%%%%%%%%%%%%%%%%%%

The purpose of this chapter is to investigate the space of pointed maps from an arbitrary CW complex
$X$ to an Eilenberg-MacLane space $K(G,n)$, where $G$ is an abelian group. 

\begin{lem}[Minimal decomposition]\label{minimal_decomposition}
Given a simply connected CW complex X and a specific free presentation of each of its homology groups
$H_n(X)\cong\scal{S_n;\Sigma_n}$ ($n\gqs 2$), there is a CW complex $Z$ and a homotopy equivalence
$f\colon Z\to X$ such that each cell of $Z$ is either
\begin{abc}
\item	a `generator' $n$-cell $e^n_\alpha$ which is a cycle in cellular homology mapped by $f$ to
	a cellular cycle representing the specified generator $\alpha\in S_n$; or
\item	a `relator' $(n+1)$-cell $e^{n+1}_\beta$ with cellular boundary corresponding to the specified relator
	$\beta\in\Sigma_{n+1}$.
\end{abc}
In addition, we may assume that all cells are attached along based maps of spheres, and that
\begin{itemize}
\item[($\bullet$)] The closure of an $n$-cell $e^n$ meets an $(n-1)$-cell $e^{n-1}$ if and only
	if the incidence number $[e^n:e^{n-1}]$ is non-zero. These are necessarily only generator
	$(n-1)$-cells.

	In particular, each generator $n$-cell $e^n_\alpha$ is attached along a (based) map
	$\varphi^n_\alpha\colon S^{n-1}\to Z^{(n-2)}$ of the $(n-1)$-sphere into the $(n-2)$-skeleton.
\end{itemize}
\end{lem}

\begin{proof}
Theorem 4C.1 of Hatcher \cite{hatcher} states {\bf a.} and {\bf b.} for $X$ of finite type.
The finite type restriction is unnecessary, and the generalization to {\bf b.'} is easy. (Compare
also Rutter \cite{rutter}, Lemma 2.1.)
\end{proof}

\begin{defn}
A CW decomposition with the properties of $Z$ as in \lemref{minimal_decomposition} is called {\it minimal}.

To a minimal decomposition of a complex $X$ we may associate a homology filtration
\[ \setof{*}=X_1\lqs X_2\lqs X_2\lqs X_3\lqs\dots \] by letting $X_i$ be the union of the $i$-skeleton $X^{(i)}$ of $X$
and all the relator $(i+1)$-cells. Then the inclusion induced morphism $H_j(X_i)\to H_j(X)$ is bijective for $j\lqs i$,
and $H_j(X_i)=0$ for $j>i$. Moreover, if we set $H_i=H_i(X)$, then the quotient $X_i/X_{i-1}$ is a Moore complex
$M(H_i,i)$, and there exists a map $\varphi\colon M(H_i,i-1)\to X_{i-1}$ which induces a homotopy equivalence
$C_\varphi\to X_i$. Here $C_\varphi$ denotes the homotopy cofibre of $\varphi$.
\end{defn}

%%%%%%%%%%%%%%%%%%%%%%%%%%%%%%%%%%%%%%%%%%%%%%%%%%%%%%%%%%%%%%%%%%%%%%%%%
%									%
%	Homotopy type of $K(G,n)^X$ depends only upon homology of $X$	%
%									%
%%%%%%%%%%%%%%%%%%%%%%%%%%%%%%%%%%%%%%%%%%%%%%%%%%%%%%%%%%%%%%%%%%%%%%%%%

\section{Homotopy type of $K(G,n)^X$ depends only upon homology of $X$}\label{thom_section}

Let $Y$ be a CW complex of type $K(G,n)$ for $G$ abelian and let $X$ be any CW complex.
Then $Y$ and consequently $(Y,*)^{(X,*)}$ are topological abelian monoids (if $G$ is uncountable
then in the category of compactly generated spaces; see also Appendix \ref{quasi_groups}).
As has been observed by Thom \cite{thom}, it follows that all Postnikov invariants of
(a CW approximation of) $(Y,*)^{(X,*)}$ vanish, and $(Y,*)^{(X,*)}$ is a {\it weak} product
of Eilenberg-MacLane spaces. Thus the weak homotopy type of $(Y,*)^{(X,*)}$ only depends upon
homology groups of $X$. We show that the actual homotopy type of $(Y,*)^{(X,*)}$ only depends
upon homology groups of $X$.

\begin{thm}\label{thom_enhanced}
Let $X$ be a connected CW complex and let $Y$ be a CW complex of type $K(G,n)$ where $G$ is abelian.
Then $(Y,*)^{(X,*)}$ has the homotopy type of the product \[ (Y,*)^{(M_1,*)}\times\dots\times(Y,*)^{(M_n,*)}, \]
where $M_i=M(H_iX,i)$. In particular, $(Y,*)^{(X,*)}$ has the homotopy type of a CW complex if and only if
the spaces $(Y,*)^{(M_i,*)}$, for $i\lqs n$, have.
\end{thm}

\begin{rem_n}
Since $Y$ is homotopy equivalent to $\Omega K(G,n+1)$, the function space $(Y,*)^{(X,*)}$ is homotopy equivalent
to $(\Omega K(G,n+1),*)^{(X,*)}\approx(K(G,n+1),*)^{(SX,*)}$, and it suffices to consider both $X$ and $Y$ simply
connected. In view of this we define $M(A,1)$ to be a CW complex $M$ with $H_1(M)\cong A$ and $H_k(M)\cong 0$ for
$k\gqs 2$.

Note that by \propref{quasi_base_point} we may arrange for a pointed homotopy equivalence in \thmref{thom_enhanced},
by choosing an appropriate space $Y$ of type $K(G,n)$.
\end{rem_n}

The rest of this section will be devoted to proving \thmref{thom_enhanced}.

\begin{lem}\label{homology_sees_it}
Let $Y$ be an Eilenberg-MacLane space $K(G,n)$ and let $X$ be a simply connected CW complex with homology
filtration $X_2\lqs X_3\lqs\dots$ associated to a minimal decomposition. Then $(Y,*)^{(X,*)}\to(Y,*)^{(X_n,*)}$
is a homotopy equivalence.
\end{lem}
\begin{proof}
By \propref{postnikov_adjunction} we know that $(Y,*)^{(X,*)}\to(Y,*)^{(X^{(n+1)},*)}$ is a homotopy equivalence,
and that \begin{equation*} \tag{$*$} (Y,*)^{(X^{(n+1)},*)}\to(Y,*)^{(X_n,*)} \end{equation*}
is a homotopy equivalence onto image. It suffices to prove that ($*$) is surjective. By assumption $X^{(n+1)}$
is obtained from $X_n$ by attaching generator $(n+1)$-cells, that is $X^{(n+1)}$ is the cofibre of a map
\[ \varphi\colon\bigvee_\lambda S^n\to X^{(n-1)}\hookrightarrow X_n. \]
Then the restriction ($*$) is the homotopy fibre of the induced map
$\varphi^\#\colon(Y,*)^{(X_n,*)}\to(Y,*)^{(\vee_\lambda S^n,*)}$ which factors through $(Y,*)^{(X^{(n-1)},*)}$.
Since $Y$ is a $K(G,n)$, on the group of path components $\varphi^\#$ transforms as
\[ \H^n(X_n;G)\to\H^n(X^{(n-1)};G)\to\H^n(\vee_\lambda S^n;G). \]
But $\H^n(X^{(n-1)};G)$ is trivial hence the homotopy fibre ($*$) is surjective.
\end{proof}

\begin{proof}[Proof of \thmref{thom_enhanced}]
Thus we may assume a minimal decomposition for $X$ and take the associated homology filtration $X_2\lqs X_3\lqs\dots$.
By \lemref{homology_sees_it} we know that the restriction $(Y,*)^{(X,*)}\to(Y,*)^{(X_n,*)}$ is a homotopy equivalence,
and we may replace $X$ by its $n$-th homology stage $X_n$. For convenience concerning connectedness, we may view $Y$
as the triple loop space of the $H$-group $K(G,n+3)$. Thus it suffices to show that $(K(G,n+3),*)^{(X_n,*)}$ has the
homotopy type of $\prod_{i=2}^n(K(G,n+3),*)^{(M_i,*)}$. By renumbering, it suffices to prove \lemref{principal_lemma}
below.
\end{proof}

\begin{lem}\label{principal_lemma}
Let $Y=K(G,n)$ and $3\lqs i\lqs n-3$. The fibration $(Y,*)^{(X_i,*)}\to(Y,*)^{(X_{i-1},*)}$ is trivial.
\end{lem}

The following step is crucial.

\begin{lem}\label{generator}
Let $(L',L)$ be the adjunction of an $i$-cell $e$ attached along a based map
\[ \varphi\colon(S^{i-1},*)\to(L^{(i-2)},x_0) \]
of the $(i-1)$-sphere to the $(i-2)$-skeleton of $L$. Denote by $K$ the smallest subcomplex
of $L$ that contains the image of $\varphi$. Note that $K$ is finite and contained in $L^{(i-2)}$. Set $K'=K\cup e$.

Then the induced map $\varphi^\#\colon(Y,*)^{(K,*)}\to(Y,*)^{(S^{i-1},*)}$ is nullhomotopic as a pointed map.
Hence the restriction fibration $(Y,*)^{(K',*)}\to(Y,*)^{(K,*)}$ is fibre-homotopy trivial and consequently so is
$(Y,*)^{(L',*)}\to(Y,*)^{(L,*)}$. If $\Phi$ denotes the fibre of $(Y,*)^{(B^i,*)}\to(Y,*)^{(S^{i-1},*)}$ (over the
constant), the canonical fibre-homotopy equivalence $(Y,*)^{(L,*)}\times\Phi\to(Y,*)^{(L',*)}$ restricts to
$(Y,*)^{(K,*)}\times\Phi\to(Y,*)^{(K',*)}$. In particular, the canonical sections give rise to a commutative diagram
as follows.
\begin{equation*}\begin{diagram}
\node{(Y,*)^{(L',*)}} \arrow{e} \node{(Y,*)^{(K\cup e,*)}} \\
\node{(Y,*)^{(L,*)}} \arrow{e} \arrow{n,l}{s} \node{(Y,*)^{(K,*)}} \arrow{n,r}{s_K}
\end{diagram}\end{equation*}
In other words, if $f,g\colon L\to Y$ are functions with $f\vert_K=g\vert_K$,
the extensions $s(f)$ and $s(g)$ restrict to the same function on $K\cup e$.
\end{lem}

\begin{rem_n}
We use the term `canonical section' according to definition following \lemref{fib1}.
\end{rem_n}

\begin{proof}
\lemref{principal_fibration} yields the following pullback.
\begin{equation*}\begin{diagram}
\node{(Y,*)^{(L',*)}} \arrow{s} \arrow{e} \node{(Y,*)^{(B^i,*)}} \arrow{s} \\
\node{(Y,*)^{(L,*)}} \arrow{e,t}{\varphi^\#} \node{(Y,*)^{(S^{i-1},*)}}
\end{diagram}\end{equation*}
This is to say that the fibration $(Y,*)^{(L',*)}\to(Y,*)^{(L,*)}$ is the principal fibration
obtained as the homotopy fibre of $\varphi^\#$. The latter factors as
\[ (Y,*)^{(L,*)}\to(Y,*)^{(K,*)}\xrightarrow{\varphi^\#}(Y,*)^{(S^{i-1},*)}. \]
Since $K$ is finite, $(Y,*)^{(K,*)}$ is globally well-pointed (see \prevlemref{cofibration}) 
and has the homotopy type of a CW complex (hence has the pointed homotopy type of a CW complex). Thus we compute
\[ [(Y,*)^{(K,*)},(Y,*)^{(S^{i-1},*)}]_*\cong\H^{n-i+1}\big((Y,*)^{(K,*)};G\big). \]
But $\pi_k\big((Y,*)^{(K,*)},\const\big)\cong\H^{n-k}(K;G)$ is trivial for $n-k\gqs (i-2)+1$ since $K$ is $(i-2)$-dimensional.
By the Hurewicz theorem the group $[(Y,*)^{(K,*)},(Y,*)^{(S^{i-1},*)}]_*$ is trivial, therefore $\varphi^\#$ is
nullhomotopic as a pointed map. Let $h\colon(Y,*)^{(K,*)}\times I\to(Y,*)^{(S^{i-1},*)}$ denote a homotopy between
the constant map and $\varphi^\#$.

The canonical section $(Y,*)^{(K,*)}\to(Y,*)^{(K',*)}$ is homotopic as a pointed map to the section
\[ s_K\colon(Y,*)^{(K,*)}\to(Y,*)^{(K',*)},\,\,\,s_K(f)\vert_K=f,\,\,\,s_K(f)(\phi[\zeta,t])=h(f,t)(\zeta). \]
Here $[\_,\_]$ denotes the quotient map $S^{i-1}\times I\to S^{i-1}\times I/S^{i-1}\times 0=B^i$, while
$\phi\colon B^i\to K\cup e$ denotes the characteristic map.

The pre-composition of $h$ with restriction $(Y,*)^{(L,*)}\to(Y,*)^{(K,*)}$ gives a pointed trivialization of
$\varphi^\#\colon(Y,*)^{(L,*)}\to(Y,*)^{(S^{i-1},*)}$ and the induced fibre homotopy equivalence
$(Y,*)^{(L,*)}\times\Phi\to(Y,*)^{(L',*)}$ clearly lifts $(Y,*)^{(K,*)}\times\Phi\to(Y,*)^{(K',*)}$.

In particular, the section $s\colon(Y,*)^{(L,*)}\to(Y,*)^{(L',*)}$, defined by
\[ s(f)\vert_L=f,\,\,\,s(f)\big(\phi[\zeta,t]\big)=h(f\vert_K,t)(\zeta) \]
(is homotopic to the canonical one and) lifts $s_K$.
\end{proof}

\begin{proof}[Proof of \lemref{principal_lemma}]
We construct $X_i$ from $X_{i-1}$ in two stages. First, we attach all the generator $i$-cells
via an attaching map
\[ \varphi=\vee_\lambda\varphi_\lambda\colon\bigvee_\lambda S^{i-1}\to X_{i-1}^{(i-2)}\hookrightarrow X_{i-1} \]
to obtain $X^{(i)}$. By \lemref{principal_fibration} the map $\varphi$ induces the pullback diagram below.
\begin{equation*}\begin{diagram}
\node{(Y,*)^{(X^{(i)},*)}} \arrow{s} \arrow{e} \node{(Y,*)^{(\vee_{\lambda}B^i,*)}} \arrow{s} 
	\arrow{e,t}{=} \node{\prod_\lambda(Y,*)^{(B^i,*)}} \arrow{s} \\
\node{(Y,*)^{(X_{i-1},*)}} \arrow{e,t}{\varphi^\#} \node{(Y,*)^{(\vee_{\lambda}S^{i-1},*)}} \arrow{e,t}{=}
	\node{\prod_\lambda(Y,*)^{(S^{i-1},*)}}
\end{diagram}\end{equation*}
The induced map $\varphi^\#$ maps into the product $\prod_\lambda(Y,*)^{(S^{i-1},*)}$. The composition with the
$\lambda$-th projection $\pr_\lambda\circ\varphi^\#$ equals $\varphi_\lambda^\#$ which is nullhomotopic as a pointed
map by \lemref{generator}. Hence so is $\varphi^\#$, by a product homotopy. Let $K_\lambda$ denote the smallest
subcomplex of $X$ that contains $\im\varphi_\lambda$ and fix homotopies
$h_\lambda\colon(Y,*)^{(K_\lambda,*)}\times I\to(Y,*)^{(S^{i-1},*)}$
between $*$ and $\varphi_\lambda^\#$. Then the `canonical section'
 $s\colon(Y,*)^{(X_{i-1},*)}\to(Y,*)^{(X^{(i)},*)}$ has the
form \[ s(f)\vert_{X_{i-1}}=f,\,\,\,s(f)\big(\phi_\lambda[\zeta,t]\big)=h_\lambda(f\vert_{K_\lambda},t)(\zeta). \]

Now we attach the relator $(i+1)$-cells to $X^{(i)}$ via an attaching map
\[ \psi=\vee_\mu\psi_\mu\colon\bigvee_\mu S^{i}\to X^{(i)}. \]
We obtain a pullback diagram similar to the one above
\begin{equation*}\begin{diagram}
\node{(Y,*)^{(X_{i},*)}} \arrow{s} \arrow{e} \node{(Y,*)^{(\bigvee_{\mu}B^{i+1},*)}} \arrow{s} \\
\node{(Y,*)^{(X^{(i)},*)}} \arrow{e,t}{\psi^\#} \node{(Y,*)^{(\bigvee_{\mu}S^{i},*)}}
\end{diagram}\end{equation*}
We claim that we can lift $s$ to a map (necessarily a section)
$s'\colon(Y,*)^{(X_{i-1},*)}\to(Y,*)^{(X_{i},*)}$. To prove this we observe projections of the composite
\[ (Y,*)^{(X_{i-1},*)}\xrightarrow{s}(Y,*)^{(X^{(i)},*)}\xrightarrow{\psi^\#}(Y,*)^{(\bigvee_{\mu}S^{i},*)}
	=\prod_\mu(Y,*)^{(S^i,*)} \]
onto the factors. Let $e_\mu$ be an $(i+1)$-relator cell with attaching map $\psi_\mu\colon S^i\to X^{(i)}$,
and let $L$ denote the smallest subcomplex containing the image of $\psi_\mu$. Note that $L$ contains only
{\it generator} $i$-cells; by property {\bf b.'} of \lemref{minimal_decomposition}. (On $L\cap X^{(i-1)}$ there are no
particular restrictions.) We number the $i$-cells
$e_1,\dots,e_r$ and let $K$ denote $L\setminus\setof{e_1,\dots,e_r}$. Then $K$ is a finite subcomplex contained in
$X^{(i-1)}\subset X_{i-1}$. By construction, the following diagram is commutative.
\begin{equation*}\begin{diagram}
\node{(Y,*)^{(X^{(i)},*)}} \arrow{e} \node{(Y,*)^{(L,*)}} \arrow{e,t}{\psi_\mu^\#} \node{(Y,*)^{(S^{i},*)}} \\
\node{(Y,*)^{(X_{i-1},*)}} \arrow{n,l}{s} \arrow{e} \node{(Y,*)^{(K,*)}} \arrow{n,r}{s_K}
\end{diagram}\end{equation*}
We compute \[ [(Y,*)^{(K,*)},(Y,*)^{(S^{i},*)}]_*\cong\H^{n-i}\big((Y,*)^{(K,*)};G\big). \]
But $\pi_k\big((Y,*)^{(K,*)},\const\big)\cong\H^{n-k}(K;G)$ is trivial for $n-k\gqs(i-1)+1$, i.e. $k\lqs n-i$,
since $K$ is $(i-1)$-dimensional. By the Hurewicz theorem, the group $[(Y,*)^{(K,*)},(Y,*)^{(S^{i},*)}]_*$ is trivial,
and therefore the composite $\psi_\mu^\#\circ s_K$ is nullhomotopic (as a pointed map). Hence so is the
pre-composite with restriction $(Y,*)^{(X_{i-1},*)}\to(Y,*)^{(K,*)}$. By commutativity, this equals
$\psi_\mu^\#\circ s\colon(Y,*)^{(X_{i-1},*)}\to(Y,*)^{(S^i,*)}$. Hence also the map into the product
$\psi^\#\circ s$ is nullhomotopic. Therefore $s$ may be lifted to $\sigma\colon(Y,*)^{(X_{i-1},*)}\to(Y,*)^{(X_i,*)}$, and
$\sigma$ is automatically a section. But there exists a (pointed cellular) map $k\colon M(H_iX,i-1)\to X_{i-1}$ such that
the inclusion $X_{i-1}\to X_i$ extends to a homotopy equivalence $F\colon C_k\to X_i$ with $C_k$ the mapping cone of $k$.
\lemref{principal_fibration} gives the following commutative diagram where the square is a pullback and
$F^\#$ is a fibre homotopy equivalence.
\begin{equation*}\begin{diagram}
\node{(Y,*)^{(X_i,*)}} \arrow{e,t}{F^\#} \arrow{se} \node{(Y,*)^{(C_k,*)}} \arrow{e}
	\arrow{s} \node{(Y,*)^{(CM(H_iX,i-1),*)}} \arrow{s} \\
\node[2]{(Y,*)^{(X_{i-1},*)}} \arrow{e,t}{k^\#} \node{(Y,*)^{(M(H_iX,i-1),*)}}
\end{diagram}\end{equation*}
The section of $(Y,*)^{(X_i,*)}\to(Y,*)^{(X_{i-1},*)}$ trivially defines a section of
$(Y,*)^{(C_k,*)}\to(Y,*)^{(X_{i-1},*)}$. Since $(Y,*)^{(CM(H_iX,i-1),*)}$ is contractible,
the existence of a section implies that $k^\#$ is nullhomotopic. Hence $(Y,*)^{(C_k,*)}\to(Y,*)^{(X_{i-1},*)}$
is fibre homotopy equivalent to a product fibration. Since $F^\#$ is a fibre homotopy equivalence also
$(Y,*)^{(X_i,*)}\to(Y,*)^{(X_{i-1},*)}$ is fibre homotopy equivalent to a product fibration.
\end{proof}

%%%%%%%%%%%%%%%%%%%%%%%%%%%%%%%%%%%%%%%%%%%%%%%%%%%%%%%%%%%%%%%%%%%%%%%%%%%%%%%%
%
%		Spaces of maps $M(A,m)\to K(G,n)$
%
%%%%%%%%%%%%%%%%%%%%%%%%%%%%%%%%%%%%%%%%%%%%%%%%%%%%%%%%%%%%%%%%%%%%%%%%%%%%%%%%

\section{Spaces of maps $M(A,m)\to K(G,n)$}\label{maps_moore_eilenberg}

The purpose of this section is to strengthen \thmref{double_wicked} in the particular case
of function spaces $K(G,n)^{M(A,m)}$, see \corref{cor_delooping} below. We apply the latter
in the following section where we give necessary and sufficient conditions for $K(G,n)^{M(A,m)}$
to have CW homotopy type.

The following delooping proposition is an immediate consequence of \thmref{v}.
\begin{prop}
Let $X=M(A,m)$ and $Y=K(G,n)$ where $m\lqs n-2$. If $(Y,*)^{(X,*)}$ has CW homotopy type, then
so also has $(Y',*)^{(X,*)}$ where $Y'=K(G,n+1)$. \qed
\end{prop}

The case $m=n-1$ can, by a different method, be obtained in a weaker form.

\begin{prop}\label{delooping}
Let $X=M(A,m)$ and $Y=K(G,n)$ where $m\lqs n-1$. Assume that $(Y,*)^{(X,*)}$ is contractible.
Then so also is $(Y',*)^{(X,*)}$ where $Y'=K(G,n+1)$.
\end{prop}
\begin{proof}
We assume that $X$ has a minimal decomposition. Let $K_\infty$ be a countable subcomplex of $X$, and let
$K_1\lqs K_2\lqs\dots$ be a filtration of finite subcomplexes for $K_\infty$. By \thmref{double_wicked}
and \propref{strong_obstruction_general} there exists a sequence of finite subcomplexes $L_0\lqs L_1\lqs L_2\lqs\dots$
such that $L_i\gqs K_i$ for all $i$ and, moreover
\begin{enroman}
\item	the inclusion $H_*(L_i)\to H_*(X)$ is injective on the image of $H_*(L_{i-1})\to H_*(L_i)$.
\item	The restriction $\Omega\big((Y,*)^{(L_i,*)},*\big)\to\Omega\big((Y,*)^{(L_{i-1},*)},*\big)$
	is nullhomotopic.
\end{enroman}

The subcomplexes $L_i$ are Moore spaces of type $M(H_mL_i,m)$. In particular,
\[ 	\pi_{n-m}\big((Y,*)^{(L_i,*)},*\big)\to\pi_{n-m}\big((Y,*)^{(L_{i-1},*)},*\big)		\]
transforms naturally as
\begin{equation*}\tag{\dag}	\Hom(H_mL_i,G)\to\Hom(H_mL_{i-1},G)		\end{equation*}
which is trivial by property \navedi{2} since $n-m\gqs 1$.
For each $i$ let $B_i$ denote the image of $H_mL_i\to H_mX=A$. By property \navedi{1}
actually $B_{i-1}\lqs H_mL_i$, for all $i\gqs 1$. For $i\gqs 1$ the induced morphism
$\Hom(B_{i+1},G)\to\Hom(B_{i-1},G)$ factors as
\[	\Hom(B_{i+1},G)\to\Hom(H_mL_{i+1},G)\to\Hom(H_mL_i,G)\to\Hom(B_{i-1},G)	\]
and is therefore trivial.

Set $L_\infty=\cup_iL_i$. By \navedi{1} (see \lemref{homology_P}) the morphism $H_*L_\infty\to H_*X$ is injective.
Denote $H_mL_\infty=B_\infty$. The group $B_\infty$ is the union of groups $B_i$. We may inductively form an expanding
sequence of finite CW complexes
\[ 	M(B_1,m)\hookrightarrow M(B_2,m)\hookrightarrow M(B_3,m)\hookrightarrow\dots		\]
the union of which is a CW complex of type $M(B_\infty,m)$. Set $M(B_i,m)=M_i$, $1\lqs i\lqs\infty$.
Since $B_i\lqs A$ for all $i$, the exact sequence $\Ext(A,G)\twoheadrightarrow\Ext(B_i,G)\to 0$ shows
that $\Ext(B_i,G)=0$. Therefore the space $(Y,*)^{(M_i,*)}$ has the type of $K\big(\Hom(B_i,G),n-m\big)$,
and $(Y',*)^{(M_i,*)}$ has the type of $K\big(\Hom(B_i,G),n-m+1\big)$.
Then since the restriction $(Y',*)^{(M_{i+1},*)}\to (Y',*)^{(M_{i-1},*)}$ induces the trivial morphism
$\Hom(B_{i+1},G)\to\Hom(B_{i-1},G)$ on the single non-zero homotopy group, it is nullhomotopic.
By \propref{contractible_limit}, the space $(Y',*)^{(M_\infty,*)}$ is contractible. Since $L_\infty\simeq M_\infty$, also
$(Y',*)^{(L_\infty,*)}$ is contractible.

This shows that every countable subcomplex $K_\infty$ of $X$ is contained in a bigger countable subcomplex $L_\infty$
such that $(Y',*)^{(L_\infty,*)}$ is contractible. Hence by \corref{cor_wicked} also $(Y',*)^{(X,*)}$ is contractible
which completes the proof.
\end{proof}

\begin{cor}\label{cor_delooping}
Let $X=M(A,m)$, $Y=K(G,n)$ where $m\lqs n-1$. If $(Y,*)^{(X,*)}$ has CW homotopy type, there exists
a countable subcomplex $L$ (containing a given countable subcomplex) of $X$ such that the restriction
$(Y,*)^{(X,*)}\to(Y,*)^{(L,*)}$ is a homotopy equivalence. In addition, we may assume $L=M(A_L,m)$ with $A_L\to A$
injective.
\end{cor}
\begin{proof}
Since $(Y,*)^{(X,*)}\simeq(K(G,n+i),*)^{(M(A,m+i),*)}$ for all $i\gqs 0$ we may assume $m\gqs 3$ with no loss
of generality.

By \thmref{double_wicked} there exists a countable subcomplex $L$ (containing a given countable subcomplex)
of $X$ such that the homology of $L$ injects
into that of $X$ and the space $(Y,*)^{(X/L,*)}$ is contractible. Since $B=H_mL\to H_mX=A$ is injective, $X/L$
is a Moore complex of type $M(A/B,m)$. By \propref{delooping} also $(Y',*)^{(X/L,*)}$ is contractible where $Y'=K(G,n+1)$.

Since all spaces in question are simply connected, there exists a pointed cellular map $\varphi\colon M(A/B,m-1)\to L$
such that the mapping cone of $\varphi$ is homotopy equivalent to $X$. By \lemref{principal_fibration} the fibration
$(Y,*)^{(X,*)}\to(Y,*)^{(L,*)}$ is (fibre homotopically equivalent to) the principal fibration obtained by
taking the homotopy fibre of $\varphi^\#\colon(Y,*)^{(L,*)}\to(Y,*)^{(M(A/B,m-1),*)}$.
Since $(Y,*)^{(M(A/B,m-1),*)}\simeq(Y',*)^{(M(A/B,m),*)}$ is contractible, $(Y,*)^{(X,*)}\to(Y,*)^{(L,*)}$
is a homotopy equivalence.
\end{proof}

Case $m=n$ is particular.

\begin{lem}\label{lem_m=n}
Let $W$ be an $(n-1)$-connected CW complex. Let $C=H_n(W)$ and assume that $\Hom(C,G)=0$.
Then $(K(G,n),*)^{(W,*)}$ is contractible.
\end{lem}
\begin{proof}
By \lemref{homology_sees_it} we may assume that $W=M(C,n)$. Set $Y=K(G,n)$.
Note that $Z=(Y,*)^{(W,*)}$ is path-connected. Let $K_\infty$ be a countable subcomplex of $W$,
and let $K_1\lqs K_2\lqs\dots$ be a filtration of finite complexes for $K_\infty$. For any subcomplex $L$ of $W$
denote by $Z_L$ the image of $Z=(Y,*)^{(W,*)}\to(Y,*)^{(L,*)}$. Note that $Z_L$ is the path component of the
constant map in $(Y,*)^{(L,*)}$. Then $\setof{Z_{K_i}}$ is an inverse sequence of contractible spaces. By \navedi{1} of
\lemref{restricted_trick} its inverse limit is $Z_{K_\infty}$ and is contractible.
\end{proof}

\begin{prop}\label{m=n}
Let $X=M(A,n)$ and $Y=K(G,n)$. Then $(Y,*)^{(X,*)}$ has CW homotopy type if and only if there exists a finitely
generated subgroup $B$ of $A$ such that $\Hom(A/B,G)=0$. In this case, if the decomposition of $X$ is minimal,
there exists a finite subcomplex $L$ of $X$ such that $(Y,*)^{(X,*)}\to(Y,*)^{(L,*)}$ is a homotopy equivalence
onto image.
\end{prop}

\begin{proof}
Assume that $X$ has a minimal decomposition. If $(Y,*)^{(X,*)}$ has CW type, then by \thmref{u_M-L_general} and
\propref{group_components} there exists a finite subcomplex $L$ of $X$ such that the morphisms
$\pi_k((Y,*)^{(X,*)},*)\to\pi_k((Y,*)^{(L,*)},*)$ are injective for all $k\gqs 0$. All higher homotopy
groups vanish, so automatically $(Y,*)^{(X,*)}\to(Y,*)^{(L,*)}$ is an equivalence onto image. On $\pi_0$ the
condition says that $\H^n(X;G)\to\H(L;G)$ is injective. Since $L$ is a complex of type $M(\Lambda,n)$, the
morphism $\Hom(A,G)\to\Hom(\Lambda,G)$ is injective. Let $B$ denote the image of $\Lambda\to A$. Then
also $\Hom(A,G)\to\Hom(B,G)$ is injective which shows $\Hom(A/B,G)=0$.

Conversely, if $\Hom(A/B,G)=0$ for a finitely generated subgroup $B$ of $A$, then $(Y,*)^{(M(A/B,n),*)}$
is contractible by \lemref{lem_m=n}, hence $(Y,*)^{(X,*)}$ has CW homotopy type by \lemref{subgroup}.
\end{proof}

\begin{cor}\label{cor_m=n}
Let $X=M(A,n)$, $Y=K(G,n)$. Assume that $A$ admits a finitely generated subgroup $B$ such that $\Hom(A/B,G)=0$.
Assume $X$ is in a minimal decomposition and let $K_\infty$ be a countable subcomplex of $X$. Then there exists
a countable subcomplex $L_\infty$ of $X$ such that $L_\infty\gqs K_\infty$, the morphism
$H_*(L_\infty)\to H_*(X)$ is injective and the restriction fibration $(Y,*)^{(X,*)}\to(Y,*)^{(L_\infty,*)}$ is
a homotopy equivalence onto image.
\end{cor}
\begin{proof}
There exists a finite subcomplex $L_0=M(A_0,n)$ of $X$ such that the image $B_0$ of $H_nL_0\to H_nX=A$ contains $B$.
By \lemref{homology_P} and \lemref{P-sequence} we may construct an ascending chain $L_1\lqs L_2\lqs\dots$ such that
$L_0\lqs L_1$, $K_\infty\lqs L_\infty$ and $H_*(L_\infty)\to H_*(X)$ is injective. Note that for all $i$ the image
$B_i$ of $H_nL_i\to H_nX$ contains $B$. Consequently $\Hom(A/B_i,G)=0$ for all $i$ and $(Y,*)^{(X,*)}\to(Y,*)^{(L_i,*)}$
is a homotopy equivalence onto image which we denote by $\Gamma_i$. Let $\Gamma_\infty$ denote the image of
$(Y,*)^{(X,*)}\to(Y,*)^{(L_\infty,*)}$. Evidently then $\Gamma_i$ is the image of $\Gamma_\infty$
under $(Y,*)^{(L_\infty,*)}\to(Y,*)^{(L_i,*)}$. Since for each $i$, restriction $\Gamma_i\to\Gamma_{i-1}$
is a homotopy equivalence, $\Gamma_\infty\to\Gamma_i$ is a homotopy equivalence for all $i$, by \lemref{restricted_trick}.
\end{proof}

%%%%%%%%%%%%%%%%%%%%%%%%%%%%%%%%%%%%%%%%%%%%%%%%%%%%%%%%%%%%%%%%%%%%%%%%%%%%%%%%%%%%%%%%%%%%
%	EXPLICIT DETERMINATIONS
%%%%%%%%%%%%%%%%%%%%%%%%%%%%%%%%%%%%%%%%%%%%%%%%%%%%%%%%%%%%%%%%%%%%%%%%%%%%%%%%%%%%%%%%%%%%

\section{Explicit determinations}\label{explicit_determinations}

Here we give necessary and sufficient conditions on abelian groups $A$ and $G$ for CW homotopy type of
$K(G,n)^{M(A,m)}$ and that of $K(G,n)^{K(A,m)}$ for $m\gqs 2$.

\begin{thm}\label{moore_eilenberg}
Let $A$ and $G$ be abelian groups. Let $T(G)$ denote the torsion subgroup of $G$, and $F$ the quotient $G/T(G)$.
Define \[ P=\setof{p\in\PP\,\vert\,G*\ZZ_p\neq 0},\,\,\,S=\setof{p\in\PP\,\Big\vert\,\frac{F\otimes\QQ}{F}_{(p)}\neq 0}. \]
The function space $K(G,n)^{M(A,m)}$ has CW homotopy type if and only if
\begin{enroman}
\item	$m>n$, or
\item	$m=n$, and there exists a finitely generated subgroup $B$ of $A$ such that $\Hom(A/B,G)=0$, or
\item	$m<n$, and there exists a short exact sequence $0\to B\to A\to A'\to 0$ such that $B$ is finitely generated,
	$A'$ is $P$-divisible, the divisible $T(A')_{(P)}$ has finite rank, and, in addition, either
	\begin{itemize}
		\item	$G$ is a $P$-bounded group, or
		\item	$A'$ is a torsion group such that for $R=\setof{p\in P\,\big\vert\,A'_{(p)}\neq 0}$
		\begin{enumerate}
			\item	$T(G)_{(R)}$ is bounded and $F_{(R)}$ is rational (if and only if \linebreak
				$R\cap S=\emptyset$),
			\item	$A'_{(S)}$ is a finite group.
		\end{enumerate}
	\end{itemize}
\end{enroman}
\end{thm}

(See page \pageref{divisible_reduced_bounded} for definitions concerning divisible and reduced abelian groups.)

An abelian group is called {\it cocyclic} if it is isomorphic with a finite cyclic group $\ZZ_{p^k}$
for some prime $p$ and $k\gqs 1$ or with a quasicyclic group $\ZZ_{p^\infty}$.

\begin{cor}\label{c_moore_eilenberg}
Let $A$ be a torsion, and $G$ an arbitrary abelian group. Let $P$ and $S$ be defined as above, and let $m<n$.
The space $K(G,n)^{M(A,m)}$ has CW homotopy type if and only there exists $R\subset P$ such that
$T(G)_{(R)}$ is bounded, $F_{(R)}$ is rational, $A_{(R)}$ is isomorphic with the direct sum of finitely many
cocyclic groups, and $A_{(P\cup S\setminus R)}$ is a finite group.\qed
\end{cor}

\begin{thm}\label{eilenberg_eilenberg}
Let $A$ and $G$ be abelian groups, and $m,n\gqs 2$. Then the function space $K(G,n)^{K(A,m)}$
has CW type if and only if $K(G,n)^{M(A,m)}$ has, and consequently if and only if the conditions of
\thmref{moore_eilenberg} hold.
\end{thm}

In subsequent sections we will frequently use some well known results on abelian groups,
and for convenience we summarize a few in \propref{abelian_groups}.

A subgroup $S$ of $A$ is {\it pure} in $A$ if $nS=S\cap nA$ for every integer $n$.

\begin{prop}\label{abelian_groups}
\begin{enumerate}
\item	(Pr\"{u}fer and Baer.) A bounded abelian group splits into the direct sum of finite cyclic groups
	(bounded by a common bound).
\item	(Kulikov.) A bounded pure subgroup is a direct summand.
	In particular if the $P$-torsion part of a group $A$ is bounded then
	it is a direct summand.
\item	(Kulikov.) If an abelian group contains elements of finite order, then
	it contains a cocyclic direct summand. In particular, if
	a reduced torsion abelian group is unbounded, it contains
	finite cyclic direct summands of arbitrarily large orders.
\end{enumerate}
\end{prop}
\begin{proof}
See Theorem 17.2, Theorem 27.5, and Corollary 27.3 of Fuchs \cite{fuchs1}.
\end{proof}

We are now ready to pursue the proof of \thmref{moore_eilenberg}.

\begin{prop}\label{f_l_kgn}
Let $G$ be an abelian group and let $Y$ denote the Eilenberg-MacLane space $K(G,n)$.
Decompose $G$ as $G\cong D\oplus R$ where $D$ is divisible and $R$ a reduced group.
Further let $X$ be an $(r-1)$-connected rational CW complex where $r\gqs 2$. Set $Q_k=H_k(X)$. Assume $Q_r\neq 0$.
The space $(Y,*)^{(X,*)}$ has CW homotopy type if and only if one of the following requirements is met.
\begin{enroman}
\item	$r>n$.
\item	$r=n$, $D$ is torsion free, and $\dim_\QQ(Q_r)$ is finite if $D$ is nontrivial.
\item	$r<n$, the group $D$ is torsion free, $R$ is bounded, and for $r\lqs k\lqs n$, the
	dimension $\dim_\QQ(Q_k)$ is finite if $D$ is nontrivial.
\end{enroman}
Moreover $(Y,*)^{(X,*)}$ is always contractible if $r>n$ (case \navedi{1}) or $D$ is trivial.
\end{prop}

\begin{proof}
By \thmref{thom_enhanced} it suffices to consider $X=M(Q,m)$ where $Q$ is a rational group.
More precisely, $Q$ is isomorphic to the direct sum $\oplus_{\Lambda}\QQ$ for some indexing set $\Lambda$.
Then $X=\bigvee_{\Lambda}M(\QQ,m)=\bigvee_{\Lambda}S^m_{(0)}$, and
$(Y,*)^{(X,*)}\approx\prod_{\Lambda}(Y,*)^{(S^m_{(0)},*)}$. Thus (see \exref{product}) the space $(Y,*)^{(X,*)}$
has CW type if and only if $(Y,*)^{(S^m_{(0)},*)}$ has CW type and is either contractible or else $\Lambda$ is finite.

If $m>n$, then $S^m_{(0)}$ may be assumed to have trivial $n$-skeleton, hence $(Y,*)^{(S^m_{(0)},*)}$ is contractible
by \propref{postnikov_adjunction}. 

Let $m\lqs n$.
Set $Y_D=K(D,n)$ and $Y_R=K(R,n)$. Then $Y\simeq Y_D\times Y_R$ and $(Y,*)^{(S^m_{(0)},*)}\simeq(Y_D,*)^{(S^m_{(0)},*)}
\times(Y_R,*)^{(S^m_{(0)},*)}$. If $D$ is torsion free then by \thmref{localization} the space $(Y_D,*)^{(S^m_{(0)},*)}$
is equivalent to $(Y_D,*)^{(S^m,*)}$, and has nontrivial homotopy group $\pi_{n-m}$ if and only if $D$ is nontrivial.

By {\bf a.} of \thmref{finite_localization} and the above remark concerning finiteness of $\Lambda$ it already follows
that for $m\lqs n$ the stated conditions are necessary. 

Consider the case $m=n$.
Since there are no morphisms from a divisible to a reduced group, the space $(Y_R,*)^{(S^n_{(0)},*)}$ is contractible
by \lemref{lem_m=n}.

If $m<n$ and $R$ is bounded, then the H-spaces $Y_R$ and consequently also $(Y_R,*)^{(S^m,*)}=\Omega^m(Y_R,*)$ admit
H-space exponents. Thus $(Y_R,*)^{(S^m_{(0)},*)}$ is contractible by {\bf c.} of \thmref{finite_localization}.
\end{proof}

\begin{prop}\label{fundamental_special_case}
Let $T$ be a divisible torsion group, and write $T\cong\oplus_{p\in P}\oplus_{\lambda\in\Lambda_p}\ZZ_{p^\infty}$
where $\Lambda_p$ is nonempty for each $p\in P$. Set $X=M(T,m)$ and $Y=K(G,n)$. Then $(Y,*)^{(X,*)}$ has CW type
if and only if one of the following is true.
\begin{enroman}
\item	$m>n$,
\item	$m=n$, and $d(G_{(P)})$ is rational,
\item	$m<n$, $d(G_{(P)})$ is rational, $r(G_{(P)})$ is bounded, and for
	\[ R=\setof{p\in P\,\vert\,r(G_{(P)})_{(p)}\neq 0} \]
	the set $\cup_{p\in R}\Lambda_p$ is finite.
\end{enroman}
\end{prop}

\begin{proof}
By \navedi{2} of \corref{razcep} the space $(Y,*)^{(X,*)}$ is equivalent to $(Y_{(P)},*)^{(X,*)}$.

Set $T'=\oplus_{p\in P}\ZZ_{p^\infty}$ and $X'=M(T',m)$. Then $X'$ is a wedge summand of $X$, and by
\lemref{wedge_splitting} the space $(Y_{(P)},*)^{(X,*)}$ dominates $(Y_{(P)},*)^{(X',*)}$. Since $T'$
is the cokernel of the monomorphism $\ZZ\to\ZZ_{(\PP\setminus P)}$, the map
$(Y_{(P)},*)^{(S^m_{(\PP\setminus P)},*)}\to(Y_{(P)},*)^{(S^m,*)}$ is a principal fibration
(see \lemref{principal_fibration}) with fibres either empty or equivalent to $(Y_{(P)},*)^{(X',*)}$. Hence
the latter has CW type if and only if $(Y_{(P)},*)^{(S^m_{(\PP\setminus P)},*)}$ has. By \thmref{localization},
the space $(Y_{(P)},*)^{(S^m_{(\PP\setminus P)},*)}$ is homotopy equivalent to $(Y_{(P)},*)^{(S^m_{(0)},*)}$.

Thus the necessity of the stated conditions follows immediately from \propref{f_l_kgn},
with the exception of the condition on the set $\cup_{p\in R}\Lambda_p$.

Now assume that $m<n$, $d(G_{(P)})$ is rational, and $r(G_{(P)})$ is bounded.
Since $X$ splits as $X\simeq\bigvee_{p\in P}\bigvee_{\lambda\in\Lambda_p}M(\ZZ_{p^\infty},m)$, $(Y_{(P)},*)^{(X,*)}$
factors as \[ (Y_{(P)},*)^{(X,*)}\simeq\prod_{p\in P}\prod_{\lambda\in\Lambda_p}(Y_{(P)},*)^{(M(\ZZ_{p^\infty},m),*)}. \]
Note that \[ \pi_{n-m-1}\big((Y_{(P)},*)^{(M(\ZZ_{p^\infty},m),*)},*\big)\cong\Ext(\ZZ_{p^\infty},r(G_{(P)}))\cong r(G_{(P)})_{(p)}, \]
since $r(G_{(P)})$ is bounded. A space of CW type only admits factorizations into products of finitely many non-contractible
factors (see \exref{product}), and $\cup_{p\in R}\Lambda_p$ must be a finite set.

Conversely, if $\cup_{p\in R}\Lambda_p$ is finite, then $(Y_{(P)},*)^{(X,*)}$ is homotopy equivalent to
a product of finitely many factors each of which is dominated by $(Y_{(P)},*)^{(X',*)}$, and
by sufficiency part of \propref{f_l_kgn} the stated conditions evidently suffice here as well.
\end{proof}

\begin{example}\label{out_of_milnor}
Let $X=M(\ZZ_{p^\infty},m)$ and $Y=K(\ZZ_p,n)$. By \propref{fundamental_special_case} the space $(Y,*)^{(X,*)}$ has the type
of a CW complex (of type $K(\ZZ_p,n-m-1)$). Contrary to examples provided in Kahn \cite{kahn} and in \cite{smrekar}, this function
space of CW type is not covered by Milnor's theorem in the sense that if $X$ is any CW complex of type $M(\ZZ_{p^\infty},m)$
and $L$ is a finite subcomplex of $X$, then $(Y,*)^{(X,*)}\to (Y,*)^{(L,*)}$ is not a homotopy equivalence. \qed
\end{example}

\begin{lem}\label{subgroup}
Let $B$ be a finitely generated subgroup of $A$, and let $Y$ be an arbitrary CW complex.
Then $(Y,*)^{(M(A,m),*)}$ has CW homotopy type if and only if $(Y,*)^{(M(A/B,m),*)}$ has
if $m\gqs 3$.
\end{lem}

\begin{rem_n}
If $Y$ is an infinite loop space, for example an Eilenberg-MacLane space, then the restriction
on $m$ is unnecessary, since then $(Y,*)^{(M(A,m),*)}$ is homotopy equivalent to $(Y',*)^{(M(A,m+2),*)}$,
and $(Y,*)^{(M(A/B,m),*)}$ is homotopy equivalent to $(Y',*)^{(M(A/B,m+2),*)}$.
\end{rem_n}

\begin{proof}
There is a cellular map $M(B,m-1)\to M(A,m-1)$ which realizes the inclusion $B\lqs A$ in homology,
since $m-1\gqs 2$. Expanding this to the Puppe sequence yields a subcomplex inclusion $M(B,m)\lqs M(A,m)$
which is the cofibre of a map $\varphi\colon M(A/B,m-1)\to M(B,m)$.

Thus by \lemref{principal_fibration} the map $(Y,*)^{(M(A,m),*)}\to(Y,*)^{(M(B,m),*)}$ is a principal fibration
with fibres either empty or of the type of $(Y,*)^{(M(A/B,m),*)}$, and the assertion follows by Stasheff's theorem.
\end{proof}

\begin{lem}\label{lem_torsion_free}
Let $\Phi$ be a nontrivial countable torsion-free abelian group, and let $G$ be an arbitrary abelian group.
Assume $m<n$. If $(K(G,n),*)^{(M(\Phi,m),*)}$ is contractible then $G$ is bounded by a number $b$, and $\Phi$
is $b$-divisible.
\end{lem}

\begin{proof}
The group $\Phi$ may be represented as the union group of an expanding sequence of subgroups
\begin{equation*} \tag{\dag} 0=\Phi_0\lqs\Phi_1\lqs\Phi_2\lqs\Phi_3\lqs\dots \end{equation*}
where for each $i$, the group $\Phi_i$ is a free group of finite rank $r_i$, and $r_{i-1}\lqs r_i\lqs i$,
for all $i\gqs 1$. Inductively we can form a sequence of inclusions
\[ M(\Phi_1,m)\lqs M(\Phi_2,m)\lqs M(\Phi_3,m)\lqs\dots \]
with union space a CW complex homotopy equivalent to $M(\Phi,m)$.
Set $L_i=M(\Phi_i,m)$ and $Y=K(G,n)$. Note that $(Y,*)^{(L_i,*)}$ is a space of type $K(G^{r_i},n-m)$.

By \propref{delooping} and \thmref{inverse_sequence} the space $(K(G,n),*)^{(M(\Phi,m),*)}$ is contractible
if and only if there exists a sequence $i_1<i_2<i_3<\dots$ such that for each $j$, the restriction
\begin{equation*} \tag{$*$} (Y,*)^{(L_{i_j},*)}\to(Y,*)^{(L_{i_{j-1}},*)} \end{equation*}
is nullhomotopic. 

We replace the original sequence $\setof{L_i\,\vert\,i}$ by the subsequence $\setof{L_{i_j}\,\vert\,j}$.

We fix isomorphisms $\ZZ^{r_i}\cong\Phi_i$ and rewrite (\dag) as
\begin{equation*} \tag{\ddag} \ZZ^{r_1}\xrightarrow{A_2}\ZZ^{r_2}\xrightarrow{A_3}\ZZ^{r_3}\xrightarrow{A_4}\dots \end{equation*}
where the $A_i$ are matrices with integral entries.
The map $(Y,*)^{(L_i,*)}\to(Y,*)^{(L_{i-1},*)}$ is nullhomotopic if and only if the $\pi_{n-m}$-induced morphism
\begin{equation*} \tag{$**$}
\Hom(\ZZ^{r_i},G)\cong G^{r_i}\xrightarrow{A_i^T} G^{r_{i-1}}\cong\Hom(\ZZ^{r_{i-1}},G) \end{equation*}
is trivial. Here $A_i^T$ denotes the transpose of $A_i$, that is, the morphism
\[ (g_1,\dots,g_{r_i})\mapsto(h_1,\dots,h_{r_{i-1}}),\,\,\,h_k=\sum_{l=1}^{r_i}(A_i)_{l,k}g_l. \]
If this is the trivial morphism, then $(A_i)_{l,k}g=0$, for all $l$, all $k$, and all $g$.

Since this is valid for all $i$, either $A_i=0$ for all $i$ which is a contradiction,
or else $G$ is bounded by a number $b$. Supposing $b$ minimal it follows that $A_i=b\cdot A_i'$ for all $i$.
Thus $\Phi$ is $b$-divisible.
\end{proof}

\begin{prop}\label{torsion_free}
Let $A$ be a nontrivial torsion free abelian group, and let $G$ be an arbitrary abelian group.
Assume $m<n$. Then $(K(G,n),*)^{(M(A,m),*)}$ is contractible if and only if $G$ is bounded by
a number $b$, and $A$ is $b$-divisible.
\end{prop}

\begin{proof}
Set $Y=K(G,n)$.
By \corref{cor_delooping} there exists a countable subcomplex $L$ of $M(A,m)$ with 
$H_*(L)\to H_*(M(A,m))$ injective, so that $(Y,*)^{(M(A,m),*)}\to(Y,*)^{(L,*)}$
is a homotopy equivalence. Then $L$ is of type $M(\Phi,m)$ with $\Phi\lqs A$, and by \lemref{lem_torsion_free}
$G$ is bounded by a number $b$. Assume a minimal $b$.

By \propref{delooping} and \thmref{double_wicked} the space $(Y,*)^{(M(A,m),*)}$ is contractible
if and only if for each countable subgroup $\Phi$ of $A$ there exists a bigger countable subgroup $\Phi'$
of $A$ such that $(Y,*)^{(M(\Phi',m),*)}$ is contractible. By the lemma and minimality of $b$, the group
$\Phi'$ is $b$-divisible. But this implies that $A$ is $P$-divisible for $P$ the set of prime divisors of $b$.

Conversely, assume that $A$ is $P$-divisible and $G$ is $P$-bounded.
Then \[ (Y,*)^{(M(A_{(P)},m),*)}\to(Y,*)^{(M(A,m),*)} \] is a homotopy equivalence by \thmref{localization}.
Since $A$ is torsion-free, $A_{(P)}$ is rational, and \propref{f_l_kgn} shows that $(Y,*)^{(M(A_{(P)},m),*)}$ has CW type.
\end{proof}

\begin{lem}\label{torsion_away_from_G}
Assume the notation of the theorem, and let $T$ be a torsion group with trivial $(P\cup S)$-part.
Then $(K(G,n),*)^{(M(T,m),*)}$ is contractible.
\end{lem}

\begin{proof}
The rationalization $K(G,n)\to K(G_{(0)},n)$ may be viewed as a principal fibration with fibre $\Phi$
such that $\pi_n(\Phi)\cong T(G)$ and $\pi_{n-1}(\Phi)\cong (F\otimes\QQ)/F$. Thus $\Phi$ is a $P\cup S$-local
torsion space, so \thmref{localization} implies that $(\Phi,*)^{(M(T,m),*)}$ is homotopy equivalent to
$(\Phi,*)^{(M(T,m)_{(P\cup S)},*)}$. The latter is contractible since $M(T,m)_{(P\cup S)}$ is.
By \thmref{localization} also $(K(G_{(0)},n),*)^{(M(T,m),*)}$ is contractible. Hence
so is $(K(G,n),*)^{(M(T,m),*)}$, as claimed.
\end{proof}

\begin{proof}[Proof of \thmref{moore_eilenberg}]
Assume a minimal decomposition for $X=M(A,m)$, and set $Y=K(G,n)$. Case $m>n$ follows from \propref{postnikov_adjunction},
and case $m=n$ by \propref{m=n}, so we assume $m<n$. Let $Y^X$ have CW homotopy type.
By \corref{cohomology_obstruction} there exists a finite subcomplex $L$ of $X$ such that for every
$k\in\setof{0,\dots,n}$ the inclusion induced morphism 
\begin{equation*}\tag{$*$} \H^k(X;G)\to\H^k(L;G) \end{equation*} is injective.

Minimal decomposition of $X$ implies that $L$ is of type $M(A_L,m)$. Denote by $\varphi\colon A_L\to A$ the morphism
induced by inclusion $L\lqs X$, and set $B=\im\varphi$.

Injectivity of ($*$) says that the morphisms $\varphi^*\colon\Hom(A,G)\to\Hom(A_L,G)$
and $\varphi^*\colon\Ext(A,G)\to\Ext(A_L,G)$ are injective.
Consequently also
\begin{equation*} \Hom(A,G)\to\Hom(B,G)\text{ and }\Ext(A,G)\to\Ext(B,G) \end{equation*}
are injective.

Plugging this into the exact sequence
\begin{align*} 0\to &\Hom(A/B,G)\to\Hom(A,G)\to\Hom(B,G) \\
	\to &\Ext(A/B,G)\to\Ext(A,G)\to\Ext(B,G)\to 0 \end{align*}
it follows that $\Hom(A/B,G)$ is trivial and
\begin{equation*}
\Ext(A,G)\cong\Ext(B,G)\cong T(B)\otimes G\text{ is a bounded abelian group.}
\end{equation*}
Set $A'=A/B$. By \lemref{subgroup} the space $K(G,n)^{M(A',m)}$ also has CW type.
Reapplying the above with $A'$ in place of $A$ yields a finitely generated subgroup $B'$ of $A'$ with
\[ \Ext(A',G)\cong\Ext(B',G)\cong T(B')\otimes G\text{ a bounded abelian group.} \]
Set $X'=M(A',m)$. We proceed with investigating $Y^{X'}$ of CW type with $\Hom(A',G)$ trivial and $\Ext(A',G)$
a bounded abelian group.

Let $p\in P$. The epimorphism $A'\to A'\otimes\ZZ_p$ induces the monomorphism
$\Hom(A'\otimes\ZZ_p,G)\to\Hom(A',G)$. Furthermore the monomorphism $G*\ZZ_p\to G$
induces the monomorphism $\Hom(A'\otimes\ZZ_p,G*\ZZ_p)\to\Hom(A'\otimes\ZZ_p,G)$.
Since $\Hom(A',G)$ is trivial and $G*\ZZ_p$ is nontrivial, $A'\otimes\ZZ_p$ must be trivial.
In other words, $A'$ is $p$-divisible.

Consider the exact sequence
\begin{align*}\tag{$**$} 0\to &\Hom(A'/T(A'),H)\to\Hom(A',H)\to\Hom(T(A'),H) \\
	\to &\Ext(A'/T(A'),H)\to\Ext(A',H)\to\Ext(T(A'),H)\to 0 \end{align*}
Since $\Hom(A',G)=0$, substituting $H=G$ simplifies ($**$) to the exact sequence
\[ 0\to\Hom(T(A'),G)\xrightarrow{\delta}\Ext(A'/T(A'),G)\xrightarrow{\alpha}\Ext(A',G)\to\Ext(T(A'),G)\to 0. \]
It follows that $\Hom(T(A'),G)\cong\ker\alpha$, and that $\coker\delta\to\Ext(A',G)$ is injective.
However, since $A'/T(A')$ is torsion-free, $\Ext(A'/T(A'),G)$ is divisible, and hence so is its quotient group
$\coker\delta$. Since no divisible group is bounded, $\coker\delta=0$, and consequently
\begin{equation*} \Ext(A',G)\xrightarrow{\cong}\Ext(T(A'),G) \end{equation*}
and $\Hom(T(A'),G)\xrightarrow{\cong}\Ext(A'/T(A'),G)$ are isomorphisms. Since $T(A')$ is a torsion group,
$\Hom(T(A'),G)$ is a reduced group (see Theorem 46.1 of Fuchs \cite{fuchs1}). The only possibility is
\begin{equation*} \Hom(T(A'),G)=\Ext(A'/T(A'),G)=0. \end{equation*}
Applying $\Hom(T(A'),\_)$ to $0\to T(G)\to G\to F\to 0$ yields the equality
$\Hom(T(A'),G)\cong\Hom(T(A'),T(G))=0$ and the exact sequence
\begin{equation*} \tag{\dag} 
0\to\Ext(T(A'),T(G))\to\Ext(T(A'),G)\to\Ext(T(A),F)\to 0.
\end{equation*}
Since \[ \Hom(T(A'),G)\cong\Hom(T(A'),T(G))\cong\prod_{p\in P}\Hom(T(A')_{(p)},T(G)_{(p)}) \]
it follows immediately that for $p\in P$, either $T(A')_{(p)}$ is trivial or else $T(G)_{(p)}$ is reduced.

The sequence (\dag) shows that both $\Ext(T(A'),T(G))$ and $\Ext(T(A),F)$ are bounded.

Consider \[ \Ext(T(A'),T(G))\cong\prod_{p\in P}\Ext(T(A')_{(p)},T(G)_{(p)}), \]
and let $R=\setof{p\in P\,\vert\,T(A')_{(p)}\neq 0}$. Take $p\in R$. Then $T(G)_{(p)}$ is reduced.
If it is unbounded, then it admits cyclic direct summands of arbitrarily large orders, and consequently
so does $\Ext(T(A')_{(p)},T(G)_{(p)})$, contradiction. Furthermore, if $R$ is infinite, then
$\prod_{p\in P}\Ext(T(A')_{(p)},T(G)_{(p)})$ contains elements of arbitrarily high orders which
again contradicts boundedness of $\Ext(T(A'),T(G))$. Thus $R$ is finite, and we conclude that
\[	T(G)_{(R)}\text{ is a bounded group.}		\]
In the short exact sequence \[ 0\to T(G)_{(R)}\to G_{(R)}\to F_{(R)}\to 0 \] the group $T(G)_{(R)}$ is
a bounded pure subgroup of $G_{(R)}$, hence it is a direct summand. It follows that the divisible part
$d(G_{(R)})$ of $G_{(R)}$ is isomorphic to the divisible part of $F_{(R)}$, and as such is rational.
Furthermore,
\[ r(G_{(R)})\cong T(G)_{(R)}\oplus r(F_{(R)}). \]
Since $T(A')_{(R)}$ is divisible, it is a direct summand of $A'$; we decompose $A'\cong T(A')_{(R)}\oplus A_1$.
Then $M(A',m)\simeq M(T(A')_{(R)},m)\vee M(A_1,m)$, and by \lemref{wedge_splitting}, the space
$(Y,*)^{(M(T(A')_{(R)},m),*)}$ has CW homotopy type.
By \propref{fundamental_special_case} it follows that $r(F_{(R)})$ is trivial, and that
$T(A')_{(R)}$ only admits finitely many nontrivial direct summands.

Since $r(F_{(R)})$ is trivial, $F_{(R)}$ is a rational group. In particular, $F\otimes\QQ\cong F_{(R)}$,
and consequently $(F\otimes\QQ)/F$ is torsion away from $R$. In other words $R\cap S=\emptyset$.

Now we consider
\[ \Ext(T(A'),F)\cong\Hom(T(A'),\tfrac{F\otimes\QQ}{F})\cong\prod_{p\in S}\Hom(T(A')_{(p)},(F_{(R)}/F)_{(p)}). \]
(Note that from boundedness of $\Ext(T(A'),F)$ it also follows that $R\cap S=\emptyset$.)

Since for $p\in S$ the torsion group $(F_{(R)}/F)_{(p)}$ is divisible, and $\Ext(T(A'),F)$ is bounded,
$T(A')_{(p)}$ must be reduced, and bounded as well. Furthermore, $T(A')_{(p)}$ can be nontrivial for
at most finitely many $p\in S$. Hence $T(A')_{(S)}$ itself is bounded, and as such a direct summand
of $A'$. As above, it follows that $(Y,*)^{(M(T(A')_{(S)},m),*)}$ has CW type. As a bounded group, $T(A')_{(S)}$
splits into a direct sum of finite cyclic groups, say $T(A')_{(S)}\cong\oplus_{\lambda\in\Lambda}C_\lambda$.
Consequently $(Y,*)^{(M(T(A')_{(S)},m),*)}$ is homotopy equivalent to the product
\[ \prod_{\lambda\in\Lambda}(Y,*)^{(M(C_\lambda,m),*)}. \] Since all factors are noncontractible,
$\Lambda$ must be finite by \exref{product}. In particular, $T(A')_{(S)}$ is a finite group.

We decompose \[ T(A')\cong T(A')_{(R)}\oplus T(A')_{(S)}\oplus T''. \]
Since $T''$ has trivial $(P\cup S)$-torsion, the space $(Y,*)^{(M(T'',m),*)}$
is contractible by \lemref{torsion_away_from_G}. 
Hence $(Y,*)^{(M(T(A'),m),*)}$ has CW homotopy type by \lemref{wedge_splitting}. By the above
we already know that \[ (Y,*)^{(M(A',m),*)}\to(Y,*)^{(M(T(A'),m),*)} \]
is a weak homotopy equivalence, hence it is a genuine homotopy equivalence, and
the fibres are contractible. The fibre over the constant map is homotopy equivalent to
\[ (Y,*)^{(M(A'/T(A'),m),*)}. \]
If $A'/T(A')$ is nontrivial then by \propref{torsion_free}, $G$ is bounded. Thus $G=T(G)=T(G)_{(P)}$,
and the bound of $G$ is a number whose prime divisors equal $P$. This concludes the proof of necessity.

Conversely, if $G$ is $P$-bounded, $A'$ is $P$-divisible and $T(A')_{(P)}$ only admits finitely many
direct summands, then $K(G,n)^{(M(A',m),*)}$ which is equivalent to $K(G,n)^{(M(A'_{(P)},m),*)}$ by
\thmref{localization} has CW homotopy type by \propref{f_l_kgn} and \propref{fundamental_special_case}.

If $A'$ is a torsion group with $A'_{(P)}=A'_{(R)}$ divisible with finitely many summands, $A'_{(S)}$ finite,
$T(G)_{(R)}$ bounded and $F_{(R)}$ rational, then $K(G,n)^{M(A'_{(R)},m)}$ has CW type by
\propref{fundamental_special_case}, $K(G,n)^{M(A'_{(S)},m)}$ has CW type by Milnor's theorem,
and $K(G,n)^{M(A'_{(\PP\setminus(P\cup S))},m)}$ has CW type by \lemref{torsion_away_from_G}.
Then $K(G,n)^{M(A',m)}$ has CW type by \lemref{wedge_splitting}.

An application of \lemref{subgroup} concludes also the proof of sufficiency.
\end{proof}

Now we turn to the proof of \thmref{eilenberg_eilenberg}.

By \thmref{thom_enhanced} the necessity part is clear, and we have to prove sufficiency.

\begin{defn}
Let $\CCC$ be a class of abelian groups. Then $\CCC$ is a Serre class if
\begin{itemize}
\item	for a short exact sequence $0\to A'\to A\to A'{}'\to 0$ the group $A$ belongs to $\CCC$
	if and only if both $A'$ and $A'{}'$ do,
\item	if $A$ and $B$ belong to $\CCC$ then so do $A\otimes B$ and $A*B$, and
\item	for $A\in\CCC$ also $\H_i(K(A,1);\ZZ)\in\CCC$, for all $i$.
\end{itemize}
\end{defn}

\begin{lem}\label{wicked_serre}
Let $R$ be a (finite) set of primes, and let $Q$ be a set of primes (possibly empty) disjoint from $R$.
Let $\CCC$ be the class of torsion abelian groups $A$ such that
\begin{itemize}
	\item	$A_{(R)}$ is the direct sum of finitely many cocyclic groups, and
	\item	$A_{(Q)}$ is finite.
\end{itemize}
Then $\CCC$ is a Serre class with the additional property that if $A\in\CCC$ and $B$ is finitely generated,
then $A\otimes B$ and $A*B$ belong to $\CCC$.
\end{lem}

\begin{proof}
It suffices to show that the required properties hold for the class $\CCC_p$ of torsion abelian groups
that are isomorphic with a direct sum of finitely many cocyclic $p$-groups.

Let $D$ be a divisible torsion $p$-group of finite rank, i.e. $D$ is isomorphic to the direct sum of
finitely many quasicyclic groups $\ZZ_{p^\infty}$. It follows easily by induction on the rank of $D$ that
\begin{itemize}
\item	if $E$ is a quotient group of $D$ it is divisible with $\rank E\lqs\rank D$,
\item	if $C$ is a subgroup of $D$ then $C$ is the direct sum of cocyclic $p$-groups
	(and $\rank C\lqs\rank D$).
\end{itemize}
Consequently if $A\in\CCC_p$ and \begin{equation*}\tag{$*$} 0\to A'\to A\to A''\to 0\end{equation*} is exact,
then $A'\in\CCC_p$ and $A''\in\CCC_p$.

Conversely, assume that in ($*$) the groups $A'$ and $A''$ belong to $\CCC$. We may decompose $A'\cong D\oplus F$
where $F$ a finite $p$-group, and $D$ is a divisible $p$-group of finite rank. Since $D$ is injective, it also splits
$A$, say $A\cong D\oplus\Gamma$ where $\Gamma$ is an extension of $F$ by $A''$, i.e.
\[	0\to F\to\Gamma\to A''\to 0. \]
There exists a bound $p^m$ for $F$, and since $A''$ is also isomorphic to the direct sum of a finite $p$-group
and a divisible $p$-group of finite rank, $p^nA''$ is divisible for some $n$. Consequently also $p^{m+n}\Gamma$
is divisible. Considering a decomposition $\Gamma\cong d(\Gamma)\oplus r(\Gamma)$ with $d(\Gamma)$ divisible
and $r(\Gamma)$ reduced it follows that $r(\Gamma)$ is bounded, and hence isomorphic to the direct sum of cyclics.
Rank considerations imply that $\Gamma$ belongs to $\CCC_p$, and hence so does $A$.

For tensor and torsion products we note that if $T$ is any torsion group,
$\ZZ_{p^\infty}\otimes T=0$, and for any group $G$, the torsion product $\ZZ_{p^\infty}*G$ is isomorphic to $T(G)_{(p)}$.

Since $K(\ZZ_{p^\infty},1)$ is the colimit of lens spaces, we know that $\H_i\big(K(\ZZ_{p^\infty},1);\ZZ\big)$
is isomorphic to $\ZZ_{p^\infty}$ for $i$ odd and zero for $i$ even.

If $A\in\CCC_p$ then $\H_i(K(A,1);\ZZ)\in\CCC_p$ by the K\"{u}nneth theorem, combined with what we have shown above.
\end{proof}

\begin{lem}\label{eilenberg_homology}
Let $\CCC$ be as in \lemref{wicked_serre} and let $0\to B\to A\to A'\to 0$ be a an exact sequence with
$B$ finitely generated and $A'\in\CCC$. Then for each $i\gqs m$, the group $H_i(K(A,m);\ZZ)$ is given by
an extension \[ 0\to B_i\to H_i(K(A,m);\ZZ)\to A_i'\to 0 \] with $B_i$ finitely generated and $A_i'\in\CCC$.
\end{lem}

\begin{proof}
We consider the Serre homology spectral sequence of the fibration $K(B,m)\to K(A,m)\to K(A',m)$.
By \lemref{wicked_serre}, for the $E^2$-term \[ E^2_{p,q}=H_p\big(K(A',m);H_q(K(B,m))\big) \] we have
$E^2_{0,q}$ a finitely generated group and $E^2_{p,q}\in\CCC$ for $p\gqs 1$. Since finitely generated
groups and $\CCC$ are Serre classes, also $E^{\infty}_{0,q}$ is finitely generated, and $E^{\infty}_{p,q}\in\CCC$
for $p\gqs 1$. This implies that for each $i$, the group $H_i(K(A,m))$ has a filtration
$F_0\lqs F_1\lqs\dots\lqs F_n=H_i(K(A,m))$ with $F_0$ finitely generated, and $F_j/F_{j-1}\in\CCC$ for positive $j$.
Set $B_i=F_0$ and suppose that for some $j$, the group $F_{j-1}/B_i$ belongs to $\CCC$. Then
also $B_i\lqs F_j$, and $F_j/B_i$ is given by the extension
\[	0\to F_{j-1}/B_i\to F_j/B_i\to F_j/F_{j-1}\to 0.	\]
By inductive hypothesis, $F_{j-1}/B_i\in\CCC$, and since $F_j/F_{j-1}\in\CCC$, also $F_j/B_i\in\CCC$, and
our assertion follows.
\end{proof}

\begin{proof}[Proof of \thmref{eilenberg_eilenberg}]
By \thmref{thom_enhanced} the case $m\gqs n$ is already contained in \thmref{moore_eilenberg}. So
we assume that $m<n$.

Assume that $G$ is a $P$-bounded group (with $G*\ZZ_p\neq 0$ for all $p\in P$) and that $A$ is given by
an extension $0\to B\to A\to A'\to 0$ where $A'$ is $P$-divisible and $T(A')_{(P)}$ is the sum of finitely
many quasicyclic groups. Let $T'=T(A')$ and $\Phi=A'/T'$. The kernel $\Gamma$ of the composite $A\to A'\to\Phi$
is given by the extension $B\to\Gamma\to T'$.

Consider the fibration $K(\Gamma_{(P)},m)\to K(A_{(P)},m)\to K(\Phi_{(P)},m)$. Since $\Phi$ is $P$-divisible
and torsion-free, $\Phi_{(P)}$ is a rational group. Hence $K(\Phi_{(P)},m-1)$ is a rational space, and thus
$(K(G,n),*)^{(K(\Phi_{(P)},m-1),*)}\simeq(K(G,n+1),*)^{(SK(\Phi_{(P)},m-1),*)}$ is contractible by \propref{f_l_kgn}.
Since $K(\Phi_{(P)},m-1)$ is the homotopy fibre of $K(\Gamma_{(P)},m)\to K(A_{(P)},m)$, the map
\[ K(G,n)^{K(A_{(P)},m)}\to K(G,n)^{K(\Gamma_{(P)},m)} \]
is a homotopy equivalence by \propref{z_prop}. By \thmref{localization},
$K(G,n)^{K(A_{(P)},m)}$ is homotopy equivalent to $K(G,n)^{K(A,m)}$, so
it suffices to prove that the space $K(G,n)^{K(\Gamma_{(P)},m)}\simeq K(G,n)^{K(\Gamma,m)}$ has CW homotopy type.

By \lemref{eilenberg_homology}, for each $i$ the homology group $H_i=\H_i(K(\Gamma,m))$ is the extension
of a finitely generated group by a torsion group whose $P$-part is isomorphic with the direct sum of
finitely many cocyclic groups, say $0\to B_i\to H_i\to T_i\to 0$. Set $Y=K(G,n)$. Since $B_i$ is finitely generated,
$(Y,*)^{(M(H_i,i),*)}$ has CW type if and only if $(Y,*)^{(M(T_i,i),*)}$ has, by \lemref{subgroup}. We split
\[ T_i=T_1\oplus T_2\oplus T_3 \] where $T_1=T_i{}_{(\PP\setminus P)}$, $T_2$ is the direct sum of finitely many
quasicyclic groups and $T_3$ is a finite $P$-group. By \corref{wedge_splitting} the space $(Y,*)^{(M(T_i,i),*)}$ has CW
type if and only if the $(Y,*)^{(M(T_j,i),*)}$, for $1\lqs j\lqs 3$, have.
The space $(Y,*)^{(M(T_1,i),*)}$ is contractible by \propref{disjoint_local}, and $(Y,*)^{(M(T_3,i),*)}$ has CW
type by Milnor's theorem. Furthermore, $(Y,*)^{(M(T_2,i),*)}$ has CW type by \propref{fundamental_special_case}.

Thus $(K(G,n),*)^{(K(\Gamma,m),*)}$ has CW type by \thmref{thom_enhanced}, and consequently $K(G,n)^{K(A,m)}$
has CW type, as claimed.

Assume the notation of the statement of \thmref{moore_eilenberg}. 
Our second case is that $A'$ is torsion with $A'_{(R)}$ the sum of finitely many cocyclic groups,
$A'_{(S\cup P\setminus R)}$ is a finite group, $T(G)_{(R)}$ is bounded, and $F_{(R)}$ is rational.

By \lemref{eilenberg_homology} every $H_i=\H_i(K(A,m))$ is given by an exact sequence
$0\to B_i\to H_i\to A_i'\to 0$ where $B_i$ is finitely generated, and $A_i'$ is torsion such that
$A_i'{}_{(R)}$ is isomorphic with a direct sum of finitely many cocyclic groups, and $A_i'{}_{(S\cup P\setminus R)}$
is finite.

Since $B_i$ is finitely generated, $(Y,*)^{(M(H_i,i),*)}$ has CW homotopy type if and only if $(Y,*)^{(M(A_i',i),*)}$
has. Split \[ A_i'\cong A_1'\oplus A_2'\oplus A_3' \]
where $A_1'$ is a $\PP\setminus(S\cup P)$-group, $A_2'$ is isomorphic with a direct sum of finitely many
$R$-quasicyclic groups, and $A_3'$ is finite. Then $(Y,*)^{(M(A_1',i),*)}$ and $(Y,*)^{(M(A_2',i),*)}$ have
CW type by \corref{c_moore_eilenberg}, and $(Y,*)^{(M(A_3',i),*)}$ has CW type by Milnor's theorem.
Consequently $(Y,*)^{(M(A_i',i),*)}$ has CW type by \corref{wedge_splitting}, hence so has $(Y,*)^{(M(H_i,i),*)}$,
and hence so has $(Y,*)^{(K(A,m),*)}$, by \thmref{thom_enhanced}.
\end{proof}

%%%%%%%%%%%%%%%%%%%%%%%%%%%%%%%%%%%%%%%%%%%%%%%%%%%%%%%%%%%%%%%%%%%%%%%%%%%%%%%%%%%%%%%%
%		Spaces of maps into $K(G,1)$
%%%%%%%%%%%%%%%%%%%%%%%%%%%%%%%%%%%%%%%%%%%%%%%%%%%%%%%%%%%%%%%%%%%%%%%%%%%%%%%%%%%%%%%%

\section{Spaces of maps into $K(G,1)$}\label{aspherical_target}

\begin{prop}
Let $X$ be a connected CW complex with base point $x_0$. Pick a base point $y_0$ in $K(G,1)$.
Then $(K(G,1),y_0)^{(X,x_0)}$ has CW homotopy type if and only if for each $g\colon(X,x_0)\to(K(G,1),*)$
there exists a finite subcomplex $L$ of $X$ (containing $x_0$) such that the induced function
\begin{equation*}\tag{$*$}	\Hom(\pi_1(X,x_0),G)\to\Hom(\pi_1(L,x_0),G)	\end{equation*}
only sends the class $[g]$ to $[g\vert_L]$.
\end{prop}

\begin{proof}
The condition is necessary by \lemref{obstruction}.

For sufficiency, note that `injectivity' of ($*$) implies that the path component of $g$ in $(K(G,1),y_0)^{(X,x_0)}$
is open, and \lemref{aspherical_main} below concludes the proof.
\end{proof}

We note the easy
\begin{lem}\label{aspherical_easy}
Let $Y$ be an aspherical space of CW type with base point $y_0$ and let $X$ be a connected CW complex with base point $x_0$.
Then for any $g\colon(X,x_0)\to(Y,y_0)$ the path component of $g$ in $(Y,y_0)^{(X,x_0)}$ is weakly contractible. \qed
\end{lem}

\begin{lem}\label{aspherical_main}
Let $(X,x_0)$ be a based connected CW complex and let $Y=K(G,1)$. The space $(Y,y_0)^{(X,x_0)}$ has
contractible path components, for any $y_0\in Y$.
\end{lem}
\begin{proof}
With no loss of generality assume that $X$ has a single $0$-cell which is $x_0$. Let $g\colon(X,x_0)\to(Y,y_0)$ be a map,
and let $C$ denote the path component of $g$ in $(Y,y_0)^{(X,x_0)}$.
Let $K_\infty$ be any countable subcomplex of $X$ and let $K_1\lqs K_2\lqs\dots$ be a filtration of finite
subcomplexes for $K_\infty$. For $L\lqs X$ let $C_L$ denote the image of $C$ under $(Y,y_0)^{(X,x_0)}\to(Y,y_0)^{(L,x_0)}$.

By \lemref{aspherical_easy} and Milnor's theorem the spaces $C_{K_i}$ are contractible. By \navedi{1} of
\lemref{restricted_trick} so is $C_{K_\infty}=\lim_iC_{K_i}$. Then by \thmref{wicked} also $C$ is contractible,
as claimed.
\end{proof}

%%%%%%%%%%%%%%%%%%%%%%%%%%%%%%%%%%%%%%%%%%%%%%%%%%%%%%%%%%%%%%%%%%%%%%%%%%%%%%%%%%%%%%%%
%		MILNOR'S DONE IT ALL
%%%%%%%%%%%%%%%%%%%%%%%%%%%%%%%%%%%%%%%%%%%%%%%%%%%%%%%%%%%%%%%%%%%%%%%%%%%%%%%%%%%%%%%%

\section{CW complexes $X$ such that $Y^X$ has CW type for `all' $Y$}\label{for_all_Y}

We close this chapter by showing that Milnor's theorem cannot be improved in the sense
that if for a space $X$ the function space $Y^X$ has CW homotopy type for {\it all} CW
complexes $Y$, then $X$ is `almost finitely dominated'. We begin with a few preliminary results.

\begin{lem}\label{character_group}
Let $A$ be an arbitrary abelian group.
\begin{enroman}
\item	The torsion-free rank of $\Hom(A,\ZZ)$ is at least $d$ if and only if $A$ admits a direct summand
	which is a free group of rank $d$.
\item	If the torsion-free rank of $\Hom(A,\ZZ)$ is exactly $d$ then there exists a decomposition
	$A\cong\ZZ^d\oplus A'$ with $\Hom(A',\ZZ)=0$.\qed
\end{enroman}
\end{lem}
\begin{proof}
See Fuchs \cite{fuchs1}, \S 43, Exercise 10.
\end{proof}

\begin{prop}\label{target_kzn}
Let $X$ be a connected CW complex. Then $K(\ZZ,n)^X$ has CW homotopy type if and only
if the groups $H_i(X)$ are finitely generated for $1\lqs i\lqs n-1$, and there exists a decomposition
$H_n(X)\cong\ZZ^d\oplus H'$ with $\Hom(H',\ZZ)=0$.
\end{prop}

\begin{proof}
By \thmref{thom_enhanced} and \thmref{moore_eilenberg} combined with Milnor's theorem
the sufficiency part is clear, and so is the necessity part except for the splitting of $H_n(X)$.

If $K(\ZZ,n)^X$ has CW type, then by \thmref{moore_eilenberg} there exists a finitely generated
subgroup $B$ of $H_n(X)$ such that $\Hom(H_n(X),\ZZ)\to\Hom(B,\ZZ)$ is injective. Since $\Hom(B,\ZZ)$
is free of finite rank, the desired splitting follows from \lemref{character_group}.
\end{proof}

Recall that a CW complex $X$ is called quasifinite (see Mislin \cite{mislin2}) if $\oplus_nH_n(X;\ZZ)$
is a finitely generated abelian group.

\begin{thm}\label{milnor_forever}
\begin{abc}
\item	Let $X$ be a connected CW complex. Then the following are equivalent.
	\begin{enumerate}
		\item	For each $i\gqs 1$, the group $H_i(X)$ is finitely generated.
		\item	For each nilpotent CW complex $Y$ with finitely many nontrivial homotopy groups,
			the space $Y^X$ has CW homotopy type.
		\item	For each $n\gqs 2$ the path component $C$ of the constant map in $K(\ZZ,n)^X$ has CW
			homotopy type.
	\end{enumerate}
\item	Let $X$ be a connected CW complex and let $n\gqs 2$. Then the following are equivalent.
	\begin{enumerate}
		\item	For each $1\lqs i\lqs n$, the group $H_i(X)$ is finitely generated.
		\item	For each nilpotent CW complex $Y$ with $\pi_k(Y)=0$ for $k\gqs n+1$
			the space $Y^X$ has CW homotopy type.
		\item	For each abelian group $G$ the space $K(G,n)^X$ has CW
			homotopy type.
	\end{enumerate}
\item	Let $X$ be a nilpotent connected CW complex. Then the following are equivalent.
	\begin{enumerate}
		\item	The space $Y^X$ has CW homotopy type for every CW complex $Y$.
		\item	The path component $C$ of the constant map in $Y^X$ has CW homotopy type for every simply connected
			CW complex $Y$ of finite type.
		\item	$X$ is dominated by a finite CW complex.
	\end{enumerate}
\item	Let $X$ be any connected CW complex. Then the following are equivalent.
	\begin{enumerate}
		\item	For every CW complex $Y$ the path component $C$ of the constant map in $Y^X$
			is open and has CW homotopy type.
		\item	There exists a finite subcomplex $L$ of $X$ such that the quotient $X/L$
			is homotopy equivalent to a finite CW complex.
	\end{enumerate}
\item	Let $X$ be a connected CW complex and $n\gqs 2$. Then $X$ is homotopy equivalent to a complex with
	finite $n$-skeleton if and only if $\pi_1(X)$ is finitely presented and $Y^X$ has CW type for all
	CW complexes $Y$ with $\pi_k(Y)=0$ for $k\gqs n+1$.
\item	Let $X$ be a connected CW complex. Then $X$ is finitely dominated if and only if $\pi_1(X)$
	is finitely presented and $Y^X$ has CW type for all CW complexes $Y$.
\end{abc}
\end{thm}

\begin{proof}
\begin{abc}
\item % a.
Implication $\navedi{1}\implies\navedi{2}$ follows by \propref{finite_postnikov_decomposition},
\thmref{thom_enhanced} and Milnor's theorem.

Let $n\gqs 2$ and assume that the path component $C$ of the constant map in $(K(\ZZ,n),*)^{(X,*)}$
(if and only if the path component of the constant map in $K(\ZZ,n)^X$) has CW homotopy type.
By \corref{cohomology_obstruction} it follows that the cohomology groups $H^i(X;\ZZ)$ are finitely
generated for $i\lqs n-1$. Hence the homology groups $H_i(X)$ are finitely generated for $i\lqs n-2$.
If this holds for all $n\gqs 2$, it follows that the $H_i(X)$ are finitely generated for all $i$.
This shows $\navedi{3}\implies\navedi{1}$.

\item % b.
We need only show $\navedi{3}\implies\navedi{1}$. As above, by taking $G=\ZZ$ it follows
from \corref{cohomology_obstruction} that the $H_i(X)$ are finitely generated for $i\lqs n-1$.
Let $\BBB$ be the set of all finitely generated subgroups of $H_n(X)$, and define $G=\prod_{B\in\BBB}H_n(X)/B$.
By \corref{cohomology_obstruction} it follows that $H_n(X)=B$ for $B$ the image of $H_n(L)\to H_n(X)$ for some
finite subcomplex $L$ of $X$.

\item % c.
Assume that the path component $C$ of the constant map in $Y^X$ has CW homotopy type for every simply
connected CW complex $Y$ of finite type. By {\bf a.} it follows that $H_*(X)$ is of finite type. Let $Y$ be a simply
connected CW complex with $\pi_k(Y)\cong\ZZ$ for all $k\gqs 2$ and all $k$-invariants trivial. In other words, $Y$ is
weakly equivalent to the product $\prod_{k\gqs 2}K(\ZZ,k)$. By \lemref{obstruction} there exists a finite complex
$L$ such that for all $k\gqs 1$ the morphism \[ [S^kX,Y]_*\to[S^kL,Y]_* \] is injective. By choice of $Y$ this
morphism is the product morphism \[ \prod_{j\gqs k+1}\big(\H^{j-k}(X;\ZZ)\to\H^{j-k}(L;\ZZ)\big). \]
It follows that there exists a number $N$ with $H^i(X;\ZZ)=0$ for $i\gqs N$ and consequently $H_i(X)=0$ for $i\gqs N$.
Thus $X$ is quasifinite, and nilpotent by assumption. Hence it is finitely dominated
by Mislin \cite{mislin2}. This shows $\navedi{2}\implies\navedi{3}$.

\item % d.
Assume \navedi{1}. The proof of $\navedi{2}\implies\navedi{3}$ of {\bf c.} shows that $X$ is quasifinite.
Let $L$ be any finite subcomplex of $X$. Consider the fibration $Y^X\to Y^L$. Let $C$, respectively $C_L$,
denote the path component of the constant map in $Y^X$, respectively $Y^L$. Then $C\to C_L$ is a fibration,
and since $C$ has CW type, so has the fibre $F_L$ over the constant map. Note that $F_L=C\cap (Y,*)^{(X,L)}$ is
open in $(Y,*)^{(X,L)}$. The path component $C'$ of the constant map in $(Y,*)^{(X,L)}$ is a path component of $F_L$,
and since $F_L$ has CW type, $C'$ is open in $F_L$, and has CW type. Hence $C'$ is open in $(Y,*)^{(X,L)}$ which
is homeomorphic with $(Y,*)^{(X/L,*)}$.

By \lemref{nasty} there exists a finite subcomplex $K$ of $X$ such that the function
\[ [X,K(\pi_1(X/M,*),1)]_*\to[L,K(\pi_1(X/M,*),1)]_* \]
is injective for all finite subcomplexes $M$ and all finite subcomplexes $L$ containing $K$.
It follows that $\pi_1(X/L,*)$ is trivial for all finite $L$ containing $K$. Thus $X/L$ is simply connected,
and is quasifinite, hence is homotopy equivalent to a finite complex.

Conversely, assume that for some finite subcomplex $L$ of $X$ the quotient $X/L$ is homotopy equivalent
to a finite complex. Consider the fibration $Y^X\to Y^L$. By Milnor's theorem the fibre over the constant
map has CW homotopy type, hence by Stasheff's theorem the preimage $\Tilde C$ of the path component of the
constant map $C_L$ in $Y^L$ has CW type, and is open since $C_L$ is open in $Y^L$. Since $\Tilde C$ is
a union of path components of $Y^X$ which contains $C$, $C$ has CW type and is open in $\Tilde C$, and hence
in $Y^X$.

\item % e.
	Let $\pi$ be a group and let $M$ be a left $\pi$-module; that is we are given
	a morphism $h\colon\pi\to\Aut(M)$. Further let $k\gqs 2$. By Gitler \cite{gitler}
	there exists a (pointed) space $L_\pi(M,k)$ and a `fundamental class' $u\in H^k(L_\pi(M,k);h,M)$
	such that the assignment
	\begin{equation*}\tag{$*$} [X,L_\pi(M,k)]_\alpha\to H^k(X;h\alpha,M),\,\,\,[f]\mapsto f^*(u), \end{equation*}
	is a bijection. Here $[X,L_\pi(M,k)]_\alpha$ denotes those pointed classes $[f]$ for
	which $f_\#\colon\pi_1(X)\to\pi_1(L_\pi(M,k))=\pi$ realizes a given morphism $\alpha\colon\pi_1(X,x_0)\to\pi$.

	Let now $\pi=\pi_1(X,x_0)$. By our assumption and \lemref{nasty} there exists a finite subcomplex $K$ of $X$
	such that the restriction induced function \begin{equation*} \tag{\dag} [X,L_\pi(M,k)]_*\to[K,L_\pi(M,k)]_*
	\end{equation*} is an injection for all $k$ with $2\lqs k\lqs n$ and all $(h,M)$ belonging to a given set
	$\MMM$ of left $\pi$-modules.

	Denote the inclusion $i\colon K\to X$. Note that (\dag) maps the set $[X,L_\pi(M,k)]_1$
	to $[K,L_\pi(M,k)]_{i_\#}$. Hence the natural bijection ($*$) together with (\dag) show that
	\begin{equation*}\tag{\ddag}	H^k(X;h,M)\to H^k(K;hi_\#,M)		\end{equation*}
	is a monomorphism for $2\lqs k\lqs n$ and all $M\in\MMM$.

	Assume that $\MMM=\setof{(h_\lambda,M_\lambda)\,\vert\,\lambda\in\Lambda}$ is a system of left $\pi$-modules
	indexed by a directed set $\Lambda$ and let $(h_\infty,M_\infty)=\colim_\lambda(h_\lambda,M_\lambda)$.
	
	Consider the following commutative diagram.
	\begin{equation*}\begin{diagram}
		\node{\colim_\lambda H^k(X;h_\lambda,M_\lambda)} \arrow{e} \arrow{s} \node{H^k(X;h_\infty,M_\infty)}
			\arrow{s} \\
		\node{\colim_\lambda H^k(K;h_\lambda i_\#,M_\lambda)} \arrow{e} \node{H^k(K;h_\infty i_\#,M_\infty)}
	\end{diagram}\end{equation*}
	Since $K$ is a finite complex, the bottom horizontal arrow is an isomorphism by Theorem 1 of
	K.~Brown \cite{ken-brown}. Since $H^k(X;h_\lambda,M_\lambda)\to H^k(K;h_\lambda i_\#,M_\lambda)$
	is injective for all $\lambda$ so is the induced map of colimits. By commutativity also 
	\begin{equation*}
		\tag{$\star$}	\colim_\lambda H^k(X;h_\lambda,M_\lambda)\to H^k(X;h_\infty,M_\infty)
	\end{equation*}
	is injective. In particular if $M_\infty=0$, also $\colim_\lambda H^k(X;h_\lambda,M_\lambda)=0$,
	for $2\lqs k\lqs n$.

	Let $\Tilde X$ denote the universal cover of $X$ at $x_0$.
	If $\pi_1(X)$ is finitely generated, then we may assume $X$ to have a single $0$-cell $x_0$
	and finitely many $1$-cells. Thus $C_i(\Tilde X)$, for $i=0,1$, are finitely generated (free) $\ZZ\pi$-modules.
	By Theorem 2 of \cite{ken-brown} it follows that if $M_\infty=0$ then $\colim_\lambda H^k(X;h_\lambda,M_\lambda)=0$
	for $k=0,1$. Reapplying Theorem 2 of \cite{ken-brown} it follows that for any direct system
	$\setof{(h_\lambda,M_\lambda)\,\vert\,\lambda}$ with colimit $(h_\infty,M_\infty)$ the morphism ($\star$)
	is an isomorphism for $k<n$ and a monomorphism for $k=n$. Our claim follows by Wall \cite{wall}, Theorem A.
\item % f.
	Assume $Y^X$ has CW type for all CW complexes $Y$. Proceeding as above we may assume
	that, in addition, (\ddag) is a monomorphism for all $k\gqs 2$. Thus $X$ has finite cohomological
	dimension and is finitely dominated by \cite{wall}, Theorem F. (See also Mislin \cite{mislin}, Theorem 3.4).
\qed\end{abc}\nqed
\end{proof}

We say that a space $Z$ is $n$-coconnected if $\pi_k(Z)=0$ for all $k\gqs n+1$.

\begin{lem}\label{nasty}
Let $X$ be a connected CW complex such that $Y^X$ has CW homotopy type for every CW complex $Y$
(respectively every CW complex $Y$ with $\pi_k(Y)=0$ for $k\gqs n+1$). Then for any set
\[ \setof{Y_\lambda\,\vert\,\lambda\in\Lambda} \] of arbitrary CW complexes (respectively of $n$-coconnected
CW complexes) there exists a finite subcomplex $K$ of $X$ such that for any $L\gqs K$ and every $\lambda$,
the restriction induced function \[	[X,Y_\lambda]_*\to[L,Y_\lambda]_*		\] is injective.
\end{lem}

\begin{proof}
Let $Y$ be a connected CW complex. Pick a choice function $[X,Y]_*\to(Y,*)^{(X,*)}$, say $C\mapsto g(C)$.
Set \[ Z=\prod_{C\in[X,Y]_*}Y \] and let $\omega\colon W\to Z$ be a CW approximation for $Z$. By contracting
a path in $W$ if necessary we may assume that for some $w_0\in W$ we have $\omega(w_0)=z_0=\setof{y_0}$.
Note that $Z$ is path-connected, and since $(Z,z_0)^{(S^k,*)}\approx\prod_{C}(Y,y_0)^{(S^k,*)}$, also $Z$
is $n$-coconnected if $Y$ is. The same properties hold for $W$, and thus by assumption, 
$(W,w_0)^{(X,x_0)}$ has CW homotopy type.

Note that $\omega$ induces a natural equivalence of functors $\omega_\#\colon[\_,W]_*\to[\_,Z]_*$.

Define $\Gamma\colon X\to Z$ by $\Gamma=\setof{\Gamma_C}_C$ where $\Gamma_C=g(C)$,
and let $G\colon X\to W$ be a representative for $\omega_\#^{-1}[\Gamma]$.

There exists a finite subcomplex $K$ of $X$ such that for any $L$ containing $K$ the restriction induced
function $[X,W]_*\to[L,W]_*$ is `injective at $G$', that is, if for some function $F\colon X\to W$ the
restriction $F\vert_L$ is homotopic to $G\vert_L$, then $F$ and $G$ are homotopic.

In view of the natural equivalence the function $[X,Z]_*\to[L,Z]_*$ is `injective at $\Gamma$'.
Assume that for $f,g\colon X\to Y$ the restrictions $f\vert_L\colon L\to Y$ and $g\vert_L\colon L\to Y$ are homotopic.
With no loss of generality we may assume $g=g(C_0)$ for some $C_0$. We define $\Phi=\setof{\Phi_C}_C\colon X\to Z$ by
\[ \Phi_C=g(C)\text{ if }C\neq C_0\text{ and }\Phi_{C_0}=f. \]
Then clearly $\Phi\vert_L$ and $\Gamma\vert_L$ are homotopic. Therefore $\Phi$ and $\Gamma$ are homotopic.
By projecting the homotopy, it follows that $f$ and $g$ are homotopic.
We have shown that $[X,Y]_*\to[L,Y]_*$ is injective for any $L$ containing $K$.
Applying this with $Y$ the CW approximation of $\prod_{\lambda}Y_\lambda$ we prove the lemma.
\end{proof}

%%%%%%%%%%%%%%%%%%%%%%%%%%%%%%%%%%%%%%%%%%%%%%%%%%%%%%%%%%%%%%%%%%%%%%%%%%%%%%%%%%%%%%%%%%%%%%%%%%%
%%%%%%%%%%%%%%%%%%%%%%%%%%%%%%%%%%%%%%%%%%%%%%%%%%%%%%%%%%%%%%%%%%%%%%%%%%%%%%%%%%%%%%%%%%%%%%%%%%%
%%%												%%%
%%%			Spaces of maps into finite type Postnikov sections			%%%
%%%												%%%
%%%%%%%%%%%%%%%%%%%%%%%%%%%%%%%%%%%%%%%%%%%%%%%%%%%%%%%%%%%%%%%%%%%%%%%%%%%%%%%%%%%%%%%%%%%%%%%%%%%
%%%%%%%%%%%%%%%%%%%%%%%%%%%%%%%%%%%%%%%%%%%%%%%%%%%%%%%%%%%%%%%%%%%%%%%%%%%%%%%%%%%%%%%%%%%%%%%%%%%

\chapter{Spaces of maps into finite type Postnikov sections}\label{iff}

\setcounter{thm}{0}

For a CW complex $Y$ let $Y\!\!\scal{k}$ denote the $k$-connected cover of $Y$.
In this chapter our aim is to prove \thmref{converse2kahn} below. In \exref{discrepancy}
at the end we illustrate one of the reasons for the discrepancy between the necessary
conditions and the sufficient conditions given.

%%
%%	STATEMENT OF CONVERSE TO KAHN'S RESULT (AKA MONSTER)
%%

\begin{thm}\label{converse2kahn}\label{monster}
Let $X$ be a simply connected CW complex and let $Y$ be a simply connected complex of finite type
with $\pi_k(Y)=0$ for all $k\gqs n+1$, and $\pi_n(Y)\neq 0$.

Set $r=\sup\setof{k\,\vert\,\pi_k(Y)\otimes\QQ\neq 0}$ and let $r_p=\sup\setof{k\,\vert\,\pi_k(Y)*\ZZ_p\neq 0}$
for a prime $p$. As usual, we understand $\sup\emptyset=-\infty$.

If $Y^X$ has CW type then the following hold.
\begin{enumerate}
\item	For $i\lqs r-2$ the group $H_i(X)$ is finitely generated, and there exists a decomposition
	$H_{r-1}(X)\cong\ZZ^{d_{r-1}}\oplus H_{r-1}'$ with $\Hom(H_{r-1}',\ZZ)=0$.
\item	For each $i\lqs n$ there exists a finitely generated subgroup $B_i\lqs H_i(X)$
	(we may assume $B_{r-1}\gqs\ZZ^{d_{r-1}}$ and $B_i=H_i(X)$ for $i\lqs r-2$) so that
	\begin{enroman}
		\item	For $i\lqs r_p$ the quotient $H_i(X)/B_i$ is $p$-divisible, and moreover
		\item	for $i\lqs r_p-1$ the divisible $p$-localization $(H_i(X)/B_i)_{(p)}$ admits
			a direct sum decomposition into $Q\oplus\big(\oplus_j\ZZ_{p^\infty}\big)$ where
			$Q$ is a rational group and $j$ ranges over a finite set.
	\end{enroman}
\end{enumerate}

If, in addition, either $X$ is an H-cogroup or $Y\!\!\scal{r-1}$ is an H-group or $r\gqs n-1$
then in fact
\begin{enumerate}
\item[(1*)]	for $i\lqs r-1$ the group $H_i(X)$ is finitely generated and there exists a decomposition
	$H_r(X)\cong\ZZ^{d_r}\oplus H_r'$ with $\Hom(H_r',\ZZ)=0$, and we may assume $B_r\gqs\ZZ^{d_r}$.
\end{enumerate}

Conversely, conditions (1*) and (2) are sufficient for $Y^X$ to have CW type.
\end{thm}

\begin{rem_n} The $(r-1)$-connected cover of $Y$ in the statement of the theorem admits the structure of an
$H$-group for example if $n\lqs 2r-3$, see Copeland \cite{copeland}, section 3.2, in combination with Spanier
\cite{spanier}, Chapter 7, Exercise H.
\end{rem_n}

We record a few auxiliary results. See Section \ref{explicit_determinations} for a recollection of some definitions
and results from the theory of abelian groups.

\begin{lem}\label{lem_ovaj}
Let $P$ be a finite set of primes and let $A$ be a $P$-divisible abelian group.
\begin{enroman}
\item	If $A$ has an element of infinite order, then for each $p\in P$ there exists a subgroup of $A$,
	isomorphic to $\ZZ[\frac{1}{p}]$.
\item	If $A$ has an element of order $p\in P$ then $A$ contains $\ZZ_{p^\infty}$ as a direct summand.\qed
\end{enroman}
\end{lem}

%%  \begin{proof}
%%  Let $y_1$ be a nonzero element of $A$ and let $p\in P$. The divisibility assumption guarantees a sequence of
%%  $y_i\in A$ such that $py_{i+1}=y_i$ for each $i$. Assume that \[ b_{n-1}y_1+\dots+b_0y_n=0. \]
%%  Substituting $y_i=p^{n-i}y_n$ we obtain \begin{equation*} \tag{$*$} (b_{n-1}p^{n-1}+\dots+b_1p+b_0)y_n=0. \end{equation*}
%%  
%%  Since $p^{n-1}y_n=y_1$ we note that if $y_1$ has infinite order, then so has $y_n$, and if $y_1$ has order $p$, then
%%  $y_n$ has order $p^n$. In particular, if $y_1$ has infinite order, then
%%  \begin{equation*}\tag{\dag} b_{n-1}p^{n-1}+\dots+b_{1}p+b_0=0. \end{equation*}
%%  If $y_1$ has order $p$, then 
%%  \begin{equation*}\tag{\ddag} p^n\,\vert\,b_{n-1}p^{n-1}+\dots+b_{1}p+b_0. \end{equation*}
%%  Thus in the first case we may evidently define a monomorphism $\ZZ[\frac{1}{p}]\to A$
%%  while in the second case we may define a monomorphism $\ZZ_{p^\infty}\to A$.
%%  \end{proof}

\def\P{\ZZ_{p^n}}

\begin{prop}\label{ovaj}
Let $A$ be an arbitrary abelian group.
\begin{abc}
\item If $\Hom(A,\P)$ is finitely generated, then there exists
	a finitely generated subgroup $B$ of $A$ so that $A/B$ is $p$-divisible, and consequently
	$(A/B)_{(p)}=A_{(p)}/B_{(p)}$ is a divisible $p$-local group which is trivial if and only
	if $A/B$ is $\PP\setminus\setof{p}$-torsion.
	In particular $\Hom(A,\P)$ is finite. If $(A/B)_{(p)}$ is nontrivial, then $B$ can be so chosen that
	$A/B$ contains $\ZZ_{p^\infty}$ as a direct summand. Moreover, $B$ may be replaced with
	any bigger finitely generated subgroup without affecting the properties stated.
\item If also $\Ext(A,\ZZ_{p^n})$ is finitely generated (and hence finite), then in the splitting
	$(A/B)_{(p)}\cong\big(\oplus_i\QQ\big)\oplus\big(\oplus_j\ZZ_{p^\infty}\big)$ the index $j$
	must range over a finite family.
\end{abc}
\end{prop}
\begin{proof}
Since $A/pA$ is a vector space over $\ZZ_p$, it is isomorphic to a direct sum $A/pA\cong\oplus_N\ZZ_p$.
Apply $\Hom(\_,\P)$ to the exact sequence $0\to p\cdot A\to A\to A/pA\to 0$ to note that
$\Hom(\oplus_N\ZZ_p,\P)\cong\prod_N\ZZ_p$ is a subgroup of $\Hom(A,\P)$. Then $N$ must be finite
if $\Hom(A,\P)$ is finitely generated.

Choose a finite subset $S$ of $A$ whose image generates $A/pA$ and let $B=\scal{S}$. Then $B/pB\to A/pA$ is onto
which shows that $A/B$ is divisible by $p$. By \lemref{lem_ovaj} the localization $(A/B)_{(p)}$ is trivial if and only
if $A/B$ is $\PP\setminus\setof{p}$-torsion. If this is not the case, then $A/B$ contains a subgroup isomorphic
to either $\ZZ_{p^\infty}$ or $\ZZ[\frac{1}{p}]$. In the latter case enlarge $S$ to kill the subgroup $\ZZ$ of
$\ZZ[\frac{1}{p}]$ in $A/B$ to obtain $\ZZ_{p^\infty}$.

Let $B'$ be a finitely generated subgroup of $A$ that contains $B$. The sequence $0\to B'/B\to A/B\to A/B'\to 0$
is exact. This is to say that if $A/B$ is $p$-divisible, so is $A/B'$, and if $A/B$ is
$\PP\setminus\setof{p}$-torsion, so is $A/B'$. Moreover if $A/B$ contains a subgroup isomorphic to $\ZZ_{p^\infty}$
then $A/B'$ contains a subgroup isomorphic to a quotient of $\ZZ_{p^\infty}$ modulo a finitely generated subgroup,
which is again isomorphic to $\ZZ_{p^\infty}$.

Since $(A/B)_{(p)}$ is a divisible $p$-local group it is isomorphic to the direct sum of a rational group $R$ and
some copies of $\ZZ_{p^\infty}$, say of cardinality $M$. Then $\Hom(A/B,\P)$ is trivial
and \[ \Ext(A/B,\P)\cong\Ext\big((A/B)_{(p)},\P\big)\cong\prod_{M}\P. \]
An application of $\Hom(\_,\P)$ to $0\to B\to A\to A/B\to 0$ concludes the proof.
\end{proof}

\begin{lem}\label{ext}
Let $R$ be an infinite set of primes and let $A$ be an $R$-torsion group. If $A$ is not finitely generated,
then $\Ext(A,\ZZ)$ contains an uncountable $R$-local subgroup.
\end{lem}

\begin{proof}
We split $A$ into the direct sum of $r$-primary components $A\cong\oplus_{r\in R}A_{(r)}$. Since $A_{(r)}$ is torsion, the
sequence $0\to\Ext(A_{(r)}\otimes\ZZ_r,\ZZ)\to\Ext(A_{(r)},\ZZ)$ is exact. If $A_{(r)}\otimes\ZZ_r$ is infinite then
$\Ext(A_{(r)},\ZZ)$ contains as a subgroup an infinite product $\prod\ZZ_r$.

If $A_{(r)}\otimes\ZZ_r$ is finite then for a finitely generated subgroup $B$ of $A_{(r)}$ the quotient
$A_{(r)}/B$ is $r$-divisible. Since $A_{(r)}/B$ is $r$-torsion, it contains $\ZZ_{r^\infty}$ as a direct summand,
if it is nontrivial.
Again $0\to\Ext(A_{(r)}/B,\ZZ)\to\Ext(A_{(r)},\ZZ)$ is exact, hence $\Ext(A_{(r)},\ZZ)$ contains a subgroup isomorphic
to $\Hat\ZZ_r$.

If for all $r$ the group $A_{(r)}$ is finitely generated, and $A$ itself is not finitely generated, then
$A$ splits into a direct sum $\oplus_{\lambda\in\Lambda}C_\lambda$ where $\Lambda$ is infinite and for
each $\lambda\in\Lambda$, the summand $C_\lambda$ is a cyclic group of order a power of some element of $R$.
Then $\Ext(A,\ZZ)\cong\prod_{\lambda}C_\lambda$, and the result follows.
\end{proof}

\begin{lem}\label{lem_cor_character_group}
Let $H$ be an abelian group with a finitely generated subgroup $B$ such that the quotient $H/B$ is
$p$-divisible for some prime $p$. Then $\Hom(H,\ZZ)$ is finitely generated hence $H$ admits a splitting
$H\cong\ZZ^d\oplus H'$ with $\Hom(H',\ZZ)=0$. 
\end{lem}
\begin{proof}
The exact sequence $0\to B\to H\to H/B\to 0$ induces the exact sequence
$0\to\Hom(H/B,\ZZ)\to\Hom(H,\ZZ)\to\Hom(B,\ZZ)\to\dots$. Since $H/B$ is $p$-divisible, $\Hom(H/B,\ZZ)=0$,
hence $\Hom(H,\ZZ)$ is finitely generated by \lemref{character_group}.
\end{proof}

For convenience we record another easy lemma.
\begin{lem}\label{hom_ext_to_finite}
Let $Q,P\subset\PP$ where $\PP$ denotes the set of all primes. Set $P'=\PP\setminus P$.
Let $\pi$ be a finite $Q$-torsion group. Let $H$ be an arbitrary abelian group with a
finitely generated subgroup $B$.
\begin{enroman}
\item	Suppose that the quotient $H/B$ is $P$-divisible. Then $\Hom(H,\pi)$ is isomorphic to a direct
	sum of a finite $Q\cap P$-group and a $Q\cap P'$-bounded group.
\item	Suppose that the quotient $H/B$ is $Q$-divisible, and the (divisible) localization $(H/B)_{(Q)}$
	is isomorphic to the direct sum of some copies of the rationals $\QQ$, some copies of groups
	$\ZZ_{q^{\infty}}$ with $q\in Q\setminus P$ and at most finitely many copies of groups $\ZZ_{q^\infty}$
	with $q\in Q\cap P$. Then $\Ext(H,\pi)$ is an extension of a $Q\cap P'$-bounded group by a finite $Q\cap P$-group.
	\qed
\end{enroman}
\end{lem}

\begin{lem}\label{no_cancellation}
Let $B$ be a bounded abelian group. Let $\Hat\ZZ_p$ denote the group of $p$-adic integers,
$\Hat\ZZ_p=\lim\ZZ_{p^i}$. Any morphism $\Hat\ZZ_p\to B$ factors through a canonical projection
$\pi_n\colon\Hat\ZZ_p\to\ZZ_{p^n}$; hence the kernel is isomorphic to $\Hat\ZZ_p$.
\end{lem}

\begin{proof}
With no loss of generality we may assume that $B$ is a bounded $p$-group.
By the theorem of Pr\"{u}fer and Baer (see \propref{abelian_groups})
$B$ decomposes as $B\cong\bigoplus_\mu C_\mu$ where the $C_\mu$ are
cyclic $p$-groups bounded by a common bound, say $p^n$. Set $G_r=\ZZ_{p^r}$.

First we determine $\Hom(\Hat\ZZ_p,G_r)$. The quotient $\Hat\ZZ_p/\ZZ$ is divisible, $G_r$ is a reduced group, and $\ZZ$
is free, hence the exact sequence obtained by applying $\Hom(\_,G_r)$ to $0\to\ZZ\to\Hat\ZZ_p\to\Hat\ZZ_p/\ZZ\to 0$
simplifies to
\[ 0\to\Hom(\Hat\ZZ_p,G_r)\to G_r\to \Ext(\Hat\ZZ_p/\ZZ,G_r)\to\Ext(\Hat\ZZ_p,G_r)\to 0. \]
The natural projection $\pi_r\colon\Hat\ZZ_p\to G_r$ is clearly nontrivial, hence
$\Hom(\Hat\ZZ_p,G_r)\cong G_r$ and every morphism $\Hat\ZZ_p\to G_r$
is obtained as the composite of $\pi_r$ followed by an endomorphism $G_r\to G_r$.

Evidently then for $r\lqs n$ every morphism $\Hat\ZZ_p\to G_r$ factors through $G_n$ since $G_r\lqs G_n$.

As for the general case, since $\bigoplus_\mu C_\mu\lqs\prod_\mu C_\mu$,
a morphism $\Hat\ZZ_p\to B$ may be viewed as a morphism $\varphi\colon\Hat\ZZ_p\to\prod C_\mu$.
Since $C_\mu\cong G_{r_\mu}$ with $r_\mu\lqs n$, it follows by the above that
$\varphi$ factors as a composite $\Hat\ZZ_p\xrightarrow{\pi_n}\ZZ_{p^n}\xrightarrow{\psi}\prod C_\mu$.
Since the projection $\pi_n$ is surjective and $\varphi$ maps to $\bigoplus C_\mu$, so does $\psi$.
\end{proof}

\begin{lem}\label{no_cancellation_2}
Let $\PP=P\cup P'$ be a partition of the primes.
Let $A$ be a $P'$-local group, and let $G$ be a $P$-bounded group.
Then $\Hom(A,G)=0$.
\end{lem}
\begin{proof}
Since $G$ is $P$-local, $\Hom(A,G)$ is isomorphic to $\Hom(A_{(P)},G)$. Since $A$ is $P'$-local,
its $P$-localization $A_{(P)}$ is a rational group. In particular, it is divisible. As it is bounded,
$G$ is a reduced group, hence $\Hom(A_{(P)},G)=0$.
\end{proof}

\begin{conv_n}
As discussed in section \ref{explicit_determinations}, the torsion-free rank and $p$-ranks
for primes $p$ are invariants of a divisible abelian group $A$.
Instead of saying that the $p$-rank of a divisible group $A$ is finite we will also say
that $A$ {\it admits} only finitely many $p$-torsion (or quasicyclic) summands.
\end{conv_n}

%************************************************************************************************************

\begin{lem}\label{injectivity}
Assume that $\pi_k(Y)=0$ for $k\gqs n+1$, and that for some $m\lqs n$, the skeleton $T=X^{(m-1)}$ is compact.
If $Y^X$ has the type of a CW complex, then there exists a finite subcomplex $L$ containing $T$ so that the morphisms
\begin{align*}
\tag{\dag} H^m(X,T;\pi_n(Y))&\to H^m(L,T;\pi_n(Y)),\text{ and } \\
\tag{\ddag} H^{m+1}(X,T;\pi_n(Y)&\to H^{m+1}(L,T;\pi_n(Y))
\end{align*}
induced by inclusion are injective.
\end{lem}

\begin{proof}
By Stasheff's theorem the fibre $(Y,*)^{(X/T,*)}$ of $(Y,*)^{(X,*)}\to(Y,*)^{(T,*)}$ over the constant map 
also has CW type. We appeal to \lemref{obstruction} with $X/T$ in place of $X$ and infer the existence of
a finite subcomplex $L$ of $X$ containing $T$ such that the functions
\[ \beta_k\colon[S^k\wedge(X/T),Y]_*\to [S^k\wedge(L/T),Y]_* \]
are `monomorphisms'. We are interested in cases $k=n-m$ and $k=n-m-1$. Let $W_i$ denote the $i$-connected cover
of $Y$. For connectivity reasons we may replace $Y$ with $W_{n-1}$ in case $k=n-m$, respectively $W_{n-2}$
in case $k=n-m-1$. Since $W_{n-1}$ is an Eilenberg-MacLane space $K(\pi_n(Y),n)$, the function $\beta_{n-m}$ is equivalent
to the natural morphism $H^m(X/T;\pi_n(Y))\to H^m(L/T;\pi_n(Y))$ and the assertion follows if $k=n-m$.

If $k=n-m-1$ then we may assume a fibration $F\hookrightarrow W_{n-2}\to B$ with $F$ and $B$ Eilenberg-MacLane spaces.
Then the induced map $(W_{n-2},*)^{(A,*)}\to(B,*)^{(A,*)}$ is a fibration for any CW complex $A$ (see
\propref{mapping_space_covariant}). The fibre over
the constant map is $(F,*)^{(A,*)}$. Using naturality of this construction in $A$ (applied to inclusion $L/T\to X/T$)
and naturality of the arising long homotopy exact sequence, the assertion follows easily also in this case.
\end{proof}

\begin{proof}[Proof of sufficiency of \thmref{monster}]
By \propref{finite_postnikov_decomposition} it suffices to show that for each $j$ the space $(K(\pi_jY,j),*)^{(X,*)}$ has CW
type. Then by \propref{homology_decomposition} it suffices to show that for each $i$ the space
$(K(\pi_jY,j),*)^{(M(H_iX,i),*)}$ has CW type.

Let $\tau_j$ denote the torsion subgroup of $\pi_jY$. Then $\pi_jY\cong\tau_j\oplus\ZZ^{\nu_j}$ where
$\nu_j$ is the rank of the free abelian group $\pi_jY/\tau_j$. Hence
$K(\pi_jY,j)\simeq K(\tau_j,j)\times K(\ZZ,j)^{\nu_j}$ and 
\[ K(\pi_jY,j)^{M(H_iX,i)}\simeq K(\tau_j,j)^{M(H_iX,i)}\times \big[K(\ZZ,j)^{M(H_iX,i)}\big]^{\nu_j}. \]
Conditions (2) and (1*) now justify the application of \thmref{moore_eilenberg} and \propref{target_kzn}
to conclude the proof.
\end{proof}

We turn to necessity.

First note that if $r<n$, and (2) of \thmref{monster} holds then the quotient $H_r(X)/B_r$ is divisible by some
prime, and \lemref{lem_cor_character_group} implies the existence of a decomposition $H_r(X)\cong\ZZ^{d_r}\oplus H_r'$
with $\Hom(H_r',\ZZ)=0$.

\begin{step}\label{one}
There exists a finite subcomplex $L$ of $X$ so that the space $F_L=(Y,*)^{(X/L,*)}$ has finitely generated
abelian homotopy groups $\pi_k$ for $k\gqs 1$ and at most a countable set of path components.
\end{step}

\begin{proof}
By \propref{strong_obstruction_general} there exists a finite subcomplex $K$ of $X$ such that for all $L\gqs K$ the fibre $F_L$
of $Y^X\to Y^L$ over the constant map deforms in $Y^X$ to a point. Since $F_L$ is homeomorphic to $(Y,*)^{(X/L,*)}$, we
infer immediately from the fibration homotopy exact sequence that for $k\gqs 1$ the group $\pi_k(F_L,*)$ is finitely
generated abelian and that the set of path components $\pi_0(F_L)$ is at most countable. Moreover $F_L$ is an H-space
(not necessarily homotopy associative or with inverse). We may assume a minimal decomposition for $X$ (in the sense of
\lemref{minimal_decomposition}) and hence for some $m\gqs 2$ the skeleton $X^{(m-1)}$ is finite. Therefore we may choose
$L$ so that it contains $X^{(m-1)}$ and consequently that the $(m-1)$-skeleton of $X/L$ is trivial.
\end{proof}

\begin{step}\label{two}
Assume that $r=n$, i.e. the top group $\pi_n(Y)$ contains a $\ZZ$ summand.
Then the groups $H_i(X)$ are finitely generated for $i\lqs n-1$,
and there exists a decomposition $H_n(X)\cong\ZZ^d\oplus H'$ with
$\Hom(H',\ZZ)=0$. If the torsion subgroup $\tau_n$ of $\pi_n$ is
nonempty, then there exists a finitely generated subgroup $B'\lqs H'$
with the quotient $H'/B'$ divisible by any prime $t$ with $\tau_n\otimes\ZZ_t$
nontrivial.
\end{step}

\begin{proof}
Assume inductively that $X$ has finite $(m-1)$-skeleton for some $m\gqs 2$.
Assume also $m\lqs n-1$. It follows from \lemref{injectivity} that both $H^m(X;\pi_nY)$ and $H^{m+1}(X;\pi_nY)$
are finitely generated. In particular, $H^m(X;\ZZ)$ and $H^{m+1}(X;\ZZ)$ are finitely generated, and hence $H_m(X)$
is finitely generated. The decomposition of $X$ is assumed minimal and therefore $X^{(m)}$ is finite. This may be
continued until $m=n-1$, as asserted. For $m=n-1$ we can still go on and apply \lemref{injectivity} to conclude
that $H^n(X;\pi_n)$ is finitely generated, and hence that $\Hom(H_n(X),\ZZ)$ and $\Hom(H_n(X),\tau_n)$ are finitely
generated. The asserted decomposition $H_n(X)\cong\ZZ^d\oplus H'$ follows immediately from \lemref{character_group}.
Now $\Hom(H_n(X),\pi_n)$ is isomorphic with a direct sum of $\Hom(H',\tau_n)$ and a finitely generated group.
This implies the existence of $B'$ as in the statement by \propref{ovaj}.
\end{proof}

\begin{conv_n} Let $A$ be an abelian group and $G,H$ two finite abelian groups. 
Let $P$ be the (finite) set of primes $p$ for which $G\otimes\ZZ_p\neq 0$ and $Q$
the set of primes $q$ for which $H\otimes\ZZ_q\neq 0$.

By abuse of language, saying that $A$ is a $G$-group we mean it is a torsion $P$-primary group;
saying it is $G$-divisible, we mean it is divisible by primes belonging to $P$; saying it is $G$-bounded
we mean that it is bounded by a number all of whose prime divisors belong to $P$; saying it is a $G\cap H$-group
we mean it is a torsion $P\cap Q$-primary group; saying it is $G\setminus H$-bounded we mean it is bounded
by a number all of whose prime divisors belong to $P\setminus Q$.

In particular we note that if $B$ is $G\setminus H$-bounded and $A$ is $H$-bounded then $\Hom(A,B)=0$.
\end{conv_n}

\begin{step}\label{three}
Assume that $r<n$. There exist finitely generated subgroups $B_i\lqs H_i(X)$ for $2\lqs i\lqs n$ so that the
following hold.
\begin{enroman}
\item	If $r=-\infty$ then for each $i\lqs n$
\begin{itemize}
	\item[($D_i$)] the quotient $H_i(X)/B_i$ is $T(\pi_{i+1})\cup\pi_{i+2}\cup\dots\cup\pi_n$-divisible, and
		for $t\in T(\pi_{i+1})\cup\pi_{i+2}\cup\dots\cup\pi_n$ the divisible localization $(H_{i-1}(X)/B_{i-1})_{(t)}$
		contains at most finitely many torsion summands.
\end{itemize}
\item	If $r$ is positive then for $i\lqs r-2$ the groups $H_i(X)$ are finitely generated,
	and there exists a decomposition $H_{r-1}(X)\cong\ZZ^{d_{r-1}}\oplus H_{r-1}'$ with $\Hom(H_{r-1}',\ZZ)$ trivial.
	(Therefore it makes sense to assume $r-m>1$.)
	Moreover, the property ($D_i$) can be established for $i\gqs r-1$.
\end{enroman}
\end{step}

\begin{proof}
Set $\Xi=X/L$. We consider the Federer spectral sequence converging to the homotopy groups
of $Y^\Xi$. This is a second quadrant homology type spectral sequence with $E^2_{-p,q}\cong H^p(\Xi;\pi_q Y)$
for $q-p\gqs 1$ and $E^2_{-p,q}\lqs H^p(\Xi;\pi_q Y)$ for $q-p=0$. For details see Federer \cite{federer}.

Set $H_i=H_i(\Xi)$ and $\pi_j=\pi_j(Y)$.

Assume for a decreasing induction on $k$ that for $q-p=k$ (with $q\gqs r$ and $p\gqs 2$) there exists a finitely
generated subgroup $B_p$ of $H_p$ such that
\begin{enroman}
\item[($*$)]	the quotient group $H_p/B_p$ is $T(\pi_q)\cup\pi_{q+1}\cup\dots\cup\pi_{n}$-divisible,
\item[($**$)]	for a prime $t$ with $t\in\pi_{q+1}\cup\dots\cup\pi_n$ the divisible $t$-local group
	$\big(H_p/B_p\big)_{(t)}$ admits only a finite number of torsion (i.e. $\ZZ_{t^\infty}$) summands.
\end{enroman}

Using \lemref{hom_ext_to_finite} it is easy to see that ($*$) and ($**$) imply that $H^{p}(\Xi;\pi_q)$ is a finite
$\pi_q$-group and $E^2_{-(p+1),q}\cong H^{p+1}(\Xi;\pi_q)$ is a split extension of a
$\pi_q\setminus(\pi_{q+1}\cup\dots\cup\pi_n)$-bounded
group by a finite $\pi_q\cap\big(\pi_{q+1}\cup\dots\cup\pi_n\big)$-group, say
\[ 0\to I^2_{-(p+1),q}\to H^{p+1}(\Xi;\pi_q)\to F^2_{-(p+1),q}\to 0. \]
Evidently then for each $u\gqs 3$ there exists an exact sequence
\[ 0\to I^u_{-(p+1),q}\to E^u_{-(p+1),q}\to F^u_{-(p+1),q}\to 0 \]
where $I^u_{-(p+1),q}$ is a homology group of $I^{u-1}_{-(p+1),q}$,
and $F^u_{-(p+1),q}$ is a homology group of $F^{u-1}_{-(p+1),q}$.
In particular, $I^u_{-(p+1),q}$ is a $\pi_q\setminus(\pi_{q+1}\cup\dots\cup\pi_n)$-bounded group
and $F^u_{-(p+1),q}$ is a finite $\pi_q\cap\big(\pi_{q+1}\cup\dots\cup\pi_n\big)$-group.

For $k=n-m+1$ there is nothing to prove.
Assume that the inductive hypothesis holds for $k>\max\setof{r-m,1}$ and consider the diagonal $q-p=k-1$.
Suppose that for some $u\gqs 2$ the subgroup $I^u_{-p,q}$ of $E^u_{-p,q}$ is not finitely generated.
The group $I^u_{-p,q}$ is $\pi_q\setminus(\pi_{q+1}\cup\dots\cup\pi_n)$-bounded, while $E^u_{-p-u,q+u-1}$ is
$\pi_{q+u-1}$-bounded (or even trivial), hence the kernel of $d^u\colon E^u_{-p,q}\to E^u_{-p-u,q+u-1}$ contains
$I^u_{-p,q}$. Since by inductive hypothesis the diagonal $j-i=k$ has all entries $E^u_{-i,j}$ finitely generated,
the group $I^{u+1}_{-p,q}$ contains a quotient of $I^u_{-p,q}$ by a finitely generated subgroup and hence cannot be finitely
generated. Hence $I^\infty_{-p,q}$ is not finitely generated. The contradiction implies that $I^2_{-p,q}$ is finite,
hence also the group $H^p(\Xi;\pi_q)$ is finite.

Let $2\lqs p$, $q<n$, and $q-p=k-1$. Since $q+1-p=k$, our inductive hypothesis guarantees that the group $H_p/B_p$ is
$\pi_{q+1}\cup\dots\cup\pi_n$-divisible and that for $t\in\pi_{q+2}\cup\dots\cup\pi_n$ its localization $(H_p/B_p)_{(t)}$
admits only finitely many torsion summands. By the above the groups $H^p(\Xi;\pi_q)$ and $H^{p+1}(\Xi;\pi_{q-1})$ are
finite. In particular, the groups $\Hom(H_p,\pi_q)$ and $\Ext(H_p,\pi_{q-1})$ are finite. By \propref{ovaj} we may enlarge
$B_p$ to yield ($*$) and ($**$) for the pair $(p,q)$. For $q=n$ we get a finitely generated subgroup $B_{n-k+1}$
of $H_{n-k+1}$ such that the quotient group $H_{n-k+1}/B_{n-k+1}$ is $\pi_n$-divisible.

The above works as stated unless $k-1=r-m$ or $k-1=0$.
Note that $k-1=0$ occurs if $r=-\infty$.

\begin{itemize}
\item	If $k-1=0$, then we have shown that the groups $H^{i-1}(\Xi;\pi_i)$ are finite for all $i$.
Let $W_i$ denote the $i$-connected cover of $Y$ for $i\gqs 1$. We may assume that
$K(\pi_n,n)=W_{n-1}\to W_{n-2}\to\dots\to W_2\to W_1=Y$ is a sequence of principal fibrations
where for each $i$ the fibre of $W_i\to W_{i-1}$ is a $K(\pi_{i},i-1)$. Therefore since
$[\Xi,W_i]_*\to[\Xi,W_{i-1}]_*$ collapses precisely the orbits of the action of $H^{i-1}(\Xi;\pi_i)$ on $[\Xi,W_i]_*$,
it follows by a trivial induction that if $[\Xi,W_{n-1}]_*\cong H^n(\Xi;\pi_n)$ is uncountable then so
is $[\Xi,Y]_*$ which is a contradiction. Hence $H^n(\Xi;\pi_n)$ is finitely generated,
there exists a finitely generated subgroup $B_n$ of $H_n$ such that $H_n/B_n$ is $\pi_n$-divisible,
and $B_{n-1}$ may be enlarged to assume that the divisible $(H_{n-1}/B_{n-1})_{(\pi_n)}$ contains
at most finitely many $\ZZ_{t^\infty}$ summands if $\pi_n\otimes\ZZ_t\neq 0$. In particular, $H^n(\Xi;\pi_n)$ is finite.

\item
If $k-1=r-m$, we proceed just as above, showing that the groups $H^p(\Xi;\pi_q)$ are finite for pairs $(p,q)$,
where $q-p=k-1$ and $q>r$. This establishes ($*$) and ($**$) for $p>m$. Now consider the pair $(m,r)$ corresponding
to $E^2_{-m,r}\cong H^m(\Xi;\pi_r)\cong\Hom(H_m,\pi_r)$. Since $H_m/B_m$ is $\pi_{r+1}\cup\dots\cup\pi_n$-divisible
the restriction $\Hom(H_m,\ZZ)\to\Hom(B_m,\ZZ)$ is an injection which by \lemref{character_group} implies the
existence of a decomposition $H_m\cong\ZZ^{d_m}\oplus H_m'$ with $\Hom(H_m',\ZZ)=0$.
Since $B_m$ may be assumed to contain the summand $\ZZ^{d_m}$ there exists finitely generated $B_m'\lqs H_m'$
with $H_m'/B_m'$ isomorphic to $H_m/B_m$.

Using \lemref{hom_ext_to_finite} we note that $\Hom(H_m,\pi_r)$ is a split extension of a
$\tau_r\setminus(\pi_{r+1}\cup\dots\cup\pi_n)$-bounded group, say $I^2_{-m,r}$, by a finitely generated group.
Reasoning as above it follows that the group $I^2_{-m,r}$ is finitely generated. Hence also
$H^m(\Xi;\pi_r)$ is finitely generated.

In particular, $\Hom(H_m,\tau_r)$ is finitely generated (in fact finite).
By \propref{ovaj} we may enlarge $B_m$ so that $H_m/B_m$ is $\tau_r\cup\pi_{r+1}\cup\dots\cup\pi_n$-divisible.
As we already know that $H^{m+1}(\Xi;\pi_{r+1})$ is finite, so is $\Ext(H_m,\pi_{r+1})$. Applying
the inductive hypothesis we infer that the localization $(H_m/B_m)_{(t)}$ has at most finitely many
torsion summands for $t\in\pi_{r+1}\cup\dots\cup\pi_n$.

We have assumed that $r-m>1$ hence also $k-2=r-m-1\gqs 1$. On diagonal $k-1$ there appear now only
finitely generated groups and we may proceed as above to show that the groups $H^p(\Xi;\pi_q)$ are finite for all pairs $(p,q)$ with $q-p=k-2$
and $q>r$. This establishes ($*$) and ($**$) for $p>m+1$. Now we consider $E^2_{-(m+1),r}\cong H^{m+1}(\Xi;\pi_r)$. Split
\begin{align*}
 H^{m+1}(\Xi;\pi_r)\cong&\Hom(H_{m+1},\ZZ^\rho)\oplus\Ext(H_m,\ZZ^\rho) \\
	\oplus&\Hom(H_{m+1},\tau_r)\oplus\Ext(H_m,\tau_r).
\end{align*}

We know that $H_{m+1}/B_{m+1}$ is $\pi_{r+1}\cup\dots\cup\pi_n$-divisible. Using the above determined properties
of the quotient group $H_m/B_m$ and \lemref{hom_ext_to_finite}, we infer that the group $H^{m+1}(\Xi;\tau_r)$ is
a split extension of a $\tau_r\setminus(\pi_{r+1}\cup\dots\cup\pi_n)$-bounded group, say $I^2_{-(m+1),r}$, by a finite
$\tau_r\cap(\pi_{r+1}\cup\dots\cup\pi_n)$-group. Then also $E^2_{-(m+1),r}$ is a split extension of $I^2_{-(m+1),r}$
and we proceed as before to infer that $I^2_{-(m+1),r}$ is finite. Therefore $H^{m+1}(\Xi;\tau_r)$ is a finite group,
whence by \propref{ovaj} we may assume that $B_{m+1}$ is such that the quotient $H_{m+1}/B_{m+1}$ is divisible by
$\tau_r\cup\pi_{r+1}\cup\dots\cup\pi_n$, and that for $t\in\tau_r\cup\pi_{r+1}\cup\dots\cup\pi_n$ the divisible
localization $(H_m/B_m)_{(t)}$ only admits finitely many torsion summands.

By \lemref{lem_cor_character_group} the divisibility properties of $H_{m+1}/B_{m+1}$ imply a decomposition
$H_{m+1}\cong\ZZ^{d_{m+1}}\oplus H_{m+1}'$ with $\Hom(H_{m+1}',\ZZ)$ trivial (and $H_{m+1}/B_{m+1}\cong H_{m+1}'/B_{m+1}'$
for a suitable $B_{m+1}'$).

Next we take care of $\Ext(H_m,\ZZ^\rho)$. Set $P=\tau_r\cup\pi_{r+1}\cup\dots\cup\pi_n$ and let $R=\PP\setminus\setof P$.
Since $H_m/B_m$ is $P$-divisible, it is either an $R$-torsion group or we may enlarge $B_m$ if necessary to assume that
$H_m/B_m$ admits a $\ZZ_{t^\infty}$-summand for some $t\in P$. Thus if $H_m/B_m$ is not finitely generated, $\Ext(H_m,\ZZ)$
either contains a subgroup isomorphic to $\Hat\ZZ_t$ modulo a finitely generated group or it contains
an uncountable $R$-local subgroup $\Gamma_R$ (see \lemref{ext}).

On diagonal $k-3$ there are only $P$-bounded groups strictly above row $r$. Hence if $\Gamma_R\lqs E^2_{-(m+1),r}$, then
by \lemref{no_cancellation_2} also $\Gamma_R\lqs E^u_{-(m+1),r}$ for $u\gqs 3$. If $E^2_{-(m+1),r}$ contains a quotient
of $\Hat\ZZ_t$ by a finitely generated subgroup, then by \lemref{no_cancellation} each $E^u_{-(m+1),r}$ for $u\gqs 3$
admits a subgroup isomorphic to a quotient of $\Hat\ZZ_t$ by a finitely generated group. See the figure below.

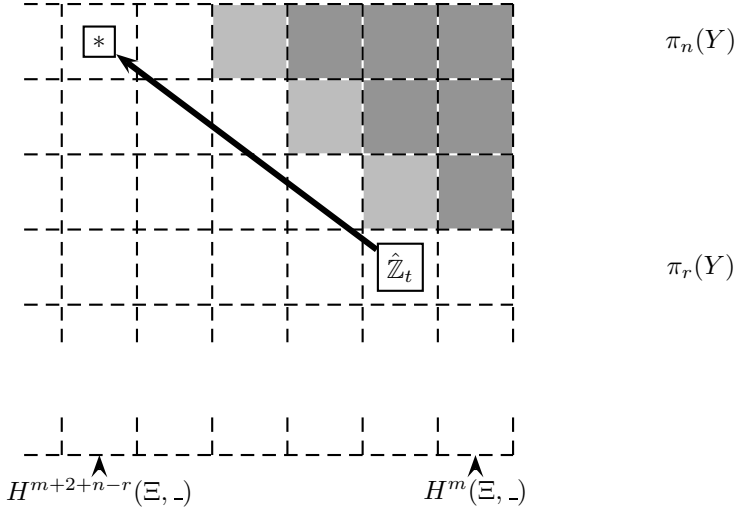
\vspace{2mm}

\centerline{
\begin{pspicture}(-0.5,0)(9,7)
  \psframe[linestyle=none,fillstyle=solid,fillcolor=gray4](3,6)(4,7)
  \psframe[linestyle=none,fillstyle=solid,fillcolor=gray4](4,6)(5,7)
  \psframe[linestyle=none,fillstyle=solid,fillcolor=gray4](5,6)(6,7)
  \psframe[linestyle=none,fillstyle=solid,fillcolor=gray4](4,5)(5,6)
  \psframe[linestyle=none,fillstyle=solid,fillcolor=gray4](5,5)(6,6)
  \psframe[linestyle=none,fillstyle=solid,fillcolor=gray4](5,4)(6,5)
  \psframe[linestyle=none,fillstyle=solid,fillcolor=gray2](2,6)(3,7)
  \psframe[linestyle=none,fillstyle=solid,fillcolor=gray2](3,5)(4,6)
  \psframe[linestyle=none,fillstyle=solid,fillcolor=gray2](4,4)(5,5)
  \psline[linestyle=dashed]{-}(-0.5,3)(6,3)
  \psline[linestyle=dashed]{-}(-0.5,4)(6,4)
  \psline[linestyle=dashed]{-}(-0.5,5)(6,5)
  \psline[linestyle=dashed]{-}(-0.5,6)(6,6)
  \psline[linestyle=dashed]{-}(-0.5,7)(6,7)
  \psline[linestyle=dashed]{-}(0,2.5)(0,7)
  \psline[linestyle=dashed]{-}(1,2.5)(1,7)
  \psline[linestyle=dashed]{-}(2,2.5)(2,7)
  \psline[linestyle=dashed]{-}(3,2.5)(3,7)
  \psline[linestyle=dashed]{-}(4,2.5)(4,7)
  \psline[linestyle=dashed]{-}(5,2.5)(5,7)
  \psline[linestyle=dashed]{-}(6,2.5)(6,7)
  \psline[linestyle=dashed]{-}(-0.5,1)(6,1)
  \psline[linestyle=dashed]{-}(0,1)(0,1.5)
  \psline[linestyle=dashed]{-}(1,1)(1,1.5)
  \psline[linestyle=dashed]{-}(2,1)(2,1.5)
  \psline[linestyle=dashed]{-}(3,1)(3,1.5)
  \psline[linestyle=dashed]{-}(4,1)(4,1.5)
  \psline[linestyle=dashed]{-}(5,1)(5,1.5)
  \psline[linestyle=dashed]{-}(6,1)(6,1.5)
  \rput(.5,.5){\rnode{A}{$H^{m+2+n-r}(\Xi,\_)$}}
  \rput(5.5,.5){\rnode{B}{$H^{m}(\Xi,\_)$}}
  \rput(.5,6.5){\rnode{C}{\psframebox{$*$}}}
  \rput(4.5,3.5){\rnode{D}{\psframebox{$\Hat\ZZ_t$}}}
  \rput(8.5,3.5){\rnode{E}{$\pi_r(Y)$}}
  \rput(8.5,6.5){\rnode{F}{$\pi_n(Y)$}}
  \ncline[linewidth=0.8mm]{->}{D}{C}
  \psline[linewidth=0.8mm]{->}(.5,.8)(.5,1)
  \psline[linewidth=0.8mm]{->}(5.5,.8)(5.5,1)
\end{pspicture}}

Both cases lead to contradiction hence $H_m/B_m$ is finitely generated, and hence so is $H_m$. The exact sequence
$H_m(L)\to H_m(X)\to H_m(X,L)=H_m$ implies that also $H_m(X)$ is finitely generated, hence the $m$-skeleton $X^{(m)}$ is
finite and we may restart with \stepref{one} in case still $m<r-2$. Therefore with no loss of generality we may assume
$m=r-2$.

We have shown that for $q-p\gqs r+1-m$ as well as for $q-p\in\setof{r-m,r-m-1}$ with $q\gqs r+1$ the groups
$H^p(\Xi;\pi_q)$ are finite. Furthermore the groups $H^m(\Xi;\pi_r)$ and $H^{m+1}(\Xi;\pi_r)$
are finitely generated. By considering exact sequences
\[ \H^p(\Xi;\pi_q)\cong H^p(X,L;\pi_q)\to H^p(X;\pi_q)\to H^p(L;\pi_q) \]
we note that exactly analogous properties hold for the groups $H^p(X;\pi_q)$. Hence we may apply
\propref{ovaj} and \lemref{character_group} to deduce the asserted properties of groups $H_i(X)$.
\qed
\end{itemize}\nqed
\end{proof}

\begin{step}\label{four}
Let $Y\to\Bar Y$ denote Zabrodsky's integral approximation with homotopy fibre $Y_t$.
Then $Y_t^X$ has the homotopy type of a CW complex, and (2) of \thmref{converse2kahn} holds.
\end{step}

\begin{proof}
The map $Y_t\to Y$ is a principal fibration with fibre $\Omega\Bar Y$ hence the fibre of the associated
fibration $(Y_t,*)^{(X,*)}\to(Y,*)^{(X,*)}$ over any point is either empty or homotopy equivalent to the
space $(\Omega\Bar Y,*)^{(X,*)}$.

According to Zabrodsky \cite{zabrodsky}, $\Omega\Bar Y$ is homotopy equivalent to the product
\[ K(\ZZ^{d_r},r-1)\times\dots\times K(\ZZ^{d_2},1) \] where $d_i$ denotes the torsion-free rank of $\pi_i(Y)$.
Then $Y_t^X$ has CW type by \propref{target_kzn}, together with Stasheff's theorem.

For $j\gqs r+1$, the morphism $\pi_j(Y_t)\to\pi_j(Y)$ is bijective, and $\pi_r(Y_t)\to\pi_r(Y)$ is a monomorphism
onto the torsion subgroup $\tau_r$ of $\pi_r(Y)$. In general there exist short exact sequences
\[ 0\to\coker(\pi_{j+1}Y\to\pi_{j+1}\Bar Y)\to\pi_j Y_t\to T(\pi_jY)\to 0. \] Since the groups
$\coker(\pi_{j+1}Y\to\pi_{j+1}\Bar Y)$ are finite, the following implication holds
\begin{equation*} \tag{$*$} \pi_jY*\ZZ_p\neq 0\implies\pi_j Y_t*\ZZ_p\neq 0. \end{equation*}

Note that $Y_t\simeq\prod_{p\in P}Y_t{}_{(p)}$ where $P$ is a finite set of primes. This implies
that for each $p\in P$ the space $(Y_t{}_{(p)},*)^{(X,*)}$ has CW type and \navedi{1} of \stepref{three}
applies to establish (2) of \thmref{converse2kahn} by virtue of ($*$), as claimed.

In particular, for $i\gqs r+1$ the groups $H^i(X;\pi_i)$ are finite.
\end{proof}

\begin{step}\label{five}
Let $Y\!\!\scal{r-1}$ denote the $(r-1)$-connected cover of $Y$. Then $Y\!\!\scal{r-1}^X$ has CW homotopy type.
\end{step}

\begin{proof}
The map $Y\!\!\scal{r-1}\to Y$ is the principal fibration obtained by taking the homotopy fibre of $Y\to Y_{r-1}$.
The fibres of the induced principal fibration $(Y\!\!\scal{r-1},*)^{(X,*)}\to(Y,*)^{(X,*)}$ are either empty or are
homotopy equivalent to $(\Omega Y_{r-1},*)^{(X,*)}$. Since $\pi_k(\Omega Y_{r-1})=0$ for $k\gqs r-1$
and $X$ has finite $(r-2)$-skeleton, the assertion follows.
\end{proof}

\begin{step}\label{six}
By the previous step we may assume that $Y$ is $(r-1)$-connected.
Assume that $X$ is an H-cogroup or that $Y$ is an H-group.
Then also $H^r(X;\pi_r)$ is finitely generated.
\end{step}

\begin{proof}
Note that in case $X$ is an H-cogroup we cannot afford to replace it by a (homologically more
convenient) quotient $X/L$ for suitable $L$. Hence we proceed with $X$ as it is.

We consider the Postnikov tower of principal fibrations $Y_n\to Y_{n-1}\to\dots\to Y_{r}$ for $Y$.
We know that for $i\gqs r+1$ the groups $H^{i}(X;\pi_{i})$ are finite, and the groups $H^{i+1}(X;\pi_i)$
are always $\pi_i$-bounded.

We know that $H_r/B_r$ is $\pi_{r+1}\cup\dots\cup\pi_n$-divisible, and therefore $\Hom(H_r,\ZZ)\to\Hom(B_r,\ZZ)$
is an injection. This guarantees $H_r\cong\ZZ^{d_r}\oplus H_r'$ with $\Hom(H_r',\ZZ)=0$.

We also know that $H_{r-1}/B_{r-1}$ is $\tau_r\cup\pi_{r+1}\cup\dots\cup\pi_n$-divisible
and that for $t\in\pi_{r+1}\cup\dots\cup\pi_n$ the localization $(H_{r-1}/B_{r-1})_{(t)}$ contains
at most finitely many torsion summands. Set $P=\tau_r\cup\pi_{r+1}\cup\dots\cup\pi_n$. If $H_{r-1}$ itself
is not finitely generated, we may assume that either $H_{r-1}/B_{r-1}$ contains a $\ZZ_{t^\infty}$ summand
for some $t\in P$ or that $H_{r-1}/B_{r-1}$ is an infinitely generated $\PP\setminus P$-torsion group.
It follows that $\Ext(H_{r-1},\ZZ)$ contains either a subgroup isomorphic to a quotient of $\Hat\ZZ_t$
modulo a finitely generated subgroup or an uncountable $\PP\setminus P$-local group, the latter by \lemref{ext}.
This leads to contradiction as follows.

Our additional hypotheses guarantee that for each $i\gqs r$ the Puppe exact sequence
\begin{equation*}
H^{i+1}(X;\pi_{i+1})\to[X,Y_{i+1}]_*\to[X,Y_i]_*\to H^{i+2}(X;\pi_{i+1})
\end{equation*}
derived from the principal fibration $Y_{i+1}\to Y_i$, is an exact sequence of groups and group-morphisms. Since
$Y_r=K(\pi_r,r)$, the group $[X,Y_r]_*$ is isomorphic to $H^r(X;\pi_r)$. If $\Ext(H_{r-1},\ZZ)$ has a subgroup
isomorphic to a quotient of $\Hat\ZZ_t$ by a finitely generated subgroup we apply
\lemref{no_cancellation}, and if $\Ext(H_{r-1},\ZZ)$ has an uncountable $\PP\setminus P$-local subgroup $\Gamma$ we apply
\lemref{no_cancellation_2} in a straightforward induction to show that $[X,Y_n]_*=[X,Y]_*$ contains a subgroup which is an
extension of a finite group by a quotient of $\Hat\ZZ_t$ by a finitely generated group or an extension of a finite group
by an uncountable $\PP\setminus P$-local subgroup. In particular, $[X,Y]_*$ is uncountable; which contradicts
\propref{group_components}. This establishes (1*) for this case, and concludes the proof of \thmref{converse2kahn}.
\end{proof}

\begin{example}\label{discrepancy}
Suppose that $Y$ is the total space of a fibration $K(\ZZ_p,n)\to Y\to K(\ZZ,r)$ where $n>r$. This is to say
that $Y$ is the homotopy fibre of a single $k$-invariant $k\colon K(\ZZ,r)\to K(\ZZ_p,n+1)$. Let $X$ be a simply
connected CW complex. If $Y^X$ has CW type, then $H_i(X)$ is finitely generated for $1\lqs i\lqs r-1$.

By \thmref{monster} we know already that for each $i\lqs n$ the group $H_i(X)$ admits a finitely generated
subgroup $B_i$ such that $H_i(X)/B_i$ is $p$-divisible and $(H_{i-1}(X)/B_{i-1})_{(p)}$ only admits finitely many
quasicyclic summands. Moreover for $i\lqs r-2$ already $H_i(X)$ is finitely generated.

By passing to cohomology, it follows that $\H^i(X;\ZZ_p)$ is finite for all $i\lqs n$, and that
$\H^{i}(X;\ZZ)$ is finitely generated for $i\lqs r-1$. If $L$ is a finite subcomplex of $X$ then
the quotient $X/L$ has the same cohomological properties as $X$ and by \propref{ovaj} and \lemref{character_group}
we may infer also the same homological properties as $X$.

Assume a minimal decomposition of $X$ with the associated homology decomposition $\setof{X_i}$. We may pick a finite
subcomplex $L$ of $X$ such that $L$ contains $X_{r-2}$, the image $B_{r-1}'$ of $H_{r-1}(L)\to H_{r-1}(X)$ contains
the subgroup $B_{r-1}$, and the fibre of the restriction fibration $Y^X\to Y^L$ deforms to a point in the total space.
Set $\Xi=X/L$. Note that $\H_i(\Xi)=0$ for $i\lqs r-2$ and that $H_{r-1}(\Xi)\cong H_{r-1}(X)/B_{r-1}'$ is $p$-divisible.
Thus if $H_{r-1}(\Xi)$ is not finitely generated, we may assume by \propref{ovaj} that either it contains
a quasicyclic summand $\ZZ_{p^\infty}$ or that it is a $\PP\setminus\setof{p}$-torsion group. Moreover,
as in \stepref{one} it follows that the set $[\Xi,Y]_*$ is countable.

We replace $X$ by $\Xi$.

We are interested in the induced function
\begin{equation*} \tag{$*$} k_\#\colon[X,K(\ZZ,r)]_*\to[X,K(\ZZ_p,n+1)]_* \end{equation*}
More precisely we would like to determine (the size of) the preimage of $k_\#^{-1}(0)$.

The homotopy class of $k$ is represented by an element $\Hat k$ of $H^{n+1}\big(K(\ZZ,r);\ZZ_p\big)$.
By results of Cartan and Borel the mod-$p$ cohomology algebra $H^*\big(K(\ZZ,r);\ZZ_p\big)$
has a $p$-simple system of generators which are transgressive with respect to the Serre cohomology spectral
sequence of the fibration

\begin{equation*} \tag{$**$} K(\ZZ,r)\to PK(\ZZ,r+1)\to K(\ZZ,r+1). \end{equation*}

Using the associated vector-space basis for $H^*\big(K(\ZZ,r);\ZZ_p\big)$ we expand $\Hat k$ as
$\Hat k=\sum_i\lambda_i\xi_i$ where $\lambda_i\in\ZZ_p$ and $\xi_i\in H^{n+1}\big(K(\ZZ,r);\ZZ_p\big)$.
For each $i$ we may write
\begin{equation*} \tag{\dag} \xi_i=\alpha_i\smile\beta_i \end{equation*}
where $\alpha_i$ is transgressive (and $\beta_i$ may be $1$).

Applying $H^{*}(\_;\ZZ_p)$ to $f\colon X\to K(\ZZ,r)$ we get the induced morphism
\[ f^{*}\colon H^{*}\big(K(\ZZ,r);\ZZ_p\big)\to H^*\big(X;\ZZ_p\big). \]
Thus we may represent the composite \[ X\xrightarrow{f} K(\ZZ,r)\xrightarrow{k} K(\ZZ_p,n+1) \]
with the cohomology class
\begin{equation*}\tag{\ddag} f^*(\Hat k)=\sum_{i}\lambda_i\,f^*(\alpha_i)\smile f^*(\beta_i). \end{equation*}
Set $a_i=\deg\alpha_i\lqs n+1$.
Transgressive elements are in the image of the cohomology suspension with respect to ($**$), hence for each $i$ there
exists an element $\eta_i\in H^{a_i+1}\big(K(\ZZ,r+1);\ZZ_p\big)$ such that $\sigma^*(\eta_i)=\alpha_i$ where
\[ \sigma^*\colon [K(\ZZ,r+1),K(\ZZ_p,a_i+1)]_*\to [\Omega K(\ZZ,r+1),\Omega K(\ZZ_p,a_i+1)]_* \]
is induced by the adjoint of the natural map \[ l\colon S\Omega K(\ZZ,r+1)\to K(\ZZ,r+1). \]
This implies that $f^*(\alpha_i)$ may be represented in $[X,K(\ZZ_p,a_i)]_*$ as the composite
\[ X\xrightarrow{f}K(\ZZ,r)\xrightarrow{\simeq}\Omega K(\ZZ,r+1)\xrightarrow{\Omega\psi_i}
	\Omega K(\ZZ_{p},a_i+1)\xrightarrow{\simeq}K(\ZZ_p,a_i) \]
for some map $\psi_i\colon K(\ZZ,r+1)\to K(\ZZ_p,a_i+1)$. Now we view $f^*(\alpha_i)$ rather as
$\Omega\psi_i{}_{\#}([f])$. Since $\Omega\psi_i{}$ is an H-group morphism, the induced function $\Omega\psi_i{}_\#$
is a group homomorphism. So we consider $\Omega\psi_i{}_\#\colon[X,K(\ZZ,r)]_*\to[X,K(\ZZ_p,a_i)]_*$.

We distinguish two possibilities. The first is that $H_{r-1}(X)$ contains a quasicyclic summand $\ZZ_{p^\infty}$.
Then $[X,K(\ZZ,r)]_*$ contains $\Ext(\ZZ_{p^\infty},\ZZ)\cong\Hat\ZZ_p$ as a direct summand. Since 
$[X,K(\ZZ_p,a_i)]_*\cong\H^{a_i}(X;\ZZ_p)$ is $p$-bounded (even finite if $a_i\lqs n$), the kernel $A_i$ of
\[ \Omega\psi_i{}_\#\vert_{\Hat\ZZ_p}\colon\Hat\ZZ_{p}\to[X,K(\ZZ_p,a_i)]_* \]
is a subgroup of finite index in $\Hat\ZZ_p$, by \lemref{no_cancellation}. The intersection $A=\cap_iA_i$
of finitely many subgroups of finite index is also a subgroup of finite index in $\Hat\ZZ_p$. This is to
say that $\Omega\psi_i{}_\#[f]=f^*(\alpha_i)=0$ for $[f]\in A\lqs\Hat\ZZ_p\lqs H^r(X;\ZZ)$. By (\ddag)
also $k_\#[f]=0$ for all $[f]\in A$. Hence the function $k_\#$ sends an uncountable set to $0\in[X,K(\ZZ_p,n+1)]_*$.
By exactness of the sequence (of pointed sets)
\[ [X,Y]_*\to[X,K(\ZZ,r)]_*\to[X,K(\ZZ_p,n+1)]_* \]
the set $[X,Y]_*$ is uncountable.

The other possibility is that $H_{r-1}(X)$ is a $\PP\setminus\setof{p}$-torsion group which is not finitely generated.
By \lemref{ext} the groups $\Ext(H_{r-1}(X),\ZZ)$ and hence $H^r(X;\ZZ)$ contain an uncountable
$\PP\setminus\setof{p}$-local subgroup $\Gamma$. By \lemref{no_cancellation_2} every morphism from $\Gamma$ into a bounded
$p$-group is trivial, and, proceeding as above, it follows that $\Gamma\subset k_\#^{-1}(0)$.

Both cases contradict the assumption that $H_{r-1}(X)$ is not finitely generated.\blackqed\qed
\end{example}

%%%%%%%%%%%%%%%%%%%%%%%%%%%%%%%%%%%%%%%%%%%%%%%%%%%%%%%%%%%%%%%%%%%%%%%%%%%%%%%%%%%%%%%%%%%%%
%%%%%%%%%%%%%%%%%%%%%%%%%%%%%%%%%%%%%%%%%%%%%%%%%%%%%%%%%%%%%%%%%%%%%%%%%%%%%%%%%%%%%%%%%%%%%

\appendix
\chapter{Pullbacks and lifting functions}\label{lifting_functions}
\setcounter{thm}{0}

\begin{lem}\label{pullback_covariant}
Let $A$ be the pullback of $B\xrightarrow{\beta} D\xleftarrow{\gamma} C$, and let $X$ be any space. Then $A^X$ is
the pullback of $B^X\xrightarrow{\beta_\#} D^X\xleftarrow{\gamma_\#} C^X$. If $\beta\colon(B,b_0)\to(D,d_0)$ and
$\gamma\colon(C,c_0)\to(D,d_0)$ are maps, then $(A,(b_0,c_0))^{(X,x_0)}$ is the pullback of 
$(B,b_0)^{(X,x_0)}\xrightarrow{\beta_\#}(D,d_0)^{(X,x_0)}\xleftarrow{\gamma_\#}(C,c_0)^{(X,x_0)}$.
\end{lem}
\begin{proof}
The natural map $F\colon(B\times C)^X\to B^X\times C^X$ is a homeomorphism (see Maunder \cite{maunder}, Theorem 6.2.34),
and it is trivial to check that $F((B\sqcap C)^X)=B^X\sqcap C^X$.
\end{proof}

\begin{defn}
Let $p\colon E\to B$ be a map and let $\eps_0\colon B^I\to B$ be evaluation at $0$.
Let $\Bar E$ denote the pullback of $E\xrightarrow{p}B\xleftarrow{\eps_0}B^I$. A lifting function for $p$
is a map $\lambda\colon\Bar E\to E^I$ that makes the following diagram commute.
\begin{equation*}\begin{diagram}
\node{\Bar E} \arrow{se,t}{\lambda} \arrow[2]{e} \arrow[2]{s} \node[2]{B^I} \arrow[2]{s,r}{\eps_0} \\
\node[2]{E^I} \arrow{ne,b}{p_\#} \arrow{sw,t}{\eps_0} \\
\node{E} \arrow[2]{e,b}{p} \node[2]{B}
\end{diagram}\end{equation*}
It is well known that $p$ has a lifting function if and only if it is a fibration. (See Fadell \cite{fadell2}.)
\end{defn}

The author was unable to find a proof of the following result which
should be folklore. (However, for the compactly generated refinement of the compact open
topology the analogous result {\it is} well known, see Fritsch and Piccinini \cite{fp}.)

\begin{prop}\label{mapping_space_covariant}
Let $p\colon E\to B$ be a fibration and let $X$ be a compactly generated space.
Then $p_\#\colon E^X\to B^X$ is a fibration.

If $p(e_0)=b_0$ where $e_0$ and $b_0$ are nondegenerate base points of their respective spaces,
then $p_\#\colon(E,e_0)^{(X,x_0)}\to(B,b_0)^{(X,x_0)}$ is a fibration, for any choice of $x_0$.
\end{prop}

\begin{proof}
Let $\lambda\colon E\sqcap B^I\to E^I$ be a lifting function for $p$. Here $E\sqcap B^I$ is the pullback
of $E\xrightarrow{p}B\xleftarrow{\eps_0}B^I$. The left-hand side diagram below commutes.
\begin{equation*}\begin{diagram}
\node{E\sqcap B^I} \arrow{se,t}{\lambda} \arrow[2]{e} \arrow[2]{s} \node[2]{B^I} \arrow[2]{s,r}{\eps_0} \\
\node[2]{E^I} \arrow{ne,b}{p_\#} \arrow{sw,t}{\eps_0} \\
\node{E} \arrow[2]{e,b}{p} \node[2]{B}
\end{diagram}\hspace{5ex}\begin{diagram}
\node{(E\sqcap B^I)^X} \arrow{se,t,..}{\Lambda} \arrow[2]{e} \arrow[2]{s} \node[2]{(B^I)^X} \arrow[2]{s,r}{(\eps_0)_\#} \\
\node[2]{(E^I)^X} \arrow{ne,b}{p_{\#\#}} \arrow{sw,t}{(\eps_0)_\#} \\
\node{E^X} \arrow[2]{e,b}{p_\#} \node[2]{B}
\end{diagram} \end{equation*}
We claim that there exists a lifting function $\Lambda\colon E^X\sqcap(B^X)^I\to(E^X)^I$ for $p_\#$.

Since $X$ is compactly generated, so is the product $X\times I$, and therefore $(B^X)^I$ is homeomorphic
to $(B^I)^X$ and $(E^X)^I$ is homeomorphic to $(E^I)^X$ (see Dugundji \cite{dugundji}, Theorem XII.5.3).
By \lemref{pullback_covariant} the pullback
$E^X\sqcap(B^I)^X$ is homeomorphic to $(E\sqcap B^I)^X$, and the lifting function diagram for $p_\#$
transcribes into the right-hand side diagram above. Thus we may set \begin{equation*}\tag{\dag} \Lambda=\lambda_\#,
\end{equation*} and the first assertion follows.

If $p(e_0)=b_0$ and $e_0,b_0$ are nondegenerate, then $\Bar e=(e_0,\const_{b_0})$ is nondegenerate in $E\sqcap B^I$,
and consequently $\lambda$ may be so chosen that $\lambda(\Bar e)=\const_{e_0}$. The function $\Lambda$ as defined
in (\dag) is a lifting function for $p_\#\colon(E,e_0)^{(X,x_0)}\to(B,b_0)^{(X,x_0)}$, as is readily verified.
\end{proof}

Following the proof of Theorem 2.8.14 of Spanier \cite{spanier} we extract

\begin{lem}\label{fib1}
Let $p\colon E\to B$ be a fibration and let $f_0,f_1\colon X\to B$ be homotopic maps. Set $E_0=f_0^*E$, and $E_1=f_1^*E$.
The fibrations $E_0\to X$ and $E_1\to X$ are fibre homotopy equivalent.

Let $\Bar E$ denote the pullback of $E\to B\xleftarrow{\eps_0}B^I$ where $\eps_0$ denotes evaluation at $0$.
Let $\lambda\colon\Bar E\to E^I$ be a lifting function for $p$. Let $\Tilde h\colon X\to B^I$ denote the adjoint
of a homotopy between $f_0$ and $f_1$. Then the map
\begin{equation*}\tag{$\star$} E_0\to E_1,\,\,\,\,(x,e)\mapsto(x,[\eps_1\circ\lambda](e,\Tilde h(x))) \end{equation*}
is a fibre homotopy equivalence with the obvious inverse. \qed
\end{lem}

\begin{defn}
Let $p\colon E\to B$ be a fibration, and let $f\colon X\to B$ be a map. If $f$ is homotopic to $\const_{b_0}$ then by
\lemref{fib1} the pullback fibration $f^*E\to X$ is fibre homotopy equivalent to the trivial fibration $X\times F\to X$
where $F$ is the fibre of $p$ over $b_0$.

We will call the fibre homotopy equivalence $f^*E\to X\times F$ given by ($\star$) the `canonical
fibre homotopy equivalence', and the associated section $X\to f^*E$ the `canonical section'. Note
that the equivalence ($\star$) depends on the choice of homotopy.
\end{defn}

The following is well known. It is contained implicitly for example in Fadell \cite{fadell2}.

\begin{lem}\label{fib2}
Let $p\colon E\to B$ be a fibration with $p(e_0)=b_0$, and let $F$ denote the fibre of $p$ over $b_0$. Let $\Lambda$ be the
homotopy fibre of the inclusion $F\hookrightarrow E$. More precisely, $\Lambda$ is the pullback of $F\hookrightarrow
E\xleftarrow{\eps_1}PE$ where $\eps_1$ is evaluation at $1$. We may view $\Lambda=(E,e_0,F)^{(I,0,1)}\subset E^I$ and
$\Omega(B,b_0)\equiv\setof{e_0}\times\Omega(B,b_0)\subset\Bar E$. In this way composition with $p$ yields a map
$\Lambda\to\Omega(B,b_0)$, while the lifting function $\lambda\colon\Bar E\to E^I$ maps $\Omega(B,b_0)$ to $\Lambda$.

The maps $\Lambda\to\Omega(B,b_0)$ and $\Omega(B,b_0)\to\Lambda$ are homotopy inverses of each other. \qed
\end{lem}

The following is a particular case of lifting functions for restriction fibrations.

\begin{lem}\label{fib3}
Let $W$ be a Hausdorff space and $PW=\setof{\gamma\,\vert\,\gamma\colon I\to W,\,\,\gamma(0)=w_0}$ the path space.
The pullback $\overline{PW}$ of $PW\xrightarrow{\eps_1}W\xleftarrow{\eps_0}W^I$ may be identified with
$(W,w_0)^{(I\times 0\cup 1\times I,(0,0))}$, the space $PW^I$ may be identified with $(W,w_0)^{(I\times I,0\times I)}$,
and if $\rho\colon I\times I\to I\times 0\cup 1\times I$ is any retraction such that $\rho(0\times I)=\setof{(0,0)}$,
then \[ \mu\colon\overline{PW}\to PW^I,\,\,\,\mu(\phi)=\phi\circ\rho \] is a lifting function for $\eps_1\colon PW\to W$.
In particular, $\rho$ may be chosen so that $\rho(s,1)=(2s,0)$ for $s\lqs\frac{1}{2}$ and $\rho(s,1)=(1,2s-1)$ for
$s\gqs\frac{1}{2}$. \qed
\end{lem} 

Lemmas \ref{fib1}, \ref{fib2}, \ref{fib3} readily yield

\begin{prop}\label{fibre_contraction}
Let $p\colon E\to B$ be a fibration and let $\Bar E$ denote the pullback of $E\to B\xleftarrow{\eps_0}B^I$ 
where $\eps_0$ denotes evaluation at $0$. Let $\lambda\colon\Bar E\to E^I$ be a lifting function for $p$.
Let $F$ be the fibre of $p$ over $b_0\in B$.

Assume that $F$ contracts in $E$ to $e_0\in E$, and let $k\colon F\to E^I$ denote the adjoint of a contracting homotopy.
Further let $\eps_1\colon E^I\to E$ denote evaluation at $1$.

Then the map
\[ \Omega(B,b_0)\to F\times\Omega(E,e_0),\,\,\,\gamma\mapsto\big(\eps_1[\lambda(e_0,\gamma)],\lambda(e_0,\gamma)*
								k(\eps_1[\lambda(e_0,\gamma)])\big) \]
is a homotopy equivalence with inverse
\[ F\times\Omega(E,e_0)\to\Omega(B,b_0),\,\,\, (x,\omega)\mapsto p\circ[\omega*k^{-1}(x)]
	=(p\circ\omega)*(p\circ k^{-1}(x)).\qed \]
\end{prop}

%%%%%%%%%%%%%%%%%%%%%%%%%%%%%%%%%%%%%%%%%%%%%%%%%%%%%%%%%%%%%%%%%%%%%%%%%%%%%%%%%%%%%%%%%%%%%%%%%%%
%%%%%%%%%%%%%%%%%%%%%%%%%%%%%%%%%%%%%%%%%%%%%%%%%%%%%%%%%%%%%%%%%%%%%%%%%%%%%%%%%%%%%%%%%%%%%%%%%%%
%%%												%%%
%%%			Quasitopological groups 						%%%
%%%												%%%
%%%%%%%%%%%%%%%%%%%%%%%%%%%%%%%%%%%%%%%%%%%%%%%%%%%%%%%%%%%%%%%%%%%%%%%%%%%%%%%%%%%%%%%%%%%%%%%%%%%
%%%%%%%%%%%%%%%%%%%%%%%%%%%%%%%%%%%%%%%%%%%%%%%%%%%%%%%%%%%%%%%%%%%%%%%%%%%%%%%%%%%%%%%%%%%%%%%%%%%

\chapter{Quasitopological groups}\label{quasi_groups}
\setcounter{thm}{0}

If $G$ is an uncountable abelian group then the geometric realization $Y$ of the simplicial Eilenberg-MacLane
group $K(G,n)$ (see Milnor \cite{milnor3}) is not a topological group in its CW topology with the cartesian product.
More precisely, there exists a function $\mu\colon Y\times Y\to Y$ which is the multiplication of an abelian
group but is continuous only on compact subsets of $Y\times Y$.

In general this is still perfectly acceptable since the category of compactly generated spaces is usually suitable
enough for homotopy theory. However, if we want to consider function spaces of maps into $Y$ equipped with the compact
open topology to have a good grip on the open sets, we encounter problems which stem from the fact that even for compact
spaces $K$ the function space $Y^K$ need not be compactly generated.

\begin{defn}
Let $Y$ be a (Hausdorff) topological space with distinguished element $y_0$ and a function $\mu\colon Y\times Y\to Y$
such that $(Y,\mu)$ is a monoid with unit $y_0$, and for every compact subset $C\subset Y$ the restrictions
\begin{equation*}\tag{\dag} \mu\vert_{Y\times C}\colon Y\times C\to Y,\,\,\,\mu\vert_{C\times Y}\colon C\times Y\to Y,
\end{equation*} are continuous. Then we say that $Y$ is a {\it quasitopological monoid}.

If, in addition, there exists a continuous function $\inv\colon Y\to Y$ such that $(Y,\mu,\inv)$ is a group
with unit $y_0$ then we say that $Y$ is a {\it quasitopological group}.

A morphism of quasitopological groups is a continuous homomorphism.
\end{defn}

Evidently if $Y$ is compactly generated then $\mu$ is continuous on compact subsets of $Y\times Y$ if and only if
the restrictions (\dag) are continuous for compact $C$.

Therefore by \cite{milnor3} for every abelian group $G$ there exists a CW complex $Y$ of type $K(G,n)$ which
is an abelian quasitopological group. (And if $G$ is countable, `quasi' may be omitted.)

For another example, by Dold and Thom \cite{dold_thom} for every connected simplicial complex $X$, the infinite symmetric
product $SP(X)$ is an abelian quasitopological monoid.

Clearly,
\begin{lem}
If $Y$ is an abelian quasitopological monoid or a quasitopological group then continuity of (\dag) for all compact $C$
is equivalent to continuity of $\mu\vert_{Y\times C}\colon Y\times C\to Y$ for all compact $C$.\qed
\end{lem}

\begin{prop}\label{quasi_function}
Let $(Y,\mu,\inv,y_0)$ be a quasitopological group and $K$ a compactum. Given $f,g\in Y^K$, the function
$K\to Y$, defined by $x\mapsto\mu(f(x),g(x))$, is continuous. Define
\[ M\colon Y^K\times Y^K\to Y^K,\,\,\,M(f,g)(x)=\mu(f(x),g(x));\,\,\,\,
	I\colon Y^K\to Y^K,\,\,\,I(f)=\inv\circ f. \]
Then $(Y^K,M,I,\const_{y_0})$ is a quasitopological group. If $A$ is any subset of $K$ then
$(Y,y_0)^{(K,A)}$ is a closed quasitopological subgroup.
\end{prop}
\begin{proof}
Let $\Gamma$ be a compact subset of $Y^K$. The evaluation map $e\colon Y^K\times K\to Y$, $e(f,x)=f(x)$,
is continuous, and hence $L=e(\Gamma\times K)=\setof{f(x)\,\vert\,f\in\Gamma,\,x\in K}$ is compact. Note that
$\Gamma\subset L^K$. Since $\mu\vert_{Y\times L}$ is continuous, so is the composite
\[ Y^K\times\Gamma\hookrightarrow Y^K\times L^K\xrightarrow{\approx}(Y\times L)^K\xrightarrow{(\mu\vert_{Y\times L})_\#}
	Y^K. \]
The other assertions are trivial.
\end{proof}

Quasitopological groups have enough structure to avoid base point problems;
the proof carries over from the standard case mutatis mutandis.
\begin{prop}\label{quasi_base_point}
\begin{enumerate}
\item	Let $G$ be a quasitopological group with unit $e$ and let $Z$ be a space with $z_0\in Z$.
	Let $f\colon Z\to G$ be a map. Then $f$ can be homotoped to a map that sends $z_0$ to
	any point in the path component of $f(z_0)$.

	If $f,g\colon Z\to G$ are freely homotopic, then they are homotopic with a homotopy preserving $z_0$.
\item	If $G_1$ and $G_2$ are freely homotopy equivalent quasitopological groups then for any choice
	of base points $g_1\in G_1$, $g_2\in G_2$, the pairs $(G_1,g_1)$ and $(G_2,g_2)$ are homotopy equivalent. \qed
\end{enumerate}
\end{prop}

The proof of the following is straightforward.
\begin{lem}\label{quasi_pullback}
\begin{enumerate}
\item	The cartesian product of finitely many quasitopological groups is a quasitopological group.
\item	If $B\xrightarrow{\varphi}A\xleftarrow{\eps}P$ is a diagram of quasitopological groups and
	morphisms, then the pullback space $E$ admits a natural quasitopological group structure so that
	the pullback square is a diagram of quasitopological groups and morphisms.\qed
\end{enumerate}
\end{lem}

\propref{quasi_function} and \lemref{quasi_pullback} yield
\begin{cor}
Let $\varphi\colon B\to A$ be a morphism of quasitopological groups. Then the canonical factorization
$B\xrightarrow{\simeq} B'\xrightarrow{\varphi'}A$ of $\varphi$ into the composite of a homotopy equivalence
followed by a fibration is a sequence of morphisms of quasitopological groups.\qed
\end{cor}

\begin{defn}
Let $G$ be a quasitopological group and $E$ a space. Then $E$ is a left $G$-space if
there exists a left action $\Phi\colon G\times E\to E$ such that the restrictions $\Phi\vert_{C\times E}$
and $\Phi\vert_{G\times D}$ are continuous for all compact subsets $C$ of $G$ and $D$ of $E$.
\end{defn}

\begin{prop}\label{compactly_open}
Let $p\colon E\to B$ be a surjective morphism of quasitopological groups and let $G=\ker p$. Then $G$
is a quasitopological group and $E$ is a free left $G$-space with the following property.

For any space $Z$ and continuous maps $f_1,f_2\colon Z\to E$ such that $p\circ f_1=p\circ f_2$ there
exists a function $f\colon Z\to G$ which is continuous on the compact subsets of $Z$ such that
$f(z)\cdot f_1(z)=f_2(z)$, for all $z\in Z$.\qed
\end{prop}

\begin{defn}
Let $G$ be a quasitopological group. A free left action of $G$ on $E$ with the property
as in \propref{compactly_open} is {\it compactly open}. (Compare the definition preceding
\lemref{lem_g_action_1} on page \pageref{lem_g_action_1}.)
\end{defn}

\begin{cor}\label{quasi_j_cohen}
Let $\setof{G_i}$ be a sequence of quasitopological groups.
Let $\setof{p_i}\colon\setof{E_i}\to\setof{B_i}$ be a level preserving morphism of inverse sequences
where for each $i$, the space $E_i$ is a free left $G_i$-space and $p_i\colon E_i\to B_i$ is a compactly open
fibration. Assume that the maps $E_i\to E_{i-1}$ are equivariant given by morphisms $\omega_i\colon G_i\to G_{i-1}$.

Let $p_\infty\colon E_\infty\to B_\infty$ denote the induced inverse limit map.
If the $\omega_i$ are fibrations then $p_\infty$ is a fibration for the class of compactly generated spaces. \qed
\end{cor}

\backmatter

%%%%%%%%%%%%%%%%%%%%%%%%%%%%%%%%%%%%%%%%%%%%%%%%%%%%%%%%%%%%%%%%%%%%%%%%%
%									%
%				Bibliography				%
%									%
%%%%%%%%%%%%%%%%%%%%%%%%%%%%%%%%%%%%%%%%%%%%%%%%%%%%%%%%%%%%%%%%%%%%%%%%%

\bibliographystyle{amsalpha}

\end{document}